%% file: main.tex
\documentclass[12pt,twoside]{article}
\usepackage[utf8]{inputenc}
\usepackage[T1]{fontenc}
\pdfoutput=1
\usepackage[english]{babel}
\usepackage{amsmath, amssymb, amsthm}            
\usepackage{amstext, amsfonts, a4}
\usepackage{graphicx}
\usepackage{hyperref}
\usepackage{color}
\usepackage{stmaryrd}
\usepackage{enumitem}
\usepackage{tikz}
\usepackage{mathrsfs}
\usepackage{pgfplots}
\usepackage{pifont}
\usepackage{comment}
\pgfplotsset{compat=1.15}
\usetikzlibrary{arrows}

\usepackage[normalem]{ulem}

\theoremstyle{plain}
	\newtheorem{Theo}{Theorem}[section] 
	\newtheorem{Prop}[Theo]{Proposition}
	\newtheorem{Lem}[Theo]{Lemma}
	\newtheorem{Cor}[Theo]{Corollary}
	\newtheorem{Conj}[Theo]{Conjecture}

\theoremstyle{definition}
	\newtheorem{Def}[Theo]{Definition}
        
	\newtheorem{Nota}[Theo]{Notation}

\theoremstyle{remark}
	\newtheorem{Rema}[Theo]{Remark}

\def\NN{{\mathbb N}}    
\def\ZZ{{\mathbb Z}}     
\def\RR{{\mathbb R}}    
\def\QQ{{\mathbb Q}}    
\def\CC{{\mathbb C}}    
\def\HH{{\mathbb H}^2}    
\def\AA{{\mathbb P}}     
\def\DD{{\mathbb D}^2} 
\def\DH{{\partial {\mathbb H}^2}}
\def\Int{{\mathrm{Int}}}
\newcommand{\Vol}{\mbox{Vol}}
\newcommand{\KVol}{\mbox{KVol}}
\def\autredir{{\cfrac{1+2 \cos(\pi/n)\cos(\pi/m)}{2\sin(\pi/n) \cos(\pi/m)}}}
\def\Tmn{{\mathcal{T}_{m,n}}}
\def\GG{{\mathfrak G}}

\def\KK{{\mathcal{K}}}
\def\Xstar{{(i,j) \xrightarrow[]{\bigstar} (i',j')}}

\newcommand{\polydec}{\textit{polygonal decomposition}}
\newcommand{\coslo}{consistent slope }
\newcommand{\coslos}{consistent slopes }
\newcommand{\HTU}{\eqref{H31} }
\newcommand{\HTD}{\eqref{H32} }

\title{Algebraic intersections on Bouw-M\"oller surfaces and more general convex polygons
}
\author{Julien Boulanger and Irene Pasquinelli}
\date{May 2026}

\begin{document}

\maketitle

\begin{abstract}
This paper focuses on intersection of closed curves on translation surfaces. 
Namely, we investigate the question of determining the intersection of two closed curves of a given length on such surfaces.
This question has been investigated in \cite{MM}, \cite{CKM}, \cite{CKMcras}, \cite{BLM22} and \cite{Bou23} and this paper complements the work of \cite{BLM22} done for double regular polygons, and extends the results to a large family of surfaces which includes in particular Bouw-M\"oller surfaces. Namely, we give an estimate for KVol on surfaces based on geometric constraints (angles and indentifications of sides). This estimate is sharp in the case of Bouw-M\"oller surfaces with a unique singularity, and it allows to compute KVol on the $SL_2(\RR)-$orbit of such surfaces.
\end{abstract}

\section{Introduction}

\input{NewIntro}

\section{Preliminaries on Bouw-M\"oller surfaces}\label{sec:preliminaries_BM}
In this section we will define Bouw-M\"oller surfaces, which are examples of translation surfaces and the main setting of our discussion. 
We will review definitions and well-known facts about translation surfaces in general and Bouw-M\"oller surfaces in particular. In Section \ref{subsec:singular_intersections_on_BM_surfaces} we study intersections of saddle connections in Bouw-M\"oller surfaces with a unique singularity. 
More details on translation surfaces can be found in \cite{ survey_Wright, Masur, survey_Massart}. See also \cite{Hooper} for more details on Bouw-M\"oller surfaces.

\subsection{Translation surfaces}
A translation surface is a topological surface $X$ with an atlas of charts on the surface minus a finite set $\Sigma$ of singularities such that transition functions are translations. 
These surfaces can also be described as the surfaces obtained by gluing parallel opposite sides of a collection of euclidean polygons by translations, and one further requires the translations to map a normal vector of the side pointing outwards with respect to the polygon to a normal vector of the identified side pointing inwards with respect to the other polygon, so that the resulting surface is orientable. 
As such, translation surfaces are instances of flat surfaces with conical singularities. 
Furthermore, because of the translation property, the cone angle around each singularity is a multiple of $2\pi$, so it is of the form $2(k+1)\pi$, where $k$ is referred to as the order of the singularity. 
The discrete version of Gauss-Bonnet formula guarantees that if a surface $S$ has $r$ singularities of order $k_i$, $1 \leq i \leq r$ and genus $g$, then $\sum 2 \pi (k_i+1)= 2\pi(2g-2)$. 
The moduli space of all translation surfaces of genus $g$ naturally carries a stratified structure, the stratification being given by the order of the singularities.
One usually denotes by $\mathcal{H}(k_1, \cdots, k_r)$ the stratum of translation surfaces of genus $g = \frac{1}{2}\left(\sum_i k_i+1\right) +1$ with $r$ singularities of respective order $k_1, \dots, k_r$.

Since the euclidean metric on the polygons defines a flat metric on the surface, where the vertices of the polygons are identified to singularities (although note that some may correspond to singularities of order zero, that is, regular points), geodesics on translation surfaces are piecewise straight lines, which are straight lines inside the polygons.
Furthermore, geodesics can only change direction at a singularity. 
A \emph{saddle connection} is a geodesic line between singularities, that is a line going from a vertex of a polygon to a vertex of a polygon, containing no singularities in its interior. 
In fact, a closed geodesic on a translation surface is always  homologous to a union of saddle connections because any non-singular closed geodesic comes with a cylinder of homologous saddle geodesics, bounded by saddle connections.\newline

The moduli space of translation surfaces carries a natural $SL_2(\RR)-$action and the stabiliser of a given surface $X$ is called the \emph{Veech group} of $X$, denoted $SL(X)$. Veech \cite{Veech} showed that the Veech group of a translation surface is always a Fuchsian group, that is, a discrete subgroup of $SL_2(\RR)$.
In particular, the $SL_2(\RR)-$orbit of a translation surface in the moduli space can be identified with $SL(X) \backslash SL_2(\RR)$. Closed $SL_2(\RR)-$orbits under this action are called \emph{Teichm\"uller curves}, and correspond to the case where $SL(X)$ is a lattice. We often quotient on the right by the action of the rotations, which gives an identification of the Teichm\"uller curve with $SL(X) \backslash SL_2(\RR)/SO_2(\RR) = SL(X) \backslash \HH$. 
Although every Veech group is a Fuchsian group, it is not true that every Fuchsian group can be realised as a Veech group, see for example \cite{HS_survey}. 
The question of determining the Fuschian groups that can be realised as Veech groups of translation surfaces is a difficult problem, and Bouw and M\"oller \cite{BM10} classified all the translation surfaces having a triangular Veech group. 
These surfaces are the Bouw-M\"oller surfaces we mentioned previously, whose polygonal description has been given by Hooper \cite{Hooper}.

\subsection{Bouw-M\"oller surfaces}
Given $m,n \geq 2$ with $mn \geq 6$, the Bouw-M\"oller surface $S_{m,n}$ is a translation surface made with $m$ semi-regular polygons, each having a symmetry of order $n$.
In this section we will explain how to describe the surface $S_{m,n}$ using polygons, following \cite{Hooper} and we will mention some useful facts that will be needed later on.

\paragraph{Definition of Bouw-M\"oller surfaces.} 
Given $m,n \geq 2$ with $mn \geq 6$, the Bouw-M\"oller surface $S_{m,n}$ is a translation surface obtained by identifying the sides of a collection of $m$ semi-regular polygons $P(0), \dots, P(m-1)$.
More precisely, for $n \geq 2$ and $a,b \in \RR^+$, let $P_n(a,b)$ be the semi regular polygon having all angles equal to $\frac{(n-1)\pi}{n}$ and sides of length alternating between $a$ and $b$, as in Figure \ref{fig:example_P_n(a,b)}. 
The edges $(v_j)_{j = 0, \dots, 2n-1 }$ of $P_n(a,b)$ are given by the vectors:
\begin{equation*}
    v_j = \left\{ 
\begin{array}{ll}
a\left(\cos\left(\frac{j\pi}{n}\right), \sin\left(\frac{j\pi}{n}\right)\right) \text{ if $j$ is even} \\
b\left(\cos\left(\frac{j\pi}{n}\right), \sin\left(\frac{j\pi}{n}\right)\right) \text{ if $j$ is odd}
\end{array}
    \right.
\end{equation*}
In the case where $a=0$ (resp. $b=0$), $P_n(a,b)$ is a regular $n$-gon of side length $b$ (resp. $a$).

\begin{figure}
\center
\begin{tikzpicture}[line cap=round,line join=round,>=triangle 45,x=0.6cm,y=0.6cm]
\clip(-2,-1) rectangle (5,6.8);
\draw [line width=1pt] (-1.8772758664047853,1.9446152422706642)-- (-0.754551732809569,0);
\draw [line width=1pt] (3.7545517328095688,0)-- (4.8772758664047835,1.9446152422706624);
\draw [line width=1pt] (0.3772758664047853,5.849613391789288)-- (2.6227241335952143,5.84961339178929);
\draw [line width=1pt] (-1.8772758664047853,1.9446152422706642)-- (0.3772758664047853,5.849613391789288);
\draw [line width=1pt] (2.6227241335952143,5.84961339178929)-- (4.8772758664047835,1.9446152422706624);
\draw [line width=1pt] (3.7545517328095688,0)-- (-0.754551732809569,0);
\draw (4.2,1.2) node[anchor=north west] {$b$};
\draw (4.1,4.5) node[anchor=north west] {$a$};
\draw (1.1,6.8) node[anchor=north west] {$b$};
\draw (-1.5,4.5) node[anchor=north west] {$a$};
\draw (1.1,0) node[anchor=north west] {$a$};
\draw (-2,1.42) node[anchor=north west] {$b$};
\end{tikzpicture}
\caption{The polygon $P_3(a,b)$}
\label{fig:example_P_n(a,b)}
\end{figure}

Now, if $n$ is odd, then define \[
P\left(i\right)=P_n\left(\sin\left(\frac{\left(i+1\right)\pi}{m}\right),\sin\left(\frac{i\pi}{m}\right)\right).
\]
If $n$ is even, define 
\begin{align*}
    P\left(i\right)&=
    \begin{cases}
        P_n\left(\sin\left(\frac{\left(i+1\right)\pi}{m}\right),\sin\left(\frac{i\pi}{m}\right)\right) & \text{if $i$ is even,} \\
        P_n\left(\sin\left(\frac{i\pi}{m}\right),\sin\left(\frac{\left(i+1\right)\pi}{m}\right)\right) & \text{if $i$ is odd,}
    \end{cases}
\end{align*}

Finally, the Bouw-M\"oller surface $S_{m,n}$ is obtained by identifying sides of $P(i)$, $i = 1, \dots m-2$ to parallel sides of either $P(i-1)$ or $P(i+1)$ respecting the rules for translation surfaces.

Note that $P(0)$ and $P(m-1)$ are regular $n$-gons and hence for them, this means that the sides are glued to the parallel sides of $P(1)$ and $P(m-2)$ respectively. 
The examples of $S_{3,4}$ and $S_{4,3}$ are represented in Figure \ref{fig:examples_S34_S43}.

\begin{Rema}\label{rk:alternation}
Note that given two adjacent sides of the polygon $P(i)$, $0 < i < m-1$, one is paired to a side in $P(i-1)$ while the other is paired to a side in $P(i+1)$.    
\end{Rema}

\begin{figure}[h]
\centering
\definecolor{ccqqqq}{rgb}{0.8,0,0}
\definecolor{qqqqff}{rgb}{0,0,1}
\definecolor{ffvvqq}{rgb}{1,0.3333333333333333,0}
\definecolor{qqwuqq}{rgb}{0,0.39215686274509803,0}
\definecolor{ffdxqq}{rgb}{1,0.8431372549019608,0}
\definecolor{wwwwww}{rgb}{0.4,0.4,0.4}
\begin{tikzpicture}[line cap=round,line join=round,>=triangle 45,x=1.3cm,y=1.3cm]
\clip(-1.5,-1.5) rectangle (3.5,3);
\draw [line width=1pt,color=wwwwww] (0,1)-- (-1,1);
\draw (-0.5,1) node[above] {$C$};  
\draw [line width=1pt,color=ffdxqq] (-1,1)-- (-1,0);
\draw (-1,0.5) node[left] {$B$};  
\draw [line width=1pt,color=qqwuqq] (-1,0)-- (0,0);
\draw (-0.5,0) node[below] {$A$};  
\draw [line width=1pt,dash pattern=on 3pt off 3pt] (0,1)-- (0,0);
\draw [line width=1pt,color=ffvvqq] (0,0)-- (0.7071067811865475,-0.7071067811865475);
\draw (0.35,-0.35) node[anchor=north east] {$D$};  
\draw [line width=1pt,color=wwwwww] (0.7071067811865475,-0.7071067811865475)-- (1.7071067811865472,-0.7071067811865475);
\draw (1.2,-0.7) node[below] {$C$};  
\draw [line width=1pt,color=qqqqff] (1.7071067811865472,-0.7071067811865475)-- (2.414213562373095,0);
\draw (2.05,-0.35) node[anchor=north west] {$E$};  
\draw [line width=1pt,color=ffdxqq] (2.414213562373095,0)-- (2.414213562373095,1);
\draw (2.41,0.5) node[right] {$B$};  
\draw [line width=1pt,dash pattern=on 3pt off 3pt] (2.414213562373095,1)-- (1.7071067811865477,1.7071067811865475);
\draw [line width=1pt,color=qqwuqq] (1.7071067811865477,1.7071067811865475)-- (0.7071067811865478,1.7071067811865477);
\draw (1.2,1.7) node[above] {$A$};  
\draw [line width=1pt,color=ccqqqq] (0.7071067811865478,1.7071067811865477)-- (0,1);
\draw (0.35,1.35) node[anchor=south east] {$F$};  
\draw [line width=1pt,color=ccqqqq] (2.414213562373095,1)-- (3.121320343559643,1.7071067811865466);
\draw (2.76,1.35) node[anchor=north west] {$F$};  
\draw [line width=1pt,color=ffvvqq] (3.121320343559643,1.7071067811865466)-- (2.414213562373096,2.414213562373095);
\draw (2.76,2.05) node[anchor=south west] {$D$};  
\draw [line width=1pt,color=qqqqff] (2.414213562373096,2.414213562373095)-- (1.7071067811865477,1.7071067811865475);
\draw (2.05,2.05) node[anchor=south east] {$E$};  
\draw (-0.96,0.8) node[anchor=north west] {$P(0)$};
\draw (0.7,0.9) node[anchor=north west] {$P(1)$};
\draw (2,1.95) node[anchor=north west] {$P(2)$};
\end{tikzpicture}
\begin{tikzpicture}[line cap=round,line join=round,>=triangle 45,x=1.3cm,y=1.3cm]
\clip(2.5,-1.5) rectangle (8.5,3);
\draw [line width=1pt,color=ffdxqq] (3,-0.8)-- (4,-0.8);
\draw (3.5,-0.8) node[below] {$A$};
\draw [line width=1pt,dash pattern=on 3pt off 3pt] (4,-0.8)-- (3.5,0.06602540378443866);
\draw [line width=1pt,color=wwwwww] (3.5,0.06602540378443866)-- (3,-0.8);
\draw (3.25,-0.37) node[anchor= east] {$B$};
\draw [line width=1pt,color=ffvvqq] (3.5,0.06602540378443866)-- (4.240044641670906,1.3478203230275505);
\draw (3.85,0.64) node[anchor= east] {$C$};
\draw [line width=1pt,color=ffdxqq] (4.240044641670906,1.3478203230275505)-- (5.240044641670906,1.3478203230275505);
\draw (4.74,1.34) node[above] {$A$};
\draw [line width=1pt,dash pattern=on 3pt off 3pt] (5.240044641670906,1.3478203230275505)-- (5.980089283341812,0.06602540378443866);
\draw [line width=1pt,color=wwwwww] (5.980089283341812,0.06602540378443866)-- (5.480089283341811,-0.8);
\draw (5.68,-0.37) node[anchor= west] {$B$};
\draw [line width=1pt,color=qqwuqq] (5.480089283341811,-0.8)-- (4,-0.8);
\draw (4.74,-0.8) node[below] {$D$};
\draw [line width=1pt,color=ccqqqq] (5.240044641670906,1.3478203230275505)-- (5.740044641670906,2.2138457268119893);
\draw (5.49,1.8) node[anchor=east] {$E$};
\draw [line width=1pt,color=qqwuqq] (5.740044641670906,2.2138457268119893)-- (7.220133925012717,2.2138457268119893);
\draw (6.48,2.21) node[above] {$D$};
\draw [line width=1pt,dash pattern=on 3pt off 3pt] (7.220133925012717,2.2138457268119893)-- (7.720133925012716,1.3478203230275505);
\draw [line width=1pt,color=qqqqff] (7.220133925012717,2.2138457268119893)-- (8.220133925012718,2.2138457268119893);
\draw (7.72,2.21) node[above] {$F$};
\draw [line width=1pt,color=ccqqqq] (7.720133925012716,1.3478203230275505)-- (8.220133925012718,2.2138457268119893);
\draw (7.97,1.8) node[anchor=west] {$E$};
\draw [line width=1pt,color=ffvvqq] (7.720133925012716,1.3478203230275505)-- (6.980089283341812,0.06602540378443866);
\draw (7.35,0.7) node[anchor= west] {$C$};
\draw [line width=1pt,color=qqqqff] (5.980089283341812,0.06602540378443866)-- (6.980089283341812,0.06602540378443866);
\draw (6.48,0.06) node[below] {$F$};
\begin{scriptsize}
\draw (3.17,-0.3) node[anchor=north west] {$P(0)$};
\draw (7.4,2.1) node[anchor=north west] {$P(3)$};
\end{scriptsize}
\draw (6,1.5) node[anchor=north west] {$P(2)$};
\draw (4.2,0.6) node[anchor=north west] {$P(1)$};  
\end{tikzpicture}
\caption{The surfaces $S_{3,4}$ and $S_{4,3}$.}
\label{fig:examples_S34_S43}
\end{figure}

\begin{Rema}\label{rk:v0horizontal}
With this construction, the side $v_0$ of $P(0)$ is always horizontal. This convention will be used later.
\end{Rema}

\paragraph{Properties of Bouw-M\"oller surfaces.}
Given $m,n \geq 2$, with $mn \geq 6$, the surface $S_{m,n}$ is a translation surface of genus $(mn-m-n-\gamma)/2 +1$ and $\gamma=\gcd(m,n)$ singularities. More precisely, one has
\begin{Prop}[Proposition 24 of \cite{Hooper}]
Let $\gamma=\gcd(m,n)$. There are $\gamma$ equivalence classes of vertices in the decomposition into polygons. 
In particular, $S_{m,n}$ will have $\gamma$ cone singularities.
Each of these singularities has cone angle $2\pi(mn-m-n)/ \gamma$.
\end{Prop}
In particular $S_{m,n}$ has only one singularity if and only if $m$ and $n$ are coprime. Notice that in this case saddle connections are closed curves, and hence we have

\begin{Prop}\label{prop:systoles}
If $m$ and $n$ are coprime, then the systoles of $S_{m,n}$ are exactly the sides of $P(0)$ and $P(m-1)$.
\end{Prop}
\begin{proof}
Any closed curve on $S_{m,n}$ has length at least the length of the shortest side of the polygons defining $S_{m,n}$. Since the sides have length $\sin \frac{k\pi}{m}$, the shortest sides are for the case $k=1$ or $k=m-1$, which corresponds to the length of the sides of $P(0)$ and $P(m-1)$.
\end{proof}

\subsection{Veech group and Fundamental domain}\label{sec:Teichmuller_disk_BM}
We now recall results about the Veech group of $S_{m,n}$ and give a model for the Teichm\"uller curve associated to $S_{m,n}$. 

\begin{Theo}[\cite{BM10}, \cite{Hooper}]
The Veech group of $S_{m,n}$ is the triangle group $\Delta^+(m,n,\infty)$. Further, $S_{m,n}$ and $S_{n,m}$ are affinely equivalent.
\end{Theo}

Then, the Teichm\"uller curve of $S_{m,n}$ can be identified with two copies of a hyperbolic triangle of angles $(0, \frac{\pi}{n}, \frac{\pi}{m})$. We parametrize the $SL_2(\RR)-$orbit of $S_{m,n}$ as follows:

\begin{Def}[Parametrization of the $SL_2(\RR)-$orbit]\label{def:teichmuller}
Given $M = \begin{pmatrix} a & b \\ c & d \end{pmatrix} \in SL_2(\RR)$, we send the surface $M \cdot S_{m,n}$ to an element of $\HH$ using the identification:
\[ \Psi : M \cdot S_{m,n}
\mapsto \frac{di+b}{ci+a} \in \HH. \]
\end{Def}

Note that with this identification, the surface $S_{m,n}$ is identified with $i \in \HH$.
We can describe precisely the shape of a fundamental domain $\Tmn$ for the $SL_2(\RR)-$orbit of $S_{m,n}$ using the following result:

\begin{Theo}[\cite{Hooper}, see also \cite{DPU}]
The horizontal cylinders of $S_{m,n}$ all have the same moduli, namely:
\[ s := 2 \cot\left(\frac{\pi}{n}\right) + \frac{2 \cos\left(\frac{\pi}{m}\right)}{\sin\left(\frac{\pi}{n}\right)} = \frac{2(\cos\left(\frac{\pi}{n}\right) +\cos\left(\frac{\pi}{m}\right))}{\sin\left(\frac{\pi}{n}\right)}.\]
\end{Theo}

\begin{Rema}

One could easily see from the geometry of the polygons that horizontal cylinders are contained in the union of exactly two polygons.
\end{Rema}

In particular, the element $\begin{pmatrix} 1 & s \\ 0 & 1 \end{pmatrix}$ belongs to the Veech group of $S_{m,n}$ and a fundamental domain consists of the union of two $(\infty, m, n)$-triangles as in Figure \ref{fig:fond_domain_0}.
We choose the fundamental domain so that it is bounded by two vertical lines, one of them passing through $i$, which is the point of order $m$. Then $S_{m,n}$ corresponds to the point of coordinates $i$ and is stabilised by the action of the rotation of angle $\frac{2\pi}{n}$ while the other corners (which are identified) correspond to the surface $S_{n,m}$, which lies in the same Teichm\"uller curve (up to normalizing the area) and is stabilised by the action of the rotation of angle $\frac{2\pi}{m}$. 

\begin{figure}[h]
\center
\definecolor{qqqqff}{rgb}{0,0,1}
\definecolor{qqwuqq}{rgb}{0,0.39215686274509803,0}
\definecolor{ccqqqq}{rgb}{0.8,0,0}
\definecolor{uuuuuu}{rgb}{0.26666666666666666,0.26666666666666666,0.26666666666666666}
\begin{tikzpicture}[line cap=round,line join=round,>=triangle 45,x=1cm,y=1cm,scale=1.2]
\clip(-6,-1) rectangle (5.5,6);
\fill[line width=0.5pt,color=uuuuuu,fill=uuuuuu,fill opacity=0.10000000149011612] (-4.528276847984907,0) -- (4.528276847984907,0) -- (4.528276847984907,9) -- (-4.528276847984907,9) -- cycle;
\draw [shift={(2.414213562373095,0)},line width=1pt,color=white,fill=white,fill opacity=1]  (0,0) --  plot[domain=0:3.141592653589793,variable=\t]({1*2.613125929752753*cos(\t r)+0*2.613125929752753*sin(\t r)},{0*2.613125929752753*cos(\t r)+1*2.613125929752753*sin(\t r)}) -- cycle ;
\draw [shift={(-2.414213562373095,0)},line width=1pt,color=white,fill=white,fill opacity=1]  (0,0) --  plot[domain=0:3.141592653589793,variable=\t]({1*2.613125929752753*cos(\t r)+0*2.613125929752753*sin(\t r)},{0*2.613125929752753*cos(\t r)+1*2.613125929752753*sin(\t r)}) -- cycle ;
\draw [shift={(4.528276847984907,1.5359568838917255)},line width=1pt,color=qqqqff,fill=qqqqff,fill opacity=0.1] (0,0) -- (90:0.4) arc (90:126:0.4) -- cycle;
\draw [shift={(0,1)},line width=1pt,color=qqwuqq,fill=qqwuqq,fill opacity=0.10000000149011612] (0,0) -- (67.5:0.4) arc (67.5:112.5:0.4) -- cycle;
\draw [shift={(-4.528276847984907,1.5359568838917255)},line width=1pt,color=qqqqff,fill=qqqqff,fill opacity=0.1] (0,0) -- (54:0.4) arc (54:90:0.4) -- cycle;
\draw [shift={(2.414213562373095,0)},line width=1pt,color=qqwuqq,fill=qqwuqq,fill opacity=0.10000000149011612] (0,0) -- (157.5:0.4) arc (157.5:180:0.4) -- cycle;
\draw [shift={(2.414213562373095,0)},line width=1pt,color=qqqqff,fill=qqqqff,fill opacity=0.1] (0,0) -- (0:0.4) arc (0:36:0.4) -- cycle;
\draw [shift={(-2.414213562373095,0)},line width=1pt,color=qqwuqq,fill=qqwuqq,fill opacity=0.10000000149011612] (0,0) -- (0:0.4) arc (0:22.5:0.4) -- cycle;
\draw [shift={(-2.414213562373095,0)},line width=1pt,color=qqqqff,fill=qqqqff,fill opacity=0.1] (0,0) -- (144:0.4) arc (144:180:0.4) -- cycle;
\draw [line width=1pt] (4.528276847984907,0) -- (4.528276847984907,6.126666666666644);
\draw [line width=1pt] (-4.528276847984907,0) -- (-4.528276847984907,6.126666666666644);
\draw [shift={(-2.414213562373095,0)},line width=1pt]  plot[domain=0.3926990816987242:2.5132741228718345,variable=\t]({1*2.613125929752753*cos(\t r)+0*2.613125929752753*sin(\t r)},{0*2.613125929752753*cos(\t r)+1*2.613125929752753*sin(\t r)});
\draw [shift={(2.414213562373095,0)},line width=1pt]  plot[domain=0.6283185307179587:2.748893571891069,variable=\t]({1*2.613125929752753*cos(\t r)+0*2.613125929752753*sin(\t r)},{0*2.613125929752753*cos(\t r)+1*2.613125929752753*sin(\t r)});
\draw [line width=1pt,domain=-5.411111111111112:6.10888888888889] plot(\x,{(-0-0*\x)/6.942490410358002});
\draw (0,0) node[below] {$0$};
\draw (0.2555555555555551,1.5) node[anchor=north west] {$S_{m,n}$};
\draw (4.6,2) node[anchor=north west] {$S_{n,m}$};
\draw (-5.4,2) node[anchor=north west] {$S_{n,m}$};
\draw (1.7,0) node[anchor=north west] {$x = \cot(\frac{\pi}{n})$};
\draw (4.3,-0.08666666666666249) node[anchor=north west] {$\frac{s}{2}$};
\draw (-3,0.05) node[anchor=north west] {$-\cot(\frac{\pi}{n})$};
\draw (-4.9,0.033333333333336976) node[anchor=north west] {$-\frac{s}{2}$};
\draw [line width=1pt,dash pattern=on 3pt off 3pt] (2.414213562373095,0)-- (4.528276847984907,1.5359568838917255);
\draw [line width=1pt,dash pattern=on 3pt off 3pt] (2.414213562373095,0)-- (0,1);
\draw [line width=1pt,dash pattern=on 3pt off 3pt] (0,1)-- (-2.414213562373095,0);
\draw [line width=1pt,dash pattern=on 3pt off 3pt] (-2.414213562373095,0)-- (-4.528276847984907,1.5359568838917255);
\draw (-4.5,2.8) node[anchor=north west] {$\frac{\pi}{m}$};
\draw (-0.5,2.3) node[anchor=north west] {$\frac{\pi}{n}$};
\draw (4,2.8) node[anchor=north west] {$\frac{\pi}{m}$};
\draw (3,0.6) node[anchor=north west] {$\frac{\pi}{m}$};
\draw (0.9,0.55) node[anchor=north west] {$\frac{\pi}{n}$};
\draw (-1.5,0.55) node[anchor=north west] {$\frac{\pi}{n}$};
\draw (-3.5,0.6) node[anchor=north west] {$\frac{\pi}{m}$};
\draw [line width=1pt,dash pattern=on 3pt off 3pt] (0,1) -- (0,6.126666666666644);
\draw (0,2.3) node[anchor=north west] {$\frac{\pi}{n}$};
\draw (2.2955555555555556,1.8066666666666626) node[anchor=north west] {$r = \frac{1}{sin( \pi /n)}$};
\begin{scriptsize}
\draw [fill=black] (0,1) circle (2.5pt);
\draw [color=uuuuuu] (2.414213562373095,0)-- ++(-2.5pt,0 pt) -- ++(5pt,0 pt) ++(-2.5pt,-2.5pt) -- ++(0 pt,5pt);
\draw [fill=black] (4.528276847984907,1.5359568838917255) circle (2.5pt);
\draw [color=uuuuuu] (-2.414213562373095,0)-- ++(-2.5pt,0 pt) -- ++(5pt,0 pt) ++(-2.5pt,-2.5pt) -- ++(0 pt,5pt);
\draw [fill=black] (-4.528276847984907,1.5359568838917255) circle (2.5pt);
\draw [color=black] (0,0)-- ++(-2.5pt,0 pt) -- ++(5pt,0 pt) ++(-2.5pt,-2.5pt) -- ++(0 pt,5pt);
\end{scriptsize}
\end{tikzpicture}
\caption{The fundamental domain $\Tmn$ of the Teichm\"uller curve of $S_{m,n}$.}
\label{fig:fond_domain_0}
\end{figure}

Using the data of the angles and the coordinates of $S_{m,n}$, one can compute all other parameters that are useful for the following sections. 
In particular, we will denote by $x$ the abscissa of the center of the circle which defines the side of the fundamental domain connecting the two vertices of order $m$ and $n$ so, the line connecting $S_{m,n}$ and $S_{n,m}$ and by $r$ its radius. 
We then have $x = \cot\left(\frac{\pi}{n}\right)$, $r = \frac{1}{\sin\left(\frac{\pi}{n}\right)}$ and that $S_{n,m}$ has coordinates $\pm \frac{s}{2} + \frac{\sin\left(\frac{\pi}{m}\right)}{\sin\left(\frac{\pi}{n}\right)}$.


Finally, notice that since $S_{m,n}$ and $S_{n,m}$ belong to the same Teichm\"uller curve (up to normalising the area), and this gives two distinct parametrisations of the $SL_2(\RR)-$orbit. In the following, we will use the parametrisation given by $S_{m,n}$ with $n<m$.\newline

\subsection{Singular intersections}\label{subsec:singular_intersections_on_BM_surfaces}

We end this section with the computation of the algebraic intersections of several pairs of closed curves in Bouw-M\"oller surfaces, which will be useful later.

\input{Systoles_2}

\section{KVol and convex polygons with obtuse angles}
\label{sec:KVol_convex}
\input{General_v2}


\input{Intersection_horizontal_saddle_connections}

\input{extension_Teichmuller_disk}

\bibliographystyle{alpha}
\bibliography{KVolBM_bibli}
\end{document}

%% file: NewIntro.tex
Given any closed oriented surface $X$ endowed with a Riemannian metric $g$ (possibly with singularities), define:

\begin{equation}
\label{eq:KVol}
\KVol(X): = \Vol(X,g)\cdot \sup_{\alpha,\beta} \frac{\Int (\alpha,\beta)}{l_g (\alpha) l_g (\beta)},
\end{equation}
where the supremum ranges over all piecewise smooth closed curves $\alpha$ and $\beta$ in $X$, $\Int(\cdot, \cdot)$ denotes the algebraic intersection number, and $l_g(\cdot)$ denotes the length with respect to the Riemannian metric $g$ (it is readily seen that multiplying by the volume Vol$(X,g)$ makes the quantity invariant by rescaling the metric). This quantity can be thought of as a way of measuring the length required for having an intersection on the surface $X$. 
The study of KVol originated in the work of Massart \cite{these_massart} (see also \cite{MM}), where it is shown that this quantity is finite. 
There, KVol$(X)$ is also shown to be the comparison constant between two norms in the homology $H_1(X,\RR)$ of an oriented closed Riemannian surface (namely the $L^2$-norm and the stable norm, see also \cite{MM} for more information). 
Furthermore, for any Riemannian surface $(X,g)$ one has that $\KVol(X,g) \geq 1$ with equality if and only if $X$ is a torus and $g$ is flat \cite{MM}. 
Although it is relatively easy to show that $\KVol(X) = 1$ on a flat torus, the question of computing explicitly KVol for a given surface turns out to be a difficult question.

The study of KVol has been extended in recent years, and 
in particular several authors (see below for a more detailed account) have studied KVol for translation surfaces, which are instances of flat surfaces with finitely many conical singularities. 
In this paper we continue the study of KVol on surfaces of this type.
Translation surfaces may be defined in several different ways (see, for example, \cite{Masur}) and in this paper we will define them in terms of Euclidean polygons, which allows us to get estimates of both intersections and lengths. 
Namely, translation surfaces can be seen as surfaces arising from a collection of Euclidean polygons by identifying pairs of parallel sides of the same length by translations satisfying certain properties (see Section \ref{sec:preliminaries_BM}).

More precisely, here is a short summary of what is known, to the best of our knowledge, for $\KVol$ on translation surfaces. 
Cheboui, Kessi and Massart \cite{CKM} study KVol on the Teichm\"uller disk associated to a family of square-tiled (or \emph{arithmetic}) translation surfaces.
The same authors also show that the infimum of KVol on the stratum $\mathcal{H}(2)$ of translation surfaces of genus two with a single conical singularity of angle $6\pi$ is bounded from above by two, by giving an explicit construction.
This construction is generalised to any genus (and any connected component of the minimal stratum) by the first author in \cite{Bou23b}.
Finally, \cite{BLM22} and \cite{Bou23} study KVol on the Teichm\"uller disks associated respectively to the double regular $n$-gon (for odd $n \geq 5$) and the regular $n$-gon (for even $n \geq 8$).


A brief mention should be given also of the case of hyperbolic surfaces.
Massart and Muetzel \cite{MM} first described the behaviour of KVol as the homological systolic length goes to zero, and this has been extended recently in \cite{Jiang_Pan}. 
Furthermore, recent work on the minimal length product over homology bases allows to get a lower bound for KVol on closed hyperbolic surfaces of a given genus, see \cite[Theorem 3.12]{BKP21}. An alternative proof has also been given in \cite{Jiang_Pan}. 
We should also mention the related work of Torkaman \cite{Torkaman} on a similar quantity in the case where the algebraic intersection has been replaced by the geometric intersection.


Going back to translation surfaces, the main goal of this paper is to calculate the value of KVol for Bouw-M\"oller surfaces, which is a class of translation surfaces whose Veech group is a triangle group \cite{BM10, Hooper}. 
For every pair of integers $m$ and $n$, with $m,n \geq 2$ and $mn > 6$, there is an associated Bouw-M\"oller surface $S_{m,n}$ constructed by gluing sides of $m$ polygons which are either regular $n$-gons or semi-regular $2n$-gons (see Section \ref{sec:preliminaries_BM}).
This class also contains double regular polygons (when $m=2$).
More precisely, we have:

\begin{Cor}\label{cor:KVol_Smn}
Let $m,n \geq 2$, $mn > 6$. Then :
\begin{equation}\label{eq:KVol_Smn}
KVol(S_{m,n}) \leq \frac{\Vol(S_{m,n})}{\sin(\pi/m)^2}
\end{equation}
with equality if $m$ and $n$ are coprime.
\end{Cor}

\begin{Rema}
If $m$ and $n$ are not coprime, the surface has several singularities and there is no pair of closed curves achieving the upper bound in \eqref{eq:KVol_Smn} (see Theorem \ref{thm:mainBM}).
 This does not directly imply that the inequality of Corollary \ref{cor:KVol_Smn} is strict, as in principle the supremum in the definition of KVol may not be a maximum. 
However, we believe that KVol is always a maximum for Veech surfaces of genus at least two (see also Remark 1.6 of \cite{BLM22}).
\end{Rema}

The strategy for proving this is to estimate the quantity 
\begin{equation}\label{eq:ratioIntLen} \frac{\Int(\gamma,\delta)}{l(\gamma)l(\delta)} \end{equation}
for a given pair of (simple) closed geodesics $(\gamma,\delta)$. In fact, our estimate of this quantity holds for a much larger class of surfaces than just Bouw-M\"oller surfaces. 
More precisely, we consider a translation surface $X$ constructed from polygons $(P_i)_{i \in I}$ (if $I$ is not finite, see Section \ref{sec:comments_Thm_convex}). We assume that: 
\begin{itemize}
\item[(P1)] each polygon is convex with obtuse or right angles,
\item[(P2)] two sides of the same polygon are never identified together. 
\end{itemize}

We denote by $l_0$ the length of the smallest sides of the polygons. We show:

\begin{Theo}\label{theo:mainCONVEX}
For a translation surface with a polygonal representation satisfying (P1) and (P2), for any two closed curves $\gamma$ and $\delta$ on $X$, we have:
\[ \frac{\Int(\gamma,\delta)}{l(\gamma)l(\delta)} \leq \frac{1}{l_0^2}. \]

Further, if the angles are all strictly obtuse, equality holds if and only if $\gamma$ and $\delta$ are sides of length $l_0$ intersecting once. If there are right angles, equality may also hold only in the following cases: 
\begin{itemize}
    \item $\gamma$ and $\delta$ are two closed curves, both of length $l_0$ and intersecting once at a singularity, 
    \item $\gamma$ and $\delta$ are both diagonals of length $\sqrt{2} l_0$ intersecting twice (once at a singularity and once outside the singularities, with the same sign),
    \item up to swapping them, $\gamma$ is a side or a diagonal of length $l_0$ and $\delta$ is a geodesic of length $2l_0$, contained in the union of exactly two polygons, and intersecting $\gamma$ twice.
\end{itemize}
\end{Theo}

The main idea in the proof of Theorem \ref{theo:mainCONVEX} is to decompose pairs of closed curves into smaller segments whose lengths, as well as the intersection of pairs of segments, can be controlled in the following way.
First, the combinatorial condition (P2) on the gluing of the polygons allows us to make estimates on the algebraic intersections of closed curves. 
Meanwhile, the assumption on the shape of the polygons (P1) allows us to bound from below the lengths of saddle connections depending on the sides of the polygons they cross. 
This will be made precise with the notion of \emph{polygonal decomposition} described in Section \ref{sec:KVol_convex}. 
Moreover, the same result holds for half-translation surfaces with the same assumptions and a similar proof. It also generalises to infinite (half-)translation surfaces. \newline

From Theorem \ref{theo:mainCONVEX}, we obtain that information on the ratio in Equation \ref{eq:ratioIntLen} provides geometric constraints on the shape of a polygonal decomposition.
In a translation surface $X$ such that there exists a pair of closed curves with $\frac{\Int(\gamma,\delta)}{l(\gamma)l(\delta)} > \frac{1}{l_0^2}$, there are no polygonal representations of $X$ satisfying both (P1) and (P2). This is for example the case on the $L$-shaped surface made of six equilateral triangles (or, more generally, on \emph{equilateral staircases}), for which there exists a pair of closed curves $\gamma,\delta$ with $\frac{\Int(\gamma,\delta)}{l(\gamma)l(\delta)} = \frac{2}{\sqrt{3} l_0^2}$.

We make additional comments on this result in \S\ref{sec:comments_Thm_convex}, but let us first discuss the special case of Bouw-M\"oller surfaces.

\subsection{KVol on Bouw-M\"oller surfaces} 
As already said, the initial motivation for proving this result was to be able to deal with the case of Bouw-M\"oller surfaces. The Bouw-M\"oller surface $S_{m,n}$ can be described gluing a collection of polygons $P(0), \dots, P(m-1)$ which are semi-regular, which means that the internal angles are all equal, but the sides have two alternating lengths. The surface $S_{m,n}$ satisfies assumptions (P1) and (P2) as soon as $n \geq 4$. 
For $n=2$ we get back to (the staircase models of) double regular polygons and this has been investigated in \cite{BLM22}\footnote{although they only deal with the case of odd $m$, the cylinder decomposition argument they use is also valid for $S_{m,2}$ for even $m$}. For $n=3$ (and $m \geq 3$), the surface $S_{m,n}$ does not satisfy (P1) but we can prove the same estimates with a slightly different proof, leading to the following result:

\begin{Theo}\label{thm:mainBM}
Let $m,n \geq 2$ with $mn > 6$ and let $S_{m,n}$ be the corresponding Bouw-M\"oller surface. Then, for any pair $(\alpha, \beta)$ of closed curves on $S_{m,n}$, we have
\[
\frac{\Int(\alpha,\beta)}{l(\alpha)l(\beta)} \leq \frac{1}{l_0^2}
\]
where $l_0 = \sin(\pi / m)$ is the length of the smallest side of the polygons forming $S_{m,n}$. Moreover, equality holds if and only if 
\begin{itemize}
    \item $m$ and $n$ are coprime and $\alpha$ and $\beta$ are intersecting systoles (hence of length $l_0$),
    \item or $n=4$ and $m \equiv 3 \mod 4$ and $\alpha$ and $\beta$ are diagonals of $P(0)$ (resp. $P(m-1)$), intersecting twice and having length $\sqrt{2}l_0$.
\end{itemize}
\end{Theo}

\begin{Rema}
    For $mn=6$ the resulting Bouw-M\"oller surface is a flat torus, and KVol is equal to 1. One sees that Theorem \ref{thm:mainBM} does not hold for $(m,n)=(2,3)$ (although it holds for $(m,n)=(3,2)$ as $S_{3,2}$ is the flat square torus).
\end{Rema}

In the case where the resulting Bouw-M\"oller surface has a single singularity (i.e. for $S_{m,n}$ with $m$ and $n$ coprime), we show that the obtained inequality is in fact an equality by showing that there are intersecting systoles, thus obtaining the value of KVol (see Corollary \ref{cor:KVol_Smn}).

\paragraph{Extension to the Teichm\"uller disk.}

In light of Theorem \ref{thm:mainBM}, it is interesting to wonder how KVol varies, as a function on the Teichm\"uller disk associated with $S_{m,n}$, that is, its $SL_2(\RR)-$orbit.
This question has been investigated in the case of the double regular $n$-gon $S_{2,n}$ for $n$ odd in \cite{BLM22}, and for the other families of Veech surfaces mentioned above in \cite{CKM, Bou23}. The second goal of this paper is to study this question for Bouw-M\"oller surfaces. 
The Teichm\"uller curve associated to $S_{m,n}$ (the quotient of the Teichm\"uller disk by the Veech group) can be identified with an hyperbolic orbifold with one cusp and two elliptic points of angles $\frac{\pi}{n}$ and $\frac{\pi}{m}$. 
A fundamental domain $\Tmn$ for the action of the Veech group on the Teichm\"uller disk is described in \S\ref{sec:Teichmuller_disk_BM}. 
We give an explicit expression for KVol on $\Tmn$ in the case where $S_{m,n}$ has a unique singularity:

\begin{Theo}\label{theo:extension_teich}
Let $m,n \geq 2$ coprime with $mn>6$. Then for any $X \in \Tmn$, we have:
\begin{equation*}
\KVol(X) = K_{m,n} \cdot \frac{1}{\cosh d_{\HH}(X, \gamma_{\infty, \pm \cot(\pi/n)})}
\end{equation*}
where $d_{\HH}$ denotes the hyperbolic distance and $\gamma_{\infty, \pm \cot(\pi/n)}$ is the union of the two hyperbolic geodesics of respective endpoints $(\infty, \cot(\pi/n))$, and $(\infty, - \cot(\pi/n))$ \footnote{The geodesics of endpoints $(\infty, \cot(\pi/n))$ and $(\infty, - \cot(\pi/n))$ are the same if $n=2$.}, and $K_{m,n} > 0$ is an explicit constant.
\end{Theo}

The case where $m$ and $n$ are not coprime is much more complicated, as there are several distinct singularities. 
In this case we do not obtain an explicit formula for KVol on the $SL_2(\RR)-$orbit of $S_{m,n}$, but we can still show that KVol is bounded. This is done using Theorem 1.5 of \cite{Bou23}. The following result is also stated as Corollary \ref{cor:boundedness}:

\begin{Theo}\label{theo:boundedness}
KVol is bounded on the $SL_2(\RR)-$orbit of $S_{m,n}$.
\end{Theo}

In fact, we describe in Section~\ref{sec:intersection_horizontal} a general approach for showing boundedness on $SL_2(\RR)-$orbits of Veech surfaces by investigating separatrix diagrams in every periodic direction, getting the following:
\begin{Theo}\label{thm:boundedness_planar_intro}
KVol is bounded on the $SL_2(\RR)-$orbit of a Veech surface if and only if the separatrix diagram associated to every periodic direction is planar.
\end{Theo}
We give a definition of separatrix diagrams in Section \ref{sec:intersection_horizontal}, following \cite{KZ03}. These are ribbon graphs encoding the intersections of closed curves made of saddle connections in a given periodic direction. In the case of Bouw-M\"oller surfaces, it turns out to be much easier to study the dual (ribbon) graph.

\subsection{Additional comments}\label{sec:comments_Thm_convex}

We continue with a few comments on Theorem \ref{theo:mainCONVEX}.

\paragraph{On the assumptions (P1) and (P2).} Let us first point out that Theorem \ref{theo:mainCONVEX} does not hold without the assumption on the polygons being convex, nor without the assumption on the angles being right or obtuse. A counter-example is given in Figure \ref{fig:counterexamples1}. We discuss further the obstruction given by the angle condition in \S\ref{sec:additional_remarks}.

On the contrary, we believe that Theorem \ref{theo:mainCONVEX} holds without the assumption (P2). Nevertheless, our proof requires a very careful count of the intersections which can not be performed in the case where there are self-identifications on the polygons. One way to bypass the problem would be to make a case-by-case analysis of the problematic cases (i.e. try to obtain better estimates of the lengths when we have too many intersections), as it is done in \cite{Bou23} in the case of the regular $2n$-gon ($n \geq 4$), but it seems unlikely that this case-by-case analysis could be performed in general.

\begin{figure}[h]
\center
\definecolor{qqqqff}{rgb}{0,0,1}
\definecolor{zzttqq}{rgb}{1,0.33,0}
\definecolor{qqwuqq}{rgb}{0,0.39215686274509803,0}
\definecolor{ccqqqq}{rgb}{0.8,0,0}
\begin{tikzpicture}[line cap=round,line join=round,>=triangle 45,x=1cm,y=1cm,scale=0.5]
\clip(-6.5,-6.5) rectangle (6.5,6.5);
\draw [shift={(-4,0)},line width=1pt,fill=black,fill opacity=0.3] (0,0) -- (-26.56505117707799:0.7643236313035815) arc (-26.56505117707799:26.56505117707799:0.7643236313035815) -- cycle;
\draw [line width=1pt] (0,2)-- (-2,6);
\draw [line width=1pt] (-2,6)-- (-6,4);
\draw [line width=1pt] (-6,4)-- (-4,0);
\draw [line width=1pt] (4,0)-- (6,4);
\draw [line width=1pt] (6,4)-- (2,6);
\draw [line width=1pt] (2,6)-- (0,2);
\draw [line width=1pt] (-4,0)-- (-6,-4);
\draw [line width=1pt] (-6,-4)-- (-2,-6);
\draw [line width=1pt] (-2,-6)-- (0,-2);
\draw [line width=1pt] (0,-2)-- (2,-6);
\draw [line width=1pt] (2,-6)-- (6,-4);
\draw [line width=1pt] (6,-4)-- (4,0);
\draw [line width=1pt,dash pattern=on 3pt off 3pt] (-4,0)-- (0,2);
\draw [line width=1pt,dash pattern=on 3pt off 3pt] (0,2)-- (4,0);
\draw [line width=1pt,dash pattern=on 3pt off 3pt] (4,0)-- (0,-2);
\draw [line width=1pt, dash pattern=on 3pt off 3pt] (0,-2)-- (-4,0);
\draw [line width=1pt,color=ccqqqq,dash pattern=on 3pt off 3pt] (-4,0)-- (4,0);
\draw [line width=1pt,color=qqwuqq,dash pattern=on 3pt off 3pt] (0,-2)-- (0,2);
\draw [color=ccqqqq](1,0.1) node[anchor=north west] {$\alpha$};
\draw [color=qqwuqq](0,1.5) node[anchor=north west] {$\beta$};
\draw (-2.65,0.15) node[anchor=north west] {$\theta$};
\draw (-4.8,-2.5) node[anchor=south east] {$F$};
\draw (-5,-6) node[anchor=south west] {$E$};
\draw (-1.1,-3.7) node[anchor=north west] {$D$};
\draw (-4.8,2.5) node[anchor=north east] {$A$};
\draw (-5,6) node[anchor=north west] {$B$};
\draw (-1.1,3.7) node[anchor=south west] {$C$};
\draw (4.8,-2.5) node[anchor=south west] {$A$};
\draw (5,-6) node[anchor=south east] {$B$};
\draw (1.1,-3.7) node[anchor=north east] {$C$};
\draw (4.8,2.5) node[anchor=north west] {$F$};
\draw (5,6) node[anchor=north east] {$E$};
\draw (1.1,3.7) node[anchor=south east] {$D$};
\begin{scriptsize}
\draw [color=black] (-4,0)-- ++(-2.5pt,-2.5pt) -- ++(5pt,5pt) ++(-5pt,0) -- ++(5pt,-5pt);
\draw [color=black] (0,2)-- ++(-2.5pt,-2.5pt) -- ++(5pt,5pt) ++(-5pt,0) -- ++(5pt,-5pt);
\draw [color=black] (4,0)-- ++(-2.5pt,-2.5pt) -- ++(5pt,5pt) ++(-5pt,0) -- ++(5pt,-5pt);
\draw [color=black] (0,-2)-- ++(-2.5pt,-2.5pt) -- ++(5pt,5pt) ++(-5pt,0) -- ++(5pt,-5pt);
\draw [color=black] (-2,6)-- ++(-2.5pt,-2.5pt) -- ++(5pt,5pt) ++(-5pt,0) -- ++(5pt,-5pt);
\draw [color=black] (-6,4)-- ++(-2.5pt,-2.5pt) -- ++(5pt,5pt) ++(-5pt,0) -- ++(5pt,-5pt);
\draw [color=black] (6,4)-- ++(-2.5pt,-2.5pt) -- ++(5pt,5pt) ++(-5pt,0) -- ++(5pt,-5pt);
\draw [color=black] (2,6)-- ++(-2.5pt,-2.5pt) -- ++(5pt,5pt) ++(-5pt,0) -- ++(5pt,-5pt);
\draw [color=black] (-6,-4)-- ++(-2.5pt,-2.5pt) -- ++(5pt,5pt) ++(-5pt,0) -- ++(5pt,-5pt);
\draw [color=black] (-2,-6)-- ++(-2.5pt,-2.5pt) -- ++(5pt,5pt) ++(-5pt,0) -- ++(5pt,-5pt);
\draw [color=black] (2,-6)-- ++(-2.5pt,-2.5pt) -- ++(5pt,5pt) ++(-5pt,0) -- ++(5pt,-5pt);
\draw [color=black] (6,-4)-- ++(-2.5pt,-2.5pt) -- ++(5pt,5pt) ++(-5pt,0) -- ++(5pt,-5pt);
\end{scriptsize}
\end{tikzpicture}
\caption{Example of a translation surface for which we can construct a decomposition into five convex polygons (four squares and a diamond) satisfying (P2), but where the diamond has an acute angle $\theta$, or a decomposition into two polygons (cutting along $\alpha$), satisfying (P2) and such that the angles are all obtuse or right, but the polygons are not convex. One checks that the curves $\alpha$ and $\beta$ intersect twice (once at the singularity and once outside the singularity with the same sign), and we have $\Int(\alpha,\beta)/l(\alpha)l(\beta) = 1/(l_0^2 \sin \theta) > 1/l_0^2$.}
\label{fig:counterexamples1}
\end{figure}

\paragraph{Non-Veech surfaces and Infinite translation surfaces.} Let us also highlight the fact that Theorem \ref{theo:mainCONVEX} allows to compute KVol on examples of translation surfaces which are not Veech. All the previous exact computations of KVol used the symmetries of the surfaces and were done for specific Veech surfaces (surfaces in the $SL_2(\RR)$-orbit of staircases \cite{CKM} and regular polygons \cite{BLM22,Bou23}). Two examples are given in Figure \ref{fig:example_2}. 

We should also mention that Theorem \ref{theo:mainCONVEX} holds for infinite translation surfaces. Of course, in this latter case $l_0$ should be replaced by the infimum of the lengths of the sides of polygons, and it could hence be zero, in which case Theorem \ref{theo:mainCONVEX} does not give any relevant inequality, but it is not surprising as the supremum of the ratio $\mathrm{Int}(\alpha,\beta)/l(\alpha)l(\beta)$ could be infinite for infinite translation surfaces.

\paragraph{More general flat surfaces.}
Further, although Theorem \ref{theo:mainCONVEX} is stated for translation surfaces, it is in fact not needed in our count of the intersections for the identified sides to be parallel of the same length. In particular, the same result is still true for half-translation surfaces (for which our estimates on the length of saddle connections also hold), and Theorem \ref{theo:mainCONVEX} may generalize to the case where the transition maps are affine maps instead of only translations. However, this would require a few modifications in the proof 
which are discussed in \S \ref{sec:additional_remarks} (see also the first author's Thesis \cite[\S 4.5.4]{these_boulanger}).

\begin{figure}[h]
\center
\definecolor{uuuuuu}{rgb}{0.26666666666666666,0.26666666666666666,0.26666666666666666}
\definecolor{qqwuqq}{rgb}{0,0.39215686274509803,0}
\definecolor{ccqqqq}{rgb}{0.8,0,0}
\definecolor{ffvvqq}{rgb}{1,0.3333333333333333,0}
\begin{tikzpicture}[line cap=round,line join=round,>=triangle 45,x=1cm,y=1cm, scale = 0.7]
\clip(-3.2,-3.5) rectangle (3.7,2.5);
\draw [line width=1pt] (0,-1)-- (1.9833112346864756,-1.2578304605092419);
\draw [line width=1pt] (1.9833112346864756,-1.2578304605092419)-- (3,0);
\draw [line width=1pt, dash pattern = on 3pt off 3pt] (0,-1)--(0,1);
\draw [line width=1pt] (3,0)-- (1.9833112346864756,1.2578304605092419);
\draw [line width=1pt] (1.9833112346864756,1.2578304605092419)-- (0,1);
\draw [line width=1pt] (0,1)-- (-1.9833112346864754,1.257830460509242);
\draw [line width=1pt] (-1.9833112346864754,1.257830460509242)-- (-3,0);
\draw [line width=1pt] (-3,0)-- (-1.9833112346864756,-1.2578304605092419);
\draw [line width=1pt] (-1.9833112346864756,-1.2578304605092419)-- (0,-1);
\draw (0,0) node[right] {$0$};
\draw (-2.5,0.5) node[anchor=south east] {$1$};
\draw (-1,1.2) node[above] {$2$};
\draw (1,1.2) node[above] {$3$};
\draw (2.5,0.5) node[anchor=south west] {$4$};
\draw (-2.5,-0.5) node[anchor=north east] {$4$};
\draw (-1,-1.2) node[below] {$3$};
\draw (1,-1.2) node[below] {$2$};
\draw (2.5,-0.5) node[anchor=north west] {$1$};
\begin{scriptsize}
\draw [fill=black] (0,1) -- ++(-2.5pt,-2.5pt) -- ++(5pt,5pt) ++(-5pt,0) -- ++(5pt,-5pt);
\draw [fill=black] (0,-1) -- ++(-2.5pt,-2.5pt) -- ++(5pt,5pt) ++(-5pt,0) -- ++(5pt,-5pt);
\draw [fill=black] (1.9833112346864756,1.2578304605092419) -- ++(-2.5pt,-2.5pt) -- ++(5pt,5pt) ++(-5pt,0) -- ++(5pt,-5pt);
\draw [fill=black] (1.9833112346864756,-1.2578304605092419) -- ++(-2.5pt,-2.5pt) -- ++(5pt,5pt) ++(-5pt,0) -- ++(5pt,-5pt);
\draw [fill=black] (-1.9833112346864754,1.257830460509242) -- ++(-2.5pt,-2.5pt) -- ++(5pt,5pt) ++(-5pt,0) -- ++(5pt,-5pt);
\draw [fill=black] (-1.9833112346864756,-1.2578304605092419) -- ++(-2.5pt,-2.5pt) -- ++(5pt,5pt) ++(-5pt,0) -- ++(5pt,-5pt);
\draw [fill=black] (3,0) -- ++(-2.5pt,-2.5pt) -- ++(5pt,5pt) ++(-5pt,0) -- ++(5pt,-5pt);
\draw [fill=black] (-3,0) -- ++(-2.5pt,-2.5pt) -- ++(5pt,5pt) ++(-5pt,0) -- ++(5pt,-5pt);
\end{scriptsize}
\end{tikzpicture}
\begin{tikzpicture}[line cap=round,line join=round,>=triangle 45,x=1cm,y=1cm, scale=0.65]
\clip(-7,-0.75) rectangle (7,8);
\draw [line width=1pt] (0,0)-- (2,0);
\draw [line width=1pt] (2,0)-- (3.414213562373095,1.414213562373095);
\draw [line width=1pt] (3.414213562373095,1.414213562373095)-- (3.414213562373095,3.4142135623730945);
\draw [line width=1pt,dash pattern=on 3pt off 3pt] (3.414213562373095,3.4142135623730945)-- (2,4.82842712474619);
\draw [line width=1pt] (2,4.82842712474619)-- (0,4.82842712474619);
\draw [line width=1pt] (0,4.82842712474619)-- (-1.414213562373095,3.4142135623730954);
\draw [line width=1pt] (-1.4142135623730954,1.4142135623730956)-- (0,0);
\draw [line width=1pt,dash pattern=on 3pt off 3pt] (-1.4142135623730954,1.4142135623730956)-- (-1.414213562373095,3.4142135623730954);
\draw [line width=1pt] (-1.414213562373095,3.4142135623730954)-- (-3.4142135623730945,3.4142135623730954);
\draw [line width=1pt,dash pattern=on 3pt off 3pt] (-3.4142135623730945,3.4142135623730954)-- (-3.414213562373095,1.414213562373096);
\draw [line width=1pt] (-3.414213562373095,1.414213562373096)-- (-1.4142135623730954,1.4142135623730956);
\draw [line width=1pt] (3.414213562373095,3.4142135623730945)-- (4.82842712474619,4.828427124746188);
\draw [line width=1pt,dash pattern=on 3pt off 3pt] (4.82842712474619,4.828427124746188)-- (3.4142135623730963,6.242640687119284);
\draw [line width=1pt] (3.4142135623730963,6.242640687119284)-- (2,4.82842712474619);
\draw [line width=1pt] (-3.4142135623730945,3.4142135623730954)-- (-5.414213562373093,3.414213562373096);
\draw [line width=1pt] (-5.414213562373093,3.414213562373096)-- (-5.414213562373094,1.414213562373097);
\draw [line width=1pt] (-5.414213562373094,1.414213562373097)-- (-3.414213562373095,1.414213562373096);
\draw [line width=1pt] (4.82842712474619,4.828427124746188)-- (6.242640687119286,6.242640687119282);
\draw [line width=1pt] (6.242640687119286,6.242640687119282)-- (4.828427124746192,7.656854249492378);
\draw [line width=1pt] (4.828427124746192,7.656854249492378)-- (3.4142135623730963,6.242640687119284);
\draw (-4.5,3.3) node[above] {$A$};
\draw (-2.5,3.3) node[above] {$B$};
\draw (-4.5,1.5) node[below] {$B$};
\draw (-2.5,1.4) node[below] {$C$};
\draw (1,4.75) node[above] {$C$};
\draw (-1.4,4.8) node[anchor=north west] {$D$};
\draw (4.1,4.3) node[anchor=north west] {$D$};
\draw (2,6.25) node[anchor=north west] {$E$};
\draw (5.5,5.75) node[anchor=north west] {$E$};
\draw (3.6,7.75) node[anchor=north west] {$F$};
\draw (2.6,0.9) node[anchor=north west] {$F$};
\draw (0.7,0.1) node[anchor=north west] {$A$};
\draw (-1.25,0.8) node[anchor=north west] {$G$};
\draw (5.5,7.55) node[anchor=north west] {$G$};
\draw (-6.3,2.85) node[anchor=north west] {$H$};
\draw (3.4,2.85) node[anchor=north west] {$H$};
\begin{scriptsize}
\draw [color=black] (0,0)-- ++(-2.5pt,-2.5pt) -- ++(5pt,5pt) ++(-5pt,0) -- ++(5pt,-5pt);
\draw [color=black] (2,0)-- ++(-2.5pt,-2.5pt) -- ++(5pt,5pt) ++(-5pt,0) -- ++(5pt,-5pt);
\draw [color=black] (3.414213562373095,1.414213562373095)-- ++(-2.5pt,-2.5pt) -- ++(5pt,5pt) ++(-5pt,0) -- ++(5pt,-5pt);
\draw [color=black] (3.414213562373095,3.4142135623730945)-- ++(-2.5pt,-2.5pt) -- ++(5pt,5pt) ++(-5pt,0) -- ++(5pt,-5pt);
\draw [color=black] (2,4.82842712474619)-- ++(-2.5pt,-2.5pt) -- ++(5pt,5pt) ++(-5pt,0) -- ++(5pt,-5pt);
\draw [color=black] (0,4.82842712474619)-- ++(-2.5pt,-2.5pt) -- ++(5pt,5pt) ++(-5pt,0) -- ++(5pt,-5pt);
\draw [color=black] (-1.414213562373095,3.4142135623730954)-- ++(-2.5pt,-2.5pt) -- ++(5pt,5pt) ++(-5pt,0) -- ++(5pt,-5pt);
\draw [color=black] (-1.4142135623730954,1.4142135623730956)-- ++(-2.5pt,-2.5pt) -- ++(5pt,5pt) ++(-5pt,0) -- ++(5pt,-5pt);
\draw [color=black] (-3.4142135623730945,3.4142135623730954)-- ++(-2.5pt,-2.5pt) -- ++(5pt,5pt) ++(-5pt,0) -- ++(5pt,-5pt);
\draw [color=black] (-3.414213562373095,1.414213562373096)-- ++(-2.5pt,-2.5pt) -- ++(5pt,5pt) ++(-5pt,0) -- ++(5pt,-5pt);
\draw [color=black] (4.82842712474619,4.828427124746188)-- ++(-2.5pt,-2.5pt) -- ++(5pt,5pt) ++(-5pt,0) -- ++(5pt,-5pt);
\draw [color=black] (3.4142135623730963,6.242640687119284)-- ++(-2.5pt,-2.5pt) -- ++(5pt,5pt) ++(-5pt,0) -- ++(5pt,-5pt);
\end{scriptsize}
\end{tikzpicture}
\caption{Examples of translation surfaces (with a single singularity) satisfying (P1) and (P2) and for which the inequality in Theorem \ref{theo:mainCONVEX} is achieved. On the left, the sides $1$ and $4$ are shorter than the others and correspond to closed curves intersecting once. On the right, there are twelve systoles and for example the closed curves corresponding to the sides $B$ and $D$ are intersecting, achieving the best possible ratio $\frac{1}{l_0^2}$. 
}
\label{fig:example_2}
\end{figure}

\paragraph{KVol and the systolic volume.} 
Finally, recall that the (homological) systolic length of a Riemannian surface $X$ is the length of a shortest non-homologically trivial loop on $X$. The systolic volume is then the volume of $X$ divided by the square of the systolic length. In other words:
\[ \mathrm{SysVol}(X) := \Vol(X) \cdot \sup_{\begin{scriptsize}
    \begin{array}{c}
\alpha \text{ closed curve}, \\
\left[ \alpha \right] \neq 0
\end{array}
\end{scriptsize}} \frac{1}{l(\alpha)^2}. \]

In particular, KVol can be thought of as a cousin of the systolic volume, twisted by the algebraic intersection, and it is natural to compare them. 

An interesting case in Theorem \ref{theo:mainCONVEX} is when all the vertices of the polygons are identified to the same point. It is for example the case on the Bouw-M\"oller surface $S_{m,n}$ when $m$ and $n$ are coprime. In this case there is only one singularity and the sides of the polygons represent closed curves on the surface. In particular, $l_0$ is also the systolic length and we have (if the surface has finite volume):

\begin{equation}\label{eq:KVol_Sysvol}
\KVol(X) \leq \mathrm{SysVol}(X).
\end{equation}

This result can be compared with Theorem 1.1 of \cite{MM}, which states that in the general setting of Riemannian surfaces, one has $\KVol(X) \leq 9 \mathrm{SysVol}(X)$. Of course, Equation~\eqref{eq:KVol_Sysvol} does not hold for every translation surface (as we have seen in the example of Figure \ref{fig:counterexamples1}, or for the equilateral L), but we conjecture that:

\begin{Conj}
For any translation surface $X$ with a single singularity, we have
\begin{equation*}
\KVol(X) \leq \frac{2}{\sqrt{3}} \mathrm{SysVol}(X).
\end{equation*}
\end{Conj}


It would be interesting to know which translation surfaces have a polygonal representation satisfying (P1) and (P2) (and whose vertices are all identified to a single singularity). From the above discussion, we see that the translation surfaces with a single singularity which do not satisfy $\KVol(X) \leq \mathrm{SysVol}(X)$ do not have such a polygonal representation. Translation surfaces with no convex representation have been studied by Lelièvre and Weiss \cite{Lelievre_Weiss_convex} but they were interested in representations with a single polygon and they had no restriction on the angles.

\subsection{Organisation of the paper}
We start in Section \ref{sec:preliminaries_BM} with some background on translation surfaces and review useful properties of Bouw-M\"oller surfaces. In Section \ref{subsec:singular_intersections_on_BM_surfaces}, we compute the algebraic intersection of some pairs of closed curves on Bouw-M\"oller with a single singularity, that will turn out to achieve the supremum in the definition of KVol.\newline
Then, we prove Theorem \ref{theo:mainCONVEX} in Section \ref{sec:KVol_convex}. We study the case of Bouw-M\"oller surfaces $S_{m,n}$ for $n=3$ in Section \ref{sec:Bouw-M\"oller}, thus obtaining Theorem \ref{thm:mainBM}.

The remaining sections are then devoted to the study of KVol on the $SL_2(\RR)-$orbit of Bouw-M\"oller surfaces. We start with the proof of Theorem \ref{thm:boundedness_planar_intro} in Section \ref{sec:intersection_horizontal}, and we study the case of Bouw-M\"oller surfaces, obtaining Theorem \ref{theo:boundedness}. In Section \ref{sec:extention_teichmuller}, we gather some facts from \cite{BLM22} on the computation of KVol on $SL_2(\RR)-$orbits of Veech surfaces and we state Theorem \ref{theo:big_statement} which we prove in Section \ref{sec:sinus_comparison} and which describes a general method to compute KVol. In Section \ref{sec:study_cylinder} we consider cylinder decompositions on Bouw-M\"oller surfaces in order to show that the assumptions of Theorem \ref{theo:big_statement} are satisfied, from which we will finally derive Theorem \ref{theo:extension_teich}.

\subsection{Acknowledgements}

The authors would like to thank Erwan Lanneau and Daniel Massart for useful conversations and comments throughout the writing of this paper. 
The second author is supported by a Leverhulme Early Career Fellowship awarded by the Leverhulme Trust.

%% file: Systoles_2.tex
Given two (simple) closed oriented curves $\alpha$ and $\beta$ on a translation surface, the algebraic intersection is given by the sum of each intersection point with its sign.
For a transverse intersection point $P$, we have $\Int_P (\alpha,\beta) = +1$ if the tangent vectors at $P$ of $\alpha$ and $\beta$ form an ordered basis with the normal to the sheet of paper, and $\Int_P (\alpha,\beta) = -1$ otherwise.

For closed saddle connections, one looks at the incoming and outgoing pieces of $\alpha$ and $\beta$ in a neighbourhood of the singularity, as in Figure \ref{fig:explanation_singular_intersection}. 
Then, the intersection will be determined by the circular order of these, in the sense that two saddle connections will intersect if and only if the two entries of one saddle connection alternate with the two entries of the other saddle connection.
Intuitively, this means that one cannot pull the two curves apart at the singularity to obtain two curves that do not intersect. 

\begin{figure}
\center
\definecolor{uuuuuu}{rgb}{0.26666666666666666,0.26666666666666666,0.26666666666666666}
\definecolor{qqqqff}{rgb}{0.4,0,0.8}
\definecolor{fuqqzz}{rgb}{0.8,0.2,0.4}
\definecolor{ffvvqq}{rgb}{0.8,0.4,0.2}
\definecolor{qqwuzz}{rgb}{0.4,0.6,0.6}
\definecolor{qqwuqq}{rgb}{0,0.5,0}
\definecolor{ccqqqq}{rgb}{0.8,0,0}
\begin{tikzpicture}[line cap=round,line join=round,>=triangle 45,x=2cm,y=2cm]
\clip(-2,-0.3) rectangle (5,3.5);
\draw [shift={(0,2.414213562373095)},line width=1pt,color=qqwuzz,fill=qqwuzz,fill opacity=0.55] (0,0) -- (-31.498117425671257:0.2666666666666664) arc (-31.498117425671257:0:0.2666666666666664) -- cycle;
\draw [shift={(-1.7071067811865475,0.707106781186548)},line width=1pt,color=ffvvqq,fill=ffvvqq,fill opacity=0.5] (0,0) -- (23.03190315777632:0.2666666666666664) arc (23.03190315777632:90:0.2666666666666664) -- cycle;
\draw [shift={(1.7071067811865475,0.7071067811865475)},line width=1pt,color=ffvvqq,fill=ffvvqq,fill opacity=0.5] (0,0) -- (90:0.2666666666666664) arc (90:225:0.2666666666666664) -- cycle;
\draw [shift={(1.7071067811865481,3.121320343559642)},line width=1pt,color=ffvvqq,fill=ffvvqq,fill opacity=0.5] (0,0) -- (-135:0.2666666666666664) arc (-135:-45:0.2666666666666664) -- cycle;
\draw [shift={(-0.7071067811865477,0.7071067811865478)},line width=1pt,color=ffvvqq,fill=ffvvqq,fill opacity=0.5] (0,0) -- (-45:0.2666666666666664) arc (-45:180:0.2666666666666664) -- cycle;
\draw [shift={(1,2.414213562373095)},line width=1pt,color=ffvvqq,fill=ffvvqq,fill opacity=0.5] (0,0) -- (-180:0.2666666666666664) arc (-180:45:0.2666666666666664) -- cycle;
\draw [shift={(1,0)},line width=1pt,color=ffvvqq,fill=ffvvqq,fill opacity=0.5] (0,0) -- (45:0.2666666666666664) arc (45:146.17024373125392:0.2666666666666664) -- cycle;
\draw [shift={(1,0)},line width=1pt,color=fuqqzz,fill=fuqqzz,fill opacity=0.5] (0,0) -- (146.17024373125392:0.2666666666666664) arc (146.17024373125392:180:0.2666666666666664) -- cycle;
\draw [shift={(-0.7071067811865475,1.7071067811865477)},line width=1pt,color=fuqqzz,fill=fuqqzz,fill opacity=0.5] (0,0) -- (-180:0.2666666666666664) arc (-180:45:0.2666666666666664) -- cycle;
\draw [shift={(1.7071067811865475,1.7071067811865472)},line width=1pt,color=fuqqzz,fill=fuqqzz,fill opacity=0.5] (0,0) -- (45:0.2666666666666664) arc (45:270:0.2666666666666664) -- cycle;
\draw [shift={(-1.7071067811865472,1.7071067811865477)},line width=1pt,color=fuqqzz,fill=fuqqzz,fill opacity=0.5] (0,0) -- (-90:0.2666666666666664) arc (-90:0:0.2666666666666664) -- cycle;
\draw [shift={(0,0)},line width=1pt,color=fuqqzz,fill=fuqqzz,fill opacity=0.5] (0,0) -- (0:0.2666666666666664) arc (0:135:0.2666666666666664) -- cycle;
\draw [shift={(2.414213562373095,2.414213562373094)},line width=1pt,color=fuqqzz,fill=fuqqzz,fill opacity=0.5] (0,0) -- (135:0.2666666666666664) arc (135:203.3185547154144:0.2666666666666664) -- cycle;
\draw [shift={(2.414213562373095,2.414213562373094)},line width=1pt,color=qqqqff,fill=qqqqff,fill opacity=0.5] (0,0) -- (-156.6814452845856:0.2666666666666664) arc (-156.6814452845856:-135:0.2666666666666664) -- cycle;
\draw [shift={(0,2.414213562373095)},line width=1pt,color=qqqqff,fill=qqqqff,fill opacity=0.5] (0,0) -- (-135:0.2666666666666664) arc (-135:-31.498117425671264:0.2666666666666664) -- cycle;
\draw [shift={(-1.7071067811865475,0.707106781186548)},line width=1pt,color=qqwuzz,fill=qqwuzz,fill opacity=0.55] (0,0) -- (0:0.2666666666666664) arc (0:23.03190315777632:0.2666666666666664) -- cycle;
\draw [shift={(3.6873208740633436,1.5452588533948983)},line width=1pt,color=fuqqzz,fill=fuqqzz,fill opacity=0.55] (0,0) -- (152.15242174021193:0.2666666666666664) arc (152.15242174021193:206.69438680280095:0.2666666666666664) -- cycle;
\draw [shift={(3.6873208740633436,1.5452588533948983)},line width=1pt,color=ffvvqq,fill=ffvvqq,fill opacity=0.5] (0,0) -- (26.694386802800977:0.2666666666666664) arc (26.694386802800977:152.1524217402119:0.2666666666666664) -- cycle;
\draw [shift={(3.6873208740633436,1.5452588533948983)},line width=1pt,color=qqwuzz,fill=qqwuzz,fill opacity=0.5] (0,0) -- (-27.847578259788147:0.2666666666666664) arc (-27.847578259788147:26.694386802800977:0.2666666666666664) -- cycle;
\draw [shift={(3.6873208740633436,1.5452588533948983)},line width=1pt,color=qqqqff,fill=qqqqff,fill opacity=0.5] (0,0) -- (-153.30561319719905:0.2666666666666664) arc (-153.30561319719905:-27.847578259788126:0.2666666666666664) -- cycle;
\draw [line width=1pt] (0,0)-- (1,0);
\draw (0.5,0) node[below] {$C$};
\draw [line width=1pt] (1,0)-- (1.7071067811865475,0.7071067811865475);
\draw (1.35,0.35) node[anchor = north west] {$E$};
\draw [line width=1pt] (1.7071067811865475,0.7071067811865475)-- (1.7071067811865475,1.7071067811865472);
\draw (1.7,1.2) node[right] {$B$};
\draw [line width=1pt] (1.7071067811865475,1.7071067811865472)-- (1,2.414213562373095);
\draw [line width=1pt] (1,2.414213562373095)-- (0,2.414213562373095);
\draw (0.5,2.41) node[above] {$A$};
\draw [line width=1pt] (0,2.414213562373095)-- (-0.7071067811865475,1.7071067811865477);
\draw (-0.35,2.05) node[anchor = south east] {$F$};
\draw [line width=1pt] (-0.7071067811865475,1.7071067811865477)-- (-0.7071067811865477,0.7071067811865478);
\draw [line width=1pt] (-0.7071067811865477,0.7071067811865478)-- (0,0);
\draw (-0.35,0.35) node[anchor=north east] {$D$};
\draw [line width=1pt] (-0.7071067811865475,1.7071067811865477)-- (-1.7071067811865472,1.7071067811865477);
\draw (-1.2,1.7) node[above] {$C$};
\draw [line width=1pt] (-1.7071067811865472,1.7071067811865477)-- (-1.7071067811865475,0.707106781186548);
\draw (-1.7,1.2) node[left] {$B$};
\draw [line width=1pt] (-1.7071067811865475,0.707106781186548)-- (-0.7071067811865477,0.7071067811865478);
\draw (-1.2,0.7) node[below] {$A$};
\draw [line width=1pt] (1.7071067811865475,1.7071067811865472)-- (2.414213562373095,2.414213562373094);
\draw (2.05,2.05) node[anchor = north west] {$F$};
\draw [line width=1pt] (2.414213562373095,2.414213562373094)-- (1.7071067811865481,3.121320343559642);
\draw (2.05,2.75) node[anchor = south west] {$D$};
\draw [line width=1pt] (1.7071067811865481,3.121320343559642)-- (1,2.414213562373095);
\draw (1.35,2.75) node[anchor = south east] {$E$};
\draw [line width=1pt,color=ccqqqq] (0,2.414213562373095)-- (0.4065615329494075,2.1650907090516776);
\draw [line width=1pt,color=ccqqqq] (1,0)-- (0.5455778630852526,0.3045509912303418);
\draw [line width=1pt,color=qqwuqq] (-1.7071067811865475,0.707106781186548)-- (-1.2404437666218397,0.9055002146743686);
\draw [line width=1pt,color=qqwuqq] (1.8520171119035886,2.1718776588779605)-- (2.414213562373095,2.414213562373094);
\draw [line width=1pt,dash pattern=on 3pt off 3pt,color=ccqqqq] (0.16,0.5629629629629622)-- (0.5455778630852526,0.3045509912303418);
\draw [line width=1pt,dash pattern=on 3pt off 3pt,color=ccqqqq] (0.4065615329494075,2.1650907090516776)-- (0.8355555555555525,1.9022222222222196);
\draw [line width=1pt,dash pattern=on 3pt off 3pt,color=qqwuqq] (1.5288888888888854,2.03259259259259)-- (1.8520171119035886,2.1718776588779605);
\draw [line width=1pt,dash pattern=on 3pt off 3pt,color=qqwuqq] (-1.2404437666218397,0.9055002146743686)-- (-0.8474074074074097,1.0725925925925912);
\draw [line width=1pt,color=ccqqqq] (0.24799871403520124,2.2622509632253305)-- (0.26074074074073794,2.1274074074074045);
\draw [line width=1pt,color=ccqqqq] (0.24799871403520124,2.2622509632253305)-- (0.37925925925925635,2.2874074074074042);
\draw [line width=1pt,color=ccqqqq] (0.5682218725919673,0.28937511184312564)-- (0.5807407407407378,0.16);
\draw [line width=1pt,color=ccqqqq] (0.5682218725919673,0.28937511184312564)-- (0.7229629629629599,0.2962962962962959);
\draw [line width=1pt,color=qqwuqq] (1.955555555555552,2.311111111111108)-- (2.0626645088947946,2.262677640318845);
\draw [line width=1pt,color=qqwuqq] (2.0626645088947946,2.262677640318845)-- (2.026666666666663,2.1451851851851824);
\draw [line width=1pt,color=qqwuqq] (-1.3451851851851873,0.9540740740740727)-- (-1.240526898666709,0.905464872572467);
\draw [line width=1pt,color=qqwuqq] (-1.240526898666709,0.905464872572467)-- (-1.28,0.788148148148147);
\draw [line width=1pt,color=ccqqqq] (4.761481481481473,0.9777777777777737)-- (2.877037037037031,1.973333333333326);
\draw [line width=1pt,color=qqwuqq] (2.859259259259253,1.1288888888888844)-- (4.432592592592585,1.92);
\draw [line width=1pt,color=qqwuqq] (4.244353781552244,1.8253488464260421)-- (4.094814814814807,1.8666666666666598);
\draw [line width=1pt,color=qqwuqq] (4.244353781552244,1.8253488464260421)-- (4.174814814814807,1.68);
\draw [line width=1pt,color=qqwuqq] (3.2925773059683836,1.3467719745222904)-- (3.0903703703703638,1.36);
\draw [line width=1pt,color=qqwuqq] (3.2925773059683836,1.3467719745222904)-- (3.179259259259253,1.1822222222222176);
\draw [line width=1pt,color=ccqqqq] (4.189712944134083,1.2798441748669604)-- (4.281481481481474,1.12);
\draw [line width=1pt,color=ccqqqq] (4.189712944134083,1.2798441748669604)-- (4.38814814814814,1.2888888888888839);
\draw [line width=1pt,color=ccqqqq] (3.0534435536440943,1.8801374377673303)-- (3.1259259259259196,1.7422222222222157);
\draw [line width=1pt,color=ccqqqq] (3.0534435536440943,1.8801374377673303)-- (3.214814814814808,1.8933333333333264);
\draw [color=ccqqqq](4.352592592592584,1.1) node[anchor=north west] {$\mathbf{\alpha}$};
\draw [color=qqwuqq](2.9,1.2) node[anchor=north west] {$\mathbf{\beta}$};
\draw (3,0.7) node[anchor=north west] {$\Int(\alpha,\beta)=-1$};
\begin{scriptsize}
\draw [fill=uuuuuu] (3.6873208740633436,1.5452588533948983) circle (2pt);
\end{scriptsize}
\end{tikzpicture}
\caption{In the example of this picture, the curves $\alpha$ and $\beta$ intersect once at the singularity. Furthermore, the sign is given by $\Int(\alpha,\beta) = -1$.}
\label{fig:explanation_singular_intersection}
\end{figure}
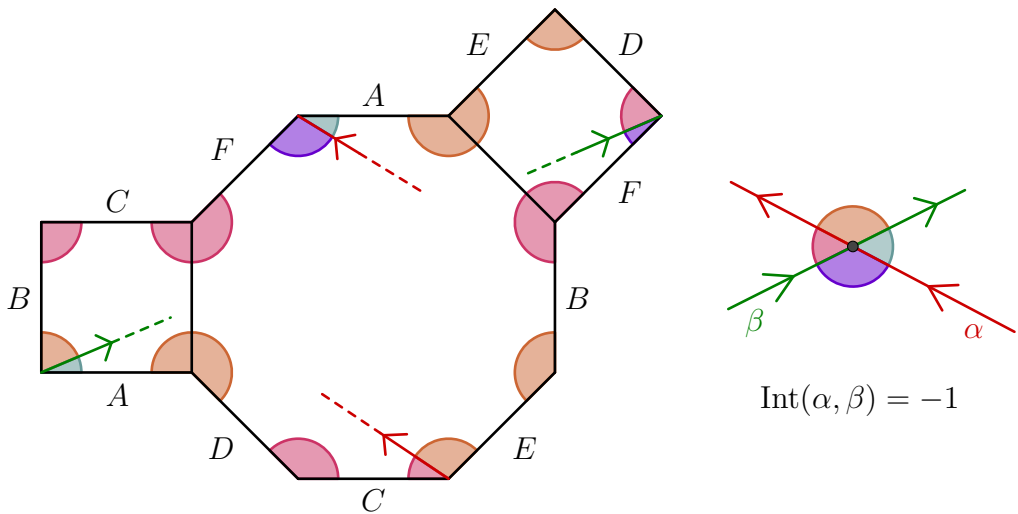

In the case of Bouw-M\"oller surfaces, this allows to prove two useful results about how some saddle connections intersect. 

\begin{Prop}\label{prop:intersection_systoles}
Assume $m$ and $n$ are coprime. Then there is a pair of intersecting systoles in $S_{m,n}$. More precisely, the horizontal systole $\alpha$ corresponding to the side $v_0$ of $P(0)$ intersects at least one systole having a direction making an angle $\frac{\pi}{n}$ with the horizontal.
\end{Prop}


\begin{proof}
As mentioned, the systoles of $S_{m,n}$ are exactly the sides of $P(0)$ and $P(m-1)$ (Proposition \ref{prop:systoles}), and in particular they cannot intersect outside the singularity.
In order to compute the algebraic intersection of systoles at the singularity, we need to understand in which order we encounter the sides of $P(0)$ and $P(m-1)$ while turning around the singularity.

The strategy of the proof is the following. 
Denote as $\alpha$ the side $v_0$ of the polygon $P(0)$, which is horizontal by Remark \ref{rk:v0horizontal} and it is a systole.
Since there is only one singularity, when looking at all the sides of polygons in a neighbourhood of the singularity, we will see two occurrences of each side corresponding to the two endpoints of the side. 
Now choose the endpoint of $\alpha$ on the left and start turning around the singularity counter-clockwise. 
If we look at the two occurrences of $\alpha$, they will be separated by a certain number of other sides and we will define a certain $\beta$ to be the side that we encounter when we are in some sense half way from meeting $\alpha$ again. 
In the following we will prove that $\alpha$ and $\beta$ are two systoles which intersect and which form an angle of $\frac{\pi}{n}$.

Recall that the polygons $P(i)$ for $i \neq 0, m-1$ are semi-regular $2n$-gons, so their internal angles are all equal to $\frac{(n-1)\pi}{n}$, while $P(0)$ and $P(m-1)$ are both regular $n$-gons so their internal angles are all $\frac{(n-2)\pi}{n}$. 

Now start turning counter-clockwise around the singularity from the left endpoint of the horizontal systole $\alpha$. 
The next side we encounter is the consecutive side in $P(0)$ (which is identified to a side in $P(1)$) having the left endpoint of $\alpha$ as an endpoint, and which we encounter after having turned by an angle $\frac{(n-2)\pi}{n}$. 
Then, by Remark \ref{rk:alternation}, we cross an angular sector in $P(1)$, then one in $P(2)$, and so on until $P(m-2)$.
In the sector contained in $P(m-2)$, if we start from a common side of $P(m-2)$ and $P(m-3)$, then the next side we encounter going around the singularity is a common side of $P(m-2)$ and $P(m-1)$.
This side is also a systole and we encounter it after having turned by a total angle $\frac{(n-2)\pi}{n} + (m-2) \frac{(n-1)\pi}{n}$, where the first term comes from the sector in $P(0)$ and the second term comes from the $m-2$ sectors in $P(1)$ to $P(m-2)$. 
Next, following the same reasoning, we encounter a second side of $P(m-1)$ after having turned by an angle $(m-2) \frac{(n-1)\pi}{n} + 2\frac{(n-2)\pi}{n}$ in total (so, an extra $\frac{(n-2)\pi}{n}$ from the first systole we encounter in $P(m-1)$).
We can then continue turning around the singularity and meeting sectors in $P(m-2)$ to $P(1)$ until we are back in $P(0)$ after an angle of $\frac{2(n-2)\pi}{n} + 2(m-2) \frac{(n-1)\pi}{n}$. 
\newline

Repeating this argument, we can see that after having turned by a total angle of $2l\frac{(n-2)\pi}{n} + 2l(m-2) \frac{(n-1)\pi}{n}$, with $l$ integer, we are back to a systole which is a side of $P(0)$ and if we add one more angular sector in $P(0)$, that is after a total angle $(2l+1)\frac{(n-2)\pi}{n} + 2l(m-2) \frac{(n-1)\pi}{n}$, we encounter another side of $P(0)$ (which is also a systole).
Moreover, when we are halfway and we turn by $l\frac{(n-2)\pi}{n} + l(m-2) \frac{(n-1)\pi}{n}$, we encounter a systole (which is either a side of $P(0)$ or a side of $P(m-1)$ depending on the parity of $l$).

Now, the next occurrence of $\alpha$ in the neighbourhood of the singularity must be at the end of a sector in $P(0)$, that is after an angle of the form $(2l_0+1)\frac{(n-2)\pi}{n} + 2l_0(m-2) \frac{(n-1)\pi}{n}$ for a certain integer $l_0$. In fact, $l_0$ must satisfy
\[ 2l_0(m-2) \frac{(n-1)\pi}{n} + (2l_0+1) \frac{(n-2)\pi}{n} \equiv \pi \mod 2\pi \] 
because, using the translation structure, one needs to turn around an angle which is an odd multiple of $\pi$ in order to come back to the same saddle connection, but at the other endpoint. 

This implies that
\[ 2l_0 (m-2) \frac{(n-1)\pi}{n} + 2l_0 \frac{(n-2)\pi}{n} \equiv \pi-\frac{(n-2)\pi}{n} \mod 2\pi .\]
and hence
\begin{equation}\label{eq:pi/n}
 l_0(m-2) \frac{(n-1)\pi}{n} + l_0 \frac{(n-2)\pi}{n} \equiv \frac{\pi}{n} \mod \pi. 
\end{equation}
Now, we define $\beta$ to be the systole that we will be encountering after an angle of $l_0\frac{(n-2)\pi}{n} + l_0(m-2) \frac{(n-1)\pi}{n}$.
In particular, Equation \eqref{eq:pi/n} gives that $\beta$ makes an angle $\frac{\pi}{n}$ with the horizontal. 

Furthermore, the above argument applied to $\beta$ tells by symmetry that the two occurences of $\beta$ in the neighbourhood of the singularity are separated by an angle $2l_0 (m-2) \frac{(n-1)\pi}{n} + (2l_0+1) \frac{(n-2)\pi}{n}$. 
This implies that turning around the singularity, we first see $\alpha$, then $\beta$ (after an angle $l_0 (m-2) \frac{(n-1)\pi}{n} + l_0 \frac{(n-2)\pi}{n}$), then $\alpha$ (after an angle $2l_0 (m-2) \frac{(n-1)\pi}{n} + (2l_0+1) \frac{(n-2)\pi}{n}$) and finally again $\beta$ (after an angle $3l_0 (m-2) \frac{(n-1)\pi}{n} + (3l_0+1) \frac{(n-2)\pi}{n}$). 
We conclude that $\Int(\alpha,\beta)=\pm 1$ (depending on the choice of the orientation of both $\alpha$ and $\beta$).
\end{proof}

Similarly, we can compute the algebraic intersection of the closed saddle connections given by the diagonals of the square $P(0)$ in the surfaces $S_{m,4}$ for $m$ odd.

\begin{Prop}\label{prop:intersection_diagonals}
Let $m \geq 3$ odd. On the surface $S_{m,4}$, the diagonals of the square $P(0)$ have algebraic intersection $2$ if $m \equiv 3 \mod 4$ and $0$ otherwise. 
\end{Prop}
\begin{proof}
As in Figure \ref{fig:intersection_diagonals}, we denote the two diagonals as $\alpha$ (positive slope) and $\beta$ (negative slope) and four sectors $\Sigma_i$ for $i=1,2,3,4$ as being a neighbourhood of the singularity intersecting the square $P(0)$ in the corners, moving counter-clockwise starting from the top-right corner.
In order to compute the algebraic intersection number at the singularity, we need to understand in which order one would encounter the sectors $\Sigma_i$ around the singularity. 
This would determine the order of occurrences of $\alpha$ and $\beta$ in the separatrix diagram and hence their intersection at the singularity. 
One would then just need to add the (positive) intersection in the centre of the square. 

The polygons $P(i)$ for $i \neq 0, m-1$ are semi-regular octagons, so the internal angles are all equal to $\frac{3\pi}{4}$, while $P(m-1)$ is a square, so its internal angles are $\frac{\pi}{2}$. 
In order to detect in which order we would see the $\Sigma_i$s around the singularity, we need to determine what is the angle of a sector sitting in between two of them. 
This, plus the gluing rules of translation surfaces, will tell us uniquely which sector follows a given $\Sigma_i$. 
Recalling Remark \ref{rk:alternation}, one can easily see that going around the singularity and starting from $P(0)$, we will encounter one vertex in each polygon $P(i)$ for $i=1$ to $m-2$, then one vertex in $P(m-1)$ and then again one vertex in each $P(i)$ for $i=m-2$ to $1$ until getting back to $P(0)$. 
This means that between each pair $\Sigma_i$, $\Sigma_{i+1}$, there is a cone angle of $2(m-2)\cdot \frac{3\pi}{4} + \frac{\pi}{2}$.

In particular, writing $m=4k+r$, we turn around the singularity by an angle $2(4k+r-2)\cdot\frac{3\pi}{4} + \frac{\pi}{2} = (3k-2)\cdot 2\pi + (3r+3)\frac{\pi}{2}$. 
Recalling that $r \neq 0,2$ because $m$ is odd, this gives an integer number of turns for $r=3$, and a half turn for $r=1$. 
Hence the order of the sectors is: 
\begin{align*}
    &\Sigma_1,\Sigma_2,\Sigma_3,\Sigma_4 & &\textit{for} &r&=3, \\ 
    &\Sigma_1,\Sigma_4,\Sigma_3,\Sigma_2 & &\textit{for} &r&=1.
\end{align*}
This would then give a positive intersection at the singularity in the first case and a negative intersection at the singularity in the second case.
\end{proof}

\begin{figure}
\centering

\definecolor{wwwwww}{rgb}{0.4,0.4,0.4}
\definecolor{ffvvqq}{rgb}{1,0.3333333333333333,0}
\definecolor{qqqqff}{rgb}{0,0,1}
\definecolor{qqwuqq}{rgb}{0,0.39215686274509803,0}
\definecolor{ccqqqq}{rgb}{0.8,0,0}
\definecolor{cqcqcq}{rgb}{0.7529411764705882,0.7529411764705882,0.7529411764705882}
\begin{tikzpicture}[line cap=round,line join=round,>=triangle 45,x=1cm,y=1cm]
\clip(-0.5,-1) rectangle (5.5,4);
\draw [shift={(3,3)},line width=1pt,color=qqqqff,fill=qqqqff,fill opacity=0.1] (0,0) -- (180:0.48185603807257593) arc (180:270:0.48185603807257593) -- cycle;
\draw [shift={(3,0)},line width=1pt,color=ccqqqq,fill=ccqqqq,fill opacity=0.1] (0,0) -- (90:0.48185603807257593) arc (90:180:0.48185603807257593) -- cycle;
\draw [shift={(0,0)},line width=1pt,color=ffvvqq,fill=ffvvqq,fill opacity=0.1] (0,0) -- (0:0.48185603807257593) arc (0:90:0.48185603807257593) -- cycle;
\draw [shift={(0,3)},line width=1pt,color=qqwuqq,fill=qqwuqq,fill opacity=0.10000000149011612] (0,0) -- (-90:0.48185603807257593) arc (-90:0:0.48185603807257593) -- cycle;
\draw [shift={(3,3)},line width=1pt,dash pattern=on 3pt off 3pt,color=wwwwww,fill=wwwwww,fill opacity=0.1] (0,0) -- (-90:0.48185603807257593) arc (-90:45:0.48185603807257593) -- cycle;
\draw [shift={(3,0)},line width=1pt,dash pattern=on 3pt off 3pt,color=wwwwww,fill=wwwwww,fill opacity=0.1] (0,0) -- (-45:0.48185603807257593) arc (-45:90:0.48185603807257593) -- cycle;
\draw [line width=1pt] (0,0)-- (3,0);
\draw [line width=1pt] (3,0)-- (3,3);
\draw [line width=1pt] (3,3)-- (0,3);
\draw [line width=1pt] (0,3)-- (0,0);
\draw [line width=1pt,dash pattern=on 3pt off 3pt] (3,3)-- (4,4);
\draw [line width=1pt,dash pattern=on 3pt off 3pt] (3,0)-- (4,-1);
\draw [line width=1pt,color=ccqqqq] (0,0)-- (3,3);
\draw [line width=1pt,color=qqwuqq] (3,0)-- (0,3);
\draw [color=qqqqff](2.3,2.55) node[anchor=north west] {$\Sigma_1$};
\draw [color=qqwuqq](0,2.55) node[anchor=north west] {$\Sigma_2$};
\draw [color=ffvvqq](0.5,0.65) node[anchor=north west] {$\Sigma_3$};
\draw [color=ccqqqq](2.3,1.2) node[anchor=north west] {$\Sigma_4$};
\draw [line width=1pt,color=qqwuqq] (0.8,2.2)-- (0.9,1.9);
\draw [line width=1pt,color=qqwuqq] (0.8,2.2)-- (1.1,2.1);
\draw [line width=1pt,color=ccqqqq] (2.2,2.2)-- (1.9,2.1);
\draw [line width=1pt,color=ccqqqq] (2.2,2.2)-- (2.1,1.9);
\draw [color=qqwuqq](0.9,2.6) node[anchor=north west] {$\beta$};
\draw [color=ccqqqq](2.059381320642478,2.1) node[anchor=north west] {$\alpha$};
\draw (3.5,1.8) node[anchor=north west] {$P(1)$};
\draw (4.5,1.7) node[anchor=north west] {$\cdots$};
\end{tikzpicture}
\caption{If $n=4$ and $m$ odd, the cyclic ordering of $\Sigma_1, \Sigma_2, \Sigma_3$ and $\Sigma_4$ determines the intersection at the singularity between the diagonals $\alpha$ and $\beta$ of $P(0)$.}
\label{fig:intersection_diagonals}
\end{figure}
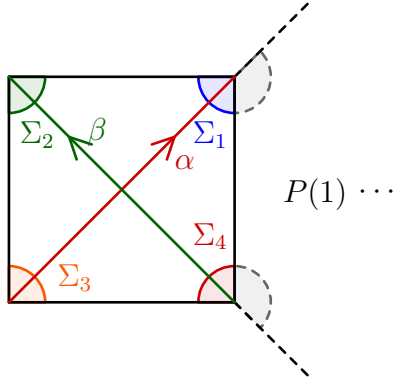

%% file: General_v2.tex
The purpose of this section is to prove Theorem \ref{theo:mainCONVEX}: in the following we consider a translation surface $X$ constructed from a collection of polygons $(P_i)_{i \in I}$ satisfying (P1) (the polygons are convex with obtuse --or right-- angles) and whose identifications of sides satisfy (P2) (sides of the same polygon are not identified together). We recall that $l_0$ denotes the length of the smallest sides of the polygons. The purpose of this section is to show that for any two closed curves $\gamma$ and $\delta$ on $X$, we have

\[\frac{\mathrm{Int}(\gamma,\delta)}{l(\gamma)l(\delta)} \leq \frac{1}{l_0^2},\]
and investigate the equality cases.

\subsection{Outline of the proof}  As the algebraic intersection does not change in a homology class, we first notice that it suffices to prove this result for closed geodesics, and, since closed geodesics on translation surfaces are homologous to unions of saddle connections (with at most the same length), it suffices to prove this result for closed curves which are union of saddle connections. We will hence deal with closed geodesics made of saddle connections\footnote{In fact, we could perform the same proof for simple closed geodesics, but it will be more convenient to deal with saddle connections here.}.
In general, saddle connections are not closed, so we cannot define the algebraic intersection of two saddle connections, as we cannot calculate whether they intersect of not at the singularities.
Nonetheless, we can still show:

\begin{Theo}\label{theo:genSC}
Under the hypotheses (P1) and (P2), for any two saddle connections $\alpha$ and $\beta$ on $X$, we have:
\[ \frac{|\alpha \cap \beta| +1}{l(\alpha) l(\beta)} \leq \frac{1}{l_0^2} \]
where $| \alpha \cap \beta|$ denotes the number of non-singular intersections points if $\alpha \neq \beta$, and is set to be $0$ if $\alpha = \beta$.

Further, if the angles are all strictly obtuse, equality holds if and only if $\alpha$ and $\beta$ are sides of the polygons both having length $l_0$. If there are right angles, equality holds if and only if:
\begin{itemize}
    \item $\alpha$ and $\beta$ are sides or also diagonals of the polygons, having both length $l_0$,
    \item $\alpha$ and $\beta$ are both diagonals of length $\sqrt{2} l_0$ and such that $|\alpha \cap \beta| = 1$,
    \item up to symmetry, $\alpha$ is a side or a diagonal of length $l_0$ and $\beta$ is a geodesic of length $2l_0$, contained the union of exactly two polygons, and $|\alpha \cap \beta | = 1$.
\end{itemize}

\end{Theo}

\begin{proof}[Proof of Theorem \ref{theo:mainCONVEX} using Theorem \ref{theo:genSC}]
Let $\gamma=\gamma_1 \cup \dots \cup \gamma_k$ and $\delta=\delta_1 \cup \dots \cup \delta_l$ be two closed curves decomposed as a union of saddle connections.
Note that $k$ (resp. $l$) is also the number of singularities that $\gamma$ (resp. $\delta$) crosses.
Denote by $|\gamma_i \cap \delta_j|$ the number of intersections of the two saddle connections $\gamma_i$ and $\delta_j$ outside of the singularity.
Then the intersection number of $\gamma$ and $\delta$ at the singularities can be at most $\min(k,l)$ and 
\begin{align}\label{eq:intersection_kl_singularities}
\Int (\gamma,\delta) &\leq \sum_{\substack{1 \leq i \leq k \\1 \leq j \leq l}}
|\gamma_i \cap \delta_j|+\min(k,l)  \\
& \leq \sum_{\substack{1 \leq i \leq k \\1 \leq j \leq l}}
|\gamma_i \cap \delta_j|+kl
= \sum_{\substack{1 \leq i \leq k \\1 \leq j \leq l}}
\left(|\gamma_i \cap \delta_j|+1\right). \label{eq:intersection_kl_singularities2}
\end{align}
with equality only if $k=l=1$. Hence, by Theorem \ref{theo:genSC},
\[ 
\Int(\gamma,\delta) \leq \sum_{\substack{1 \leq i \leq k \\1 \leq j \leq l}}
\frac{l(\gamma_i) l(\delta_j)}{l_0^2}
= \frac{1}{l_0^2}\left( \sum_{\substack{1 \leq i \leq k}} l(\gamma_i) \right)  
\left( \sum_{\substack{1 \leq j \leq l}} l(\delta_j) \right) 
= \frac{l(\gamma)l(\delta)}{l_0^2}.
\]
Further, equality holds only if we have equality in both \eqref{eq:intersection_kl_singularities2} and in Theorem \ref{theo:genSC}. 
(The condition $k=l=1$ just tells us that $\gamma$ and $\delta$ need to be (closed) saddle connections.) 
\end{proof}

We are now left to prove Theorem \ref{theo:genSC}. Given two saddle connections $\alpha$ and $\beta$ on $X$, we define the \emph{polygonal decomposition} of $\alpha$ (resp. $\beta$) by cutting $\alpha$ (resp. $\beta$) each time it goes from a polygon to another. This gives a decomposition $\alpha = \alpha_1 \cup \cdots \cup \alpha_k$ (resp. $\beta= \beta_1 \cup \cdots \cup \beta_l$) into smaller (non-closed) segments. 
We will see that this decomposition allows to estimate simultaneously the number of intersections and the lengths of the segments. 
To estimate the length of the segments we will distinguish two kinds of segments in the polygonal decomposition:

\begin{Def}
\begin{enumerate}[label=(\roman*)]
\item a \emph{non-adjacent segment} is a segment going from a side of a polygon $P$ to a non-adjacent side of $P$, or a segment having one of its endpoints as a vertex of a polygon.
\item an \emph{adjacent segment} is a segment going from the interior of a side $e$ of a polygon $P$ to the interior of a side of $P$ adjacent to $e$.
\end{enumerate}
\end{Def}
Remark that, by definition, the segments $\alpha_1$ and $\alpha_k$ (resp. $\beta_1$ and $\beta_l$) are non-adjacent segments. Let us also notice that if $k=1$, then $\alpha$ is either a side or a diagonal of a polygon: in this case we have a single non-adjacent segment according to the definition, and we will deal with this case separately.\newline

We will proceed to the proof of Theorem \ref{theo:genSC} as follows.
First, we will assume that the saddle connection in question is not a side or a diagonal of one of the polygons. 
This will be the case in the next sections until it is explicitly mentioned. 
In Section \ref{sec:length_adjacent_na} we study the length of pieces of saddle connections using its polygonal decomposition. 
Then, in Section \ref{sec:consecutive_adjacent_ppties} we study some properties of sequences of consecutive adjacent segments, as this will allow us to study the intersections of pieces of saddle connections using their polygonal decomposition in Section \ref{sec:study_intersections}. Finally, in Section \ref{sec:proofgenSC} we conclude the proof of Theorem \ref{theo:genSC} and we deal with the case where one of the saddle connections is a side or a diagonal of a polygon (Section \ref{sec:diagANDsides}).\newline

From now on and until the end of Section \ref{sec:proofgenSC}, we consider a translation surface $X$ satisfying hypotheses (P1) and (P2). We consider two saddle connections $\alpha$ and $\beta$ on $X$ and denote by $\alpha_{1} \cup \cdots \cup \alpha_{k}$ (resp. $\beta_1 \cup \cdots \cup \beta_l$) the polygonal decomposition of $\alpha$ (resp. $\beta$). We use the following:

\begin{Nota}
Given a segment $\alpha_i$ in the polygonal decomposition of the (oriented) saddle connection $\alpha$, we will denote $\alpha_{i}^-$ and $\alpha_{i}^+$ the endpoints of $\alpha_i$, so that the orientation of $\alpha$ takes us from $\alpha_{i}^-$ to $\alpha_{i}^+$.
\end{Nota}

\subsection{Study of the lengths}\label{sec:length_adjacent_na}
\subsubsection{Length of adjacent and non-adjacent segments}
In this paragraph, we use the hypothesis on the polygons (which are convex with obtuse or right angles) to obtain a first estimate on the length of the segments, namely:

\begin{Lem}[Length of adjacent and non-adjacent segments]\label{lem:length_adjacent_na}
We have:
\begin{enumerate}
\item The length of a non-adjacent segment is at least $l_0$.
\item The length of a pair of consecutive adjacent segments is greater than $l_0$.
\end{enumerate}
\end{Lem}

The Lemma will follow from elementary convex geometry. More precisely, we start by showing:
\begin{Lem}\label{lem:convex_length}
Let $P$ be a convex polygon whose angles are all obtuse or right and denote by $l_0$ the length of the smallest side of $P$. 
Let $A$ and $B$ be two points on the boundary of the polygon which do not lie on the same side of the polygon or on adjacent sides. 
Then the distance between $A$ and $B$ is at least $l_0$.
\end{Lem}

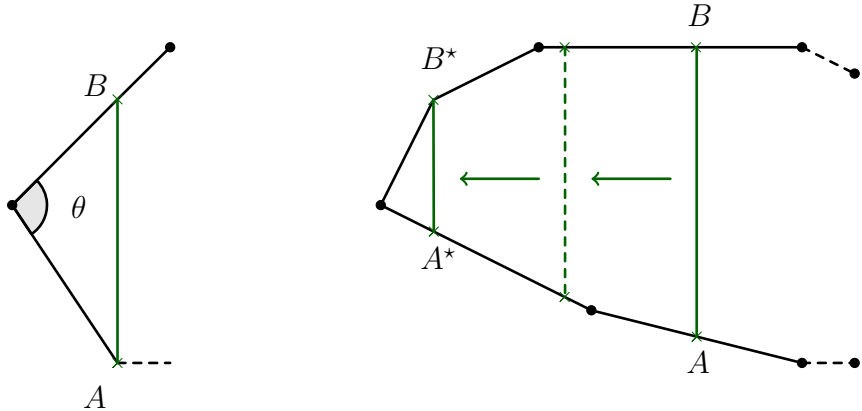
\begin{figure}
\center
\definecolor{qqwuqq}{rgb}{0,0.39215686274509803,0}
\begin{tikzpicture}[line cap=round,line join=round,>=triangle 45,x=1cm,y=1cm, scale=0.7]
\clip(-1.384839954435098,-1) rectangle (16.5,7);
\draw [shift={(-1,3)},line width=1pt,fill=black,fill opacity=0.1] (0,0) -- (-56.309932474020215:0.6609417338842879) arc (-56.309932474020215:45.1377739674746:0.6609417338842879) -- cycle;
\draw [line width=1pt] (1,0)-- (-1,3);
\draw [line width=1pt] (-1,3)-- (1,5.009641640385679);
\draw [line width=1pt] (1,5.009641640385679)-- (2,6);
\draw [line width=1pt,dash pattern=on 3pt off 3pt] (1,0)-- (2,0);
\draw [line width=1pt] (9,6)-- (14,6);
\draw [line width=1pt] (10,1)-- (14,0);
\draw [line width=1pt] (10,1)-- (6,3);
\draw [line width=1pt] (9,6)-- (7,5);
\draw [line width=1pt] (7,5)-- (6,3);
\draw [line width=1pt,color=qqwuqq] (12,6)-- (12,0.5);
\draw [line width=1pt,dash pattern=on 3pt off 3pt,color=qqwuqq] (9.5,6)-- (9.5,1.2542727253021386);
\draw [line width=1pt,color=qqwuqq] (7,5)-- (7.007002801120449,2.4964985994397755);
\draw [line width=1pt,dash pattern=on 3pt off 3pt] (14,6)-- (15,5.5);
\draw [line width=1pt,dash pattern=on 3pt off 3pt] (14,0)-- (15,0);
\draw [line width=1pt,color=qqwuqq] (1,5.009641640385679)-- (1,0);
\draw [line width=1pt,color=qqwuqq,-to] (11.5,3.5) -- (10,3.5);
\draw [line width=1pt,color=qqwuqq,to-] (7.5,3.5) -- (9,3.5);
\draw (-0.0849878777959985,3.3793186968044298) node[anchor=north west] {$\theta$};
\draw (0.14,-0.23382944842969045) node[anchor=north west] {$A$};
\draw (0.14,5.7) node[anchor=north west] {$B$};
\draw (11.61,7.014498233168027) node[anchor=north west] {$B$};
\draw (11.6,0.4271122854545999) node[anchor=north west] {$A$};
\draw (6.55,2.4) node[anchor=north west] {$A^{\star}$};
\draw (6.55,6.2) node[anchor=north west] {$B^{\star}$};
\begin{scriptsize}
\draw [color=qqwuqq] (1,0)-- ++(-2.5pt,-2.5pt) -- ++(5pt,5pt) ++(-5pt,0) -- ++(5pt,-5pt);
\draw [fill=black] (-1,3) circle (2.5pt);
\draw [color=qqwuqq] (1,5.009641640385679)-- ++(-2.5pt,-2.5pt) -- ++(5pt,5pt) ++(-5pt,0) -- ++(5pt,-5pt);
\draw [fill=black] (2,6) circle (2.5pt);
\draw [fill=black] (9,6) circle (2.5pt);
\draw [fill=black] (14,6) circle (2.5pt);
\draw [fill=black] (10,1) circle (2.5pt);
\draw [fill=black] (14,0) circle (2.5pt);
\draw [fill=black] (6,3) circle (2.5pt);
\draw [color=qqwuqq] (7,5)-- ++(-2.5pt,-2.5pt) -- ++(5pt,5pt) ++(-5pt,0) -- ++(5pt,-5pt);
\draw [color=qqwuqq] (11.984615384615385,6)-- ++(-2.5pt,-2.5pt) -- ++(5pt,5pt) ++(-5pt,0) -- ++(5pt,-5pt);
\draw [color=qqwuqq] (12,0.5)-- ++(-2.5pt,-2.5pt) -- ++(5pt,5pt) ++(-5pt,0) -- ++(5pt,-5pt);
\draw [color=qqwuqq] (9.492657342657346,6)-- ++(-2.5pt,-2.5pt) -- ++(5pt,5pt) ++(-5pt,0) -- ++(5pt,-5pt);
\draw [color=qqwuqq] (9.491454549395723,1.2542727253021386)-- ++(-2.5pt,-2.5pt) -- ++(5pt,5pt) ++(-5pt,0) -- ++(5pt,-5pt);
\draw [color=qqwuqq] (7.007002801120449,2.4964985994397755)-- ++(-2.5pt,-2.5pt) -- ++(5pt,5pt) ++(-5pt,0) -- ++(5pt,-5pt);
\draw [fill=black] (15,5.5) circle (2.5pt);
\draw [fill=black] (15,0) circle (2.5pt);
\end{scriptsize}
\end{tikzpicture}
\caption{The pushing process of Lemma \ref{lem:convex_length}.}
\label{fig:convex_polygon}
\end{figure}

\begin{proof}[Proof of Lemma \ref{lem:convex_length}]
\begin{enumerate}[label=(\roman*)]
\item Let us start with the case where one of the points (say $A$) is a vertex of the polygon and the other point $B$ lies on one of the sides adjacent to the sides containing $A$.
This means that there is a side $e$ which has $A$ as an endpoint and is adjacent to the side $e'$ containing $B$ in its interior (see the left hand side of Figure \ref{fig:convex_polygon}). 
As the angle $\theta$ between $e$ and $e'$ is obtuse or right, the distance between $A$ and $B$ is at least the length of the side $e$.
By definition, that is at least $l_0$.
\item In the general case, given two points $A$ and $B$ on the boundary of the polygon which do not lie in the same side of the polygon or in adjacent sides, we get back to case $(i)$ by ``pushing'' the points while decreasing the length. One way to do this is to draw the parallel lines to $AB$. 
Since the polygon is convex, there is at least one ''pushing'' direction for which we decrease the length, as in the right of Figure \ref{fig:convex_polygon}. This process allows to get points $A^{\star}$ and $B^{\star}$ in the configuration of case $(i)$ such that:
\[ d(A,B) \geq d(A^{\star},B^{\star}) \geq l_0. \]
\end{enumerate}
\end{proof}

\begin{proof}[Proof of Lemma \ref{lem:length_adjacent_na}]
The first part of the Lemma follows directly from Lemma \ref{lem:convex_length} applied to the polygon containing the non-adjacent segment. 

For the second part of Lemma \ref{lem:length_adjacent_na}, consider two consecutive adjacent segments $\alpha_i$ and $\alpha_{i+1}$ connected on a side $e$ shared by two polygons $P$ and $P'$. 
Since all the internal angles of polygons forming the surface $X$ are obtuse or right by hypothesis, we are in a configuration as in Figure \ref{fig:22long}, and hence the length $l(\alpha_i \cup \alpha_{i+1})$ will be greater than the length of the side $e$, which must be at least $l_0$. 
\end{proof}

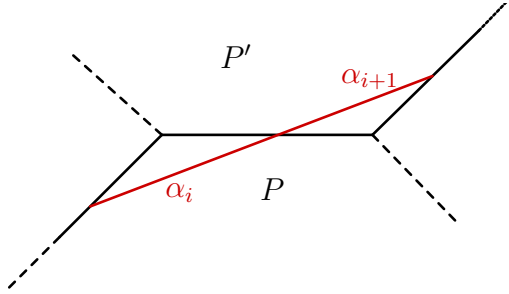
\begin{figure}[h]
\center
\definecolor{ccqqqq}{rgb}{0.8,0,0}
\begin{tikzpicture}[line cap=round,line join=round,>=triangle 45,x=1cm,y=1cm, scale =0.7]
\clip(-5.5,-3) rectangle (5,2.5);
\draw [line width=1pt] (-2,0)-- (2,0);
\draw [line width=1pt] (2,0)-- (4.02,1.98);
\draw [line width=1pt,dash pattern=on 1pt off 1pt] (4.02,1.98)-- (5,3);
\draw [line width=1pt] (-2,0)-- (-3.94,-1.94);
\draw [line width=1pt,dash pattern=on 3pt off 3pt] (-3.94,-1.94)-- (-5,-3);
\draw [line width=1pt,dash pattern=on 3pt off 3pt] (-2,0)-- (-3.76,1.58);
\draw [line width=1pt,dash pattern=on 3pt off 3pt] (2,0)-- (3.64,-1.7);
\draw [line width=1pt,color=ccqqqq] (-3.36,-1.36)-- (3.1414888511148886,1.1188851114888514);
\draw [color=ccqqqq](-2.14,-0.72) node[anchor=north west] {$\alpha_i$};
\draw [color=ccqqqq](1.2,1.46) node[anchor=north west] {$\alpha_{i+1}$};
\draw (-0.34,-0.62) node[anchor=north west] {$P$};
\draw (-1.1,1.92) node[anchor=north west] {$P'$};
\end{tikzpicture}
\caption{Two consecutive adjacent segment have length at least $l_0$.}
\label{fig:22long}
\end{figure}

\subsubsection{Lengths of saddle connections}
Now that, in the previous section, we studied the length of pieces of the polygonal decomposition, we want to use these results to estimate the length of the whole saddle connection.
As suggested by Lemma \ref{lem:length_adjacent_na}, we will do this by grouping consecutive adjacent segment by pairs. This motivates the following
\begin{Nota}
Given a saddle connection $\alpha$, we denote by $p_{\alpha}$ the number of non-adjacent segments in the decomposition of $\alpha$ and $q_{\alpha}$ the maximal number of \emph{pairs} of consecutive adjacent segments we can form in the decomposition of $\alpha$ (see Figure \ref{fig:examples_segments}) 
\end{Nota}
Using this notation, it follows directly from Lemma \ref{lem:length_adjacent_na} that
\begin{equation}\label{eq:length_alpha}
l(\alpha)  \geq (p_{\alpha} + q_{\alpha}) l_0.
\end{equation}

\begin{figure}
\center
\definecolor{qqwuqq}{rgb}{0,0.39215686274509803,0}
\definecolor{ccqqqq}{rgb}{0.8,0,0}
\begin{tikzpicture}[line cap=round,line join=round,>=triangle 45,x=1cm,y=1cm, scale=1.1]
\clip(-4,-0.6) rectangle (5.5,6.5);
\draw [line width=1pt] (0,0)-- (2,0);
\draw [line width=1pt] (2,0)-- (3.414213562373095,1.414213562373095);
\draw [line width=1pt] (3.414213562373095,1.414213562373095)-- (3.414213562373095,3.4142135623730945);
\draw [line width=1pt,dash pattern=on 3pt off 3pt] (3.414213562373095,3.4142135623730945)-- (2,4.82842712474619);
\draw [line width=1pt] (2,4.82842712474619)-- (0,4.82842712474619);
\draw [line width=1pt] (0,4.82842712474619)-- (-1.414213562373095,3.4142135623730954);
\draw [line width=1pt] (-1.4142135623730954,1.4142135623730956)-- (0,0);
\draw [line width=1pt,dash pattern=on 3pt off 3pt] (-1.4142135623730954,1.4142135623730956)-- (-1.414213562373095,3.4142135623730954);
\draw [line width=1pt] (-1.414213562373095,3.4142135623730954)-- (-3.4142135623730945,3.4142135623730954);
\draw [line width=1pt] (-3.4142135623730945,3.4142135623730954)-- (-3.414213562373095,1.414213562373096);
\draw [line width=1pt] (-3.414213562373095,1.414213562373096)-- (-1.4142135623730954,1.4142135623730956);
\draw [line width=1pt] (3.414213562373095,3.4142135623730945)-- (4.82842712474619,4.828427124746188);
\draw [line width=1pt] (4.82842712474619,4.828427124746188)-- (3.4142135623730963,6.242640687119284);
\draw [line width=1pt] (3.4142135623730963,6.242640687119284)-- (2,4.82842712474619);
\draw [color=ccqqqq](1.777037037037039,2.3) node[anchor=north west] {$\alpha_1$};
\draw [color=ccqqqq](-3,2.1) node[anchor=north west] {$\alpha_2$};
\draw [color=ccqqqq](1.4,4.75) node[anchor=north west] {$\alpha_3$};
\draw [color=ccqqqq](3.5,4.5) node[anchor=north west] {$\alpha_4$};
\draw [color=ccqqqq](1.8,3.1) node[anchor=north west] {$\alpha_5$};
\draw [color=ccqqqq](3.5,5.6) node[anchor=north west] {$\alpha_8$};
\draw [color=ccqqqq](-2.5,2.7) node[anchor=north west] {$\alpha_6$};
\draw [color=ccqqqq](1,1.6) node[anchor=north west] {$\alpha_7$};
\draw [color=qqwuqq](2,3.8) node[anchor=north west] {$\beta_1$};
\draw [color=qqwuqq](3.1,4) node[anchor=north west] {$\beta_2$};
\draw [color=qqwuqq](0.75,4.25) node[anchor=north west] {$\beta_3$};
\draw [color=qqwuqq](3.3,5.1) node[anchor=north west] {$\beta_4$};
\draw [color=qqwuqq](0,4.7) node[anchor=north west] {$\beta_5$};
\draw [line width=1pt,color=qqwuqq] (-1.414213562373095,3.4142135623730954)-- (4,4);
\draw [line width=1pt,color=qqwuqq] (-0.828427124746189,4)-- (4.585786437626907,4.585786437626906);
\draw [line width=1pt,color=qqwuqq] (-0.242640687119283,4.585786437626906)-- (2,4.82842712474619);
\draw (-4.116296296296299,2.65) node[anchor=north west] {$A$};
\draw (-2.75,4) node[anchor=north west] {$B$};
\draw (0.7,0) node[anchor=north west] {$B$};
\draw (-2.75,1.4) node[anchor=north west] {$C$};
\draw (0.7,5.5) node[anchor=north west] {$C$};
\draw (-1.4362962962962966,4.657530864197529) node[anchor=north west] {$D$};
\draw (4.123703703703708,4.270864197530862) node[anchor=north west] {$D$};
\draw (2.257037037037039,6.190864197530861) node[anchor=north west] {$E$};
\draw (2.6570370370370395,0.8308641975308648) node[anchor=north west] {$E$};
\draw (4.203703703703708,6.297530864197528) node[anchor=north west] {$F$};
\draw (-1.25,0.75) node[anchor=north west] {$F$};
\draw (3.45,2.65) node[anchor=north west] {$ A$};
\draw (5.577037037037042,7.697530864197527) node[anchor=north west] {$G$};
\draw [line width=1pt,color=ccqqqq] (-1.414213562373095,3.4142135623730954)-- (3.414213562373095,1.86);
\draw [line width=1pt,color=ccqqqq] (2.6643782628889747,5.492805387635165)-- (4.82842712474619,4.828427124746188);
\draw [line width=1pt,color=ccqqqq] (-3.414213562373095,1.86)-- (-2.0293026470448305,1.4142135623730956);
\draw [line width=1pt,color=ccqqqq] (-0.8385180769661109,3.9899090477800785)-- (3.4142135623730945,2.621005050633886);
\draw [line width=1pt,color=ccqqqq] (1.3849109153282648,4.82842712474619)-- (3.989909047780079,3.9899090477800776);
\draw [line width=1pt,color=ccqqqq] (-3.4142135623730954,2.6210050506338867)-- (2.6643782628889747,0.664378262888975);
\draw [line width=1pt,color=ccqqqq] (1.4568028020026245,2.4900673640306286)-- (1.088148148148148,2.817530864197532);
\draw [line width=1pt,color=ccqqqq] (1.4568028020026245,2.4900673640306286)-- (0.9488888888888889,2.4511111111111115);
\draw [line width=1pt,color=qqwuqq] (1.7755555555555569,3.5582716049382754)-- (2.1256213637360566,3.7972031262801216);
\draw [line width=1pt,color=qqwuqq] (2.1256213637360566,3.7972031262801216)-- (1.7320987654321,3.9079012345679067);
\begin{scriptsize}
\draw [color=black] (0,0)-- ++(-2.5pt,-2.5pt) -- ++(5pt,5pt) ++(-5pt,0) -- ++(5pt,-5pt);
\draw [color=black] (2,0)-- ++(-2.5pt,-2.5pt) -- ++(5pt,5pt) ++(-5pt,0) -- ++(5pt,-5pt);
\draw [color=black] (3.414213562373095,1.414213562373095)-- ++(-2.5pt,-2.5pt) -- ++(5pt,5pt) ++(-5pt,0) -- ++(5pt,-5pt);
\draw [color=black] (3.414213562373095,3.4142135623730945)-- ++(-2.5pt,-2.5pt) -- ++(5pt,5pt) ++(-5pt,0) -- ++(5pt,-5pt);
\draw [color=black] (2,4.82842712474619)-- ++(-2.5pt,-2.5pt) -- ++(5pt,5pt) ++(-5pt,0) -- ++(5pt,-5pt);
\draw [color=black] (0,4.82842712474619)-- ++(-2.5pt,-2.5pt) -- ++(5pt,5pt) ++(-5pt,0) -- ++(5pt,-5pt);
\draw [color=black] (-1.414213562373095,3.4142135623730954)-- ++(-2.5pt,-2.5pt) -- ++(5pt,5pt) ++(-5pt,0) -- ++(5pt,-5pt);
\draw [color=black] (-1.4142135623730954,1.4142135623730956)-- ++(-2.5pt,-2.5pt) -- ++(5pt,5pt) ++(-5pt,0) -- ++(5pt,-5pt);
\draw [color=black] (-3.4142135623730945,3.4142135623730954)-- ++(-2.5pt,-2.5pt) -- ++(5pt,5pt) ++(-5pt,0) -- ++(5pt,-5pt);
\draw [color=black] (-3.414213562373095,1.414213562373096)-- ++(-2.5pt,-2.5pt) -- ++(5pt,5pt) ++(-5pt,0) -- ++(5pt,-5pt);
\draw [color=black] (4.82842712474619,4.828427124746188)-- ++(-2.5pt,-2.5pt) -- ++(5pt,5pt) ++(-5pt,0) -- ++(5pt,-5pt);
\draw [color=black] (3.4142135623730963,6.242640687119284)-- ++(-2.5pt,-2.5pt) -- ++(5pt,5pt) ++(-5pt,0) -- ++(5pt,-5pt);
\draw [fill=qqwuqq] (4,4) circle (2pt);
\draw [fill=qqwuqq] (-0.828427124746189,4) circle (2pt);
\draw [fill=qqwuqq] (4.585786437626907,4.585786437626906) circle (2pt);
\draw [fill=qqwuqq] (-0.242640687119283,4.585786437626906) circle (2pt);
\draw [fill=ccqqqq] (3.414213562373095,1.86) circle (2.5pt);
\draw [fill=ccqqqq] (-3.414213562373095,1.86) circle (2pt);
\draw [fill=ccqqqq] (-2.0293026470448305,1.4142135623730956) circle (2pt);
\draw [fill=ccqqqq] (1.3849109153282648,4.82842712474619) circle (2pt);
\draw [fill=ccqqqq] (3.989909047780079,3.9899090477800776) circle (2pt);
\draw [fill=ccqqqq] (-0.8385180769661109,3.9899090477800785) circle (2pt);
\draw [fill=ccqqqq] (3.4142135623730945,2.621005050633886) circle (2pt);
\draw [fill=ccqqqq] (-3.4142135623730954,2.6210050506338867) circle (2pt);
\draw [fill=ccqqqq] (2.6643782628889747,0.664378262888975) circle (2pt);
\draw [fill=ccqqqq] (2.6643782628889747,5.492805387635165) circle (2pt);
\end{scriptsize}
\end{tikzpicture}
\caption{Examples of saddle connections and their polygonal decomposition. Here, the segments $\alpha_2, \alpha_3, \alpha_4, \beta_2$ and $\beta_4$ are adjacent segments while all other segments are non-adjacent, and hence $p_{\alpha} = 5, q_{\alpha} = 1$, $p_{\beta}=3$ and $q_{\beta}=0$.}
\label{fig:examples_segments}
\end{figure}
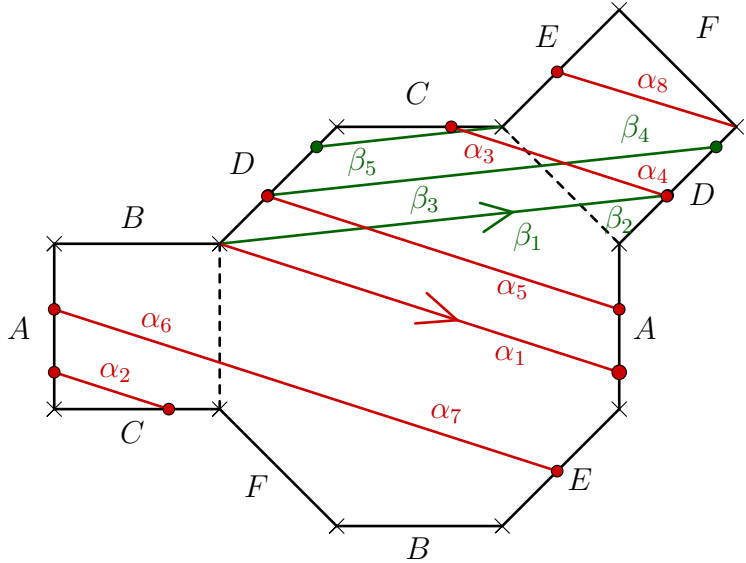

Further, one notices that the cases where this estimate is far from being sharp are the cases where there are odd sequences of consecutive adjacent segments, as then it is not possible to group all adjacent segments by pairs.

In fact, we can slightly improve the estimate of the length in this case:

\begin{Lem}\label{lem:better_lengths}
\begin{enumerate}
\item The total length of three consecutive adjacent segments is more than $\sqrt{2}l_0 = (1 + (\sqrt{2}-1))l_0$.
\item The total length of three consecutive segments $\alpha_{i-1}, \alpha_{i}$ and $\alpha_{i+1}$, being respectively non-adjacent, adjacent and non-adjacent, is more than $(1 + \sqrt{2}) l_0$.
\end{enumerate}
\end{Lem}

We will need the following elementary geometry lemma:
\begin{Lem}\label{lem:elementary_lengths} 
Let $\Delta$ and $\Delta'$ be two half lines from a point $S$ and making an angle $\theta \in ]0, \pi[$. Let $B$ be the point at distance $l_0 > 0$ from both $\Delta$ and $\Delta'$. Then, for any $P \in \Delta$ and $Q \in \Delta'$ such that the segment $[PQ]$ passes through $B$, we have $l(PQ) \geq \frac{2 l_0}{\cos \frac{\theta}{2}}$. 
\end{Lem}
\begin{proof}
Let $H_{\Delta}$ (resp. $H_{\Delta'}$) be the orthogonal projection of $B$ on $\Delta$ (resp. $\Delta'$) and let $\theta_1$ (resp. $\theta_2$) be the angle at $B$ between the segments $[BP]$
and $[BH_{\Delta}]$ (resp. $[BH_{\Delta'}]$ and $[BQ]$), see Figure \ref{lem:elementary_lengths}.
Note that if $\theta \in (0,\pi)$, then $\theta_1$ and $\theta_2$ are oriented angles, i.e. if $P$ (resp. $Q$) is between $S$ and $H_\Delta$ (resp. $H_{\Delta'}$), then $\theta_1$ (resp. $\theta_2$) is negative. Further, $\theta_1 \in (\theta - \frac{\pi}{2}, \frac{\pi}{2})$ (and similarly, $\theta_2 \in (\theta - \frac{\pi}{2}, \frac{\pi}{2})$).

Elementary geometry entails $\theta_1 + \theta_2 = \theta$. Further, $l(PQ) = l(PB) + l(BQ) = \frac{l_0}{\cos \theta_1} + \frac{l_0}{\cos \theta_2}$. Given that $\theta_2 = \theta - \theta_1$, one easily checks that this quantity is minimal for $\theta_1 = \theta_2 = \frac{\theta}{2}$.

\end{proof}

\begin{figure}
\center
\definecolor{qqwuqq}{rgb}{0,0.39215686274509803,0}
\begin{tikzpicture}[line cap=round,line join=round,>=triangle 45,x=1cm,y=1cm, scale=0.7]
\clip(-7.953333333333336,-0.89) rectangle (3.7533333333333347,6.816666666666663);
\draw[line width=1pt,fill=black,fill opacity=0.1] (-1.2067446096434562,0) -- (-1.2067446096434562,0.18856180831641273) -- (-1.3953064179598689,0.18856180831641273) -- (-1.3953064179598689,0) -- cycle; 
\draw[line width=1pt,fill=black,fill opacity=0.1] (0.45534519145768637,1.3323274042711568) -- (0.3710177871865295,1.1636725957288432) -- (0.5396725957288432,1.0793451914576864) -- (0.624,1.248) -- cycle; 
\draw [shift={(-1.3953064179598689,2.2576532089799346)},line width=1pt,color=qqwuqq,fill=qqwuqq,fill opacity=0.10000000149011612] (0,0) -- (-150.7097393870819:0.4) arc (-150.7097393870819:-90:0.4) -- cycle;
\draw [shift={(-1.3953064179598689,2.2576532089799346)},line width=1pt,color=qqwuqq,fill=qqwuqq,fill opacity=0.10000000149011612] (0,0) -- (-26.565051177077994:0.4) arc (-26.565051177077994:29.41946126452276:0.4) -- cycle;
\draw [shift={(0,0)},line width=1pt,color=qqwuqq,fill=qqwuqq,fill opacity=0.10000000149011612] (0,0) -- (63.43494882292201:0.4) arc (63.43494882292201:180:0.4) -- cycle;
\draw [line width=1pt] (0,0)-- (-8.48,0);
\draw [line width=1pt] (0,0)-- (4,8);
\draw [line width=1pt,dash pattern=on 3pt off 3pt] (-1.3953064179598689,2.2576532089799346)-- (-1.3953064179598689,0);
\draw [line width=1pt,dash pattern=on 3pt off 3pt] (-1.3953064179598689,2.2576532089799346)-- (0.624,1.248);
\draw [line width=1pt] (-5.42,0)-- (2.12,4.24);
\draw (-7.5,0) node[anchor=north west] {$\Delta$};
\draw (2.8,6.4) node[anchor=north west] {$\Delta'$};
\draw (2.45,4.4) node[anchor=north west] {$Q$};
\draw (-5.8,0) node[anchor=north west] {$P$};
\draw (-1.75,0) node[anchor=north west] {$H_{\Delta}$};
\draw (0.78,1.45) node[anchor=north west] {$H_{\Delta'}$};
\draw (0.19333333333333338,-0.0633333333333333) node[anchor=north west] {$S$};
\draw (-2,3.2) node[anchor=north west] {$B$};
\draw (-0.6066666666666669,1.136666666666666) node[anchor=north west] {$\theta$};
\draw (-2.2066666666666674,1.9366666666666654) node[anchor=north west] {$\theta_1$};
\draw (-0.6466666666666668,2.776666666666665) node[anchor=north west] {$\theta_2$};
\begin{scriptsize}
\draw [fill=black] (0,0) circle (2.5pt);
\draw [fill=black] (0.624,1.248) circle (2.5pt);
\draw [fill=black] (-1.3953064179598689,0) circle (2.5pt);
\draw [fill=black] (-1.3953064179598689,2.2576532089799346) circle (2.5pt);
\draw [fill=black] (-5.42,0) circle (2.5pt);
\draw [fill=black] (2.12,4.24) circle (2.5pt);
\end{scriptsize}
\end{tikzpicture}
\caption{The setting of Lemma \ref{lem:elementary_lengths}.}
\label{fig:elementary_lengths}
\end{figure}
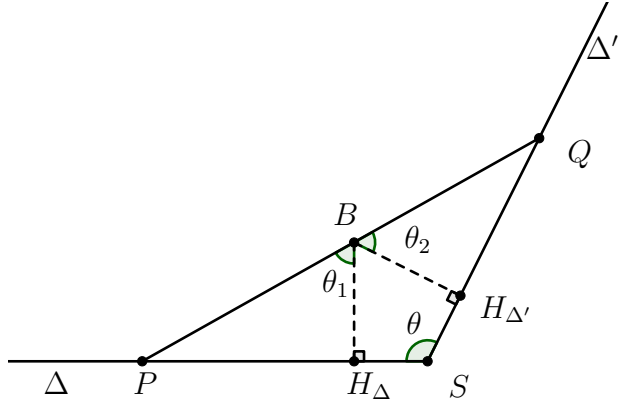

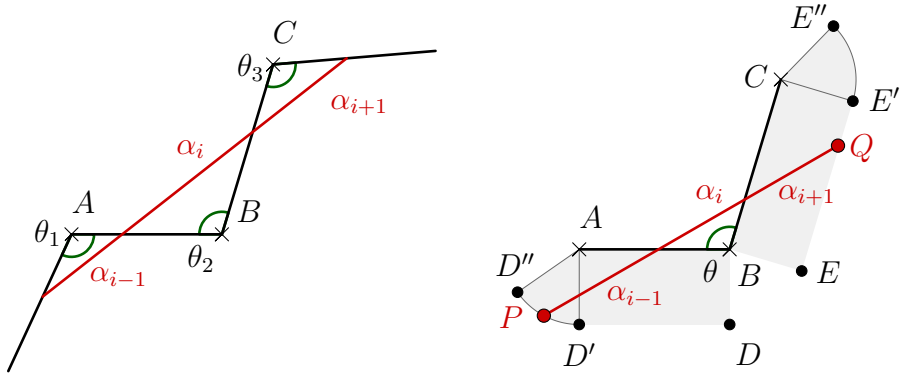
\begin{figure}[h]
\center
\definecolor{qqwuqq}{rgb}{0,0.39215686274509803,0}
\definecolor{ccqqqq}{rgb}{0.8,0,0}
\begin{tikzpicture}[line cap=round,line join=round,>=triangle 45,x=1cm,y=1cm]
\clip(-3,-2) rectangle (3,3);
\draw [shift={(-2,0)},line width=1pt,color=qqwuqq] (0,0) -- (-114.77514056883193:0.3) arc (-114.77514056883193:0:0.3) -- cycle;
\draw [shift={(0,0)},line width=1pt,color=qqwuqq] (0,0) -- (73.2542494616799:0.3) arc (73.2542494616799:180:0.3) -- cycle;
\draw [shift={(0.68,2.26)},line width=1pt,color=qqwuqq] (0,0) -- (-106.74575053832011:0.3) arc (-106.74575053832011:4.763641690726179:0.3) -- cycle;
\draw [line width=1pt] (-2,0)-- (0,0);
\draw [line width=1pt] (0,0)-- (0.68,2.26);
\draw [line width=1pt] (0.68,2.26)-- (2.84,2.44);
\draw [line width=1pt] (-2,0)-- (-2.84,-1.82);
\draw [line width=1pt,color=ccqqqq] (-2.382829268292683,-0.8294634146341464)-- (1.658206896551724,2.3415172413793104);
\draw (-2.64,0.34) node[anchor=north west] {$\theta_1$};
\draw (-0.6,0) node[anchor=north west] {$\theta_2$};
\draw (0.06,2.54) node[anchor=north west] {$\theta_3$};
\draw [color=ccqqqq](-1.9,-0.3) node[anchor=north west] {$\alpha_{i-1}$};
\draw [color=ccqqqq](-0.75,1.4) node[anchor=north west] {$\alpha_i$};
\draw [color=ccqqqq](1.3,2) node[anchor=north west] {$\alpha_{i+1}$};
\draw[color=black] (-1.84,0.43) node {$A$};
\draw[color=black] (0.36,0.31) node {$B$};
\draw[color=black] (0.84,2.69) node {$C$};
\begin{scriptsize}
\draw [color=black] (-2,0)-- ++(-2.5pt,-2.5pt) -- ++(5pt,5pt) ++(-5pt,0) -- ++(5pt,-5pt);
\draw [color=black] (0,0)-- ++(-2.5pt,-2.5pt) -- ++(5pt,5pt) ++(-5pt,0) -- ++(5pt,-5pt);
\draw [color=black] (0.68,2.26)-- ++(-2.5pt,-2.5pt) -- ++(5pt,5pt) ++(-5pt,0) -- ++(5pt,-5pt);
\end{scriptsize}
\end{tikzpicture}
\definecolor{wwwwww}{rgb}{0.4,0.4,0.4}
\definecolor{ccqqqq}{rgb}{0.8,0,0}
\definecolor{qqwuqq}{rgb}{0,0.39215686274509803,0}
\begin{tikzpicture}[line cap=round,line join=round,>=triangle 45,x=1cm,y=1cm]
\clip(-3.2,-1.8) rectangle (2.5,3.8);
\draw [shift={(0,0)},line width=1pt,color=qqwuqq] (0,0) -- (73.2542494616799:0.3) arc (73.2542494616799:180:0.3) -- cycle;
\fill[line width=1pt,color=wwwwww,fill=wwwwww,fill opacity=0.1] (-2,0) -- (-2,-1) -- (0,-1) -- (0,0) -- cycle;
\fill[line width=2pt,color=wwwwww,fill=wwwwww,fill opacity=0.1] (0,0) -- (0.957592732851901,-0.28812524705278486) -- (1.637592732851901,1.971874752947215) -- (0.68,2.26) -- cycle;
\draw [line width=1pt] (-2,0)-- (0,0);
\draw [line width=1pt] (0,0)-- (0.68,2.26);
\draw (-0.5,0) node[anchor=north west] {$\theta$};
\draw [color=ccqqqq](-1.78,-0.35) node[anchor=north west] {$\alpha_{i-1}$};
\draw [color=ccqqqq](-0.6,1) node[anchor=north west] {$\alpha_i$};
\draw [color=ccqqqq](0.5,1) node[anchor=north west] {$\alpha_{i+1}$};
\draw [shift={(0.68,2.26)},line width=0.1pt,color=wwwwww,fill=wwwwww,fill opacity=0.1]  (0,0) --  plot[domain=-0.29226848261129845:0.8,variable=\t]({1*1*cos(\t r)+0*1*sin(\t r)},{0*1*cos(\t r)+1*1*sin(\t r)}) -- cycle ;
\draw [shift={(-2,0)},line width=0.1pt,color=wwwwww,fill=wwwwww,fill opacity=0.1]  (0,0) --  plot[domain=3.75:4.71238898038469,variable=\t]({1*1*cos(\t r)+0*1*sin(\t r)},{0*1*cos(\t r)+1*1*sin(\t r)}) -- cycle ;
\draw [line width=1pt,color=ccqqqq] (-2.4705882352941178,-0.8823529411764706)-- (1.44,1.38);
\draw[color=black] (-1.84,0.43) node {$A$};
\draw[color=black] (-0.05,-0.05) node[anchor=north west] {$B$};
\draw[color=black] (0.35,2.3) node {$C$};
\draw[color=black] (0.25,-1.1) node[below] {$D$};
\draw[color=black] (-2,-1.1) node[below] {$D'$};
\draw[color=black] (-2.85,-0.2) node {$D''$};
\draw[color=black] (1.7,2) node[right] {$E'$};
\draw[color=black] (1,-0.3) node[right] {$E$};
\draw[color=black] (1.05,3.1) node {$E''$};
\draw[color=ccqqqq] (-2.6,-0.9) node[left] {$P$};
\draw[color=ccqqqq] (1.45,1.35) node[right] {$Q$};
\begin{scriptsize}
\draw [color=black] (-2,0)-- ++(-2.5pt,-2.5pt) -- ++(5pt,5pt) ++(-5pt,0) -- ++(5pt,-5pt);
\draw [color=black] (0,0)-- ++(-2.5pt,-2.5pt) -- ++(5pt,5pt) ++(-5pt,0) -- ++(5pt,-5pt);
\draw [color=black] (0.68,2.26)-- ++(-2.5pt,-2.5pt) -- ++(5pt,5pt) ++(-5pt,0) -- ++(5pt,-5pt);
\draw [fill=black] (0,-1) circle (2pt);
\draw [fill=black] (-2,-1) circle (2pt);
\draw [fill=black] (-2.82,-0.57) circle (2pt);
\draw [fill=black] (1.637592732851901,1.971874752947215) circle (2pt);
\draw [fill=black] (0.957592732851901,-0.28812524705278486) circle (2pt);
\draw [fill=black] (1.38,2.97) circle (2pt);
\draw [fill=ccqqqq] (-2.4705882352941178,-0.8823529411764706) circle (2.5pt);
\draw [fill=ccqqqq] (1.44,1.38) circle (2.5pt);
\end{scriptsize}
\end{tikzpicture}
\caption{The two cases in Lemma~\ref{lem:better_lengths}.}
\label{fig:better_lengths}
\end{figure}

\begin{proof}[Proof of Lemma \ref{lem:better_lengths}]
\begin{enumerate}
\item Given three consecutive adjacent segments $\alpha_{i-1}, \alpha_i$ and $\alpha_{i+1}$, we can unfold the trajectory to get a picture like Figure \ref{fig:better_lengths}. 
Since the angles $\theta_1, \theta_2$ and $\theta_3$ are comprised between $\frac{\pi}{2}$ and $\pi$, we directly deduce that the length of $\alpha_{i-1} \cup \alpha_i \cup \alpha_{i+1}$ is greater than the length of the segment $AC$, which is at least $\sqrt{2} l_0$.

\item Similarly, we can draw a picture like the right hand part of Figure \ref{fig:better_lengths}. By Lemma \ref{lem:convex_length}, any point on a side which is non-adjacent to the segment $AB$ (resp. $BC$) must be at a distance at least $l_0$ away from the side $AB$ (resp. $BC$). In particular, the endpoints $\alpha_{i-1}^-$ and $\alpha_{i+1}^+$ are outside the gray zone of Figure \ref{fig:better_lengths}, which is the set of points of the polygon containing $\alpha_{i-1}$ (resp. $\alpha_{i+1}$) at distance less than $l_0$ from $AB$ (resp. $BC$). The boundary of this region contains a segment $[D',D]$ (resp. $[E,E']$) parallel to $AB$ (resp. $BC$) at a distance $l_0$, and two arcs of circles centred at $B$ and $A$ (resp. $B$ and $C$), of radius $l_0$. 
Moreover, we denote by $D''$ (resp. $E''$) the point on the side adjacent to $AB$ (resp. $BC$) having $A$ (resp. $C$) as an endpoint, which is at distance $l_0$ from $AB$ (resp. $BC$). 

Now, let $P$ (resp. $Q$) be the point of $\alpha_{i-1}$ (resp. $\alpha_{i+1}$) on the boundary of the gray zone. 
We have $l(\alpha_{i-1} \cup \alpha_i \cup \alpha_{i+1}) \geq l(PQ)$. 
Let us now denote by $P'$ and $Q'$ the two points on the boundary of the grey area around $AB$ and $BC$ respectively, such that the segment $P'Q'$ is parallel to $PQ$ and passes through $B$.
Since $l(P'Q') \leq l(PQ)$, we can assume that $P=P'$ and $Q=Q'$, or in other words that $PQ$ passes through $B$.
Now,
\begin{enumerate}
\item If $P$ lies in the arc of circle $[D',D'']$, then the length of the segment $PB$ is at least $\sqrt{2} l_0$. Since the length of $BQ$ is by construction at least $l_0$, we conclude that 
\[ l(PQ) = l(PB) + l(BQ) \geq (\sqrt{2}+1)l_0.\]
By symmetry, the same holds if $Q$ lies in the arc of circle $[E',E'']$. 
\item Else, $P \in [D,D']$ and $Q \in [E,E']$, then $PQ$ has minimal length when it goes through $B$, and we can use Lemma \ref{lem:elementary_lengths} to conclude that $l(PQ) \geq 2 \sqrt{2} l_0$.
\end{enumerate}
\end{enumerate}
\end{proof}

In the remaining part of this section, we will use the following
\begin{Def}[Odd saddle connection]\label{def:oddSC}
Given a finite sequence of elements of $\{1,2\}$, we will say that the sequence is $odd$ if it starts and ends with 1, the 1s are isolated and the blocks of 2s contain an odd number of elements. 
Similarly, given a saddle connection with its polygonal decomposition, we will say that $\alpha$ is $odd$ if the non-adjacent segments are isolated (i.e. there are no consecutive pairs of pieces of non-adjacent segments) and between two isolated non-adjacent segments there is an odd number of adjacent segments.
In other words, $\alpha$ is $odd$ if the sequence given by the types of the pieces of its polygonal decomposition is $odd$ with the rule $1 =$ non-adjacent and $2=$ adjacent.
Note also that the number of pieces in the polygonal decomposition of an odd saddle connection is odd. 
\end{Def}

The first reason for this definition is that for an odd saddle connection (recall that we are excluding sides and diagonals), there is always either three consecutive adjacent segments or a sequence of three consecutive segment being respectively non-adjacent, adjacent and non-adjacent, so that we deduce from Lemma \ref{lem:better_lengths} that:
\begin{Lem}\label{lem:length_odd_sc}
If $\alpha$ is an odd saddle connection, then:
\[ l(\alpha) \geq (p_{\alpha}+q_{\alpha} + \sqrt{2}-1)l_0. \]
\end{Lem}

In fact, the notion of odd saddle connection will also turn out to be particularly useful in the study of the intersections. The main reason for this is because they form the equality case in the following
\begin{Lem}\label{lem:integer_part}
Given a saddle connection with its polygonal decomposition $\alpha = \alpha_1 \cup \cdots \cup \alpha_k$, we have
\[ \left\lceil \frac{k}{2} \right\rceil \leq p_{\alpha} + q_{\alpha} \]
with equality if and only if $\alpha$ is odd.
\end{Lem}
\begin{proof}
Let $N = p_{\alpha}$ and $k_i$, $1 \leq i \leq N-1$ be the number (possibly zero) of adjacent segments between the $i^{th}$ and the $(i+1)^{th}$ non-adjacent segment. 
Remembering that $\alpha_1$ and $\alpha_k$ are non-adjacent, we then have
\[ k = 1 + k_1 + 1 + \cdots + 1+ k_{N-1} +1 = N + \sum_{i=1}^{N-1} k_i  \]
so that 
\[ \frac{k}{2} = \frac{ N + \sum_{i=1}^{N-1} k_i}{2} = \frac{ 2N-1 + \sum_{i=1}^{N-1} (k_i-1)}{2} = N - \frac{1}{2} + \sum_{i=1}^{N-1} \frac{k_i -1}{2}. \]

Hence
\begin{align*}
\left\lceil \frac{k}{2} \right\rceil & \leq N + \sum_{i=1}^{N-1} \frac{k_i -1}{2} \\
& \leq N + \sum_{i=1}^{N-1} \left\lfloor \frac{k_i}{2} \right\rfloor = p_{\alpha} + q_{\alpha}. 
\end{align*}
The first inequality is actually strict when $k$ is even and an equality when $k$ is odd. 
The second inequality is an equality if and only if for all $i$, $\frac{k_i -1}{2} = \left\lfloor \frac{k_i}{2} \right\rfloor$, that is $k_i$ is odd and so $\alpha$ is odd.
Since for odd saddle connections $k$ is necessarily odd, equality overall is achieved exactly on odd saddle connections.
\end{proof}

\subsection{Properties of consecutive adjacent segments}\label{sec:consecutive_adjacent_ppties}
In this section, we give several properties of adjacent segments which we will use for the estimation of the intersections between two saddle connections. We start with a few definitions.

\begin{Def}[Type of a segment]
Let $\alpha_{i}$ be a segment going from the interior of a side $e$ to the interior of a side $e'$. We will say that $\alpha_i$ is of type $e \to e'$. If it goes from a vertex to the interior of a side $e'$ (resp. from the interior of a side $e$ to a vertex), we will say that $\alpha_i$ is of type $*\to e'$ (resp. $e \to *$).
\end{Def}

\subsubsection{Sector of adjacency}
In this paragraph, we consider a convex polygon $P$ whose sides are labeled in cyclic (clockwise) order $e_1, \cdots, e_N$, and a direction $\theta \in [0, 2\pi[$ which represent the direction of a saddle connection.

\begin{Lem}\label{lem:sector_adjacency}[and Definition]
There exist two sides $e_{u_-}$ and $e_{u_+}$ such that for any adjacent segment $\alpha_i$ contained in $P$ and having direction $\theta$, the type of $\alpha_i$ is either:
\begin{itemize}
\item $e_{u_+} \to e_{u_+ + 1}$,
\item or $e_{u_-} \to e_{u_- - 1}$.
\end{itemize}
In the first case, we say that $\alpha_i$ has positive sign and in the second case we say that the segment has negative sign.
\end{Lem}

Roughly speaking, we are saying that if we fix a direction, then an adjacent segment in that direction has only two possibilities for the edges that it can start from and end into.
Moreover, the cyclic ordering of the sides of the polygon determine which pair of adjacent sides the segment touches.

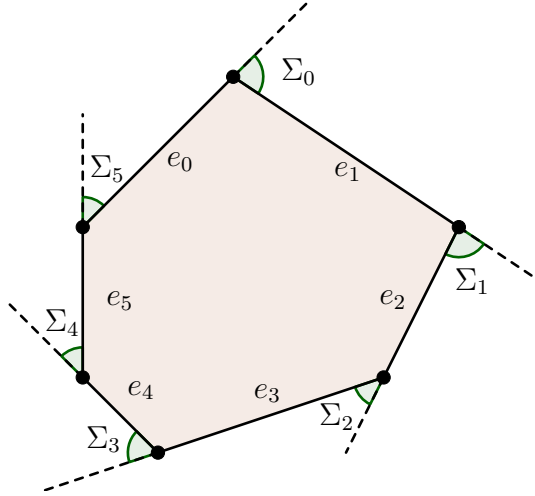
\begin{figure}
\center
\definecolor{zzttqq}{rgb}{0.6,0.2,0}
\definecolor{qqwuqq}{rgb}{0,0.39215686274509803,0}
\begin{tikzpicture}[line cap=round,line join=round,>=triangle 45,x=1cm,y=1cm]
\clip(-1,-0.5) rectangle (6,6);
\draw [shift={(0,3)},line width=1pt,color=qqwuqq,fill=qqwuqq,fill opacity=0.10000000149011612] (0,0) -- (45:0.4) arc (45:90:0.4) -- cycle;
\draw [shift={(2,5)},line width=1pt,color=qqwuqq,fill=qqwuqq,fill opacity=0.10000000149011612] (0,0) -- (-33.690067525979785:0.4) arc (-33.690067525979785:45:0.4) -- cycle;
\draw [shift={(5,3)},line width=1pt,color=qqwuqq,fill=qqwuqq,fill opacity=0.10000000149011612] (0,0) -- (-116.56505117707799:0.4) arc (-116.56505117707799:-33.69006752597978:0.4) -- cycle;
\draw [shift={(4,1)},line width=1pt,color=qqwuqq,fill=qqwuqq,fill opacity=0.10000000149011612] (0,0) -- (-161.565051177078:0.4) arc (-161.565051177078:-116.56505117707799:0.4) -- cycle;
\draw [shift={(1,0)},line width=1pt,color=qqwuqq,fill=qqwuqq,fill opacity=0.10000000149011612] (0,0) -- (135:0.4) arc (135:198.434948822922:0.4) -- cycle;
\draw [shift={(0,1)},line width=1pt,color=qqwuqq,fill=qqwuqq,fill opacity=0.10000000149011612] (0,0) -- (90:0.4) arc (90:135:0.4) -- cycle;
\fill[line width=1pt,color=zzttqq,fill=zzttqq,fill opacity=0.1] (0,1) -- (0,3) -- (2,5) -- (5,3) -- (4,1) -- (1,0) -- cycle;
\draw [line width=1pt] (1,0)-- (0,1);
\draw [line width=1pt] (0,1)-- (0,3);
\draw [line width=1pt] (0,3)-- (2,5);
\draw [line width=1pt] (2,5)-- (5,3);
\draw [line width=1pt] (5,3)-- (4,1);
\draw [line width=1pt] (1,0)-- (4,1);
\draw [line width=1pt,dash pattern=on 3pt off 3pt] (0,3)-- (0,4.5);
\draw [line width=1pt,dash pattern=on 3pt off 3pt] (0,1)-- (-1,2);
\draw [line width=1pt,dash pattern=on 3pt off 3pt] (2,5)-- (3,6);
\draw [line width=1pt,dash pattern=on 3pt off 3pt] (5,3)-- (6.2,2.2);
\draw [line width=1pt,dash pattern=on 3pt off 3pt] (4,1)-- (3.5,0);
\draw [line width=1pt,dash pattern=on 3pt off 3pt] (1,0)-- (-0.5,-0.5);
\draw (0.9718518518518443,4.1390123456789825) node[anchor=north west] {$e_0$};
\draw (3.2,4) node[anchor=north west] {$e_1$};
\draw (3.8,2.3) node[anchor=north west] {$e_2$};
\draw (2.1318518518518412,1.0590123456790064) node[anchor=north west] {$e_3$};
\draw (0.4518518518518457,1.0990123456790062) node[anchor=north west] {$e_4$};
\draw (0.17185185185184643,2.2856790123456636) node[anchor=north west] {$e_5$};
\draw (2.49185185185184,5.3923456790123065) node[anchor=north west] {$\Sigma_0$};
\draw (4.8,2.6) node[anchor=north west] {$\Sigma_1$};
\draw (3,0.8) node[anchor=north west] {$\Sigma_2$};
\draw (-0.1,0.5) node[anchor=north west] {$\Sigma_3$};
\draw (-0.65,2.05) node[anchor=north west] {$\Sigma_4$};
\draw (-0.05,4.1) node[anchor=north west] {$\Sigma_5$};
\begin{scriptsize}
\draw [fill=black] (1,0) circle (2.5pt);
\draw [fill=black] (0,1) circle (2.5pt);
\draw [fill=black] (0,3) circle (2.5pt);
\draw [fill=black] (2,5) circle (2.5pt);
\draw [fill=black] (5,3) circle (2.5pt);
\draw [fill=black] (4,1) circle (2.5pt);
\end{scriptsize}
\end{tikzpicture}
\caption{An adjacent segment from $e_{i}$ to $e_{i+1}$ has direction in the sector $\Sigma_i$ defined by the directions of the sides $e_{i}$ and $e_{i+1}$.}
\label{fig:convex_sectors}
\end{figure}

\begin{proof}
Fixing two sides $e_u$ and $e_{u+1}$, an (oriented) adjacent segment from $e_u$ to $e_{u+1}$ must have a direction within an angular sector determined by the directions of $e_u$ and $e_{u+1}$. 
We will call this the \emph{admissible sector for $e_u$ and $e_{u+1}$}.
For a convex polygon, the admissible sectors for each pair of adjacent sides form a partition of $[0, 2\pi[$. In particular, (if the direction of $\theta$ is not the direction of a side) there is only one pair of sides $e_{u_+}$ and $e_{u_++1}$ such that the direction $\theta$ lie in the admissible sector for $e_{u_+}$ and $e_{u_+ +1}$ (in this order). 

Similarly, there is only one pair of sides $e_{u_-}$ and $e_{u_- -1}$ such that the direction $\theta$ lie in the admissible sector for $e_{u_-}$ and $e_{u_-  -1}$ (in this order). 
\end{proof}

\subsubsection{Sequences of consecutive adjacent segments}\label{sec:adjacent_segments}
We now describe several properties of sequences of adjacent segments that will be used in the next section to estimate the intersections between two saddle connections using their polygonal decomposition.

\begin{Rema}\label{rk:alternates} 
In a sequence of consecutive adjacent segments, the sign of adjacent segments alternate. This is because two consecutive adjacent segments are as in Figure \ref{fig:opposite_signs}, so they must have opposite signs.
\end{Rema}

\begin{Rema}\label{rk:hypothesis_P1'_rk1}
It should be noted that Remark \ref{rk:alternates} does not in fact use the condition that the angles of the polygons are obtuse or right but it uses the weaker condition that the sum of two consecutive angles at a vertex is at least $\pi$. In fact, all the results in this section (\S\ref{sec:consecutive_adjacent_ppties}) and the next one (\S\ref{sec:study_intersections}) hold if we weaken the assumption (P1) to:
\begin{itemize}
    \item[(P1')] The polygons are convex and the sum of two consecutive angle at a vertex is at least $\pi$.
\end{itemize}
This will turn out to be useful later on, as (P1') is satisfied on Bouw-M\"oller surfaces $S_{m,n}$ for $n=3$ while (P1) is not. 

\end{Rema}
\begin{figure}[h]
\center
\definecolor{ccqqqq}{rgb}{0.8,0,0}
\begin{tikzpicture}[line cap=round,line join=round,>=triangle 45,x=1cm,y=1cm, scale =0.6]
\clip(-5.5,-3) rectangle (5,2.5);
\draw [line width=1pt] (-2,0)-- (2,0);
\draw [line width=1pt] (2,0)-- (4.02,1.98);
\draw [line width=1pt,dash pattern=on 1pt off 1pt] (4.02,1.98)-- (5,3);
\draw [line width=1pt] (-2,0)-- (-3.94,-1.94);
\draw [line width=1pt,dash pattern=on 3pt off 3pt] (-3.94,-1.94)-- (-5,-3);
\draw [line width=1pt,dash pattern=on 3pt off 3pt] (-2,0)-- (-3.76,1.58);
\draw [line width=1pt,dash pattern=on 3pt off 3pt] (2,0)-- (3.64,-1.7);
\draw [line width=1pt,color=ccqqqq] (-3.36,-1.36)-- (3.1414888511148886,1.1188851114888514);
\draw [color=ccqqqq](-2.14,-0.72) node[anchor=north west] {$\alpha_i$};
\draw [color=ccqqqq](1,1.6) node[anchor=north west] {$\alpha_{i+1}$};
\draw (-0.34,-0.62) node[anchor=north west] {$P$};
\draw (-1.1,1.92) node[anchor=north west] {$P'$};
\end{tikzpicture}
\caption{Two consecutive adjacent segment have opposite signs.}
\label{fig:opposite_signs}
\end{figure}
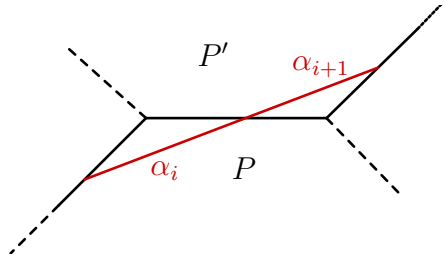

We can then use this remark to show:
\begin{Lem}\label{lem:maximal_even}
We consider a maximal sequence of adjacent segments, and we assume that there are two adjacent segments inside the same polygon having the same sign. Then there is an even number of adjacent segments in the sequence.
\end{Lem}

\begin{proof}
Let $\alpha_{i}$ and $\alpha_{i'}$ be two segments of the sequence inside the same polygon and having the same sign. Up to a change of orientation of $\alpha$, we can assume that they both have positive sign, and that we chose the indices $i$ and $i'$ so that we are in the configuration in Figure \ref{fig:same_type} (i.e. that $\alpha_i$ is closer than $\alpha_{i'}$ to the corner of $P$). We distinguish the cases $i<i'$ and $i'<i$.

\begin{figure}[h]
\center
\definecolor{ccqqqq}{rgb}{0.8,0,0}
\begin{tikzpicture}[line cap=round,line join=round,>=triangle 45,x=1cm,y=1cm, scale =0.6]
\clip(-5.5,-3) rectangle (5,2.5);
\draw [line width=1pt] (-2,0)-- (2,0);
\draw [line width=1pt] (2,0)-- (4.02,1.98);
\draw [line width=1pt,dash pattern=on 1pt off 1pt] (4.02,1.98)-- (5,3);
\draw [line width=1pt] (-2,0)-- (-3.94,-1.94);
\draw [line width=1pt,dash pattern=on 3pt off 3pt] (-3.94,-1.94)-- (-5,-3);
\draw [line width=1pt,dash pattern=on 3pt off 3pt] (-2,0)-- (-3.76,1.58);
\draw [line width=1pt,dash pattern=on 3pt off 3pt] (2,0)-- (3.64,-1.7);
\draw [line width=1pt,color=ccqqqq] (-4,-2)--(1,0);
\draw [line width=1pt, color=ccqqqq, dash pattern=on 3pt off 3pt] (1,0)--(2.6,0.6);
\draw [line width=1pt,color=ccqqqq] (-2.6,-0.6)--(-1,0);
\draw [line width=1pt, color=ccqqqq, dash pattern=on 3pt off 3pt] (-1,0)--(4,2);
\draw [color=ccqqqq](-3.6,0) node[anchor=north west] {$\alpha_i$};
\draw [color=ccqqqq](1,2.3) node[anchor=north west] {$\alpha_{i+1}$};
\draw [color=ccqqqq](-2.6,-1.2) node[anchor=north west] {$\alpha_{i'}$};
\draw [color=ccqqqq](2.5,1) node[anchor=north west] {$\alpha_{i'+1}$};
\draw (-0.34,-0.62) node[anchor=north west] {$P$};
\draw (-1.1,1.92) node[anchor=north west] {$P'$};
\end{tikzpicture}
\caption{The segments $\alpha_{i}$ and $\alpha_{i'}$ are positive adjacent segments of the same type.}
\label{fig:same_type}
\end{figure}
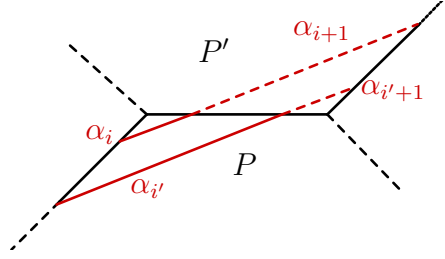

\begin{itemize}
    \item Assume $i<i'$. All the segments between $\alpha_i$ and $\alpha_{i'}$ are adjacent segments. First, we will show that the last adjacent segment of a sequence has negative sign.
    Since $\alpha_{i+1}$ is an adjacent segment parallel to $\alpha_{i'+1}$, we directly deduce from Figure \ref{fig:same_type} that $\alpha_{i'+1}$ must be an adjacent segment (and has negative sign). Now, if $\alpha_{i'+2}$ is non-adjacent, then $\alpha_{i'+1}$ is the last adjacent segment and we saw it has negative sign. Otherwise, if $\alpha_{i'+2}$ is also adjacent, then $\alpha_{i+2}$ and $\alpha_{i'+2}$ have the same type and are in the same configuration as $\alpha_{i}$ and $\alpha_{i'}$, and hence $\alpha_{i'+3}$ is also adjacent. Repeating this argument, we obtain that the last adjacent segment of the sequence must be of negative sign.\newline
    
    Now, we will prove that the first segment of the sequence has positive sign.
    Either $\alpha_{i}$ is the first adjacent segment and we have assumed it has positive sign, or
    $\alpha_{i-1}$ is an adjacent segment and then it must have negative sign and share its type with $\alpha_{i'-1}$. Next, $\alpha_{i'-2}$ is also an adjacent segment, and we deduce from Figure \ref{fig:same_type_2} that the segment $\alpha_{i-2}$, parallel to $\alpha_{i'-2}$, must also be an adjacent segment (which then has positive sign). Then, $\alpha_{i-2}$ and $\alpha_{i'-2}$ are in the same configuration as $\alpha_{i}$ and $\alpha_{i'}$ and hence we can repeat the argument until we reach a non-adjacent segment. From this we conclude that the first segment of the sequence must be of positive sign.

    \begin{figure}[h]
\center
\definecolor{ccqqqq}{rgb}{0.8,0,0}
\begin{tikzpicture}[line cap=round,line join=round,>=triangle 45,x=1cm,y=1cm, scale =0.6]
\clip(-5.5,-3) rectangle (9.5,3.3);
\draw [line width=1pt] (-2,0)-- (2,0);
\draw [line width=1pt] (2,0)-- (4.6,2.6);
\draw [line width=1pt] (4.6,2.6)-- (8.6,2.6);
\draw [line width=1pt] (-2,0)-- (-3.94,-1.94);
\draw [line width=1pt] (-3.94,-1.94)-- (-5,-3);
\draw [line width=1pt,dash pattern=on 3pt off 3pt] (-2,0)-- (-3.76,1.58);
\draw [line width=1pt,dash pattern=on 3pt off 3pt] (8.6,2.6)-- (9.5,1.7);
\draw [line width=1pt,color=ccqqqq, dash pattern=on 3pt off 3pt] (-4,-2)--(1,0);
\draw [line width=1pt, color=ccqqqq, dash pattern=on 3pt off 3pt] (1,0)--(2.6,0.6);
\draw [line width=1pt,color=ccqqqq, dash pattern=on 3pt off 3pt] (-2.6,-0.6)--(-1,0);
\draw [line width=1pt, color=ccqqqq, dash pattern=on 3pt off 3pt] (-1,0)--(4,2);
\draw [line width=1pt, color=ccqqqq] (4,2)--(5.6,2.6);
\draw [line width=1pt, color=ccqqqq] (2.6,0.6)--(7.6,2.6);
\draw [color=ccqqqq](-4.3,0) node[anchor=north west] {$\alpha_{i-2}$};
\draw [color=ccqqqq](0,1.8) node[anchor=north west] {$\alpha_{i-1}$};
\draw [color=ccqqqq](-2.6,-1.2) node[anchor=north west] {$\alpha_{i'-2}$};
\draw [color=ccqqqq](2.5,0.7) node[anchor=north west] {$\alpha_{i'-1}$};
\draw [color=ccqqqq](4.6,3.5) node[anchor=north west] {$\alpha_{i}$};
\draw [color=ccqqqq](6,2) node[anchor=north west] {$\alpha_{i'}$};
\draw (5.7,1) node[anchor=north west] {$P$};
\end{tikzpicture}
\caption{If $\alpha_{i-1}$ is adjacent, then the segments $\alpha_{i-2}$ and $\alpha_{i'-2}$ are positive adjacent segments of the same type.}
\label{fig:same_type_2}
\end{figure}
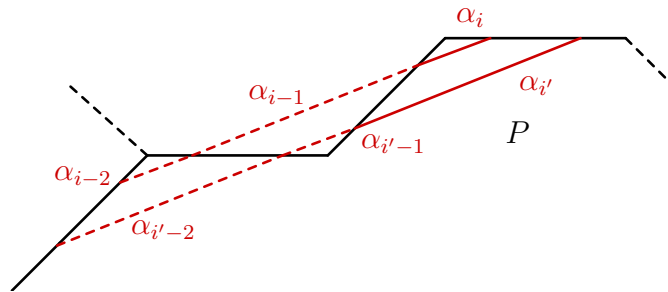

    As a conclusion, the sequence starts with a positive adjacent segment and ends with a negative adjacent segment. Since the signs are alternating, there must be an even number of segments.
    \item Similarly, if $i'<i$, the above arguments give that the first segment of the sequence of consecutive adjacent segment must have negative sign while the last segment of the sequence has positive sign. Hence, there is also an even number of segments.
\end{itemize}



\end{proof}


To study the intersections, we distinguish three types of sequences of consecutive adjacent segments:
\begin{itemize}
\item \textbf{(Isolated adjacent segments)} An adjacent segment which is preceded and followed by non-adjacent segments will be called an \emph{isolated} adjacent segment.

\item \textbf{(Sequence of consecutive adjacent segments contained inside a short cylinder)}
A sequence of consecutive adjacent segments containing at least two segments, and such that for each segment $\alpha_{i}$ of the sequence, $\alpha_{i+2}$ is either non-adjacent or has the same type as $\alpha_i$.\newline
Formally, we are in this case if there exist two sides $e$ and $e'$ of $P(0)$ (or $P(m-1)$) such that the segments of the sequence, which are alternatively contained in $P(0)$ and $P(1)$ (or $P(m-2)$ and $P(m-1)$), have type alternating\footnote{Note that if $e$ and $e'$ are two sides of $P(0)$ (resp. $P(m-1)$) the notion of segment of type $e \to e'$ also makes sense in $P(1)$ (resp. $ P(m-2)$).} between $e \to e'$ and $e' \to e$. A sequence of adjacent segments inside a short cylinder is represented in the left of Figure \ref{fig:short_cylinder}. Following the terminology of \cite{BLM22}, we say that the starting side $e$ of the sequence is the sandwiching side of the sequence of adjacent segments while the side $e'$ is the sandwiched side. By Lemma \ref{lem:maximal_even}, a sequence of consecutive adjacent segments contained inside a short cylinder must contain an even number of segments. Further, if we decompose the segment by pairs, each pair $\alpha_{i} \cup \alpha_{i+1}$ goes from $e$ to $e$, and the direction of $\alpha$ must lie in the sector defined by the direction of $e'$ and the direction of the diagonal of the cylinder.

\item \textbf{(Sequence of consecutive adjacent segments not contained inside a short cylinder)} The other sequences of consecutive adjacent segments are those which are not contained in a short cylinder. For those sequences, we have from Lemma \ref{lem:sector_adjacency} and Remark \ref{rk:alternates} that:
\begin{Cor}\label{cor:3_segments}
Given three consecutive segments in a sequence of consecutive adjacent segments which is not contained in a short cylinder, no two segments lie inside the same polygon.
\end{Cor}
\begin{proof}
The second segment cannot lie in the same polygon than either the first or the third segment because by hypothesis $(P2)$ consecutive adjacent segments do not lie in the same polygon.
By Lemma \ref{lem:sector_adjacency} and Remark \ref{rk:alternates}, the first and the third segment cannot lie in the same polygon unless they have the same type, but this is not the case by assumption of not being in a small cylinder. 
\end{proof}
\end{itemize}


\begin{figure}
\center
\definecolor{ccqqqq}{rgb}{0.8,0,0}
\begin{tikzpicture}[line cap=round,line join=round,>=triangle 45,x=1cm,y=1cm, scale=0.8]
\clip(-4.2,-3.5) rectangle (4.2,3.5);
\draw [line width=1pt] (-2,0)-- (2,0);
\draw [line width=1pt] (2,0)-- (4,3);
\draw [line width=1pt] (-2,0)-- (-4,-3);
\draw [line width=1pt,dash pattern=on 3pt off 3pt,color=ccqqqq] (0.9816035345463919,1.356290121134001)-- (3.6906621119299947,2.535993167894992);
\draw [line width=1pt,color=ccqqqq] (-2.3093378880700053,-0.4640068321050079)-- (3.3269314093167632,1.9903971139751455);
\draw [line width=1pt,color=ccqqqq] (-2.6730685906832368,-1.0096028860248545)-- (2.9632007067035326,1.444801060055299);
\draw [line width=1pt,color=ccqqqq] (-3.0367992932964674,-1.555198939944701)-- (2.5994700040903016,0.8992050061354526);
\draw [line width=1pt,color=ccqqqq] (-3.4005299959096984,-2.1007949938645476)-- (2.2357393014770706,0.3536089522156061);
\draw [line width=1pt,dash pattern=on 3pt off 3pt,color=ccqqqq] (-3.7642606985229294,-2.6463910477843937)-- (0.3686809206691537,-0.8466352724939681);
\draw (3.344258794968969,1.929581471089471) node[anchor=north west] {$e$};
\draw (-1.224699539524695,0.9914683529805202) node[anchor=north west] {$e'$};
\draw (-3.3788852181452818,-0.7284056968858895) node[anchor=north west] {$e$};
\draw [line width=1pt,dash pattern=on 3pt off 3pt] (-2,0)-- (-2.979318519691461,1.8079742150383107);
\draw [line width=1pt,dash pattern=on 3pt off 3pt] (2,0)-- (3.8828052146241174,-1.5622840240938458);
\begin{scriptsize}
\draw [color=black] (-2,0)-- ++(-2.5pt,-2.5pt) -- ++(5pt,5pt) ++(-5pt,0) -- ++(5pt,-5pt);
\draw [color=black] (2,0)-- ++(-2.5pt,-2.5pt) -- ++(5pt,5pt) ++(-5pt,0) -- ++(5pt,-5pt);
\draw [color=black] (4,3)-- ++(-2.5pt,-2.5pt) -- ++(5pt,5pt) ++(-5pt,0) -- ++(5pt,-5pt);
\draw [color=black] (-4,-3)-- ++(-2.5pt,-2.5pt) -- ++(5pt,5pt) ++(-5pt,0) -- ++(5pt,-5pt);
\draw [fill=ccqqqq,shift={(3.6906621119299947,2.535993167894992)}] (0,0) ++(0 pt,3pt) -- ++(2.598076211353316pt,-4.5pt)--++(-5.196152422706632pt,0 pt) -- ++(2.598076211353316pt,4.5pt);
\draw [fill=ccqqqq,shift={(-2.3093378880700053,-0.4640068321050079)}] (0,0) ++(0 pt,3pt) -- ++(2.598076211353316pt,-4.5pt)--++(-5.196152422706632pt,0 pt) -- ++(2.598076211353316pt,4.5pt);
\draw [color=ccqqqq] (3.3269314093167632,1.9903971139751455) ++(-2pt,0 pt) -- ++(2pt,2pt)--++(2pt,-2pt)--++(-2pt,-2pt)--++(-2pt,2pt);
\draw [color=ccqqqq] (-2.6730685906832368,-1.0096028860248545) ++(-2pt,0 pt) -- ++(2pt,2pt)--++(2pt,-2pt)--++(-2pt,-2pt)--++(-2pt,2pt);
\draw [fill=ccqqqq,shift={(2.9632007067035326,1.444801060055299)},rotate=90] (0,0) ++(0 pt,3pt) -- ++(2.598076211353316pt,-4.5pt)--++(-5.196152422706632pt,0 pt) -- ++(2.598076211353316pt,4.5pt);
\draw [fill=ccqqqq,shift={(-3.0367992932964674,-1.555198939944701)},rotate=90] (0,0) ++(0 pt,3pt) -- ++(2.598076211353316pt,-4.5pt)--++(-5.196152422706632pt,0 pt) -- ++(2.598076211353316pt,4.5pt);
\draw [fill=ccqqqq] (2.5994700040903016,0.8992050061354526) ++(-2pt,0 pt) -- ++(2pt,2pt)--++(2pt,-2pt)--++(-2pt,-2pt)--++(-2pt,2pt);
\draw [fill=ccqqqq] (-3.4005299959096984,-2.1007949938645476) ++(-2pt,0 pt) -- ++(2pt,2pt)--++(2pt,-2pt)--++(-2pt,-2pt)--++(-2pt,2pt);
\draw [fill=ccqqqq] (2.2357393014770706,0.3536089522156061) circle (2pt);
\draw [fill=ccqqqq] (-3.7642606985229294,-2.6463910477843937) circle (2pt);
\end{scriptsize}
\end{tikzpicture}
\begin{tikzpicture}[line cap=round,line join=round,>=triangle 45,x=1cm,y=1cm, scale=0.8]
\clip(-4.2,-3.5) rectangle (4.2,3.5);
\draw [line width=1pt] (-2,0)-- (2,0);
\draw [line width=1pt] (2,0)-- (4,3);
\draw [line width=1pt] (-2,0)-- (-4,-3);
\draw [line width=1pt,color=ccqqqq] (-2.6730685906832368,-1.0096028860248545)-- (2.9632007067035326,1.444801060055299);
\draw (3.344258794968969,1.929581471089471) node[anchor=north west] {$e$};
\draw (0,0) node[anchor=north west] {$e'$};
\draw (-3.3788852181452818,-0.7284056968858895) node[anchor=north west] {$e$};
\draw [line width=1pt,dash pattern=on 3pt off 3pt] (-2,0)-- (-2.979318519691461,1.8079742150383107);
\draw [line width=1pt,dash pattern=on 3pt off 3pt] (2,0)-- (3.8828052146241174,-1.5622840240938458);
\draw [line width=1pt,dash pattern=on 3pt off 3pt] (-2,0)-- (4,3);
\draw [shift={(-2,-0)},line width=1pt,color=black,fill=black,fill opacity=0.10000000149011612] (0,0) -- (0:0.8) arc (0:26:0.8) -- cycle;
\draw (0,2) node[anchor=north west] {$\Delta$};
\begin{scriptsize}
\draw [color=black] (-2,0)-- ++(-2.5pt,-2.5pt) -- ++(5pt,5pt) ++(-5pt,0) -- ++(5pt,-5pt);
\draw [color=black] (2,0)-- ++(-2.5pt,-2.5pt) -- ++(5pt,5pt) ++(-5pt,0) -- ++(5pt,-5pt);
\draw [color=black] (4,3)-- ++(-2.5pt,-2.5pt) -- ++(5pt,5pt) ++(-5pt,0) -- ++(5pt,-5pt);
\draw [color=black] (-4,-3)-- ++(-2.5pt,-2.5pt) -- ++(5pt,5pt) ++(-5pt,0) -- ++(5pt,-5pt);
\draw [fill=ccqqqq,shift={(3.6906621119299947,2.535993167894992)}] (0,0) ++(0 pt,3pt) -- ++(2.598076211353316pt,-4.5pt)--++(-5.196152422706632pt,0 pt) -- ++(2.598076211353316pt,4.5pt);
\draw [fill=ccqqqq,shift={(-2.3093378880700053,-0.4640068321050079)}] (0,0) ++(0 pt,3pt) -- ++(2.598076211353316pt,-4.5pt)--++(-5.196152422706632pt,0 pt) -- ++(2.598076211353316pt,4.5pt);
\draw [color=ccqqqq] (3.3269314093167632,1.9903971139751455) ++(-2pt,0 pt) -- ++(2pt,2pt)--++(2pt,-2pt)--++(-2pt,-2pt)--++(-2pt,2pt);
\draw [color=ccqqqq] (-2.6730685906832368,-1.0096028860248545) ++(-2pt,0 pt) -- ++(2pt,2pt)--++(2pt,-2pt)--++(-2pt,-2pt)--++(-2pt,2pt);
\draw [fill=ccqqqq,shift={(2.9632007067035326,1.444801060055299)},rotate=90] (0,0) ++(0 pt,3pt) -- ++(2.598076211353316pt,-4.5pt)--++(-5.196152422706632pt,0 pt) -- ++(2.598076211353316pt,4.5pt);
\draw [fill=ccqqqq,shift={(-3.0367992932964674,-1.555198939944701)},rotate=90] (0,0) ++(0 pt,3pt) -- ++(2.598076211353316pt,-4.5pt)--++(-5.196152422706632pt,0 pt) -- ++(2.598076211353316pt,4.5pt);
\draw [fill=ccqqqq] (2.5994700040903016,0.8992050061354526) ++(-2pt,0 pt) -- ++(2pt,2pt)--++(2pt,-2pt)--++(-2pt,-2pt)--++(-2pt,2pt);
\draw [fill=ccqqqq] (-3.4005299959096984,-2.1007949938645476) ++(-2pt,0 pt) -- ++(2pt,2pt)--++(2pt,-2pt)--++(-2pt,-2pt)--++(-2pt,2pt);
\draw [fill=ccqqqq] (2.2357393014770706,0.3536089522156061) circle (2pt);
\draw [fill=ccqqqq] (-3.7642606985229294,-2.6463910477843937) circle (2pt);
\end{scriptsize}
\end{tikzpicture}
\caption{A sequence of adjacent segments contained inside a short cylinder. On the right, we show a single pair of adjacent segments in the sequence. Its direction is contained in the sector defined by the direction of $e'$ and the direction of the diagonal $\Delta$.}
\label{fig:short_cylinder}
\end{figure}
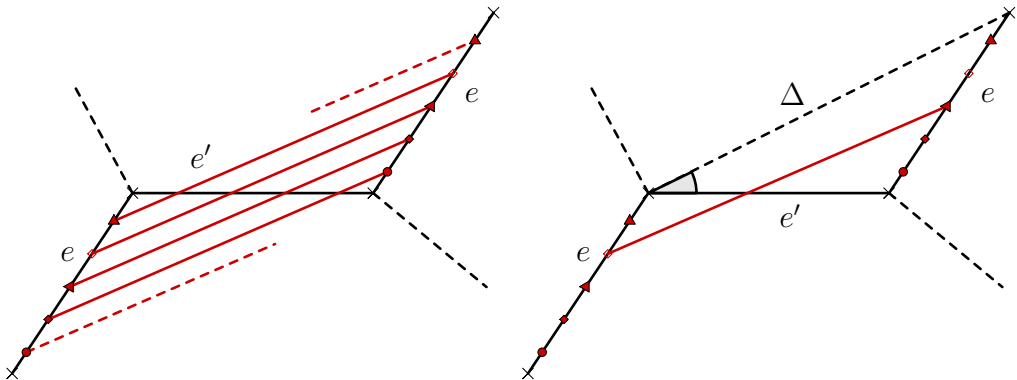

\subsection{Study of the intersections}\label{sec:study_intersections}
We now study the intersections of two saddle connections $\alpha = \alpha_1 \cup \cdots \cup \alpha_k$ and $\beta = \beta_1 \cup \cdots \cup \beta_l$ depending on their polygonal decomposition. Namely, we show:
\begin{Prop}\label{prop:total_intersections}
Let $\alpha$ and $\beta$ be two saddle connections. We have:
\[ |\alpha \cap \beta| \leq (p_{\alpha} + q_{\alpha}) (p_{\beta} + q_{\beta}) \]
Further, if equality holds then between each non-adjacent segment there is an odd number of adjacent segments (that is, both $\alpha$ and $\beta$ are odd saddle connections).
\end{Prop}

The proof of the proposition relies on the notion of \emph{configuration} $\bigstar$, which is defined in the next subsections and essentially states that an intersection involving an adjacent segment allows to remove one to the count of potential intersections. 

\begin{Rema}
The inequality of Proposition \ref{prop:total_intersections} is optimal, as given in the example of Figure \ref{fig:diagramex2}.
\end{Rema}

\begin{Rema}\label{rk:hypothesis_P1'_rk2}
As hinted in Remark \ref{rk:hypothesis_P1'_rk1} and as we will see in the proof, Proposition \ref{prop:total_intersections} is true under the weaker assumptions (P1') and (P2) on the surface $X$.
\end{Rema}

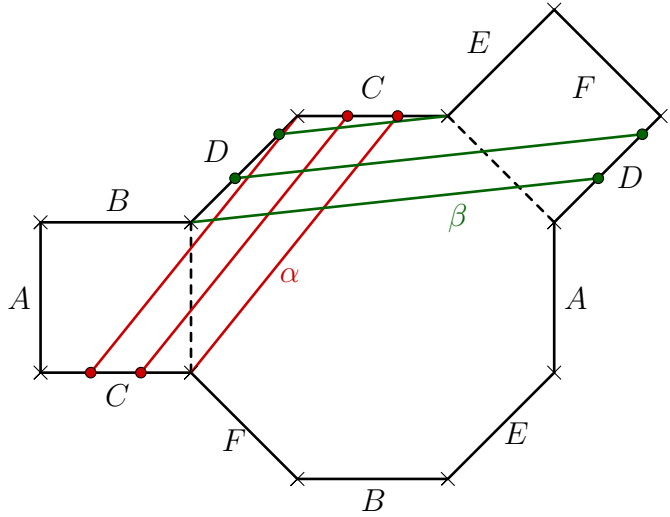
\begin{figure}
\center
\definecolor{uuuuuu}{rgb}{0.26666666666666666,0.26666666666666666,0.26666666666666666}
\definecolor{qqwuqq}{rgb}{0,0.39215686274509803,0}
\definecolor{ccqqqq}{rgb}{0.8,0,0}
\begin{tikzpicture}[line cap=round,line join=round,>=triangle 45,x=1cm,y=1cm]
\clip(-6,-0.5) rectangle (7,8);
\draw [line width=1pt] (0,0)-- (2,0);
\draw [line width=1pt] (2,0)-- (3.414213562373095,1.414213562373095);
\draw [line width=1pt] (3.414213562373095,1.414213562373095)-- (3.414213562373095,3.4142135623730945);
\draw [line width=1pt,dash pattern=on 3pt off 3pt] (3.414213562373095,3.4142135623730945)-- (2,4.82842712474619);
\draw [line width=1pt] (2,4.82842712474619)-- (0,4.82842712474619);
\draw [line width=1pt] (0,4.82842712474619)-- (-1.414213562373095,3.4142135623730954);
\draw [line width=1pt] (-1.4142135623730954,1.4142135623730956)-- (0,0);
\draw [line width=1pt,dash pattern=on 3pt off 3pt] (-1.4142135623730954,1.4142135623730956)-- (-1.414213562373095,3.4142135623730954);
\draw [line width=1pt] (-1.414213562373095,3.4142135623730954)-- (-3.4142135623730945,3.4142135623730954);
\draw [line width=1pt] (-3.4142135623730945,3.4142135623730954)-- (-3.414213562373095,1.414213562373096);
\draw [line width=1pt] (-3.414213562373095,1.414213562373096)-- (-1.4142135623730954,1.4142135623730956);
\draw [line width=1pt] (3.414213562373095,3.4142135623730945)-- (4.82842712474619,4.828427124746188);
\draw [line width=1pt] (4.82842712474619,4.828427124746188)-- (3.4142135623730963,6.242640687119284);
\draw [line width=1pt] (3.4142135623730963,6.242640687119284)-- (2,4.82842712474619);
\draw [line width=1pt,color=ccqqqq] (-1.4142135623730954,1.4142135623730956)-- (1.3333333333333335,4.82842712474619);
\draw [line width=1pt,color=ccqqqq] (-2.080880229039762,1.4142135623730958)-- (0.666666666666667,4.82842712474619);
\draw [line width=1pt,color=ccqqqq] (-2.7475468957064284,1.414213562373096)-- (0,4.82842712474619);
\draw [color=ccqqqq](-0.38,2.94) node[anchor=north west] {$\alpha$};
\draw [color=qqwuqq](1.86,3.8) node[anchor=north west] {$\beta$};
\draw [line width=1pt,color=qqwuqq] (-1.414213562373095,3.4142135623730954)-- (4,4);
\draw [line width=1pt,color=qqwuqq] (-0.828427124746189,4)-- (4.585786437626907,4.585786437626906);
\draw [line width=1pt,color=qqwuqq] (-0.242640687119283,4.585786437626906)-- (2,4.82842712474619);
\draw (-2.7,4) node[anchor=north west] {$B$};
\draw (-2.7,1.4) node[anchor=north west] {$C$};
\draw (0.7,5.5) node[anchor=north west] {$C$};
\draw (-1.4,4.65) node[anchor=north west] {$D$};
\draw (4.1,4.3) node[anchor=north west] {$D$};
\draw (2.1,6.1) node[anchor=north west] {$E$};
\draw (2.6,0.9) node[anchor=north west] {$E$};
\draw (0.7,0) node[anchor=north west] {$B$};
\draw (-1.15,0.8) node[anchor=north west] {$F$};
\draw (3.5,5.5) node[anchor=north west] {$F$};
\draw (-4,2.7) node[anchor=north west] {$A$};
\draw (3.4,2.7) node[anchor=north west] {$A$};
\begin{scriptsize}
\draw [color=black] (0,0)-- ++(-2.5pt,-2.5pt) -- ++(5pt,5pt) ++(-5pt,0) -- ++(5pt,-5pt);
\draw [color=black] (2,0)-- ++(-2.5pt,-2.5pt) -- ++(5pt,5pt) ++(-5pt,0) -- ++(5pt,-5pt);
\draw [color=black] (3.414213562373095,1.414213562373095)-- ++(-2.5pt,-2.5pt) -- ++(5pt,5pt) ++(-5pt,0) -- ++(5pt,-5pt);
\draw [color=black] (3.414213562373095,3.4142135623730945)-- ++(-2.5pt,-2.5pt) -- ++(5pt,5pt) ++(-5pt,0) -- ++(5pt,-5pt);
\draw [color=black] (2,4.82842712474619)-- ++(-2.5pt,-2.5pt) -- ++(5pt,5pt) ++(-5pt,0) -- ++(5pt,-5pt);
\draw [color=black] (0,4.82842712474619)-- ++(-2.5pt,-2.5pt) -- ++(5pt,5pt) ++(-5pt,0) -- ++(5pt,-5pt);
\draw [color=black] (-1.414213562373095,3.4142135623730954)-- ++(-2.5pt,-2.5pt) -- ++(5pt,5pt) ++(-5pt,0) -- ++(5pt,-5pt);
\draw [color=black] (-1.4142135623730954,1.4142135623730956)-- ++(-2.5pt,-2.5pt) -- ++(5pt,5pt) ++(-5pt,0) -- ++(5pt,-5pt);
\draw [color=black] (-3.4142135623730945,3.4142135623730954)-- ++(-2.5pt,-2.5pt) -- ++(5pt,5pt) ++(-5pt,0) -- ++(5pt,-5pt);
\draw [color=black] (-3.414213562373095,1.414213562373096)-- ++(-2.5pt,-2.5pt) -- ++(5pt,5pt) ++(-5pt,0) -- ++(5pt,-5pt);
\draw [color=black] (4.82842712474619,4.828427124746188)-- ++(-2.5pt,-2.5pt) -- ++(5pt,5pt) ++(-5pt,0) -- ++(5pt,-5pt);
\draw [color=black] (3.4142135623730963,6.242640687119284)-- ++(-2.5pt,-2.5pt) -- ++(5pt,5pt) ++(-5pt,0) -- ++(5pt,-5pt);
\draw [fill=ccqqqq] (-2.080880229039762,1.4142135623730958) circle (2pt);
\draw [fill=ccqqqq] (-2.7475468957064284,1.414213562373096) circle (2pt);
\draw [fill=ccqqqq] (1.3333333333333335,4.82842712474619) circle (2pt);
\draw [fill=ccqqqq] (0.666666666666667,4.82842712474619) circle (2pt);
\draw [fill=qqwuqq] (4,4) circle (2pt);
\draw [fill=qqwuqq] (-0.828427124746189,4) circle (2pt);
\draw [fill=qqwuqq] (4.585786437626907,4.585786437626906) circle (2pt);
\draw [fill=qqwuqq] (-0.242640687119283,4.585786437626906) circle (2pt);
\end{scriptsize}
\end{tikzpicture}
\caption{Example of odd saddle connections such that $|\alpha \cap \beta| = (p_{\alpha}+q_{\alpha})(p_{\beta}+q_{\beta})$.}
\label{fig:diagramex2}
\end{figure}

\subsubsection{Configurations \texorpdfstring{$\bigstar$}{bigstar}}

Roughly speaking, a configuration $\bigstar$ gives two pieces $\alpha_i$ and $\beta_j$ which intersect and have one endpoint on a shared side. In particular, we can choose to count this intersection either as an intersection between $\alpha_i$ and $\beta_j$ or deform the segments so that the intersection occur instead as an intersection between the segments of $\alpha$ and $\beta$ consecutive to $\alpha_i$ and $\beta_j$ along this shared side.

Recall that we denoted $\alpha=\alpha_1 \cup \cdots \cup \alpha_k$ and $\beta=\beta_1 \cup \cdots \cup \beta_l$ the polygonal decompositions of the saddle connection $\alpha$ and $\beta$. 

\begin{Def}
A \emph{configuration} $\bigstar$ is given by the ordered data of two pairs of indexes $((i,j) ,(i',j')) \in (\{1,\cdots,k\} \times \{ 1,\cdots, l\})^2$ such that :
\begin{enumerate}[label=(\roman*)]
\item The segments $\alpha_i$ and $\beta_j$ intersect
\item There is a side $e$ of the polygon containing both $\alpha_i$ and $\beta_j$ such that on the interior of $e$ there is one endpoint of $\alpha_i$ and one endpoint of $\beta_j$.
\item $\alpha_{i'}$ and $\beta_{j'}$ are respectively the segments consecutive to $\alpha_i$ and $\beta_j$ continued after this endpoint.
\end{enumerate} 
If $((i,j),(i',j'))$ are in a configuration $\bigstar$, we will use the notation $\Xstar$.
Moreover, we will say that $(i,j)$ \emph{induces a configuration} $\bigstar$ if there exists $(i',j')$ such that $((i,j),(i',j'))$ are in a configuration $\bigstar$.
\end{Def}

Note that $i'$ and $j'$ cannot just be any element of $\{1,\cdots,k\}$ and $\{1,\cdots,l\}$, but given $i$ and $j$, we have $i' \in \{i-1,i+1\}$ and $j \in \{j-1,j+1\}$.

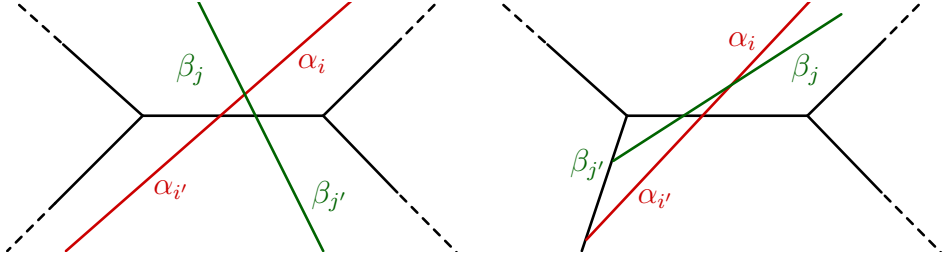
\begin{figure}[h]
\center
\definecolor{ccqqqq}{rgb}{0.8,0,0}
\definecolor{qqwuqq}{rgb}{0,0.39215686274509803,0}
\begin{tikzpicture}[line cap=round,line join=round,>=triangle 45,x=1cm,y=1cm, scale =0.6]
\clip(-5.5,-3) rectangle (5,2.5);
\draw [line width=1pt] (-2,0)-- (2,0);
\draw [line width=1pt] (2,0)-- (3.6,1.58);
\draw [line width=1pt,dash pattern=on 3pt off 3pt] (3.6,1.58)-- (5,3);
\draw [line width=1pt] (-2,0)-- (-3.94,-1.94);
\draw [line width=1pt,dash pattern=on 3pt off 3pt] (-3.94,-1.94)-- (-5,-3);
\draw [line width=1pt] (-2,0)-- (-3.76,1.58);
\draw [line width=1pt,dash pattern=on 3pt off 3pt] (-5.52,3.16)-- (-3.76,1.58);
\draw [line width=1pt] (2,0)-- (3.64,-1.7);
\draw [line width=1pt, dash pattern=on 3pt off 3pt] (5.32,-3.4)-- (3.64,-1.7);
\draw [line width=1pt,color=ccqqqq] (-3.7,-3)-- (3.1414888511148886,3);
\draw [line width=1pt,color=qqwuqq] (-1,3)-- (2,-3);
\draw [color=ccqqqq](-2,-1.2) node[anchor=north west] {$\alpha_{i'}$};
\draw [color=ccqqqq](1.2,1.6) node[anchor=north west] {$\alpha_i$};
\draw [color=qqwuqq](1.5,-1.2) node[anchor=north west] {$\beta_{j'}$};
\draw [color=qqwuqq](-1.5,1.6) node[anchor=north west] {$\beta_{j}$};
\end{tikzpicture}
\begin{tikzpicture}[line cap=round,line join=round,>=triangle 45,x=1cm,y=1cm, scale =0.6]
\clip(-5.5,-3) rectangle (5,2.5);
\draw [line width=1pt] (-2,0)-- (2,0);
\draw [line width=1pt] (2,0)-- (3.6,1.58);
\draw [line width=1pt,dash pattern=on 3pt off 3pt] (3.6,1.58)-- (5,3);
\draw [line width=1pt] (-2,0)-- (-3,-3);
\draw [line width=1pt] (-2,0)-- (-3.76,1.58);
\draw [line width=1pt,dash pattern=on 3pt off 3pt] (-5.52,3.16)-- (-3.76,1.58);
\draw [line width=1pt] (2,0)-- (3.64,-1.7);
\draw [line width=1pt, dash pattern=on 3pt off 3pt] (5.32,-3.4)-- (3.64,-1.7);
\draw [line width=1pt,color=ccqqqq] (-2.9,-2.75)-- (2.25,2.75);
\draw [line width=1pt,color=qqwuqq] (-2.3,-1)-- (2.75,2.25);
\draw [color=ccqqqq](-2,-1.4) node[anchor=north west] {$\alpha_{i'}$};
\draw [color=ccqqqq](0,2.1) node[anchor=north west] {$\alpha_i$};
\draw [color=qqwuqq](-3.5,-0.5) node[anchor=north west] {$\beta_{j'}$};
\draw [color=qqwuqq](1.4,1.6) node[anchor=north west] {$\beta_{j}$};
\end{tikzpicture}

\caption{Examples of configuration $\Xstar$.}
\label{fig:configuration_star}
\end{figure}

The picture to have in mind is given by Figure \ref{fig:configuration_star}. Let us now state several properties related to configurations $\bigstar$ which will turn out to be useful to count intersections.

\begin{Lem}\label{lem:configuration_star}
\begin{enumerate}
\item If $\Xstar$, then the segments $\alpha_{i'}$ and $\beta_{j'}$ lie in the same polygon but do not intersect.
\item If given $(i',j')$, we have $\Xstar$ and $(i'',j'') \xrightarrow[]{\bigstar} (i',j')$, then $i=i''$ and $j=j''$. In other words, the data of $(i',j')$ in a configuration $\bigstar$ determines uniquely $i$ and $j$.
\item Assume $\alpha_i$ is an adjacent segment and $j \neq 1,l$ be such that $\alpha_i \cap \beta_j \neq \varnothing$. Then $(i,j)$ induces a configuration $\bigstar$.
\item Assume $\alpha_i$ is an adjacent segment, $\alpha_i \cap \beta_1 \neq \varnothing$ and $\alpha_i \cap \beta_l \neq \varnothing$, then either $(i,1)$ or $(i,l)$ induce a configuration $\bigstar$.
\end{enumerate}
\end{Lem}
\begin{proof}
\begin{enumerate}
\item Looking at the two polygons glued along the side which contains an endpoint of $\alpha_i$ and $\beta_j$, the two lines defined by $\alpha$ and $\beta$ can only intersect once, so if $\alpha_i$ and $\beta_j$ intersect, then $\alpha_{i'}$ and $\beta_{j'}$ will not (see Figure \ref{fig:configuration_star}).
\item If we have $\Xstar$ and $(i'',j'') \xrightarrow[]{\bigstar} (i',j')$, with $i \neq i''$ and $j \neq j''$, then there exist two sides $e$ and $e'$ such that the interior of $e$ contains one endpoint of $\alpha_{i'}$ and one endpoint of $\beta_{j'}$, while $e'$ contains the other endpoints of $\alpha_{i'}$ and $\beta_{j'}$. 
Moreover, $\alpha_i$ and $\beta_j$ are segments across $e$, while $\alpha_{i''}$ and $\beta_{j''}$ are segments across $e'$. 
Now, the straight lines defined by $\alpha$ and $\beta$ can either be parallel or intersect once, so we cannot have both $\alpha_i$ and $\beta_j$ and also $\alpha_{i''}$ and $\beta_{j''}$ to intersect, as required by the first part of the definition. 
\item[3. and 4.] If $\alpha_i$ is an adjacent segment, then any segment $\beta_j$ intersecting it must have one endpoint on one of the side containing an endpoint of $\alpha_i$. The only case where it does not induce a configuration $\bigstar$ is when the endpoint lies on the vertex of the polygon which is between the two adjacent sides containing an endpoint of $\alpha_i$; this case cannot happen if $j \neq 1,l$ and cannot happen simultaneously for $j=1$ and $j=l$.
\end{enumerate}
\end{proof}

\subsubsection{Counting the intersections}
We are now ready to prove Proposition \ref{prop:total_intersections}. For this purpose, we distinguish the intersections involving non-adjacent segments and adjacent segments. Namely, we define:
\begin{flalign*}
& I_1 := \{ (i,j), \alpha_i \cap \beta_j \neq \varnothing,   \alpha_i \text{ non-adjacent}  \} &&&\\
& I_{2} := \{ (i,j), \alpha_i \cap \beta_j \neq \varnothing, \alpha_i \text{ adjacent} \}
\end{flalign*}
Further, since intersecting segments must lie in the same polygon, we have:
\begin{flalign*}
& I_1 \subset \tilde{I}_{1} := \{ (i,j), \alpha_i \text{ and } \beta_j \text{ lie in the same polygon }, \alpha_i \text{ non-adjacent} \} && &
\end{flalign*}

Next, we partition $I_2$ by distinguishing pairs of indexes of $I_2$ which induce configurations $\bigstar$ and pairs of $I_2$ which do not. We also distinguish pairs of indexes in $I_2$ inducing a configuration $\bigstar$ involving another adjacent segment or a non-adjacent segment, namely:
\begin{flalign*}
& I_{2}^{(0)} := \{ (i,j) \in I_2 \text{ which do not induce a configuration }\bigstar \} && &\\
& I_{2}^{(1)} :=  \{ (i,j) \in I_2, \text{ with } \Xstar \text{ and } \alpha_{i'} \text{ is non-adjacent}\} &&& \\
& I_{2}^{(2)} :=  \{ (i,j) \in I_2 \backslash I_{2}^{(1)}, \text{ such that } \Xstar \text{ and } \alpha_{i'} \text{ is adjacent}\} &&&
\end{flalign*}
It should be noted that a pair of indexes may induce two configurations $\bigstar$ if they intersect and the endpoints of the segments on the interior of the same two sides. This is the reason why we do not take elements of $I_2^{(1)}$ in the definition of $I_2^{(2)}$, as it avoids counting an intersection twice. 
In this way, the three sets form a partition of $I_2$.\newline

The main advantage of this partition is that by construction we have 
\[\#I_2 = \# I_2^{(0)}+\# I_2^{(1)}+\# I_2^{(2)}.  \]

Moreover, given a pair $(i,j) \in I_2^{(1)}$, the configuration $\bigstar$ determines (at least) one pair $(i',j') \in \tilde{I}_1 \backslash I_1$ since by part 1. of Lemma \ref{lem:configuration_star} the segments $\alpha_{i'}$ and $\beta_{j'}$ do not intersect. 
The pairs $(i',j')$ do not overlap by part 2. of Lemma \ref{lem:configuration_star} and hence 
\begin{align*}
\# I_2^{(1)} &\leq \# \left(\tilde{I}_1 \setminus I_1 \right),  &&\text{hence} &\# I_1 &\leq \# \tilde{I}_1 - \# I_2^{(1)}.
\end{align*}

Then,
\[ |\alpha \cap \beta| = \# I_1 + \# I_2 \leq \# \tilde{I}_1 - \# I_2^{(1)} + \# I_2^{(0)}+\# I_2^{(1)}+\# I_2^{(2)} \]
so that:
\begin{equation}\label{eq:sum_intersections}
|\alpha \cap \beta| \leq \# \tilde{I}_1 +\# I_2^{(0)}+\# I_2^{(2)}
\end{equation}

There is an easy estimate for $ \# \tilde{I}_1$, as by hypothesis (P2) two consecutive segments cannot lie in the same polygon, thus for any non-adjacent segment $\alpha_i$, there is at most $\left\lceil \frac{l}{2} \right\rceil$ segments of $\beta$ in the polygon containing $\alpha_i$, so that 
\begin{equation}\label{eq:I_1_tilde}
\# \tilde{I}_1 \leq p_{\alpha} \left\lceil \frac{l}{2} \right\rceil.
\end{equation}

Next, we study $I_2^{(0)}$ and $I_2^{(2)}$. For this purpose, we distinguish isolated adjacent segments (i.e. adjacent segments which are preceded and followed by non-adjacent segments) and sequences of (at least two) consecutive adjacent segments. The reason for this distinction is that isolated adjacent segments do not contribute to $I_2^{(2)}$, since $\alpha_i$ and $\alpha_{i'}$ in a $\bigstar$ configuration are consecutive. Then, an intersection involving an isolated adjacent segment $\alpha_i$ must be either in $I_2^{(1)}$ or in $I_2^{(0)}$. 

To compute $\# I_2^{(0)}$, we will split it further into two according to whether the adjacent segment is \emph{isolated} or a sequence of \emph{consecutive} ones and define 
\begin{align*}
I_2^{(0,i)} &:= \{ (i,j) \in I_2^{(0)}, \alpha_i \text{ is an isolated adjacent segment} \}, \\
I_2^{(0,c)} &:= I_2^{(0)} \backslash I_2^{(0,i)} = \{ (i,j) \in I_2^{(0)}, \alpha_i \text{ is not an isolated adjacent segment} \}.
\end{align*}

\begin{Lem}\label{lem:lonely_adjacent}
We have:
\[ \# I_2^{(0,i)} \leq \# \{i, \alpha_i \text{ is an isolated adjacent segment} \}. \]
\end{Lem}
\begin{proof}
The quantity $\# I_2^{(0,i)}$ counts the number of pairs $(i,j)$ such that $\alpha_i \cap \beta_j \neq \varnothing$, $\alpha_i$ is adjacent and isolated and $(i,j)$ do not induce a configuration $\bigstar$. 
Now, by part 3. of Lemma \ref{lem:configuration_star} we necessarily have that $j=1$ or $j=l$, so for each $i$, the set contains at most two pairs. 
Now, from part 4. of Lemma \ref{lem:configuration_star}, we cannot have both pairs $(i,1)$ and $(i,l)$ because one of them must induce a configuration $\bigstar$. 
In other words, we are counting the number of $i$'s such that $\alpha_i$ is adjacent and isolated and $\alpha_i$ intersects one of $\beta_1$ and $\beta_l$, which is certainly smaller than the number of $i$ such that $\alpha_i$ is adjacent and isolated.

\end{proof}

We are left to study the intersections of $\beta$ with sequences of at least two consecutive adjacent segments, i.e. we want to estimate $\# I_2^{(0,c)}$. 
This is the purpose of the next section. 

\subsubsection{Intersections in sequences of consecutive adjacent segments}
In this section we consider a maximal sequence of $q$ consecutive adjacent segments $\alpha_{i_0+1} \cup \cdots \cup \alpha_{i_0+q}$ with $q \geq 1$, that is $\alpha_{i_0}$ and $\alpha_{i_0+q+1}$ are non-adjacent segments while all the in-between segments are adjacent. We show:

\begin{Lem}\label{lem:intersection_adjacent_sequence}
There are at most $\left\lfloor \frac{q}{2} \right\rfloor \left\lfloor \frac{l}{2}+1 \right\rfloor$ intersections between $\alpha_{i_0+1} \cup \cdots \cup \alpha_{i_0+q}$ and  $\beta$ which are not in $I_2^{(1)}$.
\end{Lem}

\begin{Rema}\label{rem:extension_p}
In particular, we deduce from Lemma \ref{lem:integer_part} that the number of intersections between $\alpha_{i_0+1} \cup \cdots \cup \alpha_{i_0+q}$ and $\beta$ which belong to $I_2^{(0)}$ or $I_2^{(2)}$ is bounded above by $\left\lfloor \frac{q}{2} \right\rfloor (p_{\beta}+q_{\beta})$. Indeed, if $\l$ is odd then $\left\lfloor \frac{l}{2}+1 \right\rfloor = \frac{l+1}{2} = \left\lceil \frac{l}{2} \right\rceil \leq p_{\beta}+q_{\beta}$, and if $l$ is even then $\beta$ is not an odd saddle connection and therefore the inequality is strict in Lemma \ref{lem:integer_part} and $\left\lfloor \frac{l}{2}+1 \right\rfloor = \frac{l}{2}+1 = \left\lceil \frac{l}{2} \right\rceil +1 \leq p_{\beta}+q_{\beta}$.
\end{Rema}
Note that this counts the number of intersections for the sequence of adjacent segments both in $I_2^{(2)}$ and in $I_2^{(0,c)}$. 
Hence this gives:
\begin{Cor}\label{cor:consecutive_adjacent}
\[ \# I_2^{(2)} + \# I_2^{(0,c)} 
\leq q_{\alpha}(p_{\beta} + q_{\beta}). \]
\end{Cor}

\begin{proof}[Proof of Lemma \ref{lem:intersection_adjacent_sequence}.]
We distinguish 4 cases:
\begin{enumerate}[label=(\roman*)]
\item The sequence of adjacent segments is contained inside a short cylinder.
\item $q$ is even but the sequence of adjacent segment is not contained in a short cylinder.
\item $q=3$,
\item $q \geq 5$ is odd. 
\end{enumerate}

Note that when $q$ is odd, then the sequence cannot be contained in a small cylinder by Lemma \ref{lem:maximal_even}. 

For each of the cases, the idea is as follows: we first partition the sequence of maximal segments in pairs (with one triple if $q$ is odd), and then show that 
for each of the $\left\lfloor \frac{q}{2} \right\rfloor$ pairs $\alpha_{i} \cup \alpha_{i+1}$ (or triple $\alpha_i \cup \alpha_{i+1} \cup \alpha_{i+2}$), we can pair consecutive pieces of $\beta$ in such a way that each pair of pieces of $\beta$ intersects $\alpha_{i} \cup \alpha_{i+1}$ (or $\alpha_i \cup \alpha_{i+1} \cup \alpha_{i+2}$) only once. 
We do this by associating to every piece $\beta_j$ intersecting $\alpha_{i} \cup \alpha_{i+1}$ (or $\alpha_i \cup \alpha_{i+1} \cup \alpha_{i+2}$) a consecutive piece $\beta_{j'}$ which does not. 
Up to potentially adding an extra intersection with $\beta_1$ and $\beta_l$, we get at most $\left \lfloor \frac{l}{2} + 1 \right \rfloor$ intersections. 

More precisely, for each pair $\alpha_{i} \cup \alpha_{i+1}$ (resp. triple $\alpha_i \cup \alpha_{i+1} \cup \alpha_{i+2}$):
\begin{itemize}
\item $\beta_1$ (resp. $\beta_l$) intersect the pair (resp. triple) at most once (this is because of assumption (P2) -- and Corollary \ref{cor:3_segments} for the triple --), and
\item given $j$ such that $\beta_j$ intersects one of the segments of the pair (resp. triple), and such that the intersection belongs to $I_2^{(2)}$, we can find $j' \in \{ j-1,j+1 \}$ such that $\# \left( \{ i, i+1 \} \times \{ j,j' \} \cap \left(I_2^{(0,c)} \cup I_2^{(2)}\right) \right) \leq 1$ (resp. $\# \left( \{ i, i+1, i+2 \} \times \{ j,j' \} \cap \left(I_2^{(0,c)} \cup I_2^{(2)}\right) \right) \leq 1$). Further, we can construct the (unordered) pairs $\{j,j'\}$ such that they do not overlap.
\end{itemize}
Here, $\beta_1$ and $\beta_l$ may or may not be paired. In total,
\begin{itemize}
\item[-] If neither $\beta_1$ nor $\beta_l$ are paired, there are at most $\left\lfloor \frac{l-2}{2} \right\rfloor$ pairs, so that adding the intersections with $\beta_1$ and $\beta_l$, this implies that among all intersections between $\alpha_{i} \cup \alpha_{i+1}$ and $\beta$, at most $\left\lfloor \frac{l-2}{2} \right\rfloor + 2 = \left\lfloor \frac{l}{2} +1 \right\rfloor$ intersections in $I_2^{(0)} \cup I_2^{(2)}$.
\item[-] If $\beta_1$ (resp. $\beta_l$) is paired but $\beta_l$ (resp. $\beta_1$) is not, then we can construct at most $\left\lfloor \frac{l-1}{2} \right\rfloor$ pairs. Since $\beta_1$ (resp. $\beta_l$) is already paired, we only need to add the possible intersection with $\beta_l$ (resp. $\beta_1$), and thus we get at most $\left\lfloor \frac{l-1}{2} \right\rfloor + 1 = \left\lfloor \frac{l+1}{2} \right\rfloor$ intersections for this pair.
\item[-] Finally, if both $\beta_1$ and $\beta_l$ are paired, this gives $\left\lfloor \frac{l}{2} \right\rfloor$ intersections.
\end{itemize}
As a conclusion, we get at most $\left\lfloor \frac{l}{2} +1 \right\rfloor$ intersections.\newline

Although the intersections with pairs in $I_2^{(2)}$ induce a configuration $\bigstar$, this is not sufficient to construct the pairs $\{j,j' \}$, because some pairs might overlap. Here, we will investigate all possible configurations in each case to construct the pairs and to make sure that the pairs do not overlap.

\paragraph{Case (i).} \textbf{Assume the sequence of adjacent segments is contained inside a short cylinder.} Then $q$ is even by Lemma \ref{lem:maximal_even}, and we can group the adjacent segments by pairs. Say each pair form a sandwiched segment of type $e \to e' \to e$. As explained in \S \ref{sec:adjacent_segments}, we have a picture like Figure \ref{fig:short_cylinder}. In the following, we will denote by $\alpha_i$ and $\alpha_{i+1}$ the two adjacent segments of the pair, respectively of type $e \to e'$ and of type $e' \to e$.

Now, notice that a segment $\beta_j$ intersecting $\alpha_i \cup \alpha_{i+1}$ must have an endpoint either on the side $e$ or on $e'$ (or on both). The corresponding endpoint is not on a vertex except maybe for $j=1$ or $j=l$. We then define the index $j' = j \pm 1$ which will be paired to $j$ as follows.
\begin{itemize}
\item[1.] If $\beta_j$ has one of its endpoints on $e'$, we take $j'$ such that $\beta_{j}$ and $\beta_{j'}$ have this endpoint in common.
\item[2.] Else, $\beta_j$ has one of its endpoints on $e$ but no endpoint on $e'$, and we take $j'$ such that $\beta_{j}$ and $\beta_{j'}$ have this endpoint in common.
\end{itemize}

With this construction, we have:
\begin{Lem}\label{lem:case_short_cylinders}
\begin{enumerate}
\item[(a)] If $j \in\{ 1, \cdots, l \}$ is such that $\beta_j$ intersect $\alpha_i \cup \alpha_{i+1}$, then the segment $\beta_{j'}$ constructed as above does not intersect $\alpha_i \cup \alpha_{i+1}$.
\item[(b)] Given $j_1$ and $j_2$ two distinct indices such that $\beta_{j_1}$ and $\beta_{j_2}$ intersect $\alpha_{i} \cup \alpha_{i+1}$ and such that the intersections belong to $I_2^{(2)}$, the pairs $\{j_1, j_1' \}$ and $\{j_2, j_2' \}$ constructed as above satisfy:
 \[ \{ j_1, j_1' \} \cap \{ j_2, j_2' \} = \varnothing. \]
\end{enumerate}
\end{Lem}

\begin{proof}
\begin{enumerate}
\item[(a)] We can assume that $\beta_j$ intersects the segment $\alpha_i$. That means in particular that $\beta_{j'}$ does not lie in the same polygon as $\alpha_i$, so these segments cannot intersect. Further, 
\begin{itemize}
\item[1.] if $\beta_j$ has one of its endpoints on $e'$, then it means that we have $(i,j) \xrightarrow[]{\bigstar} (i+1,j')$ and hence $\beta_{j'}$ do not intersect $\alpha_{i} \cup \alpha_{i+1}$ by part 1. of Lemma \ref{lem:configuration_star}. 
\item[2.] Else, $\beta_j$ has one of its endpoints on $e$ but no endpoint on $e'$, and that means we have a configuration like Figure \ref{fig:short_cylinder_2}. In particular, since the direction of $\alpha$ is contained in the sector defined by the direction of $e'$ and the direction of the diagonal $\Delta$, we directly deduce that $\beta_{j'}$ cannot intersect $\alpha_{i+1}$.
\end{itemize}

\item[(b)] As in the proof of $(a)$, we will assume that $\beta_{j_1}$ intersects the segment $\alpha_i$. Since $j_1$ and $j_2$ are assumed to be distinct and both intersect the pair, we have by $(a)$ that $j_2 \neq j_1'$ and $j_1 \neq j_2'$. It remains to show that $j_1' \neq j_2'$.
\begin{itemize}
\item[$1\&1$.] If $\beta_{j_1}$ and $\beta_{j_2}$ both have one endpoint on $e'$ (case 1.), then this endpoint must be shared respectively with $\beta_{j_1'}$ and $\beta_{j_2'}$. In particular we cannot have $j_1' = j_2'$ as it would then imply $j_1=j_2$.

\item[$1\&2$.] Now, if $\beta_{j_1}$ has an endpoint on $e'$ (case 1.) but $\beta_{j_2}$ does not (case 2.), the condition $j_1'=j_2'$ would imply that $\beta_{j_1}, \beta_{j_1'}=\beta_{j_2'}$ and $\beta_{j_2}$ are consecutive segments of $\beta$, and that $\beta_{j_1'}=\beta_{j_2'}$ has an endpoint on $e'$ (shared with $\beta_{j_1}$) and an endpoint on $e$ (shared with $\beta_{j_2}$). Further, since we assumed by symmetry that $\beta_{j_1}$ intersects $\alpha_i$, $\beta_{j_1}$ lies in the polygon containing $\alpha_{i}$, we have that $\beta_{j_1'}$ lies in the polygon containing $\alpha_{i+1}$, and $\beta_{j_2}$ lies in the polygon containing $\alpha_i$. This means that $\beta_{j_2}$ intersects $\alpha_{i}$, and, combined with the fact that it has an endpoint on $e$ and that its direction is comprised between the direction of $\alpha_i$ and the direction of $e$, it is easily seen to imply that $\beta_{j_2}$ has an endpoint on $e'$, which is not the case by assumption. This gives a contradiction.\newline

By symmetry, this is the same if $\beta_{j_2}$ has an endpoint on $e'$ but $\beta_{j_1}$ does not.
\item[$2\&2$.] Else, both $\beta_{j_1}$ and $\beta_{j_2}$ have no endpoint on $e'$ but one endpoint on $e$. That means that this endpoint must be shared with $\beta_{j_1'}$ (resp. $\beta_{j_2'}$) so that we cannot have $j_1' = j_2'$ unless $j_1=j_2$.
\end{itemize}

Hence, in all cases we have $\{ j_1, j_1' \} \cap \{ j_2, j_2' \} = \varnothing$.
\end{enumerate}
\end{proof}

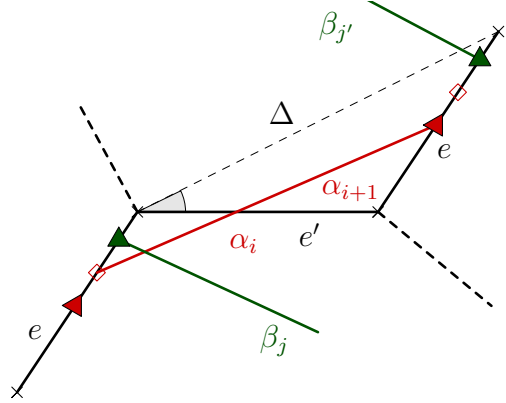
\begin{figure}[h]
\center
\definecolor{ccqqqq}{rgb}{0.8,0,0}
\definecolor{qqwuqq}{rgb}{0,0.33,0}
\begin{tikzpicture}[line cap=round,line join=round,>=triangle 45,x=1cm,y=1cm, scale=0.8]
\clip(-4.2,-3.5) rectangle (4.2,3.5);
\draw [line width=1pt] (-2,0)-- (2,0);
\draw [line width=1pt] (2,0)-- (4,3);
\draw [line width=1pt] (-2,0)-- (-4,-3);
\draw [line width=1pt,color=ccqqqq] (-2.6730685906832368,-1.0096028860248545)-- (2.9632007067035326,1.444801060055299);
\draw (2.8,1.3) node[anchor=north west] {$e$};
\draw (0.5,0) node[anchor=north west] {$e'$};
\draw (-4,-1.7) node[anchor=north west] {$e$};
\draw [line width=1pt,dash pattern=on 3pt off 3pt] (-2,0)-- (-2.979318519691461,1.8079742150383107);
\draw [line width=1pt,dash pattern=on 3pt off 3pt] (2,0)-- (3.8828052146241174,-1.5622840240938458);
\draw [line width=0.3pt,dash pattern=on 3pt off 3pt] (-2,0)-- (4,3);
\draw [shift={(-2,-0)},line width=0.3pt,color=black,fill=black,fill opacity=0.10000000149011612] (0,0) -- (0:0.8) arc (0:26:0.8) -- cycle;
\draw (0,2) node[anchor=north west] {$\Delta$};
\draw [color = qqwuqq, line width=1pt] (-2.3093378880700053,-0.4640068321050079)--(1,-2);
\draw [color = qqwuqq, line width=1pt] (3.6906621119299947,2.535993167894992)--(1.5,3.7);
\draw (0.7,-1.7) [color=qqwuqq] node[anchor=north east] {$\beta_j$};
\draw (1.8,3.5) [color=qqwuqq] node[anchor=north east] {$\beta_{j'}$};
\draw (2.2,0.7) [color=ccqqqq] node[anchor=north east] {$\alpha_{i+1}$};
\draw (-0.65,-0.2) [color=ccqqqq] node[anchor=north west] {$\alpha_{i}$};
\begin{scriptsize}
\draw [color=black] (-2,0)-- ++(-2.5pt,-2.5pt) -- ++(5pt,5pt) ++(-5pt,0) -- ++(5pt,-5pt);
\draw [color=black] (2,0)-- ++(-2.5pt,-2.5pt) -- ++(5pt,5pt) ++(-5pt,0) -- ++(5pt,-5pt);
\draw [color=black] (4,3)-- ++(-2.5pt,-2.5pt) -- ++(5pt,5pt) ++(-5pt,0) -- ++(5pt,-5pt);
\draw [color=black] (-4,-3)-- ++(-2.5pt,-2.5pt) -- ++(5pt,5pt) ++(-5pt,0) -- ++(5pt,-5pt);
\draw [fill=qqwuqq,shift={(3.6906621119299947,2.535993167894992)}] (0,0) ++(0 pt,6pt) -- ++(5.19pt,-9pt)--++(-10.39pt,0 pt) -- ++(5.19pt,9pt);
\draw [fill=qqwuqq,shift={(-2.3093378880700053,-0.4640068321050079)}] (0,0) ++(0 pt,6pt) -- ++(5.19pt,-9pt)--++(-10.39pt,0 pt) -- ++(5.19pt,9pt);
\draw [color=ccqqqq] (3.3269314093167632,1.9903971139751455) ++(-4pt,0 pt) -- ++(4pt,4pt)--++(4pt,-4pt)--++(-4pt,-4pt)--++(-4pt,4pt);
\draw [color=ccqqqq] (-2.6730685906832368,-1.0096028860248545) ++(-4pt,0 pt) -- ++(4pt,4pt)--++(4pt,-4pt)--++(-4pt,-4pt)--++(-4pt,4pt);
\draw [fill=ccqqqq,shift={(2.9632007067035326,1.444801060055299)},rotate=90] (0,0) ++(0 pt,6pt) -- ++(5.19pt,-9pt)--++(-10.39pt,0 pt) -- ++(5.19pt,9pt);
\draw [fill=ccqqqq,shift={(-3.0367992932964674,-1.555198939944701)},rotate=90](0,0) ++(0 pt,6pt) -- ++(5.19pt,-9pt)--++(-10.39pt,0 pt) -- ++(5.19pt,9pt);
\end{scriptsize}
\end{tikzpicture}
\caption{Example of pair of adjacent segments $\alpha_{i} \cup \alpha_{i+1}$ in case (i), with $\beta_{j}$ having an endpoint on $e$ but not on $e'$. We recall that the direction of $\alpha$ is contained in the sector defined by the direction of $e'$ and the direction of the diagonal $\Delta$. In particular, the segment $\beta_{j'}$ cannot intersect $\alpha_{i+1}$.}
\label{fig:short_cylinder_2}
\end{figure}

From Lemma \ref{lem:case_short_cylinders} we obtain that for each pair $\alpha_i \cup \alpha_{i+1}$, we can group segments of $\beta$ by pairs which are not overlapping and such that such that each pair intersects $\alpha_i \cup \alpha_{i+1}$ at most once. As already explained, $\beta_1$ and $\beta_l$ could remain unpaired and we conclude that $\beta$ can intersect $\alpha_{i} \cup \alpha_{i+1}$ at most $\left\lfloor \frac{l}{2}+1 \right\rfloor$ times. Since the sequence of consecutive adjacent segments is made of $\frac{q}{2}$ pairs, we get that $\beta$ intersects the sequence of adjacent segment at most $\frac{q}{2} \left\lfloor \frac{l}{2}+1 \right\rfloor$, as required.


\paragraph{Case (ii).} Now, \textbf{assume $q$ is even but the sequence of consecutive adjacent segments is not contained in a small cylinder}, that is two adjacent segment of the sequence $\alpha_{i}$ and $\alpha_{i+2}$ are never contained in the same polygon (see Corollary \ref{cor:3_segments}). Since $q$ is even we can group adjacent segments by pairs. Let $\alpha_{i} \cup \alpha_{i+1}$ be such a pair, as in Figure \ref{fig:even_adjacents}. We denote by $e$ the side containing $\alpha_{i}^-$, by $e'$ the side containing $\alpha_i^+ = \alpha_{i+1}^-$ and by $e''$ the side containing $\alpha_{i+1}^+$. The sides $e, e'$ and $e''$ are different or the sequence of consecutive adjacent segments would be contained in a short cylinder. 

Similarly to the first case, each segment $\beta_j$ intersecting the sequence of adjacent segments must have an endpoint on at least one of the sides $e,e'$ or $e''$. Assuming $j \neq 1,l$,  the corresponding endpoint is not a vertex, and hence we define an index $j' = j \pm 1$ which will be paired to $j$ as follows:
\begin{enumerate}
\item \label{case1} If $\beta_j$ has an endpoint on $e'$, then we choose $j'$ such that $\beta_{j}$ and $\beta_{j'}$ share their common endpoint on $e'$.
\item If $\beta_j$ lies in the same polygon as $\alpha_i$ and has no endpoint on $e'$ but an endpoint on $e$, then we choose $j'$ such that $\beta_j$ and $\beta_{j'}$ share their common endpoint on $e$.
\item \label{case3} Else, $\beta_j$ lies in the same polygon as $\alpha_{i+1}$ and has no endpoint on $e'$ but one endpoint on $e''$. In this case, $\beta_{j}$ intersects $\alpha_{i+1}$ and we have a configuration $\bigstar$ given by $(i+1,j) \xrightarrow[]{\bigstar} (i+2,j')$. Then,
\begin{itemize}
\item either $\alpha_{i+2}$ is non-adjacent and so $(i+1,j) \in I_2^{(1)}$, so that we do not have to count this intersection.
\item Or $\alpha_{i+2}$ is adjacent, and in this case we can perform a continuous deformation of $\beta$ which will move the intersection of $\beta_{j}$ and $\alpha_{i+1}$ to an intersection of $\alpha_{i+2}$ with $\beta_{j'}$, see Figure \ref{fig:even_adjacents}. This does not change the total number of intersections. In other words, instead of counting the intersection as an intersection between $\beta_{j}$ and $\alpha_{i+1}$, we count it as an intersection between $\beta_{j'}$ and $\alpha_{i+2}$ instead. The index paired with $j$ is still $j'$ but this procedure changes the intersection to an intersection in case \ref{case1} for the pair $\alpha_{i+2} \cup \alpha_{i+3}$.\newline

The reason why we use this argument to move the intersection to the next pair is that otherwise there could be a segment $\beta_{j}$ intersecting $\alpha_{i+1}$ and having an endpoint on $e''$, and then $\beta_{j+1}$ could have an endpoint on $e$ and $\beta_{j+2}$ intersects $\alpha_{i}$. Following the rules, $j+1$ would appear in a pair with both $j$ and $j+2$. Another way to understand this deformation argument is to notice that in the case where $j+1$ should be paired with both $j$ and $j+2$, none of the segments $\beta_j$, $\beta_{j+1}$ and $\beta_{j+2}$ intersect $\alpha_{i+2} \cup \alpha_{i+3}$: this compensates for the extra intersection of $\beta$ with the pair $\alpha_{i} \cup \alpha_{i+1}$, and the deformation argument is a way to formalize this compensation.
\end{itemize}
\end{enumerate}

\begin{figure}
\center
\definecolor{ccqqqq}{rgb}{0.8,0,0}
\definecolor{qqwuqq}{rgb}{0,0.33,0}
\begin{tikzpicture}[line cap=round,line join=round,>=triangle 45,x=1cm,y=1cm, scale=0.7]
\clip(-4.2,-3.5) rectangle (5.2,3.5);
\draw [line width=1pt] (-2,0)-- (2,0);
\draw [line width=1pt] (2,0)-- (4,3);
\draw [line width=1pt] (-2,0)-- (-4,-3);
\draw [line width=1pt,color=ccqqqq] (-2.6730685906832368,-1.0096028860248545)-- (2.9632007067035326,1.444801060055299);
\draw [line width=1pt, color=ccqqqq, dash pattern=on 3pt off 3pt] (2.96, 1.44) -- (5,2.35);
\draw (2.5,2.5) node[anchor=north west] {$e''$};
\draw (0.5,0) node[anchor=north west] {$e'$};
\draw (-3.8,-1.3) node[anchor=north west] {$e$};
\draw [line width=1pt,dash pattern=on 3pt off 3pt] (-2,0)-- (-2.979318519691461,1.8079742150383107);
\draw [line width=1pt,dash pattern=on 3pt off 3pt] (2,0)-- (3.8828052146241174,-1.5622840240938458);
\draw [color = qqwuqq, line width=1pt] (-1,3)--(5,-1);
\draw (1,3) [color=qqwuqq] node[anchor=north east] {$\beta_j$};
\draw (4.75,0.75) [color=qqwuqq] node[anchor=north east] {$\beta_{j'}$};
\draw (1,1.2) [color=ccqqqq] node[anchor=north east] {$\alpha_{i+1}$};
\draw (-1.5,-0.5) [color=ccqqqq] node[anchor=north west] {$\alpha_{i}$};
\draw (3.8,2) [color=ccqqqq] node[anchor=north west] {$\alpha_{i+2}$};
\begin{scriptsize}
\draw [color=black] (-2,0)-- ++(-2.5pt,-2.5pt) -- ++(5pt,5pt) ++(-5pt,0) -- ++(5pt,-5pt);
\draw [color=black] (2,0)-- ++(-2.5pt,-2.5pt) -- ++(5pt,5pt) ++(-5pt,0) -- ++(5pt,-5pt);
\draw [color=black] (4,3)-- ++(-2.5pt,-2.5pt) -- ++(5pt,5pt) ++(-5pt,0) -- ++(5pt,-5pt);
\draw [color=black] (-4,-3)-- ++(-2.5pt,-2.5pt) -- ++(5pt,5pt) ++(-5pt,0) -- ++(5pt,-5pt);
\end{scriptsize}
\end{tikzpicture}
\begin{tikzpicture}[line cap=round,line join=round,>=triangle 45,x=1cm,y=1cm, scale=0.7]
\clip(-4.2,-3.5) rectangle (5.2,3.5);
\draw [line width=1pt] (-2,0)-- (2,0);
\draw [line width=1pt] (2,0)-- (4,3);
\draw [line width=1pt] (-2,0)-- (-4,-3);
\draw [line width=1pt,color=ccqqqq] (-2.6730685906832368,-1.0096028860248545)-- (2.9632007067035326,1.444801060055299);
\draw [line width=1pt, color=ccqqqq, dash pattern=on 3pt off 3pt] (2.96, 1.44) -- (5,2.35);
\draw (2.5,1) node[anchor=north west] {$e''$};
\draw (0.5,0) node[anchor=north west] {$e'$};
\draw (-3.8,-1.3) node[anchor=north west] {$e$};
\draw [line width=1pt,dash pattern=on 3pt off 3pt] (-2,0)-- (-2.979318519691461,1.8079742150383107);
\draw [line width=1pt,dash pattern=on 3pt off 3pt] (2,0)-- (3.8828052146241174,-1.5622840240938458);
\draw [color = qqwuqq, line width=1pt] (-1,3)--(3.33,2);
\draw [color = qqwuqq, line width=1pt] (3.33,2)--(5,-1);
\draw (1,3.5) [color=qqwuqq] node[anchor=north east] {$\beta_j$};
\draw (4.5,0) [color=qqwuqq] node[anchor=north east] {$\beta_{j'}$};
\draw (1,1.2) [color=ccqqqq] node[anchor=north east] {$\alpha_{i+1}$};
\draw (-1.5,-0.5) [color=ccqqqq] node[anchor=north west] {$\alpha_{i}$};
\draw (3.8,2) [color=ccqqqq] node[anchor=north west] {$\alpha_{i+2}$};
\begin{scriptsize}
\draw [color=black] (-2,0)-- ++(-2.5pt,-2.5pt) -- ++(5pt,5pt) ++(-5pt,0) -- ++(5pt,-5pt);
\draw [color=black] (2,0)-- ++(-2.5pt,-2.5pt) -- ++(5pt,5pt) ++(-5pt,0) -- ++(5pt,-5pt);
\draw [color=black] (4,3)-- ++(-2.5pt,-2.5pt) -- ++(5pt,5pt) ++(-5pt,0) -- ++(5pt,-5pt);
\draw [color=black] (-4,-3)-- ++(-2.5pt,-2.5pt) -- ++(5pt,5pt) ++(-5pt,0) -- ++(5pt,-5pt);
\end{scriptsize}
\end{tikzpicture}
\caption{A pair of adjacent segments $\alpha_{i} \cup \alpha_{i+1}$ in case (ii). If a segment $\beta_j$ intersect $\alpha_{i+1}$ and has one endpoint end $e''$ but no endpoint on $e'$, we perform a small deformation of $\beta$ so that the intersection of $\beta_j$ with $\alpha_{i+1}$ moves to an intersection of $\beta_{j'}$ with $\alpha_{i+2}$. This deformation does not change the type of the segments.}
\label{fig:even_adjacents}
\end{figure}
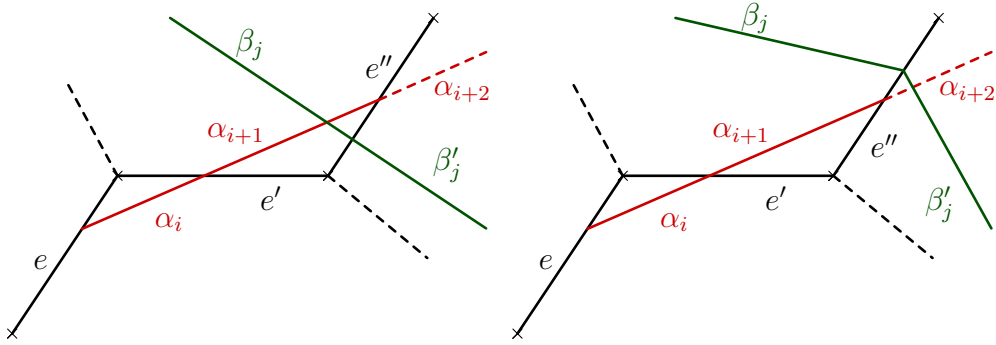

\begin{Lem}\label{lem:case_even_adjacent}
\begin{enumerate}
\item[(a)] If $j \in \{ 1, \cdots, l \}$ is such that $\beta_j$ is either in case 1 or 2, then the segment $\beta_{j'}$ constructed as above does not intersect $\alpha_i \cup \alpha_{i+1}$.
\item[(b)] Given $j_1$ and $j_2$ two distinct indices such that $\beta_{j_1}$ and $\beta_{j_2}$ are either in case 1 or 2, the pairs $\{j_1, j_1' \}$ and $\{j_2, j_2' \}$ constructed as above satisfy:
 \[ \{ j_1, j_1' \} \cap \{ j_2, j_2' \} = \varnothing. \]
\end{enumerate}
\end{Lem}

\begin{proof}
\begin{enumerate}
\item[(a)] We distinguish two cases depending whether $\beta_j$ has an endpoint on $e'$ or not.
\begin{itemize}
\item[1.] If $\beta_j$ has an endpoint on $e'$, then we have either $(i,j) \xrightarrow[]{\bigstar} (i+1,j')$ or $(i+1,j) \xrightarrow[]{\bigstar} (i,j')$ depending on whether $\alpha_i$ or $\alpha_{i+1}$ intersects $\beta_{j}$ and the result holds by 1. of Lemma \ref{lem:configuration_star}.
\item[2.] If $\beta_j$ has no endpoint on $e'$ but an endpoint on $e$, then in particular $\beta_j$ intersects $\alpha_i$, and $(i,j) \xrightarrow[]{\bigstar} (i-1,j')$. Then, either $\alpha_{i-1}$ is non-adjacent and in this case $(i,j) \in I_2^{(1)}$ and we do not have to count this intersection, or $\alpha_{i-1}$ is adjacent but by Corollary \ref{cor:3_segments} the segments $\alpha_{i-1}$ and $\alpha_{i+1}$ must lie in two different polygons. In particular, $\beta_{j'}$ and $\alpha_{i+1}$ also lie in different polygons and hence do not intersect.
\end{itemize}
\item[(b)] Similarly to the proof of $(a)$, we distinguish two cases depending whether $\beta_{j_1}$ has an endpoint on $e'$ or not. Further, as in the proof of Lemma \ref{lem:case_short_cylinders}, we know that $j_1 \neq j_2'$ and $j_1' \neq j_2$ so that we only have to prove that $j'_1 \neq j'_2$. We will proceed by contradiction:
\begin{itemize}
\item[$1\&1$.] If $\beta_{j_1}$ and $\beta_{j_2}$ both have an endpoint on $e'$, then by definition this endpoint must be shared respectively with $\beta_{j_1'}$ and $\beta_{j_2'}$, so that $j_1'=j_2'$ implies that $j_1=j_2$.
\item[$1\&2$.] If $\beta_{j_1}$ has an endpoint on $e'$ but $\beta_{j_2}$ does not (which is then in case 2., that is, it intersects $\alpha_i$ and has an endpoint on $e$, shared with $\beta_{j_2'}$), the condition $j_1'=j_2'$ implies that $\beta_{j_1'}=\beta_{j_2'}$ has one endpoint on $e'$ and one endpoint on $e$, but does not belong to the polygon containing $\alpha_i$ since $\beta_{j_2}$ already belongs to this polygon. In particular, both $\beta_{j_1}$ and $\beta_{j_2}$ belong to the polygon containing $\alpha_i$ while $\beta_{j_1'}=\beta_{j_2'}$ belongs to the polygon containing $\alpha_{i+1}$ (so that the side $e$ is also identified to a side of the polygon containing $\alpha_{i+1}$). Further, since $\beta_{j_1'}$ has its endpoints on $e'$ and $e$, the direction of $\beta$ is between the direction of $e'$ and the direction of $e$, and hence the fact that $\beta_{j_2}$ has an endpoint on $e$ combined with the fact that it intersect $\alpha_i$ implies that it must have an endpoint on $e'$, see Figure \ref{fig:argument_caseii}. This gives a contradiction.\newline

By symmetry, this is the same if $\beta_{j_2}$ has an endpoint on $e'$ and $\beta_{j_1}$ does not.
\item[$2\&2$.] Else, $\beta_{j_1}$ (resp. $\beta_{j_2}$) have no endpoint on $e'$ but one endpoint on $e$. That means this endpoint must be shared with $\beta_{j_1'}$ (resp. $\beta_{j_2'}$), and hence $j_1' = j_2'$ implies $j_1 = j_2$. 
\end{itemize}
Hence, in both cases we have $\{ j_1, j_1' \} \cap \{ j_2, j_2' \} = \varnothing$.
\end{enumerate}
\end{proof}

Similarly to case $(i)$, Lemma \ref{lem:case_even_adjacent} together with the explanation in case (ii), part 3., imply that we can group segments of $\beta$ by pairs (except $\beta_1$ and $\beta_l$ which may remain alone) such that for each pair of $\beta$ and each pair $\alpha_{i} \cup \alpha_{i+1}$, there is at most one intersection which does not belong to $I_2^{(1)}$. In total, this gives at most $\frac{q}{2} \left\lfloor \frac{l}{2}+1 \right\rfloor$ intersections, as required.

\begin{figure}
\centering
\definecolor{qqwuqq}{rgb}{0,0.39215686274509803,0}
\definecolor{ccqqqq}{rgb}{0.8,0,0}
\definecolor{ffvvqq}{rgb}{1,0.3333333333333333,0}
\begin{tikzpicture}[line cap=round,line join=round,>=triangle 45,x=1cm,y=1cm]
\clip(-3,-2) rectangle (6,5);
\draw [line width=1pt] (-2,0)-- (2,0);
\draw [line width=1pt] (2,0)-- (4,2);
\draw [line width=1pt,color=ffvvqq] (-2,0)-- (-2.636559574636737,-1.782297314914103);
\draw [line width=1pt,color=ccqqqq] (-2.2982191207248612,-0.8349809810453053)-- (2.556220772391884,0.5562207723918839);
\draw [line width=1pt,color=ffvvqq] (4.693284374828831,2.9441348300901393)-- (5.329843949465568,4.726432145004242);
\draw [line width=1pt,color=qqwuqq] (0.21479621120187709,0)-- (4.88993510998971,3.494735419914166);
\draw [line width=1pt,dash pattern=on 3pt off 3pt] (4,2)-- (4.693284374828831,2.9441348300901393);
\draw [line width=1pt,dash pattern=on 3pt off 3pt] (-2,0)-- (-2.9697342996124436,2.1848063871222445);
\draw [line width=1pt,dash pattern=on 3pt off 3pt] (2,0)-- (3.259858232491485,-1.4878638370285928);
\draw [line width=1pt,dash pattern=on 3pt off 3pt,color=qqwuqq] (0.21479621120187709,0)-- (-1.7377626421441303,-1.4258778416842748);
\draw [line width=1pt,dash pattern=on 3pt off 3pt,color=qqwuqq] (-2.439908839475858,-1.2316967250900763)-- (-0.7924762131432854,-0.12417193945359832);
\draw [color=ccqqqq](1.4390196192521394,1.0225689744162834) node[anchor=north west] {$\alpha_{i+1}$};
\draw [color=ccqqqq](0.013341726332832348,-0.21715093247007522) node[anchor=north west] {$\alpha_i$};
\draw [color=qqwuqq](-0.6065182271103446,-0.3566194219947906) node[anchor=north west] {$\beta_{j_1}$};
\draw [color=qqwuqq](2.182851563383952,2.5257293615159933) node[anchor=north west] {$\beta_{j_1'}$};
\draw [color=qqwuqq](-2.2801401014069227,-0.790521389405016) node[anchor=north west] {$\beta_{j_2}$};
\draw (-2.6520560734728287,-0.3256264243226316) node[anchor=north west] {$e$};
\draw (5.065200346894725,4.1993512358125775) node[anchor=north west] {$e$};
\draw (-0.6840007212907417,0.7436319953668528) node[anchor=north west] {$e'$};
\begin{scriptsize}
\draw [color=black] (-2,0)-- ++(-2.5pt,-2.5pt) -- ++(5pt,5pt) ++(-5pt,0) -- ++(5pt,-5pt);
\draw [color=black] (2,0)-- ++(-2.5pt,-2.5pt) -- ++(5pt,5pt) ++(-5pt,0) -- ++(5pt,-5pt);
\draw [color=black] (4,2)-- ++(-2.5pt,-2.5pt) -- ++(5pt,5pt) ++(-5pt,0) -- ++(5pt,-5pt);
\draw [color=black] (-2.636559574636737,-1.782297314914103)-- ++(-2.5pt,-2.5pt) -- ++(5pt,5pt) ++(-5pt,0) -- ++(5pt,-5pt);
\draw [color=black] (4.693284374828831,2.9441348300901393)-- ++(-2.5pt,-2.5pt) -- ++(5pt,5pt) ++(-5pt,0) -- ++(5pt,-5pt);
\draw [color=black] (5.329843949465568,4.726432145004242)-- ++(-2.5pt,-2.5pt) -- ++(5pt,5pt) ++(-5pt,0) -- ++(5pt,-5pt);
\draw [fill=qqwuqq] (0.21479621120187709,0) ++(-2.5pt,0 pt) -- ++(2.5pt,2.5pt)--++(2.5pt,-2.5pt)--++(-2.5pt,-2.5pt)--++(-2.5pt,2.5pt);
\draw [fill=qqwuqq] (4.88993510998971,3.494735419914166) circle (2.5pt);
\draw [fill=qqwuqq] (-2.439908839475858,-1.2316967250900763) circle (2.5pt);
\end{scriptsize}
\end{tikzpicture}
\caption{Illustration of the proof of Lemma \ref{lem:case_even_adjacent} (b).2.}
\label{fig:argument_caseii}
\end{figure}
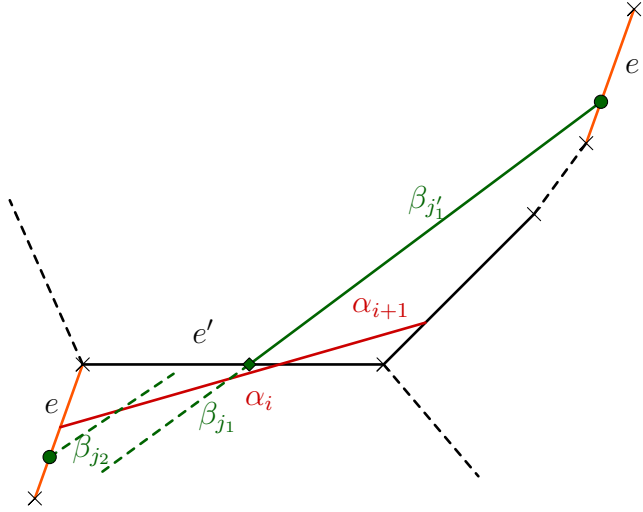

\begin{Rema}\label{rk:intersections_with_signs}
In fact, we could have avoided the use of Corollary \ref{cor:3_segments} by considering intersections with signs instead. Namely, if $\beta_j$ is as in case $2.$, $\alpha_{i-1}$ is adjacent and $\beta_{j'}$ lies in the same polygon as $\alpha_{i+1}$ (and hence $\alpha_{i+1}$ and $\alpha_{i-1}$ lie in the same polygon, which contradict Corollary \ref{cor:3_segments}), but one also notices that $\alpha_{i+1}$ and $\beta_{j'}$ can only intersect with sign opposite to the intersection of $\alpha_i$ and $\beta_j$, and thus we can choose not to count the two intersections. \newline
Although it does not matter for translation surfaces (for which Corollary \ref{cor:3_segments} holds), we claim here that the proof of Lemma \ref{lem:intersection_adjacent_sequence}, and hence of Proposition \ref{prop:total_intersections}, is only combinatorial and only uses the assumption (P2) as well as Remark \ref{rk:alternates}, but not the fact that the curves have a well defined direction.
\end{Rema}


\paragraph{Case (iii).} Let us now deal with the case $q=3$. As stated in Corollary \ref{cor:3_segments}, this means the sequence of adjacent segments is of the form $\alpha_{i} \cup \alpha_{i+1} \cup \alpha_{i+2}$ and the segments lie in three different polygons. Let $e, e', e''$ and $e'''$ be respectively the sides of the polygons containing $\alpha_{i}^-$, $\alpha_{i+1}^-$, $\alpha_{i+2}^-$ and $\alpha_{i+2}^+$. Since $q=3$ is odd, the four sides are distinct by Lemma \ref{lem:maximal_even}. 

A segment $\beta_{j}$ intersecting $\alpha_i \cup \alpha_{i+1} \cup \alpha_{i+2}$, must have at least one of its endpoints on $e, e', e''$ or $e'''$, which is not a vertex if $j \neq 1,l$. Given $j \neq 1,l$, we then define an index $j'$ which will be paired to $j$ as follows:
\begin{itemize}
\item[1.] If $\beta_j$ has an endpoint on $e'$, then we choose $j'$ such that $\beta_{j}$ and $\beta_{j'}$ share their common endpoint on $e'$.
\item[2.] If $\beta_j$ has an endpoint on $e''$ (but no endpoint on $e'$), then we choose $j'$ such that $\beta_{j}$ and $\beta_{j'}$ share their common endpoint on $e''$.
\item[3.] Else, $\beta_j$ has no endpoint on $e'$ or $e''$ but an endpoint on $e$ (resp. $e'''$). Since by assumption $\alpha_{i-1}$ (resp. $\alpha_{i+3}$) is a non-adjacent segment, that means $(i,j) \in I_2^{(1)}$ (resp. $(i+2,j) \in I_2^{(1)}$) and we do not count the intersection.
\end{itemize} 

With this construction, we have:
\begin{Lem}\label{lem:case_q=3}
\begin{enumerate}
\item[(a)] If $j \in \{ 1, \cdots, l \}$ is such that $\beta_j$ intersects $\alpha_{i} \cup \alpha_{i+1} \cup \alpha_{i+2}$ and such that the intersection belongs to $I_2^{(2)}$, then the segment $\beta_{j'}$ constructed as above does not intersect $\alpha_i \cup \alpha_{i+1} \cup \alpha_{i+2}$.
\item[(b)] Given $j_1$ and $j_2$ two distinct indices such that $\beta_{j_1}$ (resp. $\beta_{j_2}$) intersects $\alpha_{i} \cup \alpha_{i+1} \cup \alpha_{i+2}$ and such that the intersection belongs to $I_2^{(2)}$, the pairs $\{j_1, j_1' \}$ and $\{j_2, j_2' \}$ constructed as above satisfy:
 \[ \{ j_1, j_1' \} \cap \{ j_2, j_2' \} = \varnothing. \]
\end{enumerate}
\end{Lem}

\begin{proof}
Since we assume that the intersections belong to $I_2^{(2)}$, we are either in the setting of case $1.$ or case $2.$
\begin{enumerate}
    \item[(a)] If $\beta_j$ has an endpoint on the side $e'$. Then $\beta_{j}$ intersects either $\alpha_{i}$ or $\alpha_{i+1}$, so that we have either $(i,j) \xrightarrow[]{\bigstar} (i+1,j')$ or $(i+1,j) \xrightarrow[]{\bigstar} (i,j')$. In particular, $\beta_{j'}$ lies in the same polygon as $\alpha_{i+1}$ or $\alpha_{i}$ but these segments do not intersect. The other segments cannot intersect $\beta_{j'}$ since they do not lie in the same polygon.

    By symmetry, the same argument holds if $\beta_j$ has an endpoint on the side $e''$.
    \item[(b)] We first deduce from (a) that $j_1 \neq j_2'$ and $j_1' \neq j_2$, because $\beta_{j_1}$ (resp. $\beta_{j_2}$) intersects $\alpha_i \cup \alpha_{i+1} \cup \alpha_{i+2}$ while $\beta_{j_2'}$ (resp. $\beta_{j_1'}$) does not, hence we only have to show that $j_1' \neq j_2'$. Now, since $q=3$ for every $\beta_{j}$ intersecting $\alpha_{i} \cup \alpha_{i+1} \cup \alpha_{i+2}$ and such that the intersection belongs to $I_2^{(2)}$, the chosen index $j'$ to be paired with $j$ is such that either $(i,j) \xrightarrow[]{\bigstar} (i+1,j')$, $(i+1,j) \xrightarrow[]{\bigstar} (i,j')$, $(i+1,j) \xrightarrow[]{\bigstar} (i+2,j')$ or $(i+2,j) \xrightarrow[]{\bigstar} (i+1,j')$ and since $\alpha_{i}, \alpha_{i+1}$ and $\alpha_{i+2}$ do not belong to the same polygon, the data of $j'$ determines uniquely the corresponding index $i'$ in the configuration $\bigstar$, and hence the configuration $\bigstar$ by 2. of Lemma \ref{lem:configuration_star}. This exactly means that if $j_1'=j_2'$, then $j_1=j_2$, and hence the pairs do not overlap.
\end{enumerate}
\end{proof}

As a conclusion, we get that among the intersections of $\beta$ with $\alpha_i \cup \alpha_{i+1} \cup \alpha_{i+2}$, at most $\left\lfloor \frac{l}{2} +1 \right\rfloor$ contribute to $I_2^{(2)} \cup I_2^{(0)}$, as required.

\paragraph{Case (iv).} The remaining case is when \textbf{$q$ is odd, $q \geq 3$}. In this case, we make one triple with the first three adjacent segment, and then we group the remaining adjacent segment by pairs. By the arguments of cases (ii) and (iii) and using to the deformation argument as in case (ii), we can directly conclude:

\begin{Lem}
For each pair of adjacent segment, and any segment $\beta_{j}$ intersecting this pair which is not as in configuration 3. of case (ii), we can choose $j' = j+1$ such that:
\begin{itemize}
\item $\beta{j'}$ does not intersect the pair
\item The pairs $\{j,j'\}$ constructed this way do not overlap.
\end{itemize}
\end{Lem}
\begin{proof}
The construction of $j'$ and the proof is exactly the same as in Lemma \ref{lem:case_even_adjacent}.
\end{proof}

\begin{Lem}\label{lem:triple}
For the triple $\alpha_{i_0+1} \cup \alpha_{i_0+2} \cup \alpha_{i_0+3}$, and for any segment $\beta_{j}$ intersecting this triple which is not in the configuration 3. of case (iii), we can choose $j' = j+1$ such that:
\begin{itemize}
\item $\beta{j'}$ does not intersect the pair
\item The pairs $\{j,j'\}$ constructed this way do not overlap.
\end{itemize}
\end{Lem}
The proof is exactly the same as the proof of Lemma \ref{lem:case_q=3}. However, there is in fact a small difference for the reason why we do not have to consider intersections in case 3. With the notations of case (iii), we still have that if $\beta_j$ intersects $\alpha_{i_0+1}$ and has not endpoint on $e'$ but an endpoint of $e$, then the intersection belongs to $I_2^{(1)}$ so that we do not have to count this intersection. However, if $\beta_j$ intersects the triple and has an endpoint on $e'''$, the intersection is not in $I_2^{(1)}$ anymore if $q > 3$. In this case, we can again perform a small deformation of $\beta$ in order to move the intersection on the next pair $\alpha_{i_0+3} \cup \alpha_{i_0+4}$.\newline

As a conclusion, we get as required that among all intersections of $\beta$ with the sequence of adjacent segment, at most $\left\lfloor \frac{q}{2} \right\rfloor \left\lfloor \frac{l}{2}+1 \right\rfloor$ account for $I_2^{(2)} \cup I_2^{(0)}$.
\end{proof}

\subsubsection{End of the proof of Proposition~\ref{prop:total_intersections}}
The count of the intersections made in the last section allows us to conclude the proof of Proposition~\ref{prop:total_intersections}. We will start with the case where $\beta$ is not an odd saddle connection, as when $\beta$ is an odd saddle connection we need an additional argument.

\begin{Lem}
Assume $\beta$ is not an odd saddle connection. Then 
\[ |\alpha \cap \beta| < (p_{\alpha} + q_{\alpha})(p_{\beta} + q_{\beta}). \]
\end{Lem}
\begin{proof}
By Equation~\eqref{eq:sum_intersections}, we have
\begin{align*}
|\alpha \cap \beta| &\leq \# \tilde{I}_1 + \# I_2^{(0)} + \# I_2^{(2)} \\
& \leq \# \tilde{I}_1 +  \# I_2^{(0,i)} +  \# I_2^{(0,c)}  + \# I_2^{(2)} \\
\end{align*}
Using Equation~\eqref{eq:I_1_tilde}, Lemma \ref{lem:lonely_adjacent} and Corollary \ref{cor:consecutive_adjacent}, we have
\begin{align*}
|\alpha \cap \beta| &\leq p_{\alpha} \left\lceil \frac{l}{2} \right\rceil + \# \{i, \alpha_i \text{ is an isolated adjacent segment}\}+ q_{\alpha}  \left\lfloor \frac{l}{2}+1 \right\rfloor \\
&\leq p_{\alpha} \left\lceil \frac{l}{2} \right\rceil + p_{\alpha}-1+ q_{\alpha} \left\lfloor \frac{l}{2}+1 \right\rfloor
\end{align*}
where we used that $\# \{i, \alpha_i \text{ is an isolated adjacent segment}\} \leq p_{\alpha}-1$, which holds because the isolated adjacent segments are separated by the $p_\alpha$ non-adjacent ones. Using Lemma \ref{lem:integer_part}, we deduce that if $\beta$ is not an odd saddle connection, then $\left\lceil \frac{l}{2} \right\rceil \leq p_{\beta} + q_{\beta}-1$ and $\left\lfloor \frac{l}{2} +1 \right\rfloor \leq p_{\beta} + q_{\beta}$ and hence
\begin{align*}
|\alpha \cap \beta| & \leq p_{\alpha}(p_{\beta} + q_{\beta}-1) + p_{\alpha}-1 + q_{\alpha}(p_{\beta} + q_{\beta})\\
& \leq (p_{\alpha} + q_{\alpha}) (p_{\beta} + q_{\beta})- 1\\
& < (p_{\alpha} + q_{\alpha}) (p_{\beta} + q_{\beta})
\end{align*}
as required.
\end{proof}
By symmetry, the same result holds if $\alpha$ is not an odd saddle connection. Thus, we are left to prove:

\begin{Lem}\label{lem:intersection_odd_sc}
Assume that both $\alpha$ and $\beta$ are odd saddle connections, then:
\[ |\alpha \cap \beta| \leq (p_{\alpha} + q_{\alpha})(p_{\beta} + q_{\beta}). \]
\end{Lem}
\begin{proof}
In this case, we need to make an additional remark in order to show the required result:

\begin{center}
Assume the (isolated adjacent) segment $\alpha_i$ intersects $\beta_1$ (resp. $\beta_l$), then the (non-adjacent) segments $\alpha_{i-1}$ and $\alpha_{i+1}$ do not lie in the same polygon as $\beta_1$ (resp. $\beta_l$). In particular, there are at most $\left\lceil \frac{l-1}{2} \right\rceil$ segments of $\beta$ in the same polygon as $\alpha_{i-1}$, because $\alpha_{i-1}$ can intersect at most half of the $l-1$ segments $\beta_2, \dots, \beta_l$. Similarly, there are at most $\left\lceil \frac{l-1}{2} \right\rceil$ segments of $\beta$ in the same polygon as $\alpha_{i+1}$.
\end{center}

If $\beta$ is an odd saddle connection, then $l$ is an odd integer and $\left\lceil \frac{l-1}{2} \right\rceil = \left\lceil \frac{l}{2} \right\rceil - 1$. In particular, for each intersection in  $I_2^{(0,i)}$ (corresponding to the intersection of an isolated adjacent segment $\alpha_i$ with either $\beta_1$ or $\beta_l$) we can remove one in the count of $\# \tilde{I}_1$ (and more precisely in the count of intersections of $\alpha_{i-1}$ with $\beta$). Further, for the index $i_{max} = \max \{ i, \alpha_i$ is an isolated adjacent segment intersecting either $\beta_1$ or $\beta_l \}$, we can remove one more intersection in the count of $\# \tilde{I}_1$ (and more precisely in the count of intersections of $\alpha_{i_{max}+1}$ with $\beta$). Hence, we have shown:

\begin{Lem}\label{lem:3.35}
Assume $\beta$ is odd and $\# I_2^{(0,i)} \neq 0$, then:
\[ \# \tilde{I}_1 \leq p_{\alpha} \left\lceil \frac{l}{2} \right\rceil - \# I_2^{(0,i)}-1 \]
\end{Lem}

As a consequence of Equation \eqref{eq:I_1_tilde}, Lemma \ref{lem:3.35} and Corollary \ref{cor:consecutive_adjacent}, we get that:
\begin{align*}
|\alpha \cap \beta| & \leq \# \tilde{I}_1 +  \# I_2^{(0,i)} +  \# I_2^{(0,c)} + \# I_2^{(2)} \\
& \leq p_{\alpha} \left\lceil \frac{l}{2} \right\rceil + q_{\alpha} (p_{\beta} + q_{\beta}) \\
& \leq (p_{\alpha} + q_{\alpha})(p_{\beta} + q_{\beta})
\end{align*}
with equality only if $\# I_2^{(0,i)} = 0$.
\end{proof}
\begin{Rema}
Notice that having an equality above thus requires that $\# \tilde{I}_1 = p_{\alpha} \left\lceil \frac{l}{2} \right\rceil$, and that any element of $\tilde{I}_1$ appears either as an intersection or as pair $(i',j')$ of a configuration $\bigstar$.
\end{Rema}
This concludes the proof of Proposition~\ref{prop:total_intersections}.

\subsection{Proof of Theorem~\ref{theo:genSC}} \label{sec:proofgenSC}
Having studied both the length of the segments and the intersections of saddle connections according to their polygonal decomposition, we can now prove Theorem \ref{theo:genSC}, which will follow from Equation~\eqref{eq:length_alpha}, Lemma \ref{lem:length_odd_sc} and Proposition~\ref{prop:total_intersections}. We will deal with the cases where one of the saddle connections is either a side or a diagonal separately. In fact, an additional argument is required in the case where one of the saddle connections is a diagonal.

\subsubsection{No sides or diagonals}
\begin{Prop}\label{prop:non_diag_side_case}
Assume that neither $\alpha$ nor $\beta$ is a side or a diagonal of the polygons. Then:
\[ \frac{|\alpha \cap \beta|+1}{l(\alpha)l(\beta)} < \frac{1}{l_0^2}. \]
\end{Prop}
\begin{proof}
\begin{enumerate}
\item If either $\alpha$ or $\beta$ is not an odd saddle connection, we have by Proposition~\ref{prop:total_intersections},  
\[ |\alpha \cap \beta| < (p_{\alpha} + q_{\alpha})(p_{\beta} + q_{\beta}) \]
and by Equation~\eqref{eq:length_alpha},
\begin{equation}\label{eq:product_lengths}
l(\alpha) l(\beta) \geq (p_{\alpha} + q_{\alpha})(p_{\beta} + q_{\beta})l_0^2 
\end{equation}
Hence we have:
\begin{equation}\label{eq:ratiook}
\frac{|\alpha \cap \beta|+1}{l(\alpha)l(\beta)} \leq \frac{1}{l_0^2}. 
\end{equation}

Let us show that the above inequality is strict. This is because:
\begin{itemize}
    \item If there is at least one adjacent segment, then the inequality of Equation~\eqref{eq:product_lengths} is strict, and hence it is also the case for Equation~\eqref{eq:ratiook}
    \item Else, there are no adjacent segments, and $k = p_{\alpha}$ and $l = p_{\beta}$. We easily deduce from hypothesis (P2) that as soon as $k,l \geq 2$, we have
    \[ |\alpha \cap \beta| \leq \left\lceil \frac{kl}{2} \right\rceil \leq kl - 2 \leq p_{\alpha} p_{\beta}-2, \]
    and this allows to conclude that the inequality of equation~\eqref{eq:ratiook} is strict.
\end{itemize}

\item Else, both $\alpha$ and $\beta$ are odd, but neither sides nor diagonals of the polygons. From Lemma~\ref{lem:length_odd_sc}, we know that:
\[ l(\alpha) \geq (p_{\alpha} + q_{\alpha} + \sqrt{2}-1)l_0\]
and similarly
\[ l(\beta) \geq (p_{\beta} + q_{\beta} + \sqrt{2}-1)l_0\]
Thus
\[ l(\alpha)l(\beta) \geq [(p_{\alpha} + q_{\alpha})(p_{\beta} + q_{\beta})+ (\sqrt{2}-1) (p_{\alpha} + q_{\alpha} + p_{\beta} + q_{\beta}) + (\sqrt{2}-1)^2]l_0^2 \]
But $p_{\alpha} \geq 2$ as well as $p_{\beta}$, so that
\begin{align*}
 l(\alpha)l(\beta) & \geq [(p_{\alpha} + q_{\alpha})(p_{\beta} + q_{\beta}) + 4 (\sqrt{2}-1) + (\sqrt{2}-1)^2]l_0^2 \\
& > [(p_{\alpha} + q_{\alpha})(p_{\beta} + q_{\beta}) + 1]l_0^2
\end{align*}
This gives the required inequality.
\end{enumerate}
\end{proof}

\subsubsection{The case of sides}
We now turn to the case where at least one of the saddle connections is a side of a polygon. We can assume up to permutation that it is $\alpha$. We show:
\begin{Lem}\label{lem:case_alpha_side}
Assume $\alpha$ is a side of a polygon. Then
\[ \frac{|\alpha \cap \beta|+1}{l(\alpha)l(\beta)} \leq \frac{1}{l_0^2}. \]
Further, equality holds if and only if $l(\alpha) = l_0$ and:
\begin{enumerate}
\item either $\beta$ has length $l_0$. In particular it is either is a side or a diagonal and $|\alpha \cap \beta| = 0$.
\item or $|\alpha \cap \beta| = 1$ and $\beta$ is the union of two (non adjacent) segments, of total length $2l_0$, and intersecting $\alpha$ once on its interior.
\end{enumerate}

\end{Lem}
\begin{proof}
Each non-singular intersection between $\alpha$ and $\beta$ corresponds to the union of two consecutive segments $\beta_j \cup \beta_{j+1}$ which share their common endpoint on the side $\alpha$, and since by (P2) pairs of segments cannot overlap, we directly obtain:
\begin{equation}\label{eq:intersection_4.2.sides}
|\alpha \cap \beta| \leq \left\lfloor \frac{l}{2} \right\rfloor \leq (p_{\beta}+q_{\beta})-1,
\end{equation}
Further, $l(\beta) \geq (p_{\beta}+q_{\beta})l_0$ by Lemma \ref{lem:length_adjacent_na}. Hence, we conclude that
\begin{equation}\label{eq:K_4.2sides}
\frac{|\alpha \cap \beta|+1}{l(\alpha)l(\beta)} \leq \frac{1}{l_0^2}.
\end{equation}
Notice that a necessary condition to have an equality above is that $l(\alpha) = l_0$. Further,
\begin{enumerate}
\item If $\beta$ is either a side or a diagonal, then $|\alpha \cap \beta| = 0$ and we have equality in \eqref{eq:K_4.2sides} if and only if $l(\beta)=l_0$.
\item Else, $\beta$ is neither a side nor a diagonal then:
\begin{itemize}
    \item either there is an adjacent segment in the polygonal decomposition of $\beta$ and, since the length of two consecutive adjacent segments is greater than $l_0$ by Lemma \ref{lem:length_adjacent_na}, and since we do not take into account the presence of isolated adjacent segments in the estimation of Equation \eqref{eq:length_alpha}, we obtain $l(\beta) > (p_{\beta}+q_{\beta})l_0$ and hence we cannot have equality in \eqref{eq:K_4.2sides},
    \item or there are no adjacent segments in the polygonal decomposition of $\beta$ and $l = p_{\beta}$. In particular, as soon as $l \geq 3$, we have $\left\lfloor \frac{l}{2} \right\rfloor \leq p_{\beta}-2$ and hence 
    $$l(\beta)\geq p_{\beta}l_0 > \left( \left\lfloor \frac{l}{2} \right\rfloor +1 \right) l_0 $$ 
    so that by Equation~\eqref{eq:intersection_4.2.sides} we conclude that we cannot have equality in \eqref{eq:K_4.2sides}.
\end{itemize}
In particular, we cannot have equality in \eqref{eq:K_4.2sides} unless $\beta$ is the union of two non-adjacent segments and $|\alpha \cap \beta| = 1$. In that case the length of $\beta$ must be $2l_0$.
\end{enumerate}

\end{proof}

\subsubsection{The case of diagonals}\label{sec:diagANDsides}
It remains to deal with the case where either $\alpha$ or $\beta$ is a diagonal of a polygon and neither of them is a side. Up to permutation we will assume $\alpha$ is a diagonal. Namely, we show

\begin{Lem}\label{lem:cas_alpha_diagonal}
Assume $\alpha$ is a diagonal of a polygon (and $\beta$ is not a side of a polygon). Then
\[ \frac{|\alpha \cap \beta|+1}{l(\alpha)l(\beta)} \leq \frac{1}{l_0^2}. \]
Further, equality holds if and only if:
\begin{enumerate} 
\item either $\alpha$ and $\beta$ are two diagonals, both of length\footnote{This case happens for example when one of the polygons is a rectangle of size $l_0 \times 3l_0$ but with intermediate vertices along the long sides of the rectangle.} $l_0$, and $|\alpha \cap \beta|=0$. 
\item or $\alpha$ and $\beta$ are two diagonals, of length $\sqrt{2} l_0$, and intersecting once on their interior.
\item or $l(\alpha) = l_0$, $|\alpha \cap \beta| = 1$ and $\beta$ is the union of two (non adjacent) segments, of total length $2l_0$, and intersecting $\alpha$ once in its interior. 
\end{enumerate}

\end{Lem}
We will proceed to the proof as follows. First, we investigate the case where $\beta$ is not an odd saddle connection and not a diagonal.
Recall that by definition sides and diagonals are not odd saddle connections. 
Then we continue with the case where $\beta$ is also a diagonal, and finally we deal with the case where $\beta$ is an odd saddle connection.

\paragraph{A diagonal and a non-odd saddle connection.} 
Given a saddle connection $\beta$, we have that
\[ |\alpha \cap \beta| \leq \left\lceil \frac{l}{2} \right\rceil, \]
since $\alpha$ is all contained in one polygon and consecutive pieces of $\beta$ are in different polygons. 
In fact, when $\beta$ is not a diagonal and not an odd saddle connection, the inequality is strict,
with equality only if $\beta$ is either a diagonal or an odd saddle connection. In particular, we directly deduce from Equation \eqref{eq:length_alpha} that:
\[ \frac{|\alpha \cap \beta|+1}{l(\alpha)l(\beta)} \leq \frac{1}{l_0^2}, \]
and, as before, the inequality is strict unless $l(\alpha) =l_0$ and $\beta$ is the union of two non-adjacent segments such that $|\alpha \cap \beta| = 1$ and $l(\beta) = 2l_0$ (recall that we assumed $\beta$ not to be a side or a diagonal).

\paragraph{Intersections of diagonals.}
We now investigate the case where both $\alpha$ and $\beta$ are diagonals. Namely, we show:
\begin{Lem}\label{lem:intersection_diagonals}
Let $\alpha$ and $\beta$ be two distinct diagonals of a convex polygon $P$ with obtuse or right angles. Let $l_0$ be the length of the shortest side of the polygon. Then, if $\alpha$ and $\beta$ intersect in their interior, we have $l(\alpha)l(\beta) \geq 2 l_0^2$.
Further, equality holds if and only if both diagonals have length $\sqrt{2}l_0$.
\end{Lem}
\begin{proof}
First, if $\min(l(\alpha),l(\beta)) \geq \sqrt{2} l_0$, then we obviously have $l(\alpha)l(\beta) \geq 2 l_0^2$. In the rest of the proof, we will assume, up to permuting $\alpha$ and $\beta$, that $l(\alpha) \leq \sqrt{2} l_0$. Let $P$ be the polygon containing $\alpha$ and let $\theta_{min}$ be the minimum of the angles between $\alpha$ and the four sides of $P$ having a common endpoint with $\alpha$. Since the polygon has obtuse or right angles, we know from Lemma \ref{lem:convex_length} that the distance from any vertex of $P$ to a point in a side which does not have this vertex as an endpoint is at least $l_0$. In particular, we must have $\sin \theta_{min} \geq \frac{l_0}{l(\alpha)}$, otherwise by convexity there would be an endpoint of $\alpha$ and a point of $P$ in a segment which does not contain this endpoint at a distance less than $l_0$, as in Figure \ref{fig:cas_diagonales_1}.

\begin{figure}[h]
\center
\definecolor{uuuuuu}{rgb}{0.26666666666666666,0.26666666666666666,0.26666666666666666}
\definecolor{qqwuqq}{rgb}{0,0.39215686274509803,0}
\definecolor{ccqqqq}{rgb}{0.8,0,0}
\begin{tikzpicture}[line cap=round,line join=round,>=triangle 45,x=1cm,y=1cm, scale=1.3]
\clip(-1.5,-0.5) rectangle (4.2,3.3);
\draw [shift={(0,3)},line width=1pt,color=qqwuqq,fill=qqwuqq,fill opacity=0.10000000149011612] (0,0) -- (-90:0.2666666666666666) arc (-90:-31.541315391385044:0.2666666666666666) -- cycle;
\draw[line width=1pt,color=qqwuqq,fill=qqwuqq,fill opacity=0.10000000149011612] (1.1767862142196355,2.2776967374789807) -- (1.0781470321887279,2.116992452147727) -- (1.2388513175199816,2.0183532701168194) -- (1.3374904995508892,2.179057555448073) -- cycle; 
\draw [line width=1pt,color=ccqqqq] (0,0)-- (0,3);
\draw [line width=1pt] (1.9333333333333307,1.8133333333333321)-- (0,3);
\draw [line width=1pt,dash pattern=on 3pt off 3pt] (0,3)-- (-1.2666666666666686,2.8);
\draw [line width=1pt,dash pattern=on 3pt off 3pt] (0,0)-- (-1.213333333333354,0.7066666666666661);
\draw [line width=1pt,dash pattern=on 3pt off 3pt] (0,0)-- (1.4,-0.28);
\draw [line width=1pt] (0,0)-- (1.3374904995508892,2.179057555448073);
\draw [line width=1pt, to-to] (0.333333333333331,0.06666666666666662)-- (1.4933333333333307,1.9733333333333318);
\draw (0.1555555555555535,2.631111111111109) node[anchor=north west] {$\theta_{min}$};
\draw [color=ccqqqq](-0.4666666666666686,1.84) node[anchor=north west] {$\alpha$};
\draw (1.0977777777777755,1.04) node[anchor=north west] {$l = l(\alpha) \sin \theta_{min} \geq l_0$};
\begin{scriptsize}
\draw [fill=ccqqqq] (0,0) circle (2.5pt);
\draw [fill=ccqqqq] (0,3) circle (2.5pt);
\draw [color=uuuuuu] (1.3374904995508892,2.179057555448073)-- ++(-2pt,-2pt) -- ++(4pt,4pt) ++(-4pt,0) -- ++(4pt,-4pt);
\end{scriptsize}
\end{tikzpicture}
\caption{The distance from a vertex of $\alpha$ to a side of the polygon not adjacent to the vertex is at least $l_0$.}
\label{fig:cas_diagonales_1}
\end{figure}
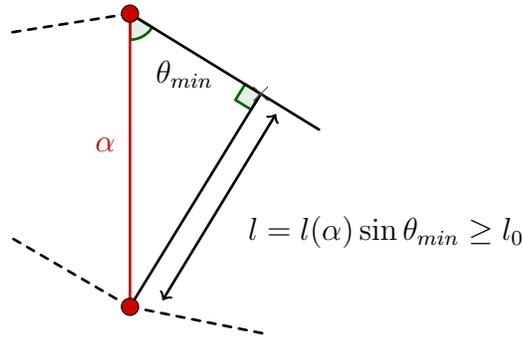

In particular, no vertex of $P$ lies inside the hexagon having as sides the four segments making an angle $\theta_{min}$ with $\alpha$ at one of its vertices and having length $l_0$ (this hexagon degenerates to a square when $l(\alpha) = \sqrt{2}l_0$). Adding this to the fact that no side of $P$ lies at a distance less than $l_0$ from the two endpoints of $\alpha$ (except the sides adjacent to the endpoints of $\alpha$), we get that no vertex of $P$ which is not a vertex of $\alpha$ lies inside the gray zone of Figure \ref{fig:argument_diagonales_2}. In particular, the length of $\beta$ has to be at least $2 l_0 \sin\theta_{min}$, and we get 
\[ l(\alpha) l(\beta) \geq l(\alpha) 2 l_0 \sin \theta_{min} \geq 2l_0^2. \]

\begin{figure}
\center
\definecolor{wwwwww}{rgb}{0.4,0.4,0.4}
\definecolor{ffvvqq}{rgb}{1,0.3333333333333333,0}
\definecolor{qqwuqq}{rgb}{0,0.39215686274509803,0}
\definecolor{ccqqqq}{rgb}{0.8,0,0}
\begin{tikzpicture}[line cap=round,line join=round,>=triangle 45,x=1cm,y=1cm, scale=1.3]
\clip(-5.386666666666671,-2.4133333333333327) rectangle (6.32,5.346666666666666);
\draw [shift={(0,3)},line width=1pt,color=qqwuqq,fill=qqwuqq,fill opacity=0.10000000149011612] (0,0) -- (-90:0.4) arc (-90:-31.541315391385044:0.4) -- cycle;
\fill[line width=1pt,color=ffvvqq,fill=ffvvqq,fill opacity=0.1] (-1.9333333333333305,1.8133333333333324) -- (0,3) -- (1.9333333333333307,1.8133333333333321) -- (1.9333333333333307,1.1866666666666683) -- (0,0) -- (-1.9333333333333305,1.1866666666666685) -- cycle;
\draw [line width=1pt,color=ccqqqq] (0,0)-- (0,3);
\draw [line width=1pt, color=qqwuqq] (-1.9,1.5)--(1.9,1.5);
\draw [line width=1pt, to-to] (-1.9,1.4)--(1.9,1.4);
\draw (0.1,1.4) node[anchor=north west] {$2l_0 \sin \theta_{min}$};
\draw [line width=1pt, color=qqwuqq, dash pattern = on 3pt off 3pt] (-1.9,1.5)--(-2.9,1.5);
\draw [line width=1pt, color=qqwuqq, dash pattern = on 3pt off 3pt] (1.9,1.5)--(2.9,1.5);
\draw [color=ccqqqq](-0.4666666666666712,1) node[anchor=north west] {$\alpha$};
\draw [color=qqwuqq](-1.5,2) node[anchor=north west] {$\beta$};
\draw [line width=1pt] (1.9333333333333307,1.8133333333333321)-- (0,3);
\draw [line width=1pt] (-1.9333333333333305,1.8133333333333324)-- (0,3);
\draw [line width=1pt] (-1.9333333333333305,1.1866666666666685)-- (0,0);
\draw [line width=1pt] (0,0)-- (1.9333333333333307,1.1866666666666683);
\draw [line width=1pt] (1.9333333333333307,1.8133333333333321)-- (1.9333333333333307,1.1866666666666683);
\draw [line width=1pt] (-1.9333333333333305,1.8133333333333324)-- (-1.9333333333333305,1.1866666666666685);
\draw [shift={(0,3)},line width=1pt,color=wwwwww,fill=wwwwww,fill opacity=0.05]  (0,0) --  plot[domain=-0.5504998039896325:3.692092457579426,variable=\t]({1*2.2684698709825417*cos(\t r)+0*2.2684698709825417*sin(\t r)},{0*2.2684698709825417*cos(\t r)+1*2.2684698709825417*sin(\t r)}) -- cycle ;
\draw [shift={(0,0)},line width=1pt,color=wwwwww,fill=wwwwww,fill opacity=0.05]  (0,0) --  plot[domain=-3.692092457579426:0.550499803989633,variable=\t]({1*2.2684698709825417*cos(\t r)+0*2.2684698709825417*sin(\t r)},{0*2.2684698709825417*cos(\t r)+1*2.2684698709825417*sin(\t r)}) -- cycle ;
\draw [line width=1pt, to-to] (1.92,1)-- (0.2,-0.06666666666666665);
\draw (0.1,2.706666666666666) node[anchor=north west] {$\theta_{min}$};
\draw (1,0.5) node[anchor=north west] {$l_0$};
\begin{scriptsize}
\draw [fill=ccqqqq] (0,0) circle (2.5pt);
\draw [fill=ccqqqq] (0,3) circle (2.5pt);
\end{scriptsize}
\end{tikzpicture}
\caption{No point on a side of the polygon which is not adjacent to one of the two endpoints of $\alpha$ lies inside the hexagon or inside the disks of radius $l_0$ centered at the endpoints of $\alpha$. In particular, the length of $\beta$ is at least $2 \sin \theta_{min} l_0$.}
\label{fig:argument_diagonales_2}
\end{figure}
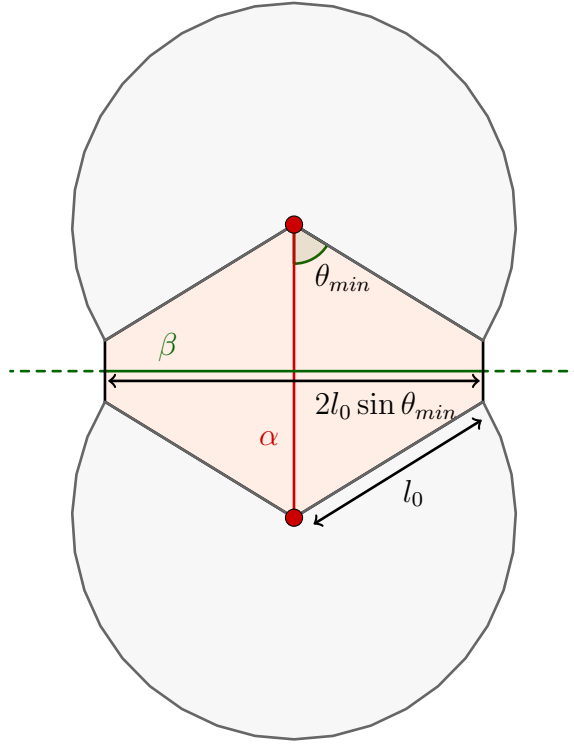

In fact, one can easily show that $l(\beta) > 2 \sin \theta_{min} l_0$ unless $l(\alpha) = l(\beta) = \sqrt{2} l_0$.
\begin{itemize}
\item If $\min (l(\alpha),l(\beta)) > \sqrt{2} l_0$, then we directly have that $l(\alpha)l(\beta) > 2 l_0^2$.
\item Else, we can assume that $l(\alpha) \leq \sqrt{2} l_0$. Of course, if $\sin \theta_{min} > \frac{l_0}{l(\alpha)}$ then we have $l(\alpha) l(\beta) > 2 l_0^2$. Otherwise, $\sin \theta_{min} = \frac{l_0}{l(\alpha)}$, and one can check that the two angles between $\alpha$ and the sides of $P$ at a vertex of $\alpha$ cannot be simultaneously equal to $\theta_{min}$ (unless $l(\alpha) = \sqrt{2}$), and hence we have $l(\beta) > 2 \frac{l_0^2}{l(\alpha)}$.
\end{itemize}
\end{proof}

As a direct corollary of Lemma \ref{lem:intersection_diagonals}, we get:
\begin{Cor}
If $\alpha$ and $\beta$ are both diagonals, then:
\[ \frac{|\alpha \cap \beta|+1}{l(\alpha)l(\beta)} \leq \frac{1}{l_0^2}. \]
Further, the inequality is strict unless $l(\alpha) = l(\beta) = \sqrt{2} l_0$ and $|\alpha \cap \beta| = 1$.
\end{Cor}
\begin{proof}
Either $\alpha$ and $\beta$ do not intersect on their interior and the result comes from the fact that $l(\alpha)l(\beta) \geq l_0^2$, or they do intersect once in their interior and the result comes from Lemma \ref{lem:intersection_diagonals}.
\end{proof}

\paragraph{A diagonal and an odd saddle connection.}
In fact, the above argument can be easily generalised to the case where $\alpha$ is a diagonal but $\beta$ is an odd saddle connection which is not a diagonal. In this case, recall that:
\[ |\alpha \cap \beta| \leq \left\lfloor \frac{l}{2} \right\rfloor \leq (p_{\beta}+q_{\beta}), \]
and 
\[ l(\beta) > (p_{\alpha} + q_{\alpha})l_0.\]
(Notice that from Lemma \ref{lem:length_odd_sc} the inequality is strict.). Hence:
\begin{enumerate}
\item if $|\alpha \cap \beta| < p_{\beta} +q_{\beta}$, we directly deduce that
\[ \frac{|\alpha \cap \beta|+1}{l(\alpha)l(\beta)} < \frac{1}{l_0^2}. \]
\item Else, $|\alpha \cap \beta| = p_{\beta} +q_{\beta}$ and in particular $\beta_1$ must intersect $\alpha$. Then similarly to the proof of Lemma \ref{lem:intersection_diagonals}, we construct the region of Figure \ref{fig:argument_diagonales_2}, and the endpoints $\beta_1^-$ and $\beta_2^+$ cannot lie inside this region, giving:
\[ l(\beta_{1} \cup \beta_{2}) l(\alpha) > 2 l_0^2. \] 
Since $\beta$ is odd, we can group the other adjacent segments after $\beta_2$ by pairs to obtain
\[ l(\alpha) l(\beta) > (p_{\beta}+q_{\beta}+1)l_0^2. \]
Using Proposition \ref{prop:total_intersections}, we conclude that
\[ \frac{|\alpha \cap \beta|+1}{l(\alpha)l(\beta)} < \frac{1}{l_0^2}, \]
as required.
\end{enumerate}

\begin{Rema}\label{rk:refinement_lengths_2_odd_sc_w/_max_intersections}
In fact, it is also possible to perform a similar argument if both $\alpha$ and $\beta$ are odd saddle connections (this means we do not need to use Lemma \ref{lem:length_odd_sc} and Lemma \ref{lem:better_lengths}), but there are a few complications and the advantage of Lemmas \ref{lem:length_odd_sc} and \ref{lem:better_lengths} (in addition to the fact that the proof is simpler to write) is that the proof generalises to the case where we assume $(P1')$ instead of $(P1)$ (up to changing $l_0$ by $l_0 \sin \theta_0$ and $\sqrt{2}$ by $2 \sin \theta_0 /2$, where $\theta_0 \leq \pi /2$ denotes the smallest angle of the polygons - or $\pi/2$ if all the angles are obtuse - see the next paragraph for a discussion on this case).
\end{Rema}

This completes the proof of Lemma \ref{lem:cas_alpha_diagonal}.

\paragraph{Conclusion.}
Combining Proposition \ref{prop:non_diag_side_case}, Lemma \ref{lem:case_alpha_side} and Lemma \ref{lem:cas_alpha_diagonal}, we finally obtain Theorem \ref{theo:genSC}, and thus Theorem \ref{theo:mainCONVEX}.

\subsubsection{Additional remarks}\label{sec:additional_remarks}
Before moving to the specific case of Bouw-M\"oller surfaces, let us make a few remarks on the proof of Theorem \ref{theo:mainCONVEX}, and in particular discuss its assumptions.

\paragraph{Surfaces made of polygons with acute angles.}
As already said in the introduction, Theorem \ref{theo:mainCONVEX} does not hold if the angles are no longer assumed to be obtuse. The main issue in this case is that two consecutive adjacent segment may have very small length, as in the example of Figure \ref{fig:acute_angles}. However, it is possible to avoid this issue by replacing (P1) with the weaker assumption (P1') from Remark \ref{rk:hypothesis_P1'_rk1}.
This property is satisfied for example by covers of flat tori ramified over a single point, which are tiled by parallelograms (although (P2) is not always satisfied). Recall from Remark \ref{rk:hypothesis_P1'_rk2} that Proposition \ref{prop:total_intersections} still holds under the assumptions (P1') and (P2), but in this case the estimations of the lengths fail. In fact, as hinted in Remark \ref{rk:refinement_lengths_2_odd_sc_w/_max_intersections} one could hope to generalize the estimates on the lengths of closed curves to obtain that, for any closed curves $\gamma$ and $\delta$, we have
\[ \frac{\Int(\gamma, \delta)}{l(\gamma) l(\delta)} \leq \frac{1}{(l_0 \sin \theta_0)^2},\]
where $\theta_0$ is the smallest angle of the polygons. 

This is discussed in more details in the first author's Thesis \cite{these_boulanger}.
\begin{figure}[h]
\center
\definecolor{ccqqqq}{rgb}{0.8,0,0}
\definecolor{qqwuqq}{rgb}{0,0.33,0}
\begin{tikzpicture}[line cap=round,line join=round,>=triangle 45,x=1cm,y=1cm, scale=0.6]
\clip(-4.2,-3.5) rectangle (4.2,3.5);
\draw [line width=1pt] (-2,0)-- (2,0);
\draw [line width=1pt] (2,0)-- (4,3);
\draw [line width=1pt] (-2,0)-- (-1,-3);
\draw [line width=1pt,color=ccqqqq] (-1.5,-1.5)--(-1.5,1.5);
\draw [line width=1pt] (-2,0)-- (-1,3);
\draw [line width=1pt,dash pattern=on 3pt off 3pt] (2,0)-- (3.8828052146241174,-1.5622840240938458);
\draw (1.5,2) node[anchor=north east] {$P$};
\draw (1.5,-2) node[anchor=south east] {$P'$};
\begin{scriptsize}
\draw [color=black] (-1,3)-- ++(-2.5pt,-2.5pt) -- ++(5pt,5pt) ++(-5pt,0) -- ++(5pt,-5pt);
\draw [color=black] (-2,0)-- ++(-2.5pt,-2.5pt) -- ++(5pt,5pt) ++(-5pt,0) -- ++(5pt,-5pt);
\draw [color=black] (2,0)-- ++(-2.5pt,-2.5pt) -- ++(5pt,5pt) ++(-5pt,0) -- ++(5pt,-5pt);
\draw [color=black] (4,3)-- ++(-2.5pt,-2.5pt) -- ++(5pt,5pt) ++(-5pt,0) -- ++(5pt,-5pt);
\draw [color=black] (-1,-3)-- ++(-2.5pt,-2.5pt) -- ++(5pt,5pt) ++(-5pt,0) -- ++(5pt,-5pt);
\end{scriptsize}
\end{tikzpicture}
\caption{If we do not assume the angles of the polygons to be obtuse or right, pairs of adjacent segments can have arbitrarily small length.}
\label{fig:acute_angles}
\end{figure}
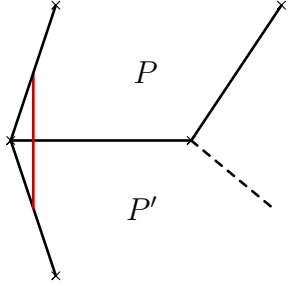

\paragraph{The Non-self-identification condition.}
Finally, one may wonder what happens if we now remove assumption (P2) instead. Although our intersection counting from Section \ref{sec:study_intersections} heavily rely on this assumption, we believe that the same result holds without (P2). However, the polygonal decomposition is not adapted anymore as pairs of adjacent segments could then intersect twice a non-adjacent segment, giving two intersection for a product of length $l_0^2$. This means we need to refine the estimates on the length in this case. This is for example done in the case of the regular $n$-gon for even $n$ in \cite{Bou23}, but setting up a general argument would rather be very technical.

\paragraph{General flat surfaces.}
Finally, although in the previous sections we dealt with translation surfaces, our count of the intersections can be generalised to any flat surface constructed from a collection of polygons by gluing pairs of sides, without assuming the identified sides to be parallel and of the same length (more precisely, we can assume that the transition maps are affine instead of only translations, and the resulting surface is still closed, orientable, and inherits a flat metric with finitely many conical singularities). 
In this generalised setting, we should point out that it is not true anymore that any closed curve is homologous to a union of vertex-to-vertex trajectory with at most the same length, and hence it is necessary to work with all simple closed geodesics instead (which may or may not go through a singularity). 
Further, the direction of a geodesic is not well defined anymore, and in particular Corollary \ref{cor:3_segments} does not hold anymore. However, as hinted in Remark \ref{rk:intersections_with_signs}, one can modify the proof of Proposition \ref{prop:total_intersections} to take into account these changes by counting intersections with signs instead (notice that the property of Remark \ref{rk:alternates} holds for simple closed geodesics on such surfaces). Further, although our lower bound on the length of a saddle connection given in Equation \eqref{eq:length_alpha} is still valid the proof of part 2 of Lemma \ref{lem:better_lengths} is not true anymore (except for half-translation surfaces, where curves do have a well defined direction) and one needs to refine the estimates on the length product of a pair of odd saddle connections achieving the maximum possible non-singular intersections in Proposition \ref{prop:total_intersections}. An additional discussion on this topic can be found in \cite{these_boulanger}.

\section{The case of Bouw-M\"oller surfaces}\label{sec:Bouw-M\"oller}
As mentioned in the introduction, one of the main motivations for Theorem \ref{theo:mainCONVEX} was to apply it to Bouw-M\"oller surfaces. In this section we conclude the proof of Theorem \ref{thm:mainBM}. If $n \geq 4$, the Bouw-M\"oller surface $S_{m,n}$ satisfies the hypotheses (P1) and (P2), and so the conclusion of Theorem \ref{theo:mainCONVEX} holds. Further:
\begin{itemize}
\item On Bouw-M\"oller surfaces for $n=2$, the hypothesis (P1) holds but not (P2). It is still possible to show that Theorem \ref{theo:mainCONVEX} holds using cylinder decompositions, as done in \cite{BLM22} for the staircase model of the double $m$-gon, $m$ odd. The case where $m$ is even works in the same way.
\item On Bouw-M\"oller surfaces for $n=3$, the hypothesis (P2) holds but (P1) does not, because there are two triangles among the polygons defining $S_{m,3}$. However, (P1') holds (if $m\neq 2$) and it is possible to adapt the proof of Theorem \ref{theo:mainCONVEX} to this case by refining our estimations of the lengths of the segments. Namely:
\begin{itemize}
\item Although Lemma \ref{lem:length_adjacent_na} does not hold anymore, we show in Section \ref{sec:pf_lbm} that:
\begin{Lem}\label{lem:lengths_bouw_moller}
Let $m \geq 3$ and let $\alpha$ be a saddle connection on $S_{m,3}$, then
\[ l(\alpha) \geq (p_{\alpha} + q_{\alpha})l_0. \]
Further, if $\alpha$ is odd but it is not a side or a diagonal of a polygon, then 
\[ l(\alpha) \geq (p_{\alpha} + q_{\alpha}+ \sqrt{2}-1)l_0. \]
\end{Lem}
\item Proposition \ref{prop:total_intersections} still holds by Remark \ref{rk:hypothesis_P1'_rk2},
\item The case of sides and diagonals does not change (as there are no diagonals on a triangle).
\end{itemize}
\end{itemize}

This will conclude the proof of Theorem \ref{thm:mainBM}. Namely, we have:
\begin{Cor}
Let $\gamma$ and $\delta$ be two closed curves on $S_{m,n}$, then:
\[ \frac{\Int(\gamma,\delta)}{l(\gamma)l(\delta)} \leq \frac{1}{l_0^2}. \]
Further, equality holds if and only $m$ and $n$ are coprime and  $\gamma$ and $\delta$ are sides of length $l_0$ intersecting at the singularity, or diagonals of length $\sqrt{2}l_0$ intersecting once at the singularity, and once outside the singularity with the same sign.
\end{Cor}

Before proving Lemma \ref{lem:lengths_bouw_moller}, let us explain the equality case for Bouw-M\"oller surfaces. First, we can deduce from Equation~\eqref{eq:intersection_kl_singularities} that equality may hold only if $\gamma$ and $\delta$ are both closed saddle connections. If $m$ and $n$ are not coprime, there are several distinct singularities, and in fact there are no closed saddle connections of length $l_0$, as the sides are never closed curves in this case. Further, there are closed diagonals of length $\sqrt{2} l_0$ only if $n=4$ and $m \equiv 2 \mod 4$, but in this case the two diagonals of $P(0)$ intersect only once (outside the singularities), giving a ratio of $\frac{1}{2 l_0^2}$ instead of $\frac{1}{l_0^2}$.

Next, if $m$ and $n$ are coprime then saddle connections are closed curves and from Theorem \ref{theo:genSC} we deduce that equality holds only if $\alpha$ and $\beta$ are sides of length $l_0$ intersecting at the singularity (diagonals of the polygons defining Bouw-M\"oller surfaces have length greater than $l_0$), or if $\alpha$ and $\beta$ are diagonals of length $\sqrt{2} l_0$ and intersecting twice (as there are no geodesics of length $2l_0$, contained the union of exactly two polygons, and intersecting a systole twice). By Proposition \ref{prop:intersection_diagonals}, this latter case can only arise if $n=4$ and $m \equiv 3 \mod 4$.

\subsection{Proof of Lemma \ref{lem:lengths_bouw_moller}}\label{sec:pf_lbm}

The Bouw-M\"oller surface $S_{m,3}$ is made of two equilateral triangles of side $l_0$ (namely $P(0)$ and $P(m-1)$) and $m-2$ hexagons which are convex with obtuse angles. In particular, non-adjacent segments which do not lie inside $P(0)$ or $P(m-1)$ have length least $l_0$. Furthermore, only the initial and terminal segments $\alpha_1$ and $\alpha_k$ can be non-adjacent on one of the triangles $P(0)$ and $P(m-1)$. Similarly, the length of two adjacent segments can be smaller than $l_0$, but this can only happen if one of the adjacent segments is contained in $P(0)$ or in $P(m-1)$. One way to compensate is then to take into account the next non-adjacent segment. More precisely, we have:
\begin{Lem}\label{lem:compensate_lengths}
Assume $\alpha$ is neither a side nor a diagonal of the polygons defining $S_{m,3}$. 
Then
\begin{enumerate}
    \item If $\alpha_1$ (resp. $\alpha_k$) lies in either $P(0)$ or $P(m-1)$, then $\alpha_2$ (resp. $\alpha_{k-1}$) is a non-adjacent segment and $l(\alpha_1 \cup \alpha_2) > 2 l_0$ (resp. $l(\alpha_{k-1} \cup \alpha_k) > 2 l_0$).
    \item Assume $\alpha_{i}$ and $\alpha_{i+1}$ are two adjacent segments, being respectively contained in $P(0)$ and $P(1)$. Then the segment $\alpha_{i+2}$ is non-adjacent and the total length of $\alpha_{i} \cup \alpha_{i+1} \cup \alpha_{i+2}$ is at least $2l_0$. 
\end{enumerate}
\end{Lem}
\begin{proof}
\begin{enumerate}
    \item By symmetry, we can consider the case of $\alpha_1$ and assume it lies inside $P(0)$. Then we have a configuration like the left of Figure \ref{fig:case_n=3_p=2} and the length of $\alpha_1 \cup \alpha_2$ is greater than $L$ (because $H_0 > L$). Recall that, by definition, the length of the sides of $P(0)$ and $P(1)$ is at least $l_0$ and the internal angles of $P(1)$ are $\frac{2\pi}{3}$. In particular, recalling that the sides of $P(1)$ have lengths $l_0$ and $2 \cos(\pi/m)l_0$, we have $H_0 =(\sqrt{3}+\sqrt{3} \cos(\pi/m))l_0) > 2 l_0$ and $L = (1 + 2 \cos(\pi/m))l_0 \geq 2 l_0$, with equality when $m=3$. Hence we deduce that the length of $\alpha_1 \cup \alpha_2$ must be greater than $2l_0$.
    \item The segments $\alpha_i, \alpha_{i+1}$ and $\alpha_{i+2}$ can be represented as in the right of Figure \ref{fig:case_n=3_p=2}. The segment $\alpha_{i+2}$ is non-adjacent because for $\alpha_{i}$ and $\alpha_{i+1}$ to be adjacent, the angle between $\alpha$ and the vertical must be less than $\frac{\pi}{6}$. In particular, the total length of $\alpha_i \cup \alpha_{i+1} \cup \alpha_{i+2}$ is at least $\min(H_0,L) \geq 2 l_0$. Further, as in part 1., the bound $2l_0$ cannot be achieved.
\end{enumerate}
\end{proof}
Now, the first part of Lemma \ref{lem:lengths_bouw_moller} is a direct corollary of Lemma \ref{lem:compensate_lengths}.

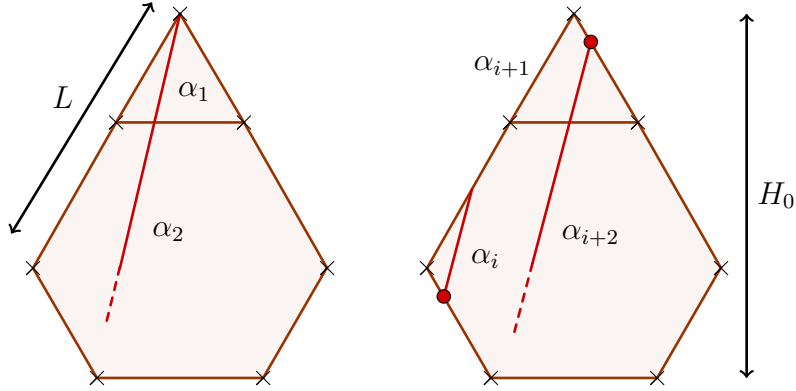
\begin{figure}[h]
\center
\definecolor{ccqqqq}{rgb}{0.8,0,0}
\definecolor{qqqqff}{rgb}{0,0,1}
\definecolor{zzttqq}{rgb}{0.6,0.2,0}
\begin{tikzpicture}[line cap=round,line join=round,>=triangle 45,x=1cm,y=1cm, scale = 1.7]
\clip(-0.5,-0.4) rectangle (3,3.2);
\fill[line width=2pt,color=zzttqq,fill=zzttqq,fill opacity=0.05] (1.15,2) -- (2.15,2) -- (1.65,2.85) -- cycle;
\fill[line width=2pt,color=zzttqq,fill=zzttqq,fill opacity=0.05] (0.5,0.86) -- (1.15,2) -- (2.15,2) -- (2.8,0.86) -- (2.3,0) -- (1,0) -- cycle;
\draw [line width=1pt,color=zzttqq] (1.15,2)-- (2.15,2);
\draw [line width=1pt,color=zzttqq] (2.15,2)-- (1.65,2.85);
\draw [line width=1pt,color=zzttqq] (1.65,2.85)-- (1.15,2);
\draw [line width=1pt,color=zzttqq] (0.5,0.86)-- (1.15,2);
\draw [line width=1pt,color=zzttqq] (2.15,2)-- (2.8,0.86);
\draw [line width=1pt,color=zzttqq] (2.800059040139864,0.8660254037844386)-- (2.3,0);
\draw [line width=1pt,color=zzttqq] (2.3,0)-- (1,0);
\draw [line width=1pt,color=zzttqq] (1,0)-- (0.5,0.86);
\draw [line width=1pt, to-to] (1.43,2.94)-- (0.33,1.13);
\draw (0.9, 2.2) node[left] {$L$};
\draw [line width=1pt,color=ccqqqq] (1.65,2.85)-- (1.18,0.86);
\draw [line width=1pt,dash pattern=on 3pt off 3pt,color=ccqqqq] (1.18,0.87)-- (1.07,0.42);
\draw (1.55,2.4) node[anchor=north west] {$\alpha_1$};
\draw (1.35,1.3) node[anchor=north west] {$\alpha_2$};
\begin{scriptsize}
\draw [color=black] (1,0)-- ++(-1.5pt,-1.5pt) -- ++(3pt,3pt) ++(-3pt,0) -- ++(3pt,-3pt);
\draw [color=black] (0.5,0.86)-- ++(-1.5pt,-1.5pt) -- ++(3pt,3pt) ++(-3pt,0) -- ++(3pt,-3pt);
\draw [color=black] (1.15,2)--++(-1.5pt,-1.5pt) -- ++(3pt,3pt) ++(-3pt,0) -- ++(3pt,-3pt);
\draw [color=black] (2.15,2)-- ++(-1.5pt,-1.5pt) -- ++(3pt,3pt) ++(-3pt,0) -- ++(3pt,-3pt);
\draw [color=black] (2.3,0)-- ++(-1.5pt,-1.5pt) -- ++(3pt,3pt) ++(-3pt,0) -- ++(3pt,-3pt);
\draw [color=black] (2.8,0.86)-- ++(-1.5pt,-1.5pt) -- ++(3pt,3pt) ++(-3pt,0) -- ++(3pt,-3pt);
\draw [color=black] (1.65,2.85)-- ++(-1.5pt,-1.5pt) -- ++(3pt,3pt) ++(-3pt,0) -- ++(3pt,-3pt);
\end{scriptsize}
\end{tikzpicture}
\begin{tikzpicture}[line cap=round,line join=round,>=triangle 45,x=1cm,y=1cm, scale=1.7]
\clip(0,-0.4) rectangle (4.5,3.2);
\fill[line width=2pt,color=zzttqq,fill=zzttqq,fill opacity=0.05] (1.15,2) -- (2.15,2) -- (1.65,2.85) -- cycle;
\fill[line width=2pt,color=zzttqq,fill=zzttqq,fill opacity=0.05] (0.5,0.86) -- (1.15,2) -- (2.15,2) -- (2.8,0.86) -- (2.3,0) -- (1,0) -- cycle;
\draw [line width=1pt,color=zzttqq] (1.15,2)-- (2.15,2);
\draw [line width=1pt,color=zzttqq] (2.15,2)-- (1.65,2.85);
\draw [line width=1pt,color=zzttqq] (1.65,2.85)-- (1.15,2);
\draw [line width=1pt,color=zzttqq] (0.5,0.86)-- (1.15,2);
\draw [line width=1pt,color=zzttqq] (2.15,2)-- (2.8,0.86);
\draw [line width=1pt,color=zzttqq] (2.800059040139864,0.8660254037844386)-- (2.3,0);
\draw [line width=1pt,color=zzttqq] (2.3,0)-- (1,0);
\draw [line width=1pt,color=zzttqq] (1,0)-- (0.5,0.86);
\draw [line width=1pt, to-to] (3,0)-- (3,2.85);
\draw (1.4,2.6) node[anchor=north east] {$\alpha_{i+1}$};
\draw (1.47,1.28) node[anchor=north west] {$\alpha_{i+2}$};
\draw (3, 1.43) node[right] {$H_0$};
\draw [line width=1pt,color=ccqqqq] (0.8521168646006797,1.4759097034746662)-- (0.6314904571969925,0.6382772512087873);
\draw [line width=1pt,color=ccqqqq] (1.7815199772669246,2.6301868101739614)-- (1.3254307791833124,0.8985939426722377);
\draw [line width=1pt,dash pattern=on 3pt off 3pt,color=ccqqqq] (1.3254307791833124,0.8985939426722377)-- (1.1838096531821927,0.36091374173455776);
\draw (0.76,1.09) node[anchor=north west] {$\alpha_i$};
\begin{scriptsize}
\draw [color=black] (1,0)-- ++(-1.5pt,-1.5pt) -- ++(3pt,3pt) ++(-3pt,0) -- ++(3pt,-3pt);
\draw [color=black] (0.5,0.86)-- ++(-1.5pt,-1.5pt) -- ++(3pt,3pt) ++(-3pt,0) -- ++(3pt,-3pt);
\draw [color=black] (1.15,2)--++(-1.5pt,-1.5pt) -- ++(3pt,3pt) ++(-3pt,0) -- ++(3pt,-3pt);
\draw [color=black] (2.15,2)-- ++(-1.5pt,-1.5pt) -- ++(3pt,3pt) ++(-3pt,0) -- ++(3pt,-3pt);
\draw [color=black] (2.3,0)-- ++(-1.5pt,-1.5pt) -- ++(3pt,3pt) ++(-3pt,0) -- ++(3pt,-3pt);
\draw [color=black] (2.8,0.86)-- ++(-1.5pt,-1.5pt) -- ++(3pt,3pt) ++(-3pt,0) -- ++(3pt,-3pt);
\draw [color=black] (1.65,2.85)-- ++(-1.5pt,-1.5pt) -- ++(3pt,3pt) ++(-3pt,0) -- ++(3pt,-3pt);
\draw [fill=ccqqqq] (0.6314904571969925,0.6382772512087873) circle (1.5pt);
\draw [fill=ccqqqq] (1.7815199772669246,2.6301868101739614) circle (1.5pt);
\end{scriptsize}
\end{tikzpicture}
\caption{On the left, the case where $\alpha_1$ belongs to the triangle $P(0)$. On the right, the case where $\alpha_i$ and $\alpha_{i+1}$ are two consecutive adjacent segments and one of them lies in $P(0)$ (here, $\alpha_{i+1}$). Further, we have $L = \left(1+ 2 \cos \frac{\pi}{m}\right)l_0$ and $H_0 = \left(2 \frac{\sqrt{3}}{2} + 2 \cos \frac{\pi}{m}\right) l_0$.}
\label{fig:case_n=3_p=2}
\end{figure}

Concerning the length of odd saddle connections, we have: 
\begin{Lem}\label{lem:3_consecutive_BM}
\begin{enumerate}[label=(\roman*)]
\item The total length of three consecutive adjacent segments is at least $\sqrt{2} l_0$.
\item Given three consecutive segments in the polygonal decomposition of $\alpha$ such that the first and the third segments are non-adjacent while the second segment is adjacent, the total length of the three segments is at least $(1 + \sqrt{2}) l_0$.
\end{enumerate}
\end{Lem}
\begin{proof}
Let us first notice that if none of the adjacent segment are contained in $P(0)$ or $P(m-1)$, then the result is a direct corollary of Lemma \ref{lem:better_lengths}. Hence, we can assume by symmetry that at least one of the adjacent segments is contained in $P(0)$.
\begin{enumerate}[label=(\roman*)]
\item Given three consecutive adjacent segments such that one is contained in $P(0)$, the geometry of Bouw-M\"oller surfaces imply that the other two segments must be contained respectively in $P(1)$ and $P(2)$. Then, we are in the setting of Figure \ref{fig:adjacents_BM} and hence the total length of the three adjacent segments is at least $H_1 = \frac{\sqrt{3}}{2} \left(1+2\cos(\pi/m)\right) l_0 \geq \sqrt{3} l_0$, which is greater than $\sqrt{2} l_0$. 
\item As in $(i)$, we can assume by symmetry that the adjacent segment is contained in $P(0)$. The preceding and following non-adjacent segment are then both contained in $P(1)$. We can see from Figure \ref{fig:na_a_na_BM} that the total length of the three segments must be at least 
$$\min\left( \left(1+ 4 \cos \frac{\pi}{m}\right)l_0, \left(2+4\cos \frac{\pi}{m}\right) \frac{\sqrt{3}}{2} l_0\right),$$
which is greater than $(1+\sqrt{2}) l_0$, as required.
\end{enumerate}
\end{proof}

Putting together Lemma \ref{lem:compensate_lengths} and Lemma \ref{lem:3_consecutive_BM} allows to conclude that the length of an odd saddle connection $\alpha$ which is not a side or a diagonal of a polygon is at least $(p_{\alpha}+q_{\alpha}+\sqrt{2}-1)l_0$. This concludes the proof of Lemma \ref{lem:lengths_bouw_moller}, and hence of Theorem \ref{thm:mainBM}.

\begin{figure}[h]
\center
\definecolor{ccqqqq}{rgb}{0.8,0,0}
\definecolor{qqqqff}{rgb}{0,0,1}
\definecolor{zzttqq}{rgb}{0.6,0.2,0}
\begin{tikzpicture}[line cap=round,line join=round,>=triangle 45,x=1cm,y=1cm]
\clip(-2,-0.6) rectangle (10,5.5);
\fill[line width=1pt,color=zzttqq,fill=zzttqq,fill opacity=0.05] (-1,0) -- (1,0) -- (0,1.7320508075688776) -- cycle;
\fill[line width=1pt,color=zzttqq,fill=zzttqq,fill opacity=0.05] (0,1.7320508075688776) -- (1.6158776665906764,4.530833024919771) -- (3.615877666590676,4.530833024919771) -- (5.2317553331813516,1.732050807568878) -- (4.2317553331813516,0) -- (1,0) -- cycle;
\draw [line width=1pt,color=zzttqq] (-1,0)-- (1,0);
\draw [line width=1pt,dash pattern=on 3pt off 3pt,color=zzttqq] (1,0)-- (0,1.7320508075688776);
\draw [line width=1pt,color=zzttqq] (0,1.7320508075688776)-- (-1,0);
\draw [line width=1pt,color=zzttqq] (0,1.7320508075688776)-- (1.6158776665906764,4.530833024919771);
\draw [line width=1pt,color=zzttqq] (1.6158776665906764,4.530833024919771)-- (3.615877666590676,4.530833024919771);
\draw [line width=1pt,color=zzttqq] (3.615877666590676,4.530833024919771)-- (5.2317553331813516,1.732050807568878);
\draw [line width=1pt,dash pattern=on 3pt off 3pt, color=zzttqq] (0,1.73)-- (-2,1.73);
\draw [line width=1pt,dash pattern=on 3pt off 3pt, color=zzttqq] (1.6158776665906764,4.530833024919771)-- (0,8);
\fill[line width=1pt,color=zzttqq,fill=zzttqq,fill opacity=0.01] (0,1.73)-- (-2,1.73) -- (1.6158776665906764,4.530833024919771)-- (0,6) -- cycle;
\draw [line width=1pt,color=zzttqq] (5.2317553331813516,1.732050807568878)-- (4.2317553331813516,0);
\draw [line width=1pt,color=zzttqq] (4.2317553331813516,0)-- (1,0);
\draw [line width=1pt, to-to] (5.5,0) -- (5.5,4.53);
\draw (5.5,2.75) node[anchor=north west] {$H_1 = \left(\frac{\sqrt{3}}{2} + \sqrt{3} \cos \frac{\pi}{m}\right) l_0$};
\draw (0.5, 2.6) node[anchor=north west] {$\alpha_{i+1}$};
\draw (-0.6,1) node[anchor=north west] {$\alpha_{i}$};
\draw (-0.6,4) node[anchor=north west] {$\alpha_{i+2}$};
\draw (-0.6,0) node[anchor=north west] {$P(0)$};
\draw (2,2.5) node[anchor=north west] {$P(1)$};
\draw (-2,4) node[anchor=north west] {$P(2)$};
\draw [line width=1pt,color=ccqqqq] (0,0)-- (0.5,4.5);
\draw [line width=1pt,dash pattern=on 3pt off 3pt, color=ccqqqq] (0.5,4.5)-- (1,9);
\end{tikzpicture}
\caption{The total length of three consecutive adjacent segments, one being in $P(0)$, is at least $H_1$. Note that in this configuration $\alpha_{i-1}$ cannot be an adjacent segment.}
\label{fig:adjacents_BM}
\end{figure}
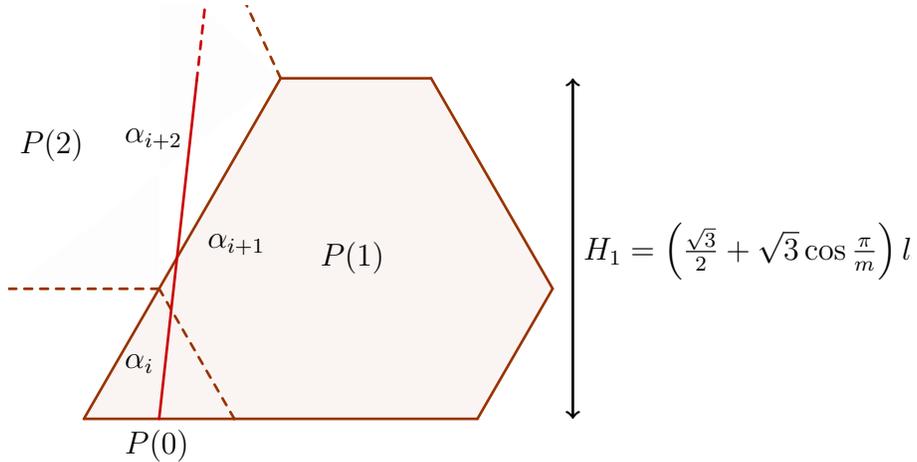

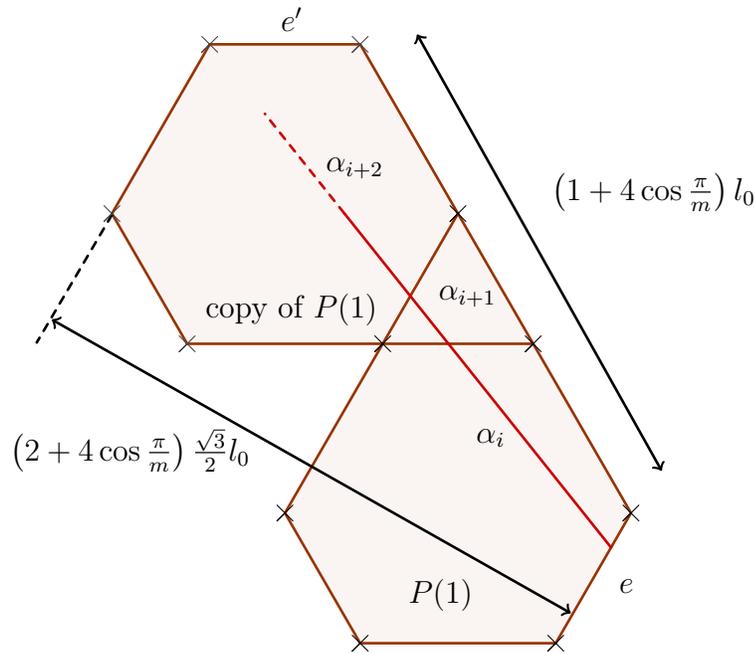
\begin{figure}
\center
\definecolor{ccqqqq}{rgb}{0.8,0,0}
\definecolor{uuuuuu}{rgb}{0.26666666666666666,0.26666666666666666,0.26666666666666666}
\definecolor{zzttqq}{rgb}{0.6,0.2,0}
\begin{tikzpicture}[line cap=round,line join=round,>=triangle 45,x=1cm,y=1cm, scale=2]
\clip(-1.5,-0.36262643384448245) rectangle (4,4.5);
\fill[line width=1pt,color=zzttqq,fill=zzttqq,fill opacity=0.05] (1.1500295200699324,1.9919095589651739) -- (2.1500295200699324,1.9919095589651739) -- (1.6500295200699324,2.8579349627496127) -- cycle;
\fill[line width=1pt,color=zzttqq,fill=zzttqq,fill opacity=0.05] (0.5,0.8660254037844386) -- (1.1500295200699324,1.9919095589651739) -- (2.1500295200699324,1.9919095589651739) -- (2.800059040139864,0.8660254037844386) -- (2.300059040139864,0) -- (1,0) -- cycle;
\fill[line width=1pt,color=zzttqq,fill=zzttqq,fill opacity=0.05] (1.6500295200699324,2.8579349627496127) -- (1,3.9838191179303477) -- (0,3.9838191179303477) -- (-0.6500295200699313,2.8579349627496127) -- (-0.15002952006993153,1.991909558965174) -- (1.1500295200699324,1.9919095589651739) -- cycle;
\draw [line width=1pt,color=zzttqq] (1.1500295200699324,1.9919095589651739)-- (2.1500295200699324,1.9919095589651739);
\draw [line width=1pt,color=zzttqq] (2.1500295200699324,1.9919095589651739)-- (1.6500295200699324,2.8579349627496127);
\draw [line width=1pt,color=zzttqq] (1.6500295200699324,2.8579349627496127)-- (1.1500295200699324,1.9919095589651739);
\draw [line width=1pt,color=zzttqq] (0.5,0.8660254037844386)-- (1.1500295200699324,1.9919095589651739);
\draw [line width=1pt,color=zzttqq] (2.1500295200699324,1.9919095589651739)-- (2.800059040139864,0.8660254037844386);
\draw [line width=1pt,color=zzttqq] (2.800059040139864,0.8660254037844386)-- (2.300059040139864,0);
\draw [line width=1pt,color=zzttqq] (2.300059040139864,0)-- (1,0);
\draw [line width=1pt,color=zzttqq] (1,0)-- (0.5,0.8660254037844386);
\draw (1.45,2.45) node[anchor=north west] {$\alpha_{i+1}$};
\draw (0.7,3.3) node[anchor=north west] {$\alpha_{i+2}$};
\draw (1.7,1.5) node[anchor=north west] {$\alpha_i$};
\draw [line width=1pt,color=ccqqqq] (2.6676932490932663,0.6367611285076861)-- (0.8945910745429955,2.8595648343772058);
\draw [line width=1pt, to-to] (1.3784035772789822,4.049743591107738)-- (3.0040135864718973,1.1468685746918046);
\draw [line width=1pt, to-to] (-1.050335186455671,2.1531985803826617)-- (2.4,0.2);
\draw [line width=1pt,dash pattern=on 3pt off 3pt,color=ccqqqq] (0.8945910745429955,2.8595648343772058)-- (0.3661392612626404,3.522044789448592);
\draw [line width=1pt,color=zzttqq] (1.6500295200699324,2.8579349627496127)-- (1,3.9838191179303477);
\draw [line width=1pt,color=zzttqq] (1,3.9838191179303477)-- (0,3.9838191179303477);
\draw [line width=1pt,color=zzttqq] (0,3.9838191179303477)-- (-0.6500295200699313,2.8579349627496127);
\draw [line width=1pt,color=zzttqq] (-0.6500295200699313,2.8579349627496127)-- (-0.15002952006993153,1.991909558965174);
\draw [line width=1pt,color=zzttqq] (-0.15002952006993153,1.991909558965174)-- (1.1500295200699324,1.9919095589651739);
\draw [line width=1pt, dash pattern=on 3pt off 3pt] (-0.65,2.85) -- (-1.15, 2);
\draw (2.2,3.2) node[anchor=north west] {$\left(1+ 4 \cos \frac{\pi}{m}\right)l_0$};
\draw (-1.4,1.5) node[anchor=north west] {$\left(2+4\cos \frac{\pi}{m}\right) \frac{\sqrt{3}}{2} l_0$};
\draw (-0.1,2.4) node[anchor=north west] {$\text{copy of } P(1)$};
\draw (1.25,0.5) node[anchor=north west] {$P(1)$};
\draw (2.65,0.5) node[anchor=north west] {$e$};
\draw (0.4,4.3) node[anchor=north west] {$e'$};
\begin{scriptsize}
\draw [color=black] (1,0)-- ++(-1.5pt,-1.5pt) -- ++(3pt,3pt) ++(-3pt,0) -- ++(3pt,-3pt);
\draw [color=black] (0.5,0.8660254037844386)-- ++(-1.5pt,-1.5pt) -- ++(3pt,3pt) ++(-3pt,0) -- ++(3pt,-3pt);
\draw [color=black] (1.1500295200699324,1.9919095589651739)-- ++(-1.5pt,-1.5pt) -- ++(3pt,3pt) ++(-3pt,0) -- ++(3pt,-3pt);
\draw [color=black] (2.1500295200699324,1.9919095589651739)-- ++(-1.5pt,-1.5pt) -- ++(3pt,3pt) ++(-3pt,0) -- ++(3pt,-3pt);
\draw [color=black] (2.300059040139864,0)-- ++(-1.5pt,-1.5pt) -- ++(3pt,3pt) ++(-3pt,0) -- ++(3pt,-3pt);
\draw [color=black] (2.800059040139864,0.8660254037844386)-- ++(-1.5pt,-1.5pt) -- ++(3pt,3pt) ++(-3pt,0) -- ++(3pt,-3pt);
\draw [color=black] (1.6500295200699324,2.8579349627496127)-- ++(-1.5pt,-1.5pt) -- ++(3pt,3pt) ++(-3pt,0) -- ++(3pt,-3pt);
\draw [color=uuuuuu] (1,3.9838191179303477)-- ++(-1.5pt,-1.5pt) -- ++(3pt,3pt) ++(-3pt,0) -- ++(3pt,-3pt);
\draw [color=uuuuuu] (0,3.9838191179303477)-- ++(-1.5pt,-1.5pt) -- ++(3pt,3pt) ++(-3pt,0) -- ++(3pt,-3pt);
\draw [color=uuuuuu] (-0.6500295200699313,2.8579349627496127)-- ++(-1.5pt,-1.5pt) -- ++(3pt,3pt) ++(-3pt,0) -- ++(3pt,-3pt);
\draw [color=uuuuuu] (-0.15002952006993153,1.991909558965174)-- ++(-1.5pt,-1.5pt) -- ++(3pt,3pt) ++(-3pt,0) -- ++(3pt,-3pt);
\end{scriptsize}
\end{tikzpicture}
\caption{Example of three consecutive segments, the first and the third ($\alpha_{i}$ and $\alpha_{i+2}$) being non-adjacent, while the second ($\alpha_{i+1}$) is adjacent and contained in $P(0)$. (Since either $\alpha_{i}$ has an endpoint on $e$ or $\alpha_{i+1}$ as an endpoint on $e'$, we assumed by symmetry that $\alpha_{i}$ has an endpoint on $e$.) The total length of the three segments is at least $\min\left( \left(1+ 4 \cos \frac{\pi}{m}\right)l_0, \left(2+4\cos \frac{\pi}{m}\right) \frac{\sqrt{3}}{2} l_0\right) > (1+\sqrt{2}) l_0$.}
\label{fig:na_a_na_BM}
\end{figure}



%% file: Intersection_horizontal_saddle_connections.tex
\section{Intersection of horizontal saddle connections}\label{sec:intersection_horizontal}

The purpose of this section is to prove the following:

\begin{Prop}\label{prop:intersection_horizontal_sc_Bouw_Moller}
Horizontal closed curves on Bouw-M\"oller surfaces are pairwise non-intersecting.
\end{Prop}

This proposition will be used in Section \ref{sec:extention_teichmuller}. 
Moreover, an interesting consequence of Proposition \ref{prop:intersection_horizontal_sc_Bouw_Moller} as well as Theorem 1.5 of \cite{Bou23} stated below, is that we directly obtain the boundedness of KVol on the Teichm\"uller disk of every Bouw-M\"oller surface, independantly of the number of singularities. Namely,

\begin{Cor}\label{cor:boundedness}
KVol is bounded on the Teichm\"uller disk of Bouw-M\"oller surfaces.
\end{Cor}

For convenience we state here Theorem 1.5 of \cite{Bou23}:
\begin{Theo}\cite[Theorem 1.5]{Bou23}
KVol is bounded on the Teichm\"uller disk of a Veech surface $X$ if and only if there are no intersecting closed curves $\eta$ and $\xi$ on $X$ such that $\eta = \eta_1 \cup \cdots \cup \eta_k$ and $\xi = \xi_1 \cup \cdots \cup \xi_l$ are union of parallel saddle connections (that is all saddle connections $\eta_1, \dots, \eta_k,\xi_1, \dots,\xi_l$ have the same direction).
\end{Theo}

To prove Proposition \ref{prop:intersection_horizontal_sc_Bouw_Moller}, we use the notion of (horizontal) \emph{separatrix diagram} defined by Kontsevich-Zorich \cite[Section 4.1]{KZ03} for any translation surface $X$ which is completely periodic in the horizontal direction (i.e. every leaf of the horizontal foliation is closed). The separatrix diagram is a \emph{ribbon graph} (that is a graph with a cyclic ordering of the edges incident to each vertex) which allows to understand the intersection of horizontal closed curves on $X$, that is, closed curves which are union of horizontal saddle connections. More precisely, the next section will show:

\begin{Prop}\label{prop:planar_separatrix_diagram}
Let $X$ be a translation surface which is completely periodic in the horizontal direction. Then horizontal closed curves are pairwise non-intersecting if and only if the horizontal separatrix diagram is planar.
\end{Prop}
\begin{Rema}
Let us stress that the separatrix diagram comes with a cylic orientation of the edges at each vertex. As such, a separatrix diagram is planar if there exist a planar representation of the separatrix diagram respecting the cyclic orientation at each vertex. 
An introduction to graphs on surfaces and ribbon graphs is presented in \cite{GraphsOS}.
\end{Rema}

It turns out that for Bouw-M\"oller surfaces it is easier to study the \emph{dual separatrix diagram}, which is planar if and only if the separatrix diagram is. 
In Section \ref{sec:separatrix_diagram} we first give the standard definition of the dual separatrix diagram as the dual of the separatrix diagram seen as a cell complex. Then we give an alternative definition using the horizontal cylinder decomposition, which is easily understood in the case of Bouw-M\"oller surfaces. This allows to show Proposition \ref{prop:intersection_horizontal_sc_Bouw_Moller}.

\subsection{The horizontal separatrix diagram}\label{sec:separatrix_diagram}
Following \cite{KZ03}, we define the horizontal separatrix diagram associated to a translation surface which is completely periodic in the horizontal direction, as the graph whose edges are the horizontal saddle connections and the vertices are the singularities. The edge corresponding to a saddle connection $A$ connects the vertices corresponding to the endpoint singularities of $A$. The orientation on the edges at a vertex comes from the circular orientation at the corresponding singularity, choosing a trigonometric order (e.g. counter-clockwise).

In particular, a horizontal closed curve corresponds to a closed path on the horizontal separatrix diagram. Further, the algebraic intersection of such closed curves on $X$ is exactly the intersection of the paths on the horizontal separatrix diagram, since the intersection can only happen at the singularity, i.e. at a vertex of the graph and so it is only determined by the cyclic order of the edges around a vertex. In particular, if the horizontal separatrix diagram is planar then closed paths on the horizontal separatrix diagram are pairwise non-intersecting since the plane is simply connected.

Conversely, if the horizontal separatrix diagram is not planar, we can construct two closed paths on the separatrix diagram having non-zero intersection. Indeed, the fact that the graph is not planar implies that we have a configuration like Figure \ref{fig:non_planar_graph}. We then choose the curves $\alpha$ and $\beta$, corresponding to horizontal closed curves on the surface $X$ and intersecting once. We have proven Proposition \ref{prop:planar_separatrix_diagram}.

\begin{figure}[h]
\center
\definecolor{qqwuqq}{rgb}{0,0.39215686274509803,0}
\definecolor{zzttqq}{rgb}{0.6,0.2,0}
\definecolor{uuuuuu}{rgb}{0.26666666666666666,0.26666666666666666,0.26666666666666666}
\begin{tikzpicture}[line cap=round,line join=round,>=triangle 45,x=1.7cm,y=1.7cm]
\clip(-1.7,-2) rectangle (1.7,1);
\draw (-1,0) .. controls (-0.5,0.3) and (0.5,0.3) .. (1,0);
\draw (-1,0) .. controls (-3.5,-2.5) and (3.5,-2.5) .. (1,0);
\draw (-1,0) .. controls (-1,-1) and (0,-1) .. (0,0.1);
\draw (1,0) .. controls (1,1) and (0,1) .. (0,0.3);
\draw [-to] (0.82,-0.1) .. controls (0.9,-0.3) and (1.3,-0.2) .. (1.12,0.1);
\draw [-to] (-0.82,-0.1) .. controls (-0.7,0.3) and (-1.3,0.3) .. (-1.18,-0.1);
\draw [fill=black] (-1,0) circle (2pt);
\draw [fill=black] (1,0) circle (2pt);
\end{tikzpicture}
\begin{tikzpicture}[line cap=round,line join=round,>=triangle 45,x=1.7cm,y=1.7cm]
\clip(-1.7,-2) rectangle (1.7,1);
\draw [color = qqwuqq] (-1,0) .. controls (-0.5,0.3) and (0.5,0.3) .. (1,0);
\draw [color = qqwuqq] (-1,0) .. controls (-3.5,-2.5) and (3.5,-2.5) .. (1,0);
\draw [color = zzttqq] (-1,0) .. controls (-1,-1) and (0,-1) .. (0,0);
\draw [color = zzttqq] (1,0) .. controls (1,1) and (0,1) .. (0,0.3);
\draw [color = zzttqq] (-1,0) .. controls (-0.5,0.2) and (0.5,0.2) .. (1,0);
\draw [-to] (-0.82,-0.1) .. controls (-0.7,0.3) and (-1.3,0.3) .. (-1.18,-0.1);
\draw [-to] (0.82,-0.1) .. controls (0.9,-0.3) and (1.3,-0.2) .. (1.12,0.1);
\draw [color=qqwuqq] (0,-1.8) node[above] {$\beta$};
\draw [color=zzttqq] (1,0.5) node[right] {$\alpha$};
\draw [fill=black] (-1,0) circle (2pt);
\draw [fill=black] (1,0) circle (2pt);
\end{tikzpicture}
\caption{A configuration in a non-planar ribbon graph (the two vertices could also be collapsed as one). From this configuration, one constructs curves $\alpha$ and $\beta$ intersecting once.}
\label{fig:non_planar_graph}
\end{figure}
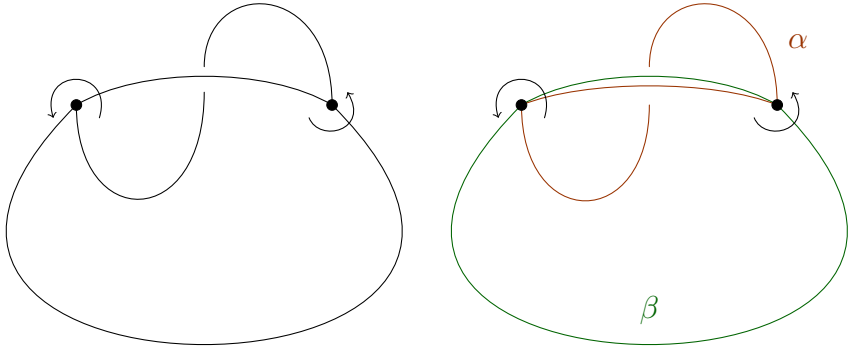

\subsection{The dual separatrix diagram}
\paragraph{The dual graph.}
Now, a ribbon graph is equivalent to a cell complex embedded on a surface (see \cite{GraphsOS}; in fact, one can also see that a ribbon graph is planar if and only if it can be embedded in a sphere).
In particular, there is a well defined notion of faces and it is possible to define the dual of a ribbon graph: replace each face by a vertex, and each edge would then become an edge connecting the two faces it separates (it could be the same face, and in this case we just get a loop around the vertex) and the orientation of the vertex at each edge is just the circular orientation of the edges of the corresponding face. Finally, the vertices of the ribbon graph give the faces of the dual ribbon graph. The dual ribbon graph is then just the dual cell complex and it can be embedded on the same surface. In particular, the dual graph of a planar ribbon graph is planar, and conversely. Hence, we have:

\begin{Prop}\label{prop:dual_planar}
Horizontal closed curves are pairwise non-intersecting if and only if the dual separatrix diagram is planar. 
\end{Prop}

In the next paragraph we give an alternative definition of the dual separatrix diagram which comes directly from the cylinder decomposition and is easier to work with in the case of Bouw-M\"oller surfaces.

\paragraph{An alternative definition.}
Let $X$ be a translation surface which is completely periodic in the horizontal direction. As such, it is decomposed into cylinders: let $\mathcal{C}_1, \dots , \mathcal{C}_n$ be the horizontal cylinders. We define the dual separatrix diagram $\GG (X)$ having $2n$ vertices as follows: 
\begin{itemize}
\item \textbf{Vertices.} For every horizontal cylinder $\mathcal{C}$, we define two vertices of the graph, denoted $\mathcal{C}^+$ and $\mathcal{C}^-$. 
One can think of them as denoting the top and the bottom of each cylinder.
\item \textbf{Edges.} For a horizontal saddle connection $A$ connecting the top of a cylinder $\mathcal{C}_i$ to the bottom of a cylinder $\mathcal{C}_{i'}$, we define an edge between the vertices of label $\mathcal{C}_i^+$ and $\mathcal{C}_{i'}^-$.

\item \textbf{Cyclic orientation on the vertices.}
Given a vertex $\mathcal{C}^+$, corresponding to the top of the cylinder $\mathcal{C}$, define the cyclic orientation on the edges from $\mathcal{C}^+$ as the left to right order on the horizontal saddle connections on the top of $\mathcal{C}^+$.\newline
Given a vertex $\mathcal{C}^-$, corresponding to the bottom of the cylinder $\mathcal{C}$, define the cyclic orientation on the edges from $\mathcal{C}^-$ as the right to left order on the horizontal saddle connections on the bottom of $\mathcal{C}^-$. See Figure \ref{fig:construction_graph}.
\end{itemize}
An example of such graph is depicted in Figure \ref{fig:example_graph}. We explain in the next section why it is the dual graph of the horizontal separatrix diagram.

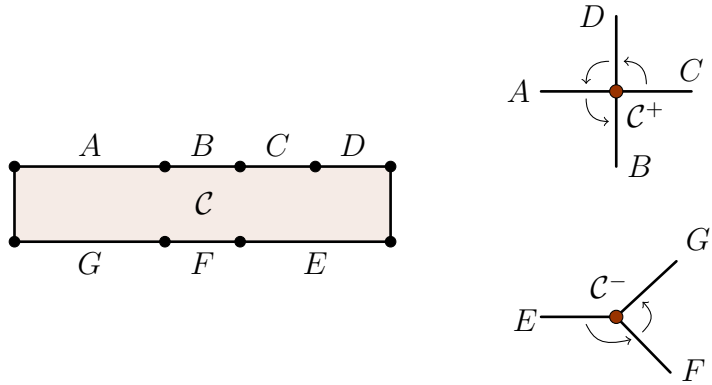
\begin{figure}[h]
\center
\definecolor{zzttqq}{rgb}{0.6,0.2,0}
\begin{tikzpicture}[line cap=round,line join=round,>=triangle 45,x=1cm,y=1cm]
\clip(-5.5,-2.5) rectangle (5,3.5);
\fill[line width=2pt,color=zzttqq,fill=zzttqq,fill opacity=0.1] (-5,1) -- (0,1) -- (0,0) -- (-5,0) -- cycle;
\draw [line width=1pt] (-5,0)-- (-5,1);
\draw [line width=1pt] (-5,1)-- (-3,1);
\draw [line width=1pt] (-3,1)-- (-2,1);
\draw [line width=1pt] (-2,1)-- (-1,1);
\draw [line width=1pt] (-1,1)-- (0,1);
\draw [line width=1pt] (0,1)-- (0,0);
\draw [line width=1pt] (0,0)-- (-2,0);
\draw [line width=1pt] (-2,0)-- (-3,0);
\draw [line width=1pt] (-3,0)-- (-5,0);
\draw [line width=1pt] (3,2)-- (2,2);
\draw [line width=1pt] (3,2)-- (3,3);
\draw [line width=1pt] (3,2)-- (4,2);
\draw [line width=1pt] (3,2)-- (3,1);
\draw [line width=1pt] (3,-1)-- (2,-1);
\draw [line width=1pt] (3,-1)-- (3.8,-0.26);
\draw [line width=1pt] (3,-1)-- (3.72,-1.74);
\draw (-4,1) node[above] {$A$};
\draw (-2.5,1) node[above] {$B$};
\draw (-1.5,1) node[above] {$C$};
\draw (-0.5,1) node[above] {$D$};
\draw (-4,0) node[below] {$G$};
\draw (-2.5,0) node[below] {$F$};
\draw (-1,0) node[below] {$E$};
\draw (2,2) node[left] {$A$};
\draw (3,3) node[left] {$D$};
\draw (4,2) node[above] {$C$};
\draw (3,1) node[right] {$B$};
\draw [-to] (3.4,2.1) .. controls (3.4,2.3) and (3.3,2.4) .. (3.1,2.4);
\draw [-to] (2.9,2.4) .. controls (2.7,2.4) and (2.6,2.3) .. (2.6,2.1);
\draw [-to] (2.6, 1.9) .. controls (2.6, 1.7) and (2.7,1.6) .. (2.9,1.6);
\draw (1.48,-0.76) node[anchor=north west] {$E$};
\draw (3.78,0.32) node[anchor=north west] {$G$};
\draw (3.72,-1.42) node[anchor=north west] {$F$};
\draw (-2.74,0.8) node[anchor=north west] {$\mathcal{C}$};
\draw (3,2) node[anchor=north west] {$\mathcal{C}^+$};
\draw (3.3,-0.9) node[anchor=south east] {$\mathcal{C}^-$};
\draw [-to] (2.6,-1.1) .. controls (2.7,-1.3) and (2.8,-1.4) .. (3.2,-1.3);
\draw [-to] (3.35,-1.2) .. controls (3.5,-1.05) and (3.5,-0.95) .. (3.35,-0.8);
\begin{scriptsize}
\draw [fill=black] (-5,0) circle (2pt);
\draw [fill=black] (-5,1) circle (2pt);
\draw [fill=black] (-3,1) circle (2pt);
\draw [fill=black] (-2,1) circle (2pt);
\draw [fill=black] (-1,1) circle (2pt);
\draw [fill=black] (0,1) circle (2pt);
\draw [fill=black] (0,0) circle (2pt);
\draw [fill=black] (-2,0) circle (2pt);
\draw [fill=black] (-3,0) circle (2pt);
\draw [fill=zzttqq] (3,2) circle (2.5pt);
\draw [fill=zzttqq] (3,-1) circle (2.5pt);
\end{scriptsize}
\end{tikzpicture}
\caption{The two vertices $\mathcal{C}^+$ and $\mathcal{C}^-$ from a cylinder $\mathcal{C}$. The circular order of the edges is determined by the order of the horizontal saddle connections.}
\label{fig:construction_graph}
\end{figure}

\begin{figure}[h]
\center
\definecolor{qqwuqq}{rgb}{0,0.39215686274509803,0}
\definecolor{zzttqq}{rgb}{0.6,0.2,0}
\definecolor{uuuuuu}{rgb}{0.26666666666666666,0.26666666666666666,0.26666666666666666}
\begin{tikzpicture}[line cap=round,line join=round,>=triangle 45,x=1cm,y=1cm]
\clip(-3.1,-2.1) rectangle (10.1,2.1);
\fill[line width=2pt,color=zzttqq,fill=zzttqq,fill opacity=0.10000000149011612] (-3,1) -- (0,1) -- (0,0) -- (-3,0) -- cycle;
\fill[line width=2pt,color=qqwuqq,fill=qqwuqq,fill opacity=0.1] (-1,0) -- (1,0) -- (1,-1) -- (-1,-1) -- cycle;
\draw [line width=1pt] (0,0)-- (-1,0);
\draw [line width=1pt] (-1,0)-- (-2,0);
\draw [line width=1pt] (-2,0)-- (-3,0);
\draw [line width=1pt] (-3,0)-- (-3,1);
\draw [line width=1pt] (-3,1)-- (-2,1);
\draw [line width=1pt] (-2,1)-- (-1,1);
\draw [line width=1pt] (-1,1)-- (0,1);
\draw [line width=1pt] (0,1)-- (0,0);
\draw [line width=1pt] (-1,0)-- (-1,-1);
\draw [line width=1pt] (-1,-1)-- (0,-1);
\draw [line width=1pt] (0,-1)-- (1,-1);
\draw [line width=1pt] (1,0)-- (0,0);
\draw [line width=1pt] (1,0)--(1,-1);
\begin{scriptsize}
\draw [fill=uuuuuu] (0,0) circle (2pt);
\draw [fill=black] (-1,0) circle (2pt);
\draw [fill=black] (-2,0) circle (2pt);
\draw [fill=black] (-3,0) circle (2pt);
\draw [fill=black] (-3,1) circle (2pt);
\draw [fill=black] (-2,1) circle (2pt);
\draw [fill=black] (-1,1) circle (2pt);
\draw [fill=black] (0,1) circle (2pt);
\draw [fill=black] (-1,-1) circle (2pt);
\draw [fill=black] (0,-1) circle (2pt);
\draw [fill=black] (1,-1) circle (2pt);
\draw [fill=black] (1,0) circle (2pt);
\draw [fill=zzttqq] (5,1) circle (2.5pt);
\draw [fill=zzttqq] (5,-1) circle (2.5pt);
\draw [fill=qqwuqq] (8,1) circle (2.5pt);
\draw [fill=qqwuqq] (8,-1) circle (2.5pt);
\end{scriptsize}
\draw (5,1) node[anchor=south west] {$\mathcal{C}_1^+$};
\draw (5,-1) node[anchor=north west] {$\mathcal{C}_1^-$};
\draw (8,1) node[anchor=south east] {$\mathcal{C}_2^+$};
\draw (8,-1) node[anchor=north east] {$\mathcal{C}_2^-$};
\draw (-1.5,0.5) node[anchor=center] {$\mathcal{C}_1$};
\draw (0,-0.5) node[anchor=center] {$\mathcal{C}_2$};
\draw (-2.5,1) node[above] {$A$};
\draw (-1.5,1) node[above] {$B$};
\draw (-0.5,1) node[above] {$C$};
\draw (-0.5,0) node[below] {$A'$};
\draw (-1.5,0) node[below] {$B$};
\draw (-2.5,0) node[below] {$C$};
\draw (0.5,0) node[above] {$D$};
\draw (0.5,-1) node[below] {$D$};
\draw (-0.5,-1) node[below] {$A$};
\draw (5,1) -- (6.4,0.07);
\draw (6.6,-0.07) -- (8, -1);
\draw (5,-1) .. controls (6.5,0) .. (8, 1);
\draw (5,1) .. controls (3,2) .. (4.3,0.1);
\draw (4.4,-0.05) -- (5,-1);
\draw (5,1) .. controls (3,-2) .. (5, -1);
\draw (8,1) .. controls (10,2) and (10,-2) .. (8, -1);
\draw (6,0.5) node[above] {$A$};
\draw (7,0.5) node[above] {$A'$};
\draw (3.4,1.5) node[left] {$B$};
\draw (3.5,-1.5) node[left] {$C$};
\draw (9.1,-1.2) node[right] {$D$};
\end{tikzpicture}
\caption{Example of decomposition into cylinders and its associated graph.}
\label{fig:example_graph}
\end{figure}
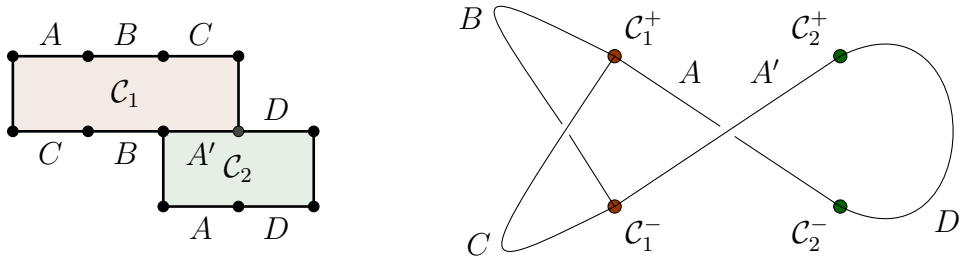

\subsection{Turning aroung the singularity}
Let us clarify how to tell the cyclic sequence of horizontal saddle connections crossed while turning around the singularity, as it will explain how to recover the separatrix diagram from the dual separatrix diagram and hence prove that the two definitions of dual separatrix diagram coincide.

Assume we turn around the singularity in the trigonometric order, as in, counter-clockwise starting from a horizontal. Then, see Figure \ref{fig:explication_turning_around_sing},
\begin{itemize}
\item after crossing a side $A$ on the top of a cylinder $\mathcal{C}$, we cross a side $B$ which is the top side of $\mathcal{C}$ directly on the right of $A$ (up to cyclic identification);
\item after crossing a side $C$ on the bottom of a cylinder $\mathcal{C}$, we cross a side $D$ which is the bottom side of $\mathcal{C}$ directly on the left of $C$ (up to cyclic identification);
\end{itemize}

\begin{figure}[h]
\center
\definecolor{qqwuqq}{rgb}{0,0.39215686274509803,0}
\begin{tikzpicture}[line cap=round,line join=round,>=triangle 45,x=1cm,y=1cm,scale=0.8]
\clip(-0.5,-0.6) rectangle (8,2.7);
\fill[line width=1pt,color=qqwuqq,fill=qqwuqq,fill opacity=0.1] (0,0) -- (4,0) -- (4,2) -- (0,2) -- cycle;
\draw [line width=1pt] (0,2)-- (2,2);
\draw [line width=1pt] (2,2)-- (4,2);
\draw [line width=1pt] (0,0)-- (2,0);
\draw [line width=1pt] (2,0)-- (4,0);
\draw [shift={(2,2)},line width=1pt, -to]  plot[domain=3.141592653589793:4.71238898038469,variable=\t]({1*0.4*cos(\t r)+0*0.4*sin(\t r)},{0*0.4*cos(\t r)+1*0.4*sin(\t r)});
\draw [shift={(2,2)},line width=1pt]  plot[domain=-1.5707963267948966:0,variable=\t]({1*0.4*cos(\t r)+0*0.4*sin(\t r)},{0*0.4*cos(\t r)+1*0.4*sin(\t r)});
\draw [shift={(2,0)},line width=1pt, -to]  plot[domain=0:1.5707963267948966,variable=\t]({1*0.4*cos(\t r)+0*0.4*sin(\t r)},{0*0.4*cos(\t r)+1*0.4*sin(\t r)});
\draw [shift={(2,0)},line width=1pt]  plot[domain=1.5707963267948966:3.141592653589793,variable=\t]({1*0.4*cos(\t r)+0*0.4*sin(\t r)},{0*0.4*cos(\t r)+1*0.4*sin(\t r)});
\draw (1,2) node[above] {$A$};
\draw (3,2) node[above] {$B$};
\draw (3,0) node[below] {$C$};
\draw (1,0) node[below] {$D$};
\draw (1,1) node[right] {$\mathcal{C}$};
\draw [line width=1pt] (5.5,2)-- (6.5,2);
\draw [line width=1pt] (6.5,2)-- (7.5,2);
\draw [line width=1pt] (6.5,0)-- (5.5,0);
\draw [line width=1pt] (6.5,0)-- (7.5,0);
\draw [shift={(6.5,0)},line width=1pt,-to]  plot[domain=0:1.5707963267948966,variable=\t]({1*0.5*cos(\t r)+0*0.5*sin(\t r)},{0*0.5*cos(\t r)+1*0.5*sin(\t r)});
\draw [shift={(6.5,0)},line width=1pt]  plot[domain=1.5707963267948966:3.141592653589793,variable=\t]({1*0.5*cos(\t r)+0*0.5*sin(\t r)},{0*0.5*cos(\t r)+1*0.5*sin(\t r)});
\draw [shift={(6.5,2)},line width=1pt,-to]  plot[domain=0:1.5707963267948966,variable=\t]({1*0.5*cos(\t r)+0*0.5*sin(\t r)},{0*0.5*cos(\t r)+1*0.5*sin(\t r)});
\draw [shift={(6.5,2)},line width=1pt]  plot[domain=1.5707963267948966:3.141592653589793,variable=\t]({1*0.5*cos(\t r)+0*0.5*sin(\t r)},{0*0.5*cos(\t r)+1*0.5*sin(\t r)});
\draw (7.5,2) node[below] {$A$};
\draw (5.5,2) node[below] {$B$};
\draw (5.5,0) node[above] {$D$};
\draw (7.5,0) node[below] {$C$};
\draw (6.37563118052241,2.0813394257655298) node[anchor=north west] {$\mathcal{C}^+$};
\draw (6.403251722021421,0.037419354838711005) node[anchor=north west] {$\mathcal{C}^-$};
\begin{scriptsize}
\draw [color=black] (2,2)-- ++(-2.5pt,-2.5pt) -- ++(5pt,5pt) ++(-5pt,0) -- ++(5pt,-5pt);
\draw [fill=black] (2,0) circle (2pt);
\draw [color=qqwuqq] (6.5,2)-- ++(-2.5pt,-2.5pt) -- ++(5pt,5pt) ++(-5pt,0) -- ++(5pt,-5pt);
\draw [fill=qqwuqq] (6.5,0)-- ++(-2.5pt,-2.5pt) -- ++(5pt,5pt) ++(-5pt,0) -- ++(5pt,-5pt);
\end{scriptsize}
\end{tikzpicture}
\caption{Turning around the singularity on the surface $X$ and on the graph $\GG (X)$.}
\label{fig:explication_turning_around_sing}
\end{figure}
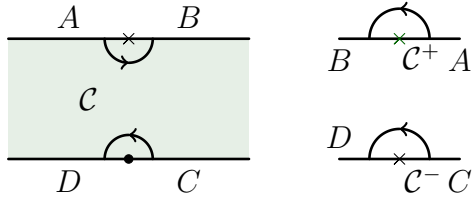

In particular, turning around a singularity on the surface may be described combinatorially as following a path in the graph $\mathfrak{G}(X)$, or more precisely in the directed double cover $\mathfrak{G}^{2,\to}(X)$ obtained by replacing each edge of $\mathfrak{G}(X)$ with two oppositely oriented copies. At each vertex, the path continues along the edge immediately following the incoming edge in the counterclockwise cyclic order around that vertex, see Figure \ref{fig:example_turning_around_sing}. Note that vertices may be visited multiple times, and that an edge of $\mathfrak{G}(X)$ may be traversed twice, once in each direction.

The number of pairwise disjoint paths\footnote{Viewed as collections of oriented edges in the directed double copy $\mathfrak{G}^{2,\to}(X)$.} obtained in this manner coincides with the number of singularities of $X$. In particular, if $X$ has a unique singularity, then $\mathfrak{G}(X)$ is connected.

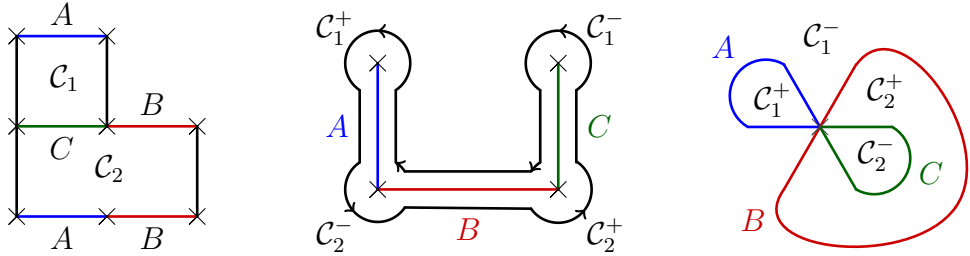
\begin{figure}[h]
\center
\definecolor{uuuuuu}{rgb}{0.26666666666666666,0.26666666666666666,0.26666666666666666}
\definecolor{qqwuqq}{rgb}{0,0.39215686274509803,0}
\definecolor{ccqqqq}{rgb}{0.8,0,0}
\definecolor{qqqqff}{rgb}{0,0,1}
\begin{tikzpicture}[line cap=round,line join=round,>=triangle 45,x=1cm,y=1cm, scale=1.2]
\clip(-4,-0.5) rectangle (8,2.5);
\draw [line width=1pt,color=qqqqff] (-3,0)-- (-2,0);
\draw [line width=1pt,color=ccqqqq] (-2,0)-- (-1,0);
\draw [line width=1pt,color=ccqqqq] (-2,1)-- (-1,1);
\draw [line width=1pt] (-1,1)-- (-1,0);
\draw [line width=1pt,color=qqwuqq] (-2,1)-- (-3,1);
\draw [line width=1pt] (-3,1)-- (-3,0);
\draw [line width=1pt,color=qqqqff] (-3,2)-- (-2,2);
\draw [line width=1pt] (-2,2)-- (-2,1);
\draw [line width=1pt] (-3,2)-- (-3,1);
\draw [line width=1pt,color=ccqqqq] (5.9,0.9928203230275512)-- (5.5,0.3);
\draw [line width=1pt,color=qqqqff] (5.9,0.9928203230275512)-- (5.1,0.9928203230275519);
\draw [line width=1pt,color=qqqqff] (5.9,0.9928203230275512)-- (5.5,1.6856406460551026);
\draw [line width=1pt,color=ccqqqq] (5.9,0.9928203230275512)-- (6.3,1.6856406460551019);
\draw [line width=1pt,color=qqwuqq] (5.9,0.9928203230275512)-- (6.7,0.9928203230275505);
\draw [line width=1pt,color=ccqqqq] (5.5,0.3) .. controls (5,-0.4) and (6.9,-0.6) .. (7.4,0.1);
\draw [line width=1pt,color=ccqqqq] (7.4,0.1) .. controls (7.9,0.8) and (6.8,2.38) .. (6.3,1.68);
\draw [line width=1pt,color=qqwuqq] (5.9,0.9928203230275512)-- (6.3,0.3);
\draw [shift={(5.3,1.3392304845413272)},line width=1pt,color=qqqqff]  plot[domain=1.0471975511965963:4.18879020478639,variable=\t]({1*0.4*cos(\t r)+0*0.4*sin(\t r)},{0*0.4*cos(\t r)+1*0.4*sin(\t r)});
\draw [shift={(6.5,0.6464101615137753)},line width=1pt,color=qqwuqq]  plot[domain=-2.0943951023931966:1.047197551196597,variable=\t]({1*0.4*cos(\t r)+0*0.4*sin(\t r)},{0*0.4*cos(\t r)+1*0.4*sin(\t r)});
\draw [line width=1pt,color=qqwuqq] (3,0.3)-- (3,1.7);
\draw [line width=1pt,color=qqqqff] (1,1.7)-- (1,0.3);
\draw (-2.5,0) node[below] {$A$};
\draw (-1.5,1) node[above] {$B$};
\draw (-1.5,0) node[below] {$B$};
\draw (-2.5,1) node[below] {$C$};
\draw (-2.5,2) node[above] {$A$};
\draw [color=qqqqff](4.576353735052029,2.118313933624504) node[anchor=north west] {$A$};
\draw [color=qqwuqq](6.874090763751652,0.7165498343549337) node[anchor=north west] {$C$};
\draw [color=ccqqqq](4.9,0.2) node[anchor=north west] {$B$};
\draw (-2.75,1.8) node[anchor=north west] {$\mathcal{C}_1$};
\draw (-2.25,0.8) node[anchor=north west] {$\mathcal{C}_2$};
\draw (0.2,2.4) node[anchor=north west] {$\mathcal{C}_1^+$};
\draw (3.2,2.4) node[anchor=north west] {$\mathcal{C}_1^-$};
\draw (3.2,0.1) node[anchor=north west] {$\mathcal{C}_2^+$};
\draw (0.2,0.1) node[anchor=north west] {$\mathcal{C}_2^-$};
\draw [color=qqqqff](0.8,1) node[left] {$A$};
\draw [color=ccqqqq](2,0.1) node[below] {$B$};
\draw [color=qqwuqq](3.2,1) node[right] {$C$};
\draw [line width=1pt,color=ccqqqq] (1,0.3)-- (3,0.3);
\draw [line width=1pt] (0.8,1.4)-- (0.8,0.6);
\draw [line width=1pt] (1.2,1.4)-- (1.2,0.6);
\draw [line width=1pt] (1.3,0.5)-- (2.7,0.5);
\draw [line width=1pt] (1.3,0.1)-- (2.6980140264410246,0.08665695808140177);
\draw [line width=1pt] (3.2,0.6)-- (3.2,1.4);
\draw [line width=1pt] (2.8,0.6)-- (2.8,1.4);
\draw [shift={(1,1.7)},line width=1pt,-to]  plot[domain=-0.9827937232473296:1.5707963267948966,variable=\t]({1*0.3605551275463989*cos(\t r)+0*0.3605551275463989*sin(\t r)},{0*0.3605551275463989*cos(\t r)+1*0.3605551275463989*sin(\t r)});
\draw [shift={(1,1.7)},line width=1pt]  plot[domain=1.5707963267948966:4.124386376837123,variable=\t]({1*0.36055512754639873*cos(\t r)+0*0.36055512754639873*sin(\t r)},{0*0.36055512754639873*cos(\t r)+1*0.36055512754639873*sin(\t r)});
\draw [shift={(1,0.3)},line width=1pt,-to]  plot[domain=2.158798930342464:3.911439915453955,variable=\t]({1*0.36055512754639885*cos(\t r)+0*0.36055512754639885*sin(\t r)},{0*0.36055512754639885*cos(\t r)+1*0.36055512754639885*sin(\t r)});
\draw [shift={(1,0.3)},line width=1pt]  plot[domain=3.911439915453955:5.695182703632019,variable=\t]({1*0.3605551275463988*cos(\t r)+0*0.3605551275463988*sin(\t r)},{0*0.3605551275463988*cos(\t r)+1*0.3605551275463988*sin(\t r)});
\draw [shift={(3,0.3)},line width=1pt,-to]  plot[domain=3.7566455297784405:5.641372015605702,variable=\t]({1*0.3697442112615733*cos(\t r)+0*0.3697442112615733*sin(\t r)},{0*0.3697442112615733*cos(\t r)+1*0.3697442112615733*sin(\t r)});
\draw [shift={(3,0.3)},line width=1pt]  plot[domain=-0.6418132915738841:0.9827937232473286,variable=\t]({1*0.36974421126157336*cos(\t r)+0*0.36974421126157336*sin(\t r)},{0*0.36974421126157336*cos(\t r)+1*0.36974421126157336*sin(\t r)});
\draw [shift={(3,1.7)},line width=1pt,-to]  plot[domain=-0.9827937232473287:1.558735750079066,variable=\t]({1*0.36055512754639907*cos(\t r)+0*0.36055512754639907*sin(\t r)},{0*0.36055512754639907*cos(\t r)+1*0.36055512754639907*sin(\t r)});
\draw [shift={(3,1.7)},line width=1pt]  plot[domain=1.558735750079066:4.124386376837122,variable=\t]({1*0.3605551275463991*cos(\t r)+0*0.3605551275463991*sin(\t r)},{0*0.3605551275463991*cos(\t r)+1*0.3605551275463991*sin(\t r)});
\draw [shift={(1,0.3)},line width=1pt,-to]  plot[domain=0.5880026035475675:0.9827937232473292,variable=\t]({1*0.360555127546399*cos(\t r)+0*0.360555127546399*sin(\t r)},{0*0.360555127546399*cos(\t r)+1*0.360555127546399*sin(\t r)});
\draw [shift={(3,0.3)},line width=1pt,-to]  plot[domain=2.1587989303424644:2.5535900500422253,variable=\t]({1*0.36055512754639907*cos(\t r)+0*0.36055512754639907*sin(\t r)},{0*0.36055512754639907*cos(\t r)+1*0.36055512754639907*sin(\t r)});
\draw (5.05,1.6) node[anchor=north west] {$\mathcal{C}_1^+$};
\draw (5.6,2.3) node[anchor=north west] {$\mathcal{C}_1^-$};
\draw (6.3,1.7) node[anchor=north west] {$\mathcal{C}_2^+$};
\draw (6.2,1) node[anchor=north west] {$\mathcal{C}_2^-$};
\begin{scriptsize}
\draw [color=black] (-3,0)-- ++(-2.5pt,-2.5pt) -- ++(5pt,5pt) ++(-5pt,0) -- ++(5pt,-5pt);
\draw [color=black] (-3,1)-- ++(-2.5pt,-2.5pt) -- ++(5pt,5pt) ++(-5pt,0) -- ++(5pt,-5pt);
\draw [color=black] (-3,2)-- ++(-2.5pt,-2.5pt) -- ++(5pt,5pt) ++(-5pt,0) -- ++(5pt,-5pt);
\draw [color=black] (-2,2)-- ++(-2.5pt,-2.5pt) -- ++(5pt,5pt) ++(-5pt,0) -- ++(5pt,-5pt);
\draw [color=black] (-2,1)-- ++(-2.5pt,-2.5pt) -- ++(5pt,5pt) ++(-5pt,0) -- ++(5pt,-5pt);
\draw [color=black] (-1,1)-- ++(-2.5pt,-2.5pt) -- ++(5pt,5pt) ++(-5pt,0) -- ++(5pt,-5pt);
\draw [color=black] (-1,0)-- ++(-2.5pt,-2.5pt) -- ++(5pt,5pt) ++(-5pt,0) -- ++(5pt,-5pt);
\draw [color=black] (-2,0)-- ++(-2.5pt,-2.5pt) -- ++(5pt,5pt) ++(-5pt,0) -- ++(5pt,-5pt);
\draw [color=uuuuuu] (5.9,0.9928203230275512)-- ++(-2.5pt,-2.5pt) -- ++(5pt,5pt) ++(-5pt,0) -- ++(5pt,-5pt);
\draw [color=black] (1,0.3)-- ++(-2.5pt,-2.5pt) -- ++(5pt,5pt) ++(-5pt,0) -- ++(5pt,-5pt);
\draw [color=black] (1,1.7)-- ++(-2.5pt,-2.5pt) -- ++(5pt,5pt) ++(-5pt,0) -- ++(5pt,-5pt);
\draw [color=black] (3,1.7)-- ++(-2.5pt,-2.5pt) -- ++(5pt,5pt) ++(-5pt,0) -- ++(5pt,-5pt);
\draw [color=black] (3,0.3)-- ++(-2.5pt,-2.5pt) -- ++(5pt,5pt) ++(-5pt,0) -- ++(5pt,-5pt);
\end{scriptsize}
\end{tikzpicture}
\caption{From the dual separatrix diagram to the separatrix diagram.}
\label{fig:example_turning_around_sing}
\end{figure}

\subsection{Saddle connections on Bouw-M\"oller surfaces}
In this paragraph we use the criterion of Proposition \ref{prop:dual_planar} to show that horizontal closed curves are pairwise non-intersecting on Bouw-M\"oller surfaces, hence proving Proposition \ref{prop:intersection_horizontal_sc_Bouw_Moller}.\newline

Let $m,n \geq 2$, $mn \geq 6$. We recall that the Bouw-M\"oller surface $S_{m,n}$ has $gcd(m,n)$ singularities. 
Moreover, $S_{m,n}$ is a Veech surface and as such it is completely periodic in every saddle connection direction. Furthermore, all the cylinder decomposition in every periodic direction have the same combinatoric structure and we have:
\begin{Lem}\label{lem:top_bottom_cyl_Smn}
The Bouw-M\"oller surface $S_{m,n}$ is decomposed into $\frac{(m-1) \times (n-1)}{2}$ cylinders, and every cylinder has either one or two top saddle connections and one or two bottom saddle connections.
\end{Lem}
\begin{proof}[Proof of Lemma \ref{lem:top_bottom_cyl_Smn}]
By construction of $S_{m,n}$, all cylinders are contained in the union of two polygons, and in each polygon there is at most one horizontal top (resp. bottom) saddle connection. 
\end{proof}
From this lemma we deduce that each vertex of the dual separatrix diagram $\GG(S_{m,n})$ of $S_{m,n}$ has either one or two incident edges. Thus, every connected component of $\GG(S_{m,n})$ is either a tree or a cycle, and $\GG(S_{m,n})$ is planar. Hence, Proposition \ref{prop:intersection_horizontal_sc_Bouw_Moller} follows from Proposition \ref{prop:dual_planar}.

\begin{Rema}
For a translation surface with only one singularity, computing the Euler characteristic of the separatrix diagram in a periodic direction, as well as the Euler characteristic of its cylinder decomposition, shows that having a planar separatrix diagram is equivalent to having exactly $g$ cylinders, with $g$ being the genus of the surface. This is the case, for example, for any periodic direction in any algebraically primitive Veech surface which has a single singularity (see \cite{Hubert_Lanneau_Parabolic}). In particular Proposition \ref{prop:dual_planar} and \cite[Theorem 1.5]{Bou23} imply that KVol is always bounded on the Teichm\"uller disk of an algebraically primitive Veech surface having a single singularity. 


\end{Rema}

%% file: extension_Teichmuller_disk.tex
\section{Extension to the Teichm\"uller disk}\label{sec:extention_teichmuller}

In the rest of the paper we study KVol as a function on the Teichm\"uller disk of $S_{m,n}$, with $m,n$ coprime. 
In this section, we follow Sections 4 and 5 of \cite{BLM22} and give another convenient expression for KVol on the Teichm\"uller disk of $S_{m,n}$. We will then study the cylinder decompositions of $S_{m,n}$ in Section \ref{sec:study_cylinder} in order to derive estimates which allow to compute KVol on the Teichm\"uller disk.

\subsection{Preliminaries : Direction decomposition}\label{sec:direction_decomposition}
Recall that we have a parametrization of the Teichm\"uller disk given by Definition \ref{def:teichmuller}. We start with a few definitions.

\begin{Def}[The \coslo of a saddle connection on $X \in \Tmn$]\label{def:directions}

\begin{enumerate}[label=(\roman*)]
\item Given a saddle connection $\alpha$ on $S_{m,n}$, of holonomy vector $\Vec{\alpha} = \begin{pmatrix} \alpha_x \\ \alpha_y \end{pmatrix}$, we define the \emph{\coslo} of $\alpha$ as $d_{\alpha} := -\cfrac{\alpha_x}{\alpha_y} \in \RR \cup \{ \infty \}$, that is minus the co-slope of the holonomy vector $\Vec{\alpha}$. If $\alpha$ is horizontal, the co-slope is defined as $\infty$. \newline
Using the classical identification $\RR \cup \{ \infty \} \equiv \partial \HH$,
this means that we are associating a point in $\partial \HH$ to each saddle connection in $S_{m,n}$. \newline
\item Further, for any $M \in GL^+_2(\RR)$, the \coslo of a saddle connection $\alpha \subset M \cdot S_{m,n}$ is defined as the \coslo of the preimage $M^{-1} \cdot \alpha$, which is a saddle connection on the surface $S_{m,n}$.
\end{enumerate}
\end{Def}

Note that for saddle connections in $S_{m,n}$ the \coslo is just the direction of the saddle connection, so we will use the two names interchangeably. In particular, the set of admissible consistent slopes corresponds to the periodic directions on the surface $S_{m,n}$. We will denote this set $\mathcal{P} \subset \partial \HH$.\newline

The above definition allows to compute the angle between two saddle connections on any surface of the $SL_2(\RR)$-orbit of $S_{m,n}$. More precisely, we will denote by $B_\theta (d,d')$ the set of elements $M \cdot S_{m,n} \in \Tmn$ such that the holonomy vectors of $M \cdot \alpha$ and $M \cdot \alpha'$ form an angle $\theta$, with $\alpha$ and $\alpha'$ two saddle connections in $S_{m,n}$ with consistent slopes  $d$ and $d'$.
Note that this is not the same as the angle between the directions $d$ and $d'$.
Recalling Definition \ref{def:teichmuller}, the set $B_\theta (d,d')$ is identified (via $\Psi$) to a subset of $\HH$, and we have:


\begin{Prop}\cite[\S 4]{BLM22}\label{prop:directions}
The set $\Psi(B_\theta(d,d')) \subset \HH$ is the banana neighbourhood
\[
\gamma_{d,d',r}= \{ z \in \HH: \mathrm{dist}_{\HH}(z, \gamma_{d,d'} ) \leq r \}
\]
where $\cosh r = \cfrac{1}{\sin \theta}$, and $\gamma_{d,d'}$ is the hyperbolic geodesic having endpoints $d,d' \in \partial \HH$ (see Figure \ref{fig:banana}).\newline

In particular, the locus of surfaces in $\Tmn$ where the saddle connections of respective consistent slopes $d$ and $d'$ are orthogonal is the hyperbolic geodesic $\gamma_{d,d'}$. \newline
\end{Prop}

\begin{figure}[h]
\begin{minipage}[l]{0.1\linewidth}
\begin{tikzpicture}[line cap=round,line join=round,>=triangle 45,x=1cm,y=1cm, scale=1.25]
\clip(-3,0) rectangle (2.6,4);
\draw [line width=2pt,color=gray,domain=-3.0190928292046935:2.895707170795306] plot(\x,{(-0-0*\x)/4});
\draw [shift={(0,0)},line width=1pt,color=black]  plot[domain=0:3.141592653589793,variable=\t]({1*2*cos(\t r)+0*2*sin(\t r)},{0*2*cos(\t r)+1*2*sin(\t r)});

\draw [->,line width=1pt,color=red] (0,0) -- (1,1);
\draw [line width=1pt,color=red] (0,0) -- (5,5);

\draw [->,line width=1pt,color=black] (0,0) -- (0,1);
\draw [line width=1pt,color=black] (0,0) -- (0,5);

\draw [->,line width=1pt,color=red] (0,0) -- (-1,1);
\draw [line width=1pt,color=red] (0,0) -- (-5,5);

\draw (-2.001161147327249,0) node[anchor=north west] {$\bar{u}$};
\draw (1.9857379400260755,0) node[anchor=north west] {$\bar{v}$};
\draw (0,3) node[anchor=north west] {$\gamma_{d,d'}=\gamma_{0,\infty}$};
\draw (-2.4,3) node[anchor=north west] {$\gamma_{d,d',r}$};
\draw (1,2.5) node[anchor=north west] {$\gamma_{d,d',r}$};

\draw (1,2.5) node[anchor=north west] {$\gamma_{d,d',r}$};

\node[draw,circle,inner sep=1.5pt,fill] at (0,2) {};
\node[draw,circle,inner sep=1.5pt,fill] at (1.4,1.4) {};
\draw[left] (0,2.2) node {\begin{scriptsize} $i$ \end{scriptsize}};
\draw[right] (1.4,1.4) node {\begin{scriptsize} $z=e^{i\theta}$ \end{scriptsize}};

\draw [shift={(0,0)},line width=1pt,fill=red,fill opacity=0.10000000149011612] (0,0) -- (45.13010235415601:0.23134810951760104) arc (45.13010235415601:0:0.23134810951760104) -- cycle;

\draw [shift={(0,0)},line width=1pt,fill=red,fill opacity=0.10000000149011612] (0,0) -- (-45.13010235415601:-0.23134810951760104) arc (-45.13010235415601:0:-0.23134810951760104) -- cycle;

\begin{scriptsize}
\draw[color=black] (-0.3,0.15) node {$\theta$};
\draw[color=black] (0.3,0.15) node {$\theta$};
\end{scriptsize}
\end{tikzpicture}
\end{minipage}
\hskip 60mm
\begin{minipage}[l]{0.1\linewidth}
\begin{tikzpicture}[line cap=round,line join=round,>=triangle 45,x=1cm,y=1cm, scale=1.25]
\clip(-3,0) rectangle (2.6,4);
\draw [shift={(0,0)},line width=1pt,color=black]  plot[domain=0:3.141592653589793,variable=\t]({1*2*cos(\t r)+0*2*sin(\t r)},{0*2*cos(\t r)+1*2*sin(\t r)});
\draw [shift={(0,1.5)},line width=1pt,color=red]  plot[domain=-0.6435011087932843:3.7850937623830774,variable=\t]({1*2.5*cos(\t r)+0*2.5*sin(\t r)},{0*2.5*cos(\t r)+1*2.5*sin(\t r)});
\draw [shift={(0,-1.5)},line width=1pt,color=red]  plot[domain=0.6435011087932844:2.498091544796509,variable=\t]({1*2.5*cos(\t r)+0*2.5*sin(\t r)},{0*2.5*cos(\t r)+1*2.5*sin(\t r)});
\draw [->,line width=1pt,color=red] (-2,0) -- (-1.4084909621903532,0.7886787170795292);
\draw [->,line width=1pt,color=black] (-2,0) -- (-2,1);
\draw [->,line width=1pt,color=red] (-2,0) -- (-2.6164052672750966,0.8218736897001288);
\draw [line width=2pt,color=gray,domain=-3.0190928292046935:2.895707170795306] plot(\x,{(-0-0*\x)/4});
\draw (-2.001161147327249,0) node[anchor=north west] {$\bar{u}$};
\draw (1.9857379400260755,0) node[anchor=north west] {$\bar{v}$};
\draw (-0.4973984354628422,2.5) node[anchor=north west] {$\gamma_{d,d'}$};
\draw (-0.3971475880052151,1.0268959582790067) node[anchor=north west] {$\gamma_{d,d',r}$};
\draw (-0.28918513689700126,3.9727285528031264) node[anchor=north west] {$\gamma_{d,d',r}$};

\draw [shift={(-2,0)},line width=1pt,fill=red,fill opacity=0.10000000149011612] (0,0) -- (53.13010235415601:0.23134810951760104) arc (53.13010235415601:0:0.23134810951760104) -- cycle;

\draw [shift={(-2,0)},line width=1pt,fill=red,fill opacity=0.10000000149011612] (0,0) -- (-53.13010235415601:-0.23134810951760104) arc (-53.13010235415601:0:-0.23134810951760104) -- cycle;

\begin{scriptsize}
\draw[color=black] (-2.3,0.15) node {$\theta$};
\draw[color=black] (-1.7,0.15) node {$\theta$};
\end{scriptsize}
\end{tikzpicture}
\end{minipage}
\caption{The set $\gamma_{d,d',r}$ for $\cosh r = \frac{1}{\sin \theta}$.}
\label{fig:banana}
\end{figure}
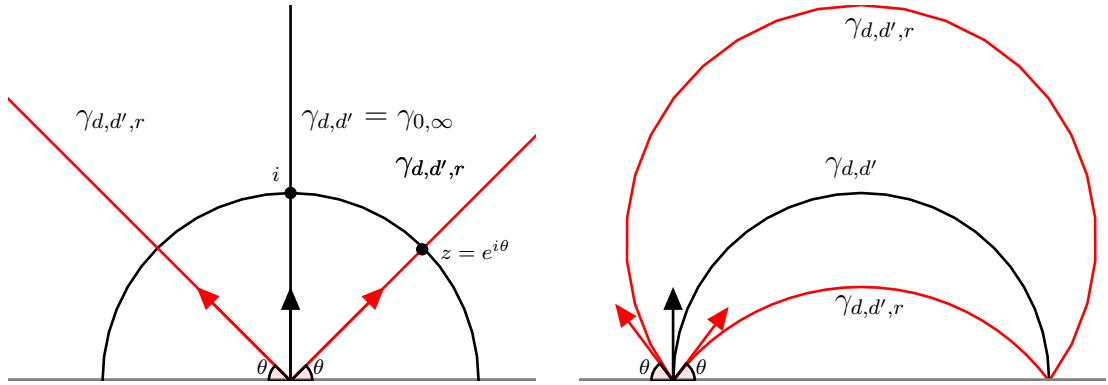

In the rest of the paper, we use the following notation.

\begin{Nota}\label{nota:sinus}
Given $X = M \cdot S_{m,n} \in \Tmn$, and $d,d' \in \partial \HH$ distinct, we define
$\theta(X,d,d') \in ]0, \frac{\pi}{2}]$ as the (unoriented) angle between saddle connections of respective consistent slopes $d$ and $d'$ in the surface $X$.
\end{Nota}

With this notation, we have by Proposition \ref{prop:directions}:
\begin{equation}\label{eq:lien_sin_cosh}
\sin \theta(X,d,d') = \cfrac{1}{\cosh(d_{\HH}(X ,\gamma_{d,d'}))}.
\end{equation}

\subsection{Another look at KVol}
The above geometric interpretation allows to rewrite KVol as a supremum over pairs of directions instead of saddle connections, by grouping together all pairs of saddle connections $(\alpha,\beta)$ having \coslos $(d,d')$. More precisely, we have:
\begin{Prop}\cite[Proposition 5.1]{BLM22}\label{prop:sup_directions}
Let $m,n \geq 2$ coprime, and let $\mathcal{P}$ be the set of consistent slopes of periodic saddle connections of $S_{m,n}$. Then for any surface $X \in \Tmn$ we have:
\begin{equation}\label{eq:reformulation}
\KVol(X) = \mathrm{Vol}(X) \cdot \sup_{
 \begin{scriptsize}
 \begin{array}{c}
d, d' \in \mathcal{P} \\
d \neq d'
\end{array}
\end{scriptsize}} K(d,d')\cdot \sin \theta(X,d,d'),
\end{equation}
where
$K(d,d') = \sup
\left\{
\frac{\mathrm{Int} (\alpha,\beta)}{\alpha\wedge \beta}, \textit{ with }
 \begin{array}{c}
\alpha\subset S_{m,n} \ {\mathrm s.\ c.\ with\ \coslo } d \\
\beta\subset S_{m,n} \ {\mathrm s.\ c.\ with\ \coslo } d'
\end{array}
\right\}
$.

\end{Prop}

The proof in \cite{BLM22} is for the case of the double regular $n$-gon, but it only uses the fact that they are dealing with Veech surfaces having a single singularity and no pairs of intersecting saddle connections in the same direction, which is also the case for $S_{m,n}$ for coprime $m,n$, as seen in Section \ref{sec:intersection_horizontal}.

\begin{Rema}\label{rk:Kdd'}
\begin{itemize}
\item The affine group of $S_{m,n}$ acts on $S_{m,n}$ preserving the intersection form, and the Veech group $\Gamma$ acts linearly on $\RR^2$, hence preserving the wedge product. In particular, $K(d,d') = K(g \cdot d, g \cdot d')$ for any element $g \in \Gamma$.
\item Since any periodic direction in $S_{m,n}$ is the image of the horizontal direction by an element of the Veech group $\Gamma$, we deduce that we can assume $d = \infty$.
\end{itemize}
\end{Rema}

From Equation~\eqref{eq:lien_sin_cosh}, the statement of Theorem \ref{theo:extension_teich} can be reformulated as follows:

\[ \forall X \in \Tmn,\text{ } \KVol(X) = K_{m,n} \cdot \sin \theta \left(X, \infty, \pm\cot \frac{\pi}{n}\right). \]
where $\sin \theta \left(X, \infty, \pm\cot \frac{\pi}{n}\right) := \max \left(\sin \theta \left(X, \infty, +\cot \frac{\pi}{n}\right), \sin \theta \left(X, \infty, -\cot \frac{\pi}{n}\right)\right).$
In fact, we will see that the constant $K_{m,n}$ of Theorem \ref{theo:extension_teich} is $\mathrm{Vol}(X) \cdot K(\infty,\pm \cot \frac{\pi}{n})$. More precisely, we can directly obtain Theorem~\ref{theo:extension_teich} from Proposition \ref{prop:sup_directions} and the following result: 

\begin{Theo}\label{theo:reformulation}
For every $X \in \Tmn$ and for every pair of distinct periodic directions $(d,d') \in (\partial \HH)^2$, we have:
\[ K(d,d') \sin \theta(X,d,d') \leq K\left(\infty,\pm \cot \frac{\pi}{n}\right) \sin \theta \left(X, \infty, \pm\cot \frac{\pi}{n}\right).\]
\end{Theo}

Before going on any further, let us remark that for $X = S_{m,n}$ (with $m,n \geq 3$ coprime), we can combine Theorem \ref{thm:mainBM} and Proposition \ref{prop:intersection_systoles} to obtain that the supremum in the definition of $\KVol(S_{m,n})$ is achieved by a horizontal curve and a curve having direction $\pm \cot(\frac{\pi}{n})$. Combining this result with Proposition \ref{prop:sup_directions}, we obtain:
\begin{Cor}\label{cor:Smn_K}
For every pair of distinct periodic directions $(d,d')$
, we have:
\[ K(d,d') \sin \theta(S_{m,n},d,d') \leq K\left(\infty,\pm \cot \frac{\pi}{n}\right)\sin \theta \left(S_{m,n}, \infty, \pm \cot \frac{\pi}{n}\right).\]
\end{Cor}

Notice that in this case $\theta\left(S_{m,n}, \infty,\cot \frac{\pi}{n} \right) = \theta\left(S_{m,n}, \infty, -\cot \frac{\pi}{n}\right)= \frac{\pi}{n}$. We will use Corollary \ref{cor:Smn_K} in the proof of Theorem~\ref{theo:reformulation}.

\subsection{Strategy for the proof of Theorem \ref{theo:reformulation}}
In order to show Theorem \ref{theo:reformulation}, we use the following result, proven in Section \ref{sec:sinus_comparison}, 

\begin{Theo}\label{theo:big_statement}
Let $a,b,c \in \RR$ with $0 < |b| \leq |c|$. Let $\mathcal{D}$ be the hyperbolic domain delimited by the geodesics $\gamma_{a,\infty}$, $\gamma_{a+b, \infty}$ and $\gamma_{a-c, a+c}$. Let also $X_0$ be the intersection point of $\gamma_{a+b,\infty}$ and $\gamma_{a-c, a+c}$, which exist by the assumption $|b| \leq |c|$. Assume that the angle between the geodesics $\gamma_{a+b,\infty}$ and $\gamma_{a+c,a-c}$ is of the form $\frac{\pi}{n}$, for $n \geq 2$ so that the group generated by the reflections along the geodesics $\gamma_{a+b,\infty}$ and $\gamma_{a-c, a+c}$ is a dihedral group. Let $\mathcal{P} \subset \partial \HH$ containing at least $\infty,a$, and $a+2b$, and $K(\cdot,\cdot): \mathcal{P}\times \mathcal{P} \backslash \Delta \to \RR$ be a map which is symmetric with respect to its two coordinates ($\Delta$ is the diagonal $\{(x,x), x \in \mathcal{P}\}$.) and invariant under the diagonal action of the dihedral group generated by the reflections along the geodesics $\gamma_{a+b,\infty}$ and $\gamma_{a-c, a+c}$. \newline

Assume that 
\begin{enumerate}
\item[(H1)] For any distinct $d,d' \in \mathcal{P}$,
 \[ K(d,d') \sin \theta(X_0, d,d') \leq K(\infty, a) \sin \theta(X_0,\infty,a) \]
\item[(H2)] For any distinct $d,d' \in \mathcal{P}$, 
\[ K(d,d') \leq K(\infty,a) \]
\item[(H3)] For any distinct $d,d' \in \mathcal{P}$ such that $\gamma_{d,d'}$ intersects $\mathcal{D}$ and $(d,d') \neq (\infty,a)$, 
\[ K(d,d') \leq \min \left( \sin \left(\frac{\pi}{4} \right), \frac{\sin \theta(X_0, \infty, a)}{\sin \theta(X_0, a, a+2b)} \right) K(\infty, a) \]
\end{enumerate}
Then, for any $X \in \mathcal{D}$ and any distinct $d,d' \in \mathcal{P}$
 \[ K(d,d') \sin \theta (X, d,d') \leq K(\infty, a) \sin \theta(X,\infty,a). \]
\end{Theo}

\begin{Nota}
Although it is more compact to state (H3) as in Theorem \ref{theo:big_statement}, it will be convenient for the proof to separate it into the following two conditions
\begin{equation}\label{H31}
K(d,d') \leq \sin \left(\frac{\pi}{4} \right) K(\infty, a)  \tag{$H3^+$}
\end{equation}
and 
\begin{equation}\label{H32}
K(d,d') \leq \frac{\sin \theta(X_0, \infty, a)}{\sin \theta(X_0, a, a+2b)} K(\infty, a) \tag{$H3^-$}
\end{equation}
\end{Nota}

In particular, it suffices to show that our function $K(\cdot,\cdot)$ from Proposition \ref{prop:sup_directions} defined over the set $\mathcal{P}$ of periodic directions, satisfies the hypotheses of Theorem \ref{theo:big_statement} in the fundamental domain described in Section \ref{sec:Teichmuller_disk_BM}. 
We will split the fundamental domain in four subdomains $\mathcal{D}_1$, $\mathcal{D}_2$, $\mathcal{D}_3$, separated by the vertical lines at 0 and at $\pm \cot \frac{\pi}{n}$ and ordered from left to right, as in Figure \ref{fig:fond_domain}. We will then use Theorem \ref{theo:big_statement} on each of these four domains.
\begin{Rema}\label{rk:parameters_Di}
On the domain $\mathcal{D}_1$, the corresponding parameters $a,b$ and $c$ are given by 
\begin{align*}
    a &= - \cot(\pi/n), &b&=-\frac{\cos(\pi/m)}{\sin(\pi/n)} &&\text{ and } &c&=\sin(\pi/n).
\end{align*}

Notice that the point $X_0$ represents the surface $S_{n,m}$.\newline
Similarly, on the domains $\mathcal{D}_2, \mathcal{D}_3$ and $\mathcal{D}_4$ the parameters are given by
\begin{itemize}
\item $(a,b,c) = \left(- \cot(\pi/n),+\cot(\pi/n),+\sin(\pi/n)\right)$ on $\mathcal{D}_2$
\item $(a,b,c) = \left(+ \cot(\pi/n),-\cot(\pi/n),+\sin(\pi/n)\right)$ on $\mathcal{D}_3$
\item $(a,b,c) = \left(+ \cot(\pi/n),+\frac{\cos(\pi/m)}{\sin(\pi/n)},+\sin(\pi/n)\right)$ on $\mathcal{D}_4$.
\end{itemize}
\end{Rema}

Notice that $(H1)$ comes directly from Corollary \ref{cor:Smn_K}. 
We show $(H2)$, $(H3)$ in the next section.

\begin{Rema}\label{rk:modified_big_statement}
In some cases, we will be able to simplify our argument by substituting conditions $(H1)$ and \HTD with the following. 
\begin{enumerate}
\item[$(H\star)$] $\theta(X_0, \infty, a) \geq \frac{\pi}{4}$.
\end{enumerate}
This is because $(H\star)$ and \HTU together imply $(H1)$ and \HTD. 
In general, $(H\star)$ does not hold but we can see from \S \ref{sec:Teichmuller_disk_BM} that it is the case on $\mathcal{D}_2$ and $\mathcal{D}_3$ if $n \leq 4$ and (by symmetry) on $\mathcal{D}_1$ and $\mathcal{D}_4$ if $m \leq 4$. 
We will use this simplification for $m,n \leq 4$. 
\end{Rema}

\begin{figure}
\center
\definecolor{qqqqff}{rgb}{0,0,1}
\definecolor{qqwuqq}{rgb}{0,0.39215686274509803,0}
\definecolor{ccqqqq}{rgb}{0.8,0,0}
\definecolor{uuuuuu}{rgb}{0.26666666666666666,0.26666666666666666,0.26666666666666666}
\begin{tikzpicture}[line cap=round,line join=round,>=triangle 45,x=1cm,y=1cm]
\clip(-6,-1) rectangle (6,8);
\fill[line width=0.5pt,color=uuuuuu,fill=uuuuuu,fill opacity=0.10000000149011612] (-4.528276847984907,0) -- (4.528276847984907,0) -- (4.528276847984907,9) -- (-4.528276847984907,9) -- cycle;
\draw [shift={(2.414213562373095,0)},line width=1pt,color=white,fill=white,fill opacity=1]  (0,0) --  plot[domain=0:3.141592653589793,variable=\t]({1*2.613125929752753*cos(\t r)+0*2.613125929752753*sin(\t r)},{0*2.613125929752753*cos(\t r)+1*2.613125929752753*sin(\t r)}) -- cycle ;
\draw [shift={(-2.414213562373095,0)},line width=1pt,color=white,fill=white,fill opacity=1]  (0,0) --  plot[domain=0:3.141592653589793,variable=\t]({1*2.613125929752753*cos(\t r)+0*2.613125929752753*sin(\t r)},{0*2.613125929752753*cos(\t r)+1*2.613125929752753*sin(\t r)}) -- cycle ;
\draw [line width=1pt] (4.528276847984907,0) -- (4.528276847984907,9.78);
\draw [line width=1pt, dash pattern=on 3pt off 3pt] (0,0) -- (0,10);
\draw [line width=1pt] (-4.528276847984907,0) -- (-4.528276847984907,9.78);
\draw [shift={(-2.414213562373095,0)},line width=1pt]  plot[domain=0.3926990816987242:2.5132741228718345,variable=\t]({1*2.613125929752753*cos(\t r)+0*2.613125929752753*sin(\t r)},{0*2.613125929752753*cos(\t r)+1*2.613125929752753*sin(\t r)});
\draw [shift={(2.414213562373095,0)},line width=1pt]  plot[domain=0.6283185307179587:2.748893571891069,variable=\t]({1*2.613125929752753*cos(\t r)+0*2.613125929752753*sin(\t r)},{0*2.613125929752753*cos(\t r)+1*2.613125929752753*sin(\t r)});
\draw [line width=1pt,domain=-6.411111111111112:11.148888888888898] plot(\x,{(-0-0*\x)/6.942490410358002});
\draw (0,0) node[below] {$0$};
\draw (-0.05,1) node[anchor=north west] {$S_{m,n}$};
\draw (4.6,1.5) node[anchor=north west] {$S_{n,m}$};
\draw (-5.7,1.5) node[anchor=north west] {$S_{n,m}$};
\draw [line width=1pt, dash pattern= on 3pt off 3pt] (-2.414213562373095,0) -- (-2.414213562373095,9.78);
\draw [line width=1pt, dash pattern= on 3pt off 3pt] (2.414213562373095,0) -- (2.414213562373095,9.78);
\draw (2.414213562373095,0.0) node[below] {$\cot(\frac{\pi}{n})$};
\draw (-2.414213562373095,0.0) node[below] {$-\cot(\frac{\pi}{n})$};
\draw (4.528276847984907,0) node[below] {$\frac{\cos(\frac{\pi}{n}) +\cos(\frac{\pi}{m})}{\sin(\frac{\pi}{n})}$};
\draw (-3.5,5) node[above] {$\mathcal{D}_1$};
\draw (-1.25,5) node[above] {$\mathcal{D}_2$};
\draw (1.25,5) node[above] {$\mathcal{D}_3$};
\draw (3.5,5) node[above] {$\mathcal{D}_4$};
\begin{scriptsize}
\draw [fill=black] (0,1) circle (2.5pt);
\draw [fill=black] (4.528276847984907,1.5359568838917255) circle (2.5pt);
\draw [fill=black] (-4.528276847984907,1.5359568838917255) circle (2.5pt);
\draw [color=black] (0,0)-- ++(-2.5pt,0 pt) -- ++(5pt,0 pt) ++(-2.5pt,-2.5pt) -- ++(0 pt,5pt);
\end{scriptsize}
\end{tikzpicture}
\caption{The fundamental domain $\Tmn$ of the Teichm\"uller disk of $S_{m,n}$ and the four domains $\mathcal{D}_1$, $\mathcal{D}_2$, $\mathcal{D}_3$ and $\mathcal{D}_4$.}
\label{fig:fond_domain}
\end{figure}
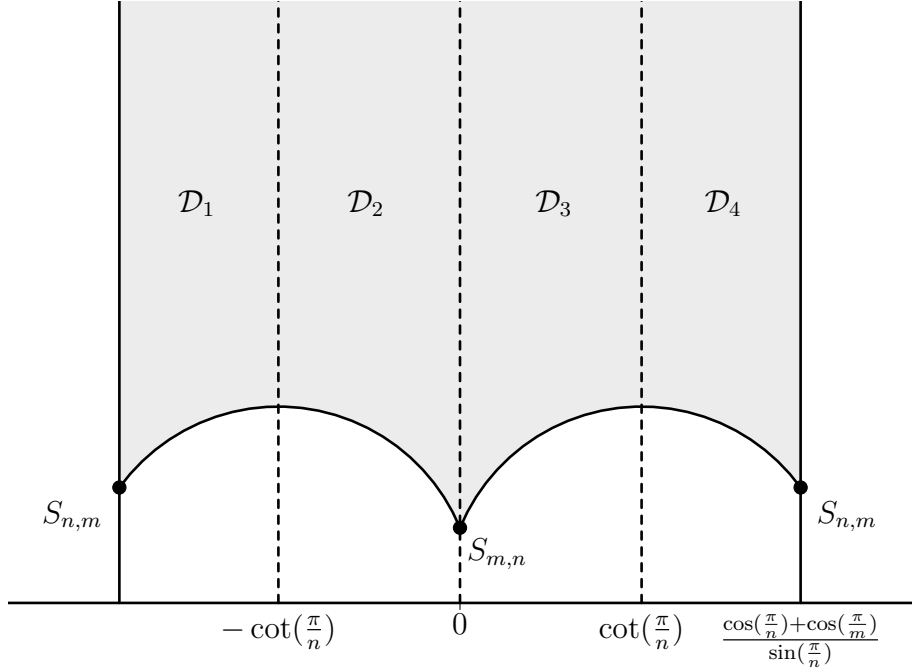


\section{Study of \texorpdfstring{$K(d,d')$}{Kdd}}\label{sec:study_cylinder}
In this section we give estimates on $K(d,d')$ which ensures that assumptions $(H2)$ and \HTU are satisfied on each of the domains $\mathcal{D}_1$ to $\mathcal{D}_4$, as well as \HTD when $n, m \geq 5$. 
More precisely:

\begin{Prop}\label{prop:etude_K}
Let $m,n$ coprime, with $3 \leq n < m$. Then,
\begin{itemize}
    \item[(i)] $K(\infty, \pm \cot(\pi/n)) = \frac{1}{\sin(\pi / n) l_0^2}$. 
    \item[(ii)] $K(\infty, \pm \cot(2\pi/n)) \leq \frac{1}{2\cos(\pi/n) \sin(\pi / n) l_0^2}$. 
\end{itemize}
Moreover, for every pair of distinct periodic directions $(d,d')$ which is not the image of one of the preceding pairs of directions by the diagonal action of the Veech group,
\begin{equation} \label{eq:boundKdd'}
    K(d,d') \leq \frac{1}{2\cos(\pi/m) \sin(\pi / n) l_0^2}.
\end{equation}  
In particular, we have
\begin{itemize}
    \item[(iii)] $K\left(\infty, \pm \cfrac{1+2 \cos(\pi/n)\cos(\pi/m)}{2\sin(\pi/n) \cos(\pi/m)}\right) \leq \frac{1}{2\cos(\pi/m) \sin(\pi / n) l_0^2}$. 
\end{itemize}
\end{Prop}

\begin{Rema}\label{rema:cas_particuliers_etude_Kdd}
\begin{itemize}
\item For $n=3$ we have $\cot(2\pi/3) = -\cot(\pi/3)$ and $2 \cos(\pi/3) = 1$: cases $(i)$ and $(ii)$ are merged.
\item For $n=4$, notice that $\cot(2\pi/n) = 0$. Interestingly, $K(\infty, \cot(2\pi/n))$ is realised by horizontal and vertical sides of $P(0)$ if $m \equiv 1 \mod 4$, while such sides do not intersect if $m \equiv 3 \mod 4$. In the latter case, it is the diagonals of $P(m-1)$, which are intersecting twice by Proposition \ref{prop:intersection_diagonals}, that achieve the best ratio for $K(\infty, \cot(2\pi/n))$. 
\item In cases $(ii)$ and $(iii)$, instead of having an equality we only give an upper bound. This is because the equality does not hold in general. More precisely, equality holds in case $(ii)$ if and only if a horizontal side of $P(0)$ intersects a side of $P(0)$ (or $P(m-1)$) in direction $d' = \pm \cot(2\pi/n)$. Similarly, equality holds in case $(iii)$ if and only if a horizontal side of $P(0)$ intersects a small diagonal of $P(1)$ (or $P(m-2)$) in direction $d' =\cfrac{1+2 \cos(\pi/n)\cos(\pi/m)}{2\sin(\pi/n) \cos(\pi/m)}$. However, the upper bound is sufficient to show that the hypotheses of Theorem \ref{theo:big_statement} are satisfied.
\end{itemize}
\end{Rema}

Before proving the proposition, we start with a few preliminaries about the horizontal cylinder decomposition of $S_{m,n}$.
\begin{Lem}[Length of horizontal saddle connections]\label{lem:lengths}
The three shortest lengths of horizontal saddle connections of $S_{m,n}$ (with $3 \leq n < m$) are 
\begin{align*}
l_0 = \sin \left(\frac{\pi}{m}\right), & &  l_1 := 2 \cos\left(\frac{\pi}{n}\right) l_0 & &\text{ and } & & l_2 := 2 \cos\left(\frac{\pi}{m}\right) l_0.
\end{align*}
Further,
\begin{itemize}
\item[-] $l_0$ is the length of the systole of $S_{m,n}$ which is realised by sides of $P(0)$ and $P(m-1)$. If $n$ is even, then there are two horizontal saddle connections of length $l_0$, and both lie in $P(0)\cap P(1)$. If $n$ is odd, there is one such saddle connection in $P(0)\cap P(1)$ and one in $P(m-1)\cap P(m-2)$.
\item[-] $l_1$ is the length of the smallest diagonal of $P(0)$ and $P(m-1)$, except when $n=3$, in which case there are no such diagonals, but $l_0=l_1$. If $n=4$ (then $m$ is odd and) there is only one horizontal saddle connection of length $l_1$, the diagonal of the square $P(m-1)$. 
In all other cases, there are two horizontal saddle connections of length $l_1$: if $n$ is even, both lie in $P(m-1)$, while if $n$ is odd, then one lie inside $P(0)$ and the other in $P(m-1)$.
\item[-] $l_2$ is the length of the long sides of $P(1)$ and $P(m-2)$. There are two such horizontal sides: if $n$ is even, then both are sides of $P(m-2)$, while if $n$ is odd, then there is one such side in $P(m-2)$ and another in $P(1)$.
\end{itemize}
\end{Lem}
\begin{proof}
First, the horizontal saddle connections contained in $P(0)$ or $P(m-1)$ correspond by construction to sides or diagonals of a regular $n$-gon of side length $\sin(\frac{\pi}{m})$, and hence have length $\sin(\frac{\pi}{m})$, $2 \cos(\frac{\pi}{n}) \sin(\frac{\pi}{m})$ (if $n > 3$), $\left(1 + 2 \cos(\frac{2\pi}{n})\right) \sin(\frac{\pi}{m})$ (if $n \geq 6$), and so on. Notice that $\left(1 + 2 \cos(\frac{2\pi}{n})\right) \sin(\frac{\pi}{m}) \geq 2 \sin \left( \frac{\pi}{m} \right)$ for $n \geq 6$, which is already greater than $l_2$.\newline
Next, 
the smallest horizontal saddle connections contained in $P(1)$ or $P(m-2)$ are the long sides having length $l_2 = \sin(\frac{2\pi}{m}) = 2 \cos(\frac{\pi}{m})\sin(\frac{\pi}{m})$. \newline
Finally, all other horizontal saddle connections (which are not contained in $P(0)$, $P(1)$, $P(m-2)$ or $P(m-1)$) are longer.
\end{proof}

\begin{figure}
\hskip -25mm
\definecolor{ccqqqq}{rgb}{0.8,0,0}
\begin{tikzpicture}[line cap=round,line join=round,>=triangle 45,x=1cm,y=1cm]
\clip(-9.538558698508323,-5) rectangle (8.865749847247395,3);
\draw [line width=2pt,color=ccqqqq] (-5,0)-- (-4,0);
\draw [line width=2pt] (-4,0)-- (-4,1);
\draw [line width=2pt,color=ccqqqq] (-4,1)-- (-5,1);
\draw [line width=2pt] (-5,1)-- (-5,0);
\draw [line width=2pt] (-4,0)-- (-2.8558771943646315,-1.144122805635369);
\draw [line width=2pt] (-2.8558771943646315,-1.144122805635369)-- (-1.8558771943646313,-1.144122805635369);
\draw [line width=2pt] (-1.8558771943646313,-1.144122805635369)-- (-0.7117543887292628,0);
\draw [line width=2pt] (-0.7117543887292628,0)-- (-0.7117543887292628,1);
\draw [line width=2pt] (-0.7117543887292628,1)-- (-1.8558771943646315,2.1441228056353685);
\draw [line width=2pt] (-1.8558771943646315,2.1441228056353685)-- (-2.8558771943646315,2.1441228056353685);
\draw [line width=2pt] (-2.8558771943646315,2.1441228056353685)-- (-4,1);
\draw [line width=2pt] (-0.7117543887292628,0)-- (-1.8558771943646313,-1.144122805635369);
\draw [line width=2pt] (-1.8558771943646313,-1.144122805635369)-- (-1.855877194364631,-2.7621567943852634);
\draw [line width=2pt] (-1.855877194364631,-2.7621567943852634)-- (-0.7117543887292628,-3.906279600020632);
\draw [line width=2pt,color=ccqqqq] (-0.7117543887292628,-3.906279600020632)-- (0.9062796000206319,-3.906279600020632);
\draw [line width=2pt] (0.9062796000206319,-3.906279600020632)-- (2.0504024056560004,-2.7621567943852634);
\draw [line width=2pt] (2.0504024056560004,-2.7621567943852634)-- (2.0504024056560004,-1.144122805635369);
\draw [line width=2pt] (2.0504024056560004,-1.144122805635369)-- (0.9062796000206325,0);
\draw [line width=2pt,color=ccqqqq] (0.9062796000206325,0)-- (-0.7117543887292628,0);
\draw [line width=2pt] (2.0504024056560004,-1.144122805635369)-- (2.757509186842548,-0.4370160244488215);
\draw [line width=2pt,color=ccqqqq] (2.757509186842548,-0.4370160244488215)-- (4.375543175592443,-0.4370160244488215);
\draw [line width=2pt] (4.375543175592443,-0.4370160244488215)-- (5.08264995677899,-1.144122805635369);
\draw [line width=2pt] (5.08264995677899,-1.144122805635369)-- (5.08264995677899,-2.7621567943852634);
\draw [line width=2pt] (5.08264995677899,-2.7621567943852634)-- (4.375543175592442,-3.469263575571811);
\draw [line width=2pt,color=ccqqqq] (4.375543175592442,-3.469263575571811)-- (2.7575091868425474,-3.469263575571811);
\draw [line width=2pt] (2.7575091868425474,-3.469263575571811)-- (2.0504024056560004,-2.7621567943852634);
\draw [line width=2pt] (5.08264995677899,-1.144122805635369)-- (5.789756737965539,-0.43701602444882215);
\draw [line width=2pt] (5.789756737965539,-0.43701602444882215)-- (5.082649956778991,0.27009075673772615);
\draw [line width=2pt] (5.082649956778991,0.27009075673772615)-- (4.375543175592443,-0.4370160244488215);
\draw [line width=2pt,color=ccqqqq] (-4,1)-- (-0.7117543887292628,1);
\draw [line width=2pt,color=ccqqqq] (-4,0)-- (-0.7117543887292628,0);
\draw [line width=2pt,color=ccqqqq] (-1.8558771943646313,-1.144122805635369)-- (2.0504024056560004,-1.144122805635369);
\draw [line width=2pt,color=ccqqqq] (-1.855877194364631,-2.7621567943852634)-- (2.0504024056560004,-2.7621567943852634);
\draw [line width=2pt,color=ccqqqq] (2.0504024056560004,-2.7621567943852634)-- (5.08264995677899,-2.7621567943852634);
\draw [line width=2pt,color=ccqqqq] (2.0504024056560004,-1.144122805635369)-- (5.08264995677899,-1.144122805635369);
\draw [line width=2pt,color=ccqqqq] (4.375543175592443,-0.4370160244488215)-- (5.789756737965539,-0.43701602444882215);
\draw (-5.2,2.2) node[anchor=north west] {$\sin(\frac{\pi}{m})$};
\draw (2.1,-3.8) node[anchor=north west] {$2 \cos(\frac{\pi}{m}) \sin(\frac{\pi}{m})$};
\draw (3.7,1.4) node[anchor=north west] {$2 \cos(\frac{\pi}{n}) \sin(\frac{\pi}{m})$};
\draw (-9.333610941874072,0.9) node[anchor=north west] {$2 \cos\left(\frac{\pi}{n}\right) \sin(\frac{\pi}{m}) \sin(\frac{\pi}{n})$};
\draw (6.1,-0.37) node[anchor=north west] {$\sin(\frac{\pi}{m}) \sin(\frac{\pi}{n})$};
\draw (-0.3364044256304639,2) node[anchor=north west] {$2 \cos(\frac{\pi}{m}) \sin(\frac{\pi}{m}) \sin(\frac{\pi}{n})$};
\draw [to-to,line width=1.2pt] (-0.500362630937868,2.1923693915856215)-- (-0.479867855274443,1.0036724031069697);
\draw [to-to,line width=1.2pt] (-5.3,1)-- (-5.3,0);
\draw [to-to,line width=1.2pt] (5.996481254367881,-0.4719514446596325)-- (5.996481254367881,-1.0458051632355334);
\draw [to-to,line width=1.2pt] (4.479867855274425,0.47080823585791887)-- (5.66856484375308,0.47080823585791887);
\draw [to-to,line width=1.2pt] (-3.9844744937201306,1.4135679163754702)-- (-4.988718501227959,1.4135679163754702);
\draw [to-to,line width=1.2pt] (2.778801475210144,-3.710125999480787)-- (4.3773939769573,-3.710125999480787);
\end{tikzpicture}
\caption{Lengths of horizontal saddle connections and height of the cylinders for $n=4$ and $m=5$.}
\end{figure}
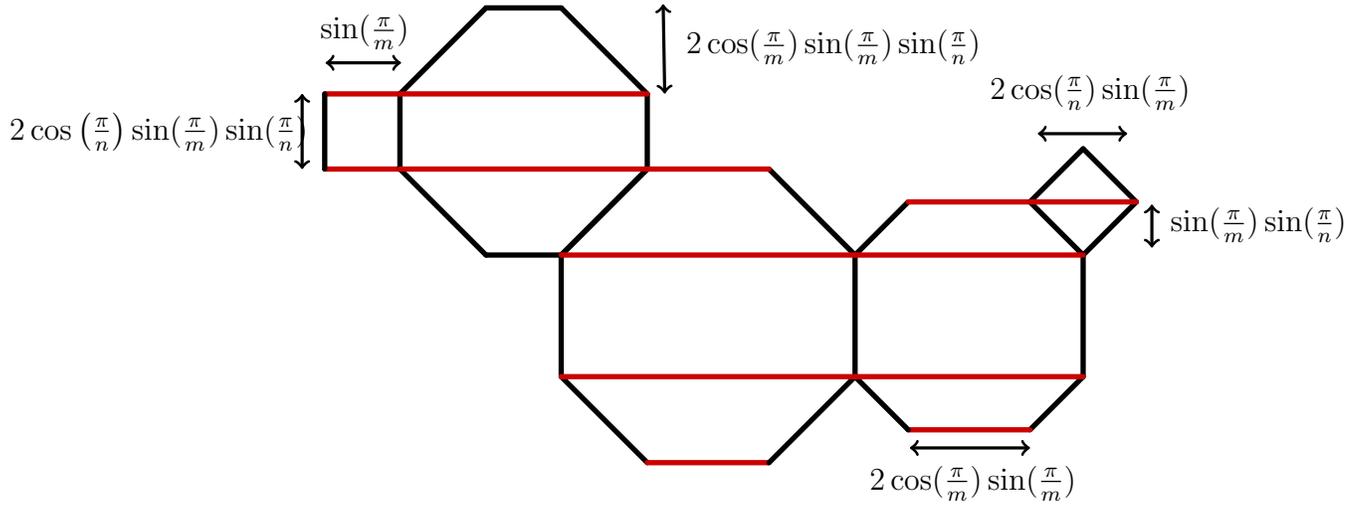

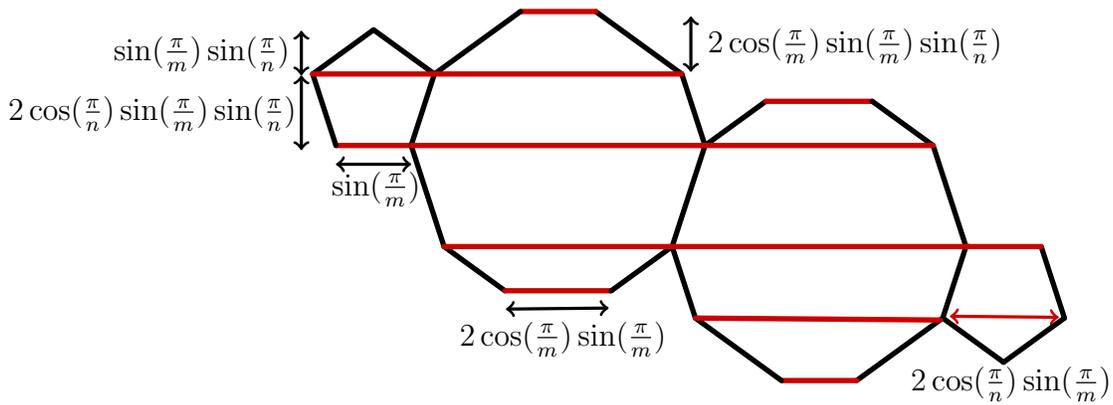
\begin{figure}
\hskip -15mm
\definecolor{ccqqqq}{rgb}{0.8,0,0}
\begin{tikzpicture}[line cap=round,line join=round,>=triangle 45,x=1cm,y=1cm]
\clip(-9.4,-4) rectangle (5.5,2);
\draw (-5.2,-0.2) node[anchor=north west] {$\sin(\frac{\pi}{m})$};
\draw (-3.5,-2.2) node[anchor=north west] {$2 \cos(\frac{\pi}{m}) \sin(\frac{\pi}{m})$};
\draw (2.5,-2.8) node[anchor=north west] {$2 \cos(\frac{\pi}{n}) \sin(\frac{\pi}{m})$};
\draw (-9.5,0.8) node[anchor=north west] {$2 \cos(\frac{\pi}{n}) \sin(\frac{\pi}{m}) \sin(\frac{\pi}{n})$};
\draw (-8.1,1.6) node[anchor=north west] {$\sin(\frac{\pi}{m}) \sin(\frac{\pi}{n})$};
\draw (-0.2,1.7) node[anchor=north west] {$2 \cos(\frac{\pi}{m}) \sin(\frac{\pi}{m}) \sin(\frac{\pi}{n})$};
\draw [to-to,line width=1.2pt] (-5.460098341486729,0.9319406882849823)-- (-5.460098341486736,-0.06205593139112908);
\draw [to-to,line width=1.2pt] (-0.2817516905279928,1.7175737553829395)-- (-0.2817516905279928,0.9797618314996394);
\draw [to-to,line width=1.2pt,color=ccqqqq] (3.1613706209274195,-2.272075907097128)-- (4.609668101142792,-2.285739090872745);
\draw [to-to,line width=1.2pt] (-4,-0.25)-- (-5,-0.25);
\draw [to-to,line width=1.2pt] (-2.7411247701390025,-2.1764336206678117)-- (-1.3748063925773304,-2.162770436892195);
\draw [line width=2pt,color=ccqqqq] (-5,0)-- (-4,0);
\draw [line width=2pt] (-4,0)-- (-3.6909830056250525,0.9510565162951532);
\draw [line width=2pt] (-3.6909830056250525,0.9510565162951532)-- (-4.5,1.5388417685876266);
\draw [line width=2pt] (-4.5,1.5388417685876266)-- (-5.3090169943749475,0.9510565162951536);
\draw [line width=2pt] (-5.3090169943749475,0.9510565162951536)-- (-5,0);
\draw [line width=2pt] (-4,0)-- (-3.562983975551179,-1.3449970239279145);
\draw [line width=2pt] (-3.562983975551179,-1.3449970239279145)-- (-2.7539669811762315,-1.932782276220388);
\draw [line width=2pt,color=ccqqqq] (-2.7539669811762315,-1.932782276220388)-- (-1.3397534188031364,-1.932782276220388);
\draw [line width=2pt] (-1.3397534188031364,-1.932782276220388)-- (-0.5307364244281889,-1.344997023927915);
\draw [line width=2pt] (-0.5307364244281889,-1.344997023927915)-- (-0.0937203999793681,0);
\draw [line width=2pt] (-0.0937203999793681,0)-- (-0.40273739435431555,0.951056516295153);
\draw [line width=2pt] (-0.40273739435431555,0.951056516295153)-- (-1.546860199989684,1.7823103918500598);
\draw [line width=2pt,color=ccqqqq] (-1.546860199989684,1.7823103918500598)-- (-2.546860199989684,1.7823103918500598);
\draw [line width=2pt] (-2.546860199989684,1.7823103918500598)-- (-3.6909830056250525,0.9510565162951532);
\draw [line width=2pt] (-0.0937203999793681,0)-- (0.7152965943955794,0.5877852522924722);
\draw [line width=2pt,color=ccqqqq] (0.7152965943955794,0.5877852522924722)-- (2.1295101567686743,0.5877852522924722);
\draw [line width=2pt] (2.1295101567686743,0.5877852522924722)-- (2.9385271511436217,0);
\draw [line width=2pt] (2.9385271511436217,0)-- (3.3755431755924428,-1.3449970239279156);
\draw [line width=2pt] (3.3755431755924428,-1.3449970239279156)-- (3.0665261812174953,-2.2960535402230686);
\draw [line width=2pt] (3.0665261812174953,-2.2960535402230686)-- (1.9224033755821268,-3.127307415777975);
\draw [line width=2pt,color=ccqqqq] (1.9224033755821268,-3.127307415777975)-- (0.9224033755821268,-3.127307415777975);
\draw [line width=2pt] (0.9224033755821268,-3.127307415777975)-- (-0.22171943005324168,-2.2960535402230686);
\draw [line width=2pt] (-0.22171943005324168,-2.2960535402230686)-- (-0.5307364244281889,-1.344997023927915);
\draw [line width=2pt] (3.0665261812174953,-2.2960535402230686)-- (3.875543175592442,-2.8838387925155415);
\draw [line width=2pt] (3.875543175592442,-2.8838387925155415)-- (4.684560169967389,-2.296053540223069);
\draw [line width=2pt] (4.684560169967389,-2.296053540223069)-- (4.375543175592442,-1.3449970239279159);
\draw [line width=2pt,color=ccqqqq] (4.375543175592442,-1.3449970239279159)-- (3.3755431755924428,-1.3449970239279156);
\draw [to-to,line width=1.2pt] (-5.460098341486729,1.5262891825243075)-- (-5.460098341486729,0.9319406882849823);
\draw [line width=2pt,color=ccqqqq] (-5.309016994374948,0.9510565162951541)-- (-3.6909830056250525,0.9510565162951535);
\draw [line width=2pt,color=ccqqqq] (-3.6909830056250525,0.9510565162951535)-- (-0.40273739435431555,0.951056516295153);
\draw [line width=2pt,color=ccqqqq] (-4,0)-- (-0.09372039997936808,0);
\draw [line width=2pt,color=ccqqqq] (-3.562983975551179,-1.3449970239279143)-- (-0.5307364244281888,-1.3449970239279145);
\draw [line width=2pt,color=ccqqqq] (-0.5307364244281888,-1.3449970239279145)-- (3.420971112664135,-1.3449970239279163);
\draw [line width=2pt,color=ccqqqq] (3.0310863262910455,-2.3218021020134745)-- (-0.22171943005324196,-2.296053540223068);
\draw [line width=2pt,color=ccqqqq] (-0.09372039997936808,0)-- (2.9385271511436217,0);
\end{tikzpicture}
\caption{Lengths of horizontal saddle connections and height of the cylinders for $n=5$ and $m=4$.}
\end{figure}

Similarly, we can compute the height of the smallest horizontal cylinders:

\begin{Lem}[Height of horizontal cylinders] \label{lem:heights}
The three smallest height of horizontal cylinders of $S_{m,n}$ (with $3 \leq n < m$) are
\begin{align*}
   h_0:= \sin(\pi / n) l_0,  & & h_1 := 2 \cos(\pi / n) h_0 & & \text{and} && h_2 := 2 \cos(\pi / m) h_0.
\end{align*}
Further,
\begin{itemize}
\item[-] There are two horizontal cylinders of height $h_0$. If $m$ is odd and $n$ is even then both are contained in $P(m-1) \cup P(m-2)$, otherwise there is one such cylinder inside $P(m-1) \cup (m-2)$ and another contained in $P(0) \cup P(1)$.
\item[-] If $n \neq 4$, then there are two horizontal cylinders of height $h_1$: if $m$ is odd and $n$ is even then both are contained in $P(0) \cup P(1)$, otherwise there is one such cylinder inside $P(m-1) \cup P(m-2)$ and another contained in $P(0) \cup P(1)$. Moreover, if $n=3$, then $h_0=h_1$. If $n=4$, then there is a single such horizontal cylinder, contained in $P(0) \cup P(1)$.
\item[-] $h_2$ is the height of the smallest horizontal cylinders which do not not intersect $P(0) \cup P(m-1)$. There are two such cylinders: if $n$ is even then both cylinders are contained in $P(1) \cup P(2)$, otherwise if $n$ is odd, then there is one such cylinder inside $P(1) \cup P(2)$ and another inside $P(m-2) \cup P(m-3)$.
\end{itemize}
\end{Lem}
\begin{proof}
First, the horizontal cylinders intersecting $P(0)$ or $P(m-1)$ have height $\sin\left(\frac{ k \pi}{n}\right) \sin\left(\frac{\pi}{m}\right)$ for some $1 \leq k \leq n-1$. For $k \geq 3$, this is already higher than $h_2$.\newline
Next, the horizontal cylinders intersecting $P(1)$ (resp. $P(m-2)$) but not $P(0)$ (resp. $P(m-1)$) have height 
$\sin\left(\frac{k \pi}{n}\right) \sin\left(\frac{2 \pi}{m}\right)$ for some $k$ as above. 
The smallest such cylinder $(k=1)$ has height $h_2$, and all other cylinders are bigger.\newline
Finally, all other horizontal cylinders, which are not intersecting $P(0)$, $P(1)$, $P(m-2)$ or $P(m-1)$, have height greater than $\sin\left(\frac{k \pi}{n}\right) \sin\left(\frac{3\pi}{m}\right)$ and hence greater than $h_2$.
\end{proof}

We are now able to prove Proposition \ref{prop:etude_K}.

\begin{proof}[Proof of Proposition \ref{prop:etude_K}.] 
The strategy here is to look at $K(d,d')$ and show that \eqref{eq:boundKdd'} holds, unless $(d,d')$ is as in cases $(i)$ and $(ii)$, in which case we can do a little better. 

First, recall from Remark \ref{rk:Kdd'} that to study $K(d,d')$ we can assume $d=\infty$ and work on $S_{m,n}$ directly. 
Let $d' \neq \infty$ and let $\alpha$ and $\beta$ be two saddle connections on $S_{m,n}$ in respective directions $d=\infty$ and $d'$.\newline

{\color{red}}

\textbf{Case 1.} \underline{If there are no non singular intersection between $\alpha$ and $\beta$}. In this case $\Int(\alpha,\beta) \leq 1$ with at most a single singular intersection. In particular:
\[ \frac{\Int(\alpha, \beta)}{\alpha \wedge \beta} \leq \frac{1}{l_v(\beta) l(\alpha)} \]
where $l_v(\beta)$ denotes the vertical length of $\beta$. This length is minimal if $\beta$ crosses once a horizontal cylinder with small height. From Lemma \ref{lem:heights}, we deduce:
\begin{itemize}
\item[a.]\label{case1.h0} If $\beta$ is contained in a cylinder of height $h_0$, then 
up to a horizontal twist, $\beta$ is a saddle connection in direction $d'=\pm \cot(\pi/n)$. Further,
\[ \frac{\Int(\alpha, \beta)}{\alpha \wedge \beta} \leq \frac{1}{l_0h_0} = \frac{1}{\sin(\pi/n)l_0^2}\]
with equality if $\Int(\alpha,\beta)=1$ and $\alpha$ is a systole. As seen in Proposition \ref{prop:intersection_systoles}, this maximal ratio can be achieved by choosing $\alpha$ a horizontal side of $P(0)$ (hence of minimal possible length, $l_0$) and $\beta$ a suitable side of either $P(m-1)$ or $P(0)$ in direction $d'= \cot(\pi/n)$  (resp. $d'=-\cot(\pi/n)$).

\item[b.]\label{case1.h1} If $\beta$ is contained in a cylinder of height $h_1$, and $n \neq 3$, then up to an horizontal twist, $\beta$ is a side of $P(0)$ (or $P(m-1)$ if $n$ is odd) and has direction either $\pm \cot(\pi/n)$ or $\pm \cot(2\pi/n)$, and we have
\begin{equation}\label{eq:Kforcot2pin}
    \frac{\Int(\alpha, \beta)}{\alpha \wedge \beta} \leq \frac{1}{l_0h_1} = \frac{1}{2 \cos(\pi/n)\sin(\pi/n)l_0^2}.
\end{equation} 
with equality if $\Int(\alpha,\beta)=1$ and $\alpha$ is a systole. For $d'=\pm \cot(\pi/n)$, we have already found a pair $(\alpha,\beta)$ in directions $(\infty,d')$ achieving a better ratio. Hence the only remaining possibility is $d'= \pm \cot(2\pi/n)$. 
In this case, the equality is achieved when there exists such a saddle connection $\beta$ intersecting singularly $\alpha$.

\item[c.]\label{case1.h2} If $\beta$ is contained in a cylinder of height $h_2$, then
\[ \frac{\Int(\alpha, \beta)}{\alpha \wedge \beta} \leq \frac{1}{l_0h_2} = \frac{1}{2 \cos(\pi/m)\sin(\pi/n)l_0^2}\]
with equality if $\Int(\alpha,\beta)=1$ and $\alpha$ is a systole. Outside the case $d'=\pm \cot(\pi/n)$, the only such $\beta$ are saddle connections in direction $d' = \pm \cfrac{1+2 \cos(\pi/n)\cos(\pi/m)}{2\sin(\pi/n) \cos(\pi/m)}$.
\end{itemize} 

In all other cases, the vertical length of $\beta$ is greater than $h_2$\footnote{Notice that crossing at least two horizontal cylinders gives a vertical length at least $2 \sin(\pi/n) l_0$, which is greater than $h_2 = 2 \cos(\pi/m)\sin(\pi/n)l_0$.} so that we have:
\[ \frac{\Int(\alpha, \beta)}{\alpha \wedge \beta} \leq \frac{1}{l_0h_2} = \frac{1}{2 \cos(\pi/m)\sin(\pi/n)l_0^2}.\]

\textbf{Case 2.} \underline{If there is at least one non-singular intersection and $\alpha$ has length $l_0$}, we can assume by symmetry that $\alpha$ is a horizontal side of $P(0)$. Then,

\begin{enumerate}[label=(\roman*)]
\item for every non singular intersection with $\alpha$, $\beta$ has to cross both cylinders $C_0$ and $C_1$ adjacent to $\alpha$, see Figure \ref{fig:cylinders_around_alpha}, hence a non-singular intersection with $\alpha$ accounts for a vertical length at least $h_1 + h_2 = (2 \cos(\pi/m) + 2 \cos(\pi/n))\sin(\pi/n) l_0$,
\item moreover, there is no saddle connection of vertical height $h_1 + h_2$ which intersects $\alpha$ only once outside of the singularity. This is because such a saddle connection would have to cross vertically $C_0$ and $C_1$ exactly once, and no other cylinder. Such a saddle connection must have an endpoint on the bottom of $C_1$, and because $C_1$ has at most two bottom saddle connections (Lemma \ref{lem:top_bottom_cyl_Smn}) we can assume up to an horizontal twist that $\beta$ has an endpoint on either $F$ or $G$ (with the notations of Figure \ref{fig:cylinders_around_alpha}) - say on $F$ by symmetry - and that the direction of $\beta$ lies between the directions of $(FA)$ and $(FB)$ (whose respective co-slopes are given by $\cot(\pi/n)$ and $\cot(\pi/n) + 1/(2\cos(\pi/m)\sin(\pi/n))$). In particular, since the points F,A and C are by construction aligned, $\beta$ must then intersect the segment $[BC]$ on its interior. Finally, since $[BC]$ is identified to the segment $[HF]$ of $P(1)$, directly below the piece of $C_1$ contained in $P(1)$, and since the co-slope of $(HG)$ (given by $\cot(\pi/n) + 2\cos(\pi/m)/\sin(\pi/n)$) is greater than the co-slope of $(FB)$, we conclude that $\beta$ has to cross the segment $(FG)$ on its interior, which means that $\beta$ crosses vertically $C_1$ at least twice. Contradiction.

As a conclusion, if $\beta$ intersects $\alpha$ exactly once in its interior, then $\beta$ has to cross vertically $C_0, C_1$, plus another cylinder, and the vertical length of $\beta$ must be at least $h_1 + h_2 + h_0 = (2 \cos(\pi/m) + 2 \cos(\pi/n)+1)\sin(\pi/n)l_0$. 
\end{enumerate}

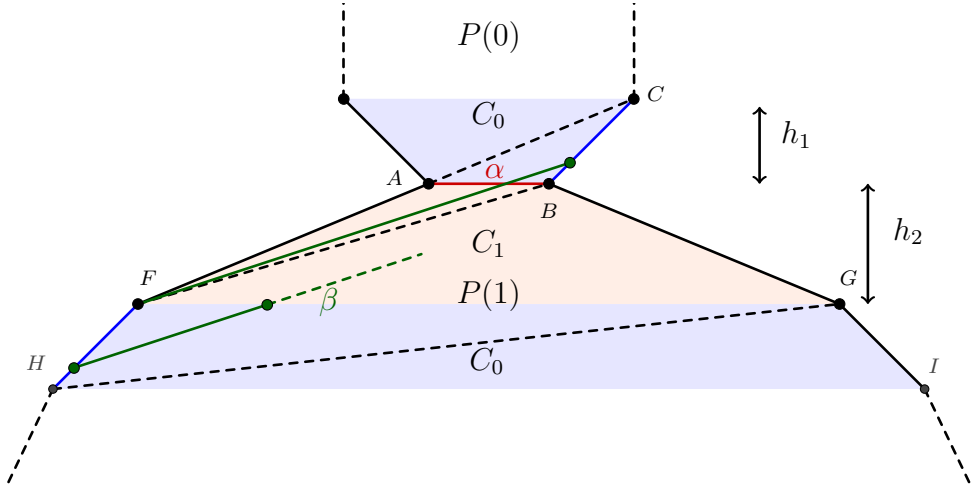
\begin{figure}[h]
\centering
\definecolor{qqwuqq}{rgb}{0,0.39215686274509803,0}
\definecolor{ccqqqq}{rgb}{0.8,0,0}
\definecolor{ffvvqq}{rgb}{1,0.3333333333333333,0}
\definecolor{qqqqff}{rgb}{0,0,1}
\definecolor{uuuuuu}{rgb}{0.26666666666666666,0.26666666666666666,0.26666666666666666}
\begin{tikzpicture}[line cap=round,line join=round,>=triangle 45,x=1cm,y=1cm, scale = 0.8]
\clip(-7,-5) rectangle (9,3);
\fill[line width=1pt,color=qqqqff,fill=qqqqff,fill opacity=0.1] (0,0) -- (-1.414213562373095,1.414213562373095) -- (3.414213562373095,1.414213562373095) -- (2,0) -- cycle;
\fill[line width=1pt,color=ffvvqq,fill=ffvvqq,fill opacity=0.1] (0,0) -- (-4.828427124746191,-2) -- (6.828427124746191,-2) -- (2,0) -- cycle;
\fill[color=qqqqff,fill=qqqqff,fill opacity=0.1] (-4.828427124746191,-2) -- (6.828427124746191,-2) -- (8.242640687119286,-3.414213562373095) -- (-6.242640687119286,-3.414213562373095) -- cycle;
\draw [line width=1pt] (0,0)-- (-1.414213562373095,1.414213562373095);
\draw [line width=1pt,color=qqqqff] (3.414213562373095,1.414213562373095)-- (2,0);
\draw [line width=1pt] (0,0)-- (-4.828427124746191,-2);
\draw [line width=1pt] (6.828427124746191,-2)-- (2,0);
\draw [line width=1pt,color=qqqqff] (-4.828427124746191,-2)-- (-6.242640687119286,-3.414213562373095);
\draw [line width=1pt] (6.828427124746191,-2)-- (8.242640687119286,-3.414213562373095);
\draw [line width=1pt,dash pattern=on 3pt off 3pt] (3.414213562373095,1.414213562373095)-- (3.414213562373095,3.414213562373095);
\draw [line width=1pt,dash pattern=on 3pt off 3pt] (-1.414213562373095,1.414213562373095)-- (-1.414213562373095,3.414213562373095);
\draw [line width=1pt,dash pattern=on 3pt off 3pt] (0,0)-- (3.414213562373095,1.414213562373095);
\draw [line width=1pt,color=ccqqqq] (0,0)-- (2,0);
\draw (1,1.54) node[below] {$C_0$};
\draw (1,-0.6) node[below] {$C_1$};
\draw (1,-1.4) node[below] {$P(1)$};
\draw (1,2.88) node[below] {$P(0)$};
\draw [color=ccqqqq](1.1,0.5) node[below] {$\alpha$};
\draw [line width=1pt, to-to] (5.5,1.28)-- (5.5,0);
\draw [line width=1pt, to-to] (7.3,0)-- (7.3,-2);
\draw (5.68,1.22) node[anchor=north west] {$h_1$};
\draw (7.54,-0.38) node[anchor=north west] {$h_2$};
\draw [line width=1pt,dash pattern=on 3pt off 3pt] (-6.242640687119286,-3.414213562373095)-- (-7,-5);
\draw [line width=1pt,dash pattern=on 3pt off 3pt] (8.242640687119286,-3.414213562373095)-- (9,-5);
\draw [line width=1pt,dash pattern=on 3pt off 3pt] (-4.828427124746191,-2)-- (2,0);
\draw [line width=1pt,dash pattern=on 3pt off 3pt] (-6.242640687119286,-3.414213562373095)-- (6.828427124746191,-2);
\draw (0.5,-2.55) node[anchor=north west] {$C_0$};
\draw [line width=1pt,color=qqwuqq] (-4.82,-2)-- (2.35,0.35);
\draw [line width=1pt,color=qqwuqq] (-5.89,-3.06)-- (-2.68,-2.01);
\draw [line width=1pt,dash pattern=on 3pt off 3pt,color=qqwuqq] (-2.68,-2.01)-- (-0.12,-1.17);
\draw [color=qqwuqq](-2,-1.55) node[anchor=north west] {$\beta$};
\begin{scriptsize}
\draw [fill=black] (0,0) circle (2.5pt);
\draw[color=black] (-0.3,0.1) node[left] {$A$};
\draw[color=black] (2,-0.2) node[below] {$B$};
\draw[color=black] (3.5,1.5) node[right] {$C$};
\draw[color=black] (-4.66,-1.57) node {$F$};
\draw[color=black] (6.98,-1.57) node {$G$};
\draw[color=uuuuuu] (-6.52,-2.97) node {$H$};
\draw[color=uuuuuu] (8.4,-3.03) node {$I$};
\draw [fill=black] (2,0) circle (2.5pt);
\draw [fill=black] (3.41,1.41) circle (2.5pt);
\draw [fill=black] (-1.41,1.41) circle (2.5pt);
\draw [fill=black] (-4.83,-2) circle (2.5pt);
\draw [fill=black] (6.83,-2) circle (2.5pt);
\draw [fill=uuuuuu] (-6.242640687119286,-3.414213562373095) circle (2pt);
\draw [fill=uuuuuu] (8.242640687119286,-3.414213562373095) circle (2pt);
\draw [fill=qqwuqq] (2.35,0.35) circle (2.5pt);
\draw [fill=qqwuqq] (-5.892640687119286,-3.064213562373095) circle (2.5pt);
\draw [fill=qqwuqq] (-2.6815735776267875,-2.013007277740906) circle (2.5pt);
\end{scriptsize}
\end{tikzpicture}
\caption{The horizontal systole $\alpha$ and the two cylinders $C_0$ and $C_1$ adjacent to $\alpha$ (of respective height $h_1$ and $h_2$). Any saddle connection $\beta$ with an endpoint on $F$ and intersecting $\alpha$ on its interior must cross $C_1$ at least twice. }
\label{fig:cylinders_around_alpha}
\end{figure}

From the first remark we deduce that if there are $p \geq 2$ non singular intersections, then $\Int(\alpha,\beta) \leq p+1$ and 

\begin{align*}
\frac{\Int(\alpha, \beta)}{\alpha \wedge \beta} & \leq \frac{p+1}{p(2 \cos(\pi/m) + 2 \cos(\pi/n))\sin(\pi / n) l_0^2}  \\
& \leq \frac{3}{2(2 \cos(\pi/m) + 2 \cos(\pi/n))\sin(\pi / n) l_0^2},
\end{align*}
which is easily shown to be less than $\cfrac{1}{2\cos(\pi/m) \sin(\pi/n) l_0^2}$ for $3 \leq n < m$.\newline

From the second remark we deduce that if there is a single non singular intersection, then $\Int(\alpha,\beta) \leq 2$ and
\[ \frac{\Int(\alpha, \beta)}{\alpha \wedge \beta} \leq \frac{2}{(2 \cos(\pi/m) + 2 \cos(\pi/n)+1)\sin(\pi / n) l_0^2}. \]
which is again less than $\cfrac{1}{2\cos(\pi/m) \sin(\pi/n) l_0^2}$.\newline

\textbf{Case 3.} \underline{If there is at least one non-singular intersection and $\alpha$ is not a systole}, we have by Lemma \ref{lem:lengths}
\begin{itemize}
\item If $n=3$, $l(\alpha) \geq l_2 = 2 \cos(\pi/m) l_0$.
\item If $n \neq 3$, $l(\alpha) \geq l_1 = 2 \cos(\pi/n) l_0$
\end{itemize}

Further, similarly to Case 2, $\beta$ has to cross vertically a horizontal cylinder $C$ before any non singular intersection with $\alpha$, and another horizontal cylinder $C'$ after any non-singular intersection with $\alpha$ (the two cylinders are not the same as no cylinder is glued to itself on $S_{m,n}$). In particular, if there are $p$ non-singular intersections, then
\[  l_v(\beta) \geq p \cdot (h(C) + h(C')). \]

Now, $h(C) + h(C') \geq 2 h_0$, and equality can occur only if the two cylinders of height $h_0$ are adjacent, which is possible if and only if either $n=4$ or by symmetry $m=4$ (and hence $n=3$). Further,
\begin{enumerate}
\item[(a)] If $n=3$, given that $\alpha$ is not a systole we have $l(\alpha) \geq l_2$ and the inequality $h(C)+h(C')\geq 2h_0$ is sufficient to obtain
\[ 
\frac{\Int(\alpha,\beta)}{\alpha \wedge \beta} \leq \frac{p+1}{2p h_0 l_2} \leq \frac{2}{2 h_0l_2} = \frac{1}{2 \cos(\pi/m) \sin(\pi/n) l_0^2}
\]
as required.
\item[(b)] If $n=4$, then $m$ is odd and the two cylinders of height $h_0$ are the two cylinders intersecting $P(m-1)$, so that we have equality above only when $\alpha$ and $\beta$ are the two diagonals of $P(m-1)$ (up to a horizontal twist). As we have seen in Proposition \ref{prop:intersection_diagonals}, these diagonals intersect twice if and only if $m \equiv 3 \mod 4$, giving
\[\frac{\Int(\alpha, \beta)}{\alpha \wedge \beta} = \frac{2}{2 h_0 l(\alpha)} = \frac{1}{2 \cos(\pi/n) \sin(\pi/n) l_0^2}. \]
This ratio appears in part $(ii)$ of Proposition \ref{prop:etude_K}, which is not surprising as the two diagonals of $P(m-1)$ have respective directions $\infty$ and $0 = \cot(2\pi/n)$ (see Remark \ref{rema:cas_particuliers_etude_Kdd}).
\end{enumerate}

In all the other cases, we have $h(C)+h(C') = h_0 + h_1$, and hence:
\begin{align*}
\frac{\Int(\alpha,\beta)}{\alpha \wedge \beta} &\leq \frac{p+1}{p(h_0 + h_1) l_1} \\
& \leq \frac{2}{(h_0+h_1)l_1} =\frac{2}{(2 \cos(\pi/n) +1)\cdot 2\cos(\pi/n)\sin(\pi/n) l_0^2}
\end{align*}
For $n \geq 5$, this last quantity is easily shown to be less than $\cfrac{1}{2 \cos(\pi/m) \sin(\pi/n) l_0^2}$. However, this is not the case for $n=4$, and one has to further notice that if $4 = n < m$ and if we are not in the setting of $(b)$, then the two cylinders of height $h_0$ (which are contained in $P(m-1) \cup P(m-2)$) are not adjacent to the (single) cylinder of height $h_1$ which is contained in $P(0) \cup P(1)$, so that $h(C) + h(C') \geq h_0 + h_2$ and we get:
\begin{align*}
\frac{\Int(\alpha,\beta)}{\alpha \wedge \beta} & \leq \frac{p+1}{p(h_0 + h_1) l_1}\\
&\leq \frac{2}{(h_0+h_2)l_1} =\frac{2}{(2 \cos(\pi/m) +1)\cdot 2\cos(\pi/n)\sin(\pi/n) l_0^2} \\
& < \frac{1}{2 \cos(\pi/m) \sin(\pi/n) l_0^2}
\end{align*}
as required.
   
\textbf{Conclusion.}
According to the above study, we have
\[ 
\frac{\Int(\alpha,\beta)}{\alpha \wedge \beta} \leq \frac{1}{2 \cos(\pi/m)\sin(\pi/n) l_0^2}
\]
unless:
\begin{itemize}
\item $\alpha$ is a systole and $\beta$ is, up to a horizontal twist, a side of $P(m-1)$ (or $P(0)$), having direction $\pm \cot(\pi/n)$ and intersecting $\alpha$ once (see Case 1 part (a) and (b)), and then 

\[ 
\frac{\Int(\alpha,\beta)}{\alpha \wedge \beta} \leq \frac{1}{\sin(\pi/n) l_0^2}.
\]

\item $\alpha$ is a systole and, up to a horizontal twist, $\beta$ has direction either $\pm \cot(\frac{2\pi}{n})$ (see Case 1 part (b) and Case 3 part (b)), and then:
\[ 
\frac{\Int(\alpha,\beta)}{\alpha \wedge \beta} \leq \frac{1}{2 \cos(\pi /n) \sin(\pi/n) l_0^2}.
\]
\end{itemize}
This concludes the proof of Proposition \ref{prop:etude_K}
\end{proof}

Using Proposition \ref{prop:etude_K}, we directly deduce that $(H2)$ is satisfied on each of the domains $\mathcal{D}_i$ (recall from Remark \ref{rk:parameters_Di} that on $\mathcal{D}_1$ and $\mathcal{D}_2$ the parameter $a$ is $-\cot \left( \frac{\pi}{n} \right)$ while on $\mathcal{D}_3$ and $\mathcal{D}_4$ we have $a=+\cot \left( \frac{\pi}{n} \right)$). Further, it is easily deduced from Proposition \ref{prop:etude_K} that \HTU is also satisfied, as the only pairs of consistent slopes $(d,d')$ such that $K(d,d')=K(\infty,a)$ are the images of $(\infty,a)$ by the diagonal action of the Veech group on $\partial \HH$, and among them only $(\infty,a)$ intersects $D_i$. For all other pairs of consistent slopes $(d,d')$, we have 
\[ K(d,d') \leq \frac{1}{2\cos(\pi/n)}K(\infty,a) < \sin\left(\frac{\pi}{4}\right) K(\infty,a).\]

Next, for $m,n \leq 4$ we already know from Remark \ref{rk:modified_big_statement} that \HTD holds. In the other cases, we can see that \HTD holds as well by looking at the images of the geodesics $\gamma_{\infty, \pm \cot(\pi/n)}$, $\gamma_{\infty, \pm \cot(2\pi/n)}$ and $\gamma_{\infty, \pm \frac{1+2 \cos(\pi/n)\cos(\pi/m)}{2\sin(\pi/n) \cos(\pi/m)})}$ under the action the Veech group which intersect the fundamental domain $\Tmn$. Those geodesics are represented in Figure \ref{fig:geodesics_fondamental_domain}.

\begin{figure}
\definecolor{qqqqff}{rgb}{0,0,1}
\definecolor{qqwuqq}{rgb}{0,0.39215686274509803,0}
\definecolor{ccqqqq}{rgb}{0.8,0,0}
\definecolor{uuuuuu}{rgb}{0.26666666666666666,0.26666666666666666,0.26666666666666666}
\begin{tikzpicture}[line cap=round,line join=round,>=triangle 45,x=1cm,y=1cm]
\clip(-6,-0.7) rectangle (9.5,7);
\draw [line width=1pt] (4.921489531140712,0) -- (4.921489531140712,9.093313002462567);
\draw [line width=1pt] (-4.921489531140712,0) -- (-4.921489531140712,9.093313002462567);
\draw [shift={(-2.414213562373095,0)},line width=1pt]  plot[domain=0.3926990816987242:2.855993321445266,variable=\t]({1*2.613125929752753*cos(\t r)+0*2.613125929752753*sin(\t r)},{0*2.613125929752753*cos(\t r)+1*2.613125929752753*sin(\t r)});
\draw [shift={(2.414213562373095,0)},line width=1pt]  plot[domain=0.28559933214452704:2.748893571891069,variable=\t]({1*2.613125929752753*cos(\t r)+0*2.613125929752753*sin(\t r)},{0*2.613125929752753*cos(\t r)+1*2.613125929752753*sin(\t r)});
\draw [line width=1pt,domain=-6.37650148292727:9.658601777524105] plot(\x,{(-0-0*\x)/7.335703093513807});
\draw (5.25,6) node[anchor=north west] {$d_{max} = \cot\left(\frac{\pi}{n}\right)$};
\draw (5.25,3) node[anchor=north west] {$s = 2 \frac{\cos(\frac{\pi}{n}) +\cos(\frac{\pi}{m})}{\sin(\frac{\pi}{n})}$};
\draw (-0.18,0.0) node[anchor=north west] {$0$};
\draw (-0.1,1) node[anchor=north west] {$S_{m,n}$};
\draw (5,1.3) node[anchor=north west] {$S_{n,m}$};
\draw (-6,1.3) node[anchor=north west] {$S_{n,m}$};
\draw [line width=1pt,color=ccqqqq] (-2.414213562373095,0) -- (-2.414213562373095,9.093313002462567);
\draw [line width=1pt,color=ccqqqq] (2.414213562373095,0) -- (2.414213562373095,9.093313002462567);
\draw [shift={(0,0)},line width=1pt,color=qqwuqq]  plot[domain=0:3.141592653589793,variable=\t]({1*2.414213562373095*cos(\t r)+0*2.414213562373095*sin(\t r)},{0*2.414213562373095*cos(\t r)+1*2.414213562373095*sin(\t r)});
\draw [shift={(4.921489531140712,0)},line width=1pt,color=qqqqff]  plot[domain=1.57:3.141592653589793,variable=\t]({1*2.507275968767617*cos(\t r)+0*2.507275968767617*sin(\t r)},{0*2.507275968767617*cos(\t r)+1*2.507275968767617*sin(\t r)});
\draw [shift={(-4.921489531140712,0)},line width=1pt,color=qqqqff]  plot[domain=0:1.57,variable=\t]({1*2.507275968767617*cos(\t r)+0*2.507275968767617*sin(\t r)},{0*2.507275968767617*cos(\t r)+1*2.507275968767617*sin(\t r)});
\draw [line width=1pt,color=qqwuqq] (1,0) -- (1,9.093313002462567);
\draw [line width=1pt,color=qqwuqq] (-1,0) -- (-1,9.093313002462567);
\draw [line width=1pt,color=qqqqff] (3.775935847794774,0) -- (3.775935847794774,9.093313002462567);
\draw [line width=1pt,color=qqqqff] (-3.775935847794774,0) -- (-3.775935847794774,9.093313002462567);
\draw (5.25,5) node[anchor=north west] {$d_1 = \cot\left(\frac{2\pi}{n}\right)$};
\draw (5.25,4) node[anchor=north west] {$d_2 = \frac{1+2\cos(\pi/n) \cos(\pi/n)}{2 \cos(\pi/m) \sin(\pi/n)}$};
\draw (0.7781746355252499,0) node[anchor=north west] {$d_1$};
\draw (2.054511330023061,0) node[anchor=north west] {$d_{max}$};
\draw (3.5465669024641646,0) node[anchor=north west] {$d_2$};
\draw (4.768973877476153,0) node[anchor=north west] {$\frac{s}{2}$};
\draw (-1.3610375707457296,0) node[anchor=north west] {$-d_1$};
\draw (-2.996905728482361,0) node[anchor=north west] {$-d_{max}$};
\draw (-4.183359557170467,0) node[anchor=north west] {$-d_2$};
\draw (-5.387789959020515,0) node[anchor=north west] {$-\frac{s}{2}$};
\begin{scriptsize}
\draw [fill=black] (0,1) circle (2.5pt);
\draw [color=uuuuuu] (2.414213562373095,0)-- ++(-2.5pt,0 pt) -- ++(5pt,0 pt) ++(-2.5pt,-2.5pt) -- ++(0 pt,5pt);
\draw [fill=black] (4.921489531140712,0.7362026495378823) circle (2.5pt);
\draw [color=uuuuuu] (-2.414213562373095,0)-- ++(-2.5pt,0 pt) -- ++(5pt,0 pt) ++(-2.5pt,-2.5pt) -- ++(0 pt,5pt);
\draw [fill=black] (-4.921489531140712,0.7362026495378823) circle (2.5pt);
\draw [color=black] (0,0)-- ++(-2.5pt,0 pt) -- ++(5pt,0 pt) ++(-2.5pt,-2.5pt) -- ++(0 pt,5pt);
\draw [color=uuuuuu] (1,0)-- ++(-2.5pt,0 pt) -- ++(5pt,0 pt) ++(-2.5pt,-2.5pt) -- ++(0 pt,5pt);
\draw [color=uuuuuu] (-1,0)-- ++(-2.5pt,0 pt) -- ++(5pt,0 pt) ++(-2.5pt,-2.5pt) -- ++(0 pt,5pt);
\draw [color=uuuuuu] (3.775935847794774,0)-- ++(-2.5pt,0 pt) -- ++(5pt,0 pt) ++(-2.5pt,-2.5pt) -- ++(0 pt,5pt);
\draw [color=uuuuuu] (-3.775935847794774,0)-- ++(-2.5pt,0 pt) -- ++(5pt,0 pt) ++(-2.5pt,-2.5pt) -- ++(0 pt,5pt);
\end{scriptsize}
\end{tikzpicture}
\caption{The geodesics $\gamma_{\infty,\pm \cot(\pi/n)}$, $\gamma_{\infty,\pm \cot(2\pi/n)}$ and $\gamma_{\infty,\pm \frac{1+2 \cos(\pi/n)\cos(\pi/m)}{2\sin(\pi/n) \cos(\pi/m)}}$ and their images by the Veech group intersecting the fundamental domain.}
\label{fig:geodesics_fondamental_domain}
\end{figure}
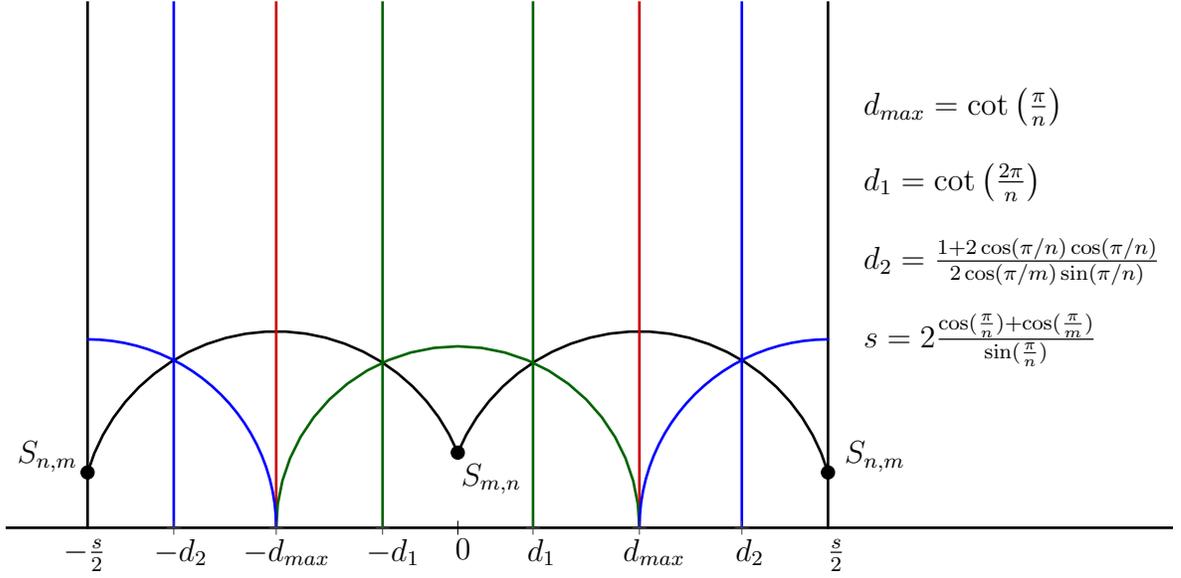

Now,
\begin{enumerate}
\item On the domain $\mathcal{D}_1$, the corresponding parameters $a,b$ and $c$ of Theorem \ref{theo:big_statement} are $a = - \cot(\pi/n)$, $b=-\frac{\cos(\pi/m)}{\sin(\pi/n)}$, $c=\sin(\pi/n)$ and $X_0 = S_{n,m}$ so that 
\begin{itemize}
    \item $\theta(X_0, \infty, a) = \theta(X_0, \infty, - \cot(\pi/n)) = \frac{\pi}{n}$ 
    \item $\theta(X_0, a, a+2b) = \theta(X_0, -\cot(\pi/n), -\cot(\pi/n) -2\sin(\pi/n)) = \frac{2 \pi}{n}$
\end{itemize}
And hence 
\[ 
\frac{\sin \theta(X_0, \infty, a)}{\sin \theta(X_0, a, a+2b)} = \frac{\sin( \pi /n)}{\sin( 2\pi/n)} = \frac{1}{2 \cos(\pi/n)}.
\]
But then by Proposition \ref{prop:etude_K}, for any $(d,d')$ which is not the image of the pair $(\infty, \pm \cot(\pi/n))$,
\[ 
K(d,d') \leq \frac{1}{2 \cos(\pi/n)} K(\infty, \pm \cot(\pi/n))
\]
hence we directly have that \HTD is satisfied.
\item On the domain $\mathcal{D}_2$, the parameters $a,b$ and $c$ are given by $a = - \cot(\pi/n)$, $b=\frac{\cos(\pi/m)}{\sin(\pi/n)}$, $c=\sin(\pi/n)$ and $X_0 = S_{m,n}$ so that 
\begin{itemize}
    \item $\theta(X_0, \infty, a) = \theta(X_0, \infty, - \cot(\pi/n)) = \frac{\pi}{m}$ 
    \item $\theta(X_0, a, a+2b) = \theta(X_0, -\cot(\pi/n), -\cot(\pi/n) -2\sin(\pi/n)) = \frac{2 \pi}{m}$
\end{itemize}
And hence 
\[ 
\frac{\sin \theta(X_0, \infty, a)}{\sin \theta(X_0, a, a+2b)} = \frac{\sin( \pi /m)}{\sin( 2\pi/m)} = \frac{1}{2 \cos(\pi/m)}.
\]
But since by Proposition \ref{prop:etude_K}, for any $(d,d')$ which is not the image of the pairs $(\infty, \pm \cot(\pi/n))$, and $(\infty, \pm \cot(2\pi/n))$
\[ 
K(d,d') \leq \frac{1}{2 \cos(\pi/m)} K(\infty, \pm \cot(\pi/n))
\]
we directly have that \HTD is satisfied.
\end{enumerate}
By symmetry, we also deduce that \HTD holds in the domains $\mathcal{D}_3$ and $\mathcal{D}_4$.

\section{Proof of Theorem \ref{theo:big_statement}}\label{sec:sinus_comparison}
This last section is devoted to prove Theorem \ref{theo:big_statement}. Given $X \in \mathcal{D}$ and $(d,d')$ an admissible pair of directions. Recall that we want to show
\begin{equation}\label{eq:trefle}
\tag{\ding{168}}
K(d,d') \sin \theta(X,d,d') \leq K(\infty, a) \sin \theta(X,\infty, a) 
\end{equation}
By convenience we will assume that $b > 0$, which is possible by symmetry. 
We start with the following proposition, which is a direct generalisation of Proposition 7.8 of \cite{BLM22} and is proven in the same way (see also Proposition 5.8 of \cite{Bou23}).

\begin{Prop}\label{prop:sinus_comparison}
Let $a \in \RR$, and $0 < b < c$. Let $\mathcal{D} \subset \HH$ be the domain delimited by the geodesics $\gamma_{a, \infty}$, $\gamma_{a + b, \infty}$ and $\gamma_{a-c, a+c}$.
Let also $X_0$ be the intersection point of $\gamma_{a + b, \infty}$ and $\gamma_{a-c, a+c}$. Then for any $(d,d')$ with $a \leq d \leq a + b \leq d'$, and such that $\gamma_{d,d'}$ intersect the domain $\mathcal{D}$, the function
\[
F_{(d,d')} : X \in \mathcal{D} \mapsto \frac{\sin \theta(X,\infty,a)}{\sin \theta(X,d,d')}
\]
is minimal at $X_0$ on the domain $\mathcal{D}$.
\end{Prop}

We now show Equation~\eqref{eq:trefle} by distinguishing the case where $\gamma_{d,d'}$ intersects the domain $\mathcal{D}$ and the case where it does not.

\paragraph{\underline{Case 1: $\gamma_{d,d'} \cap \mathcal{D} \neq \varnothing$.}}
In this case, let us first remark that if $X \in \mathcal{R}:=\{ X \in \mathcal{D}, \theta(X,a,\infty) \geq \frac{\pi}{4}\}$, then we can directly deduce \eqref{eq:trefle} from \HTU. In the following we will then assume that $X \notin \mathcal{R}$, see Figure \ref{fig:domain_R}.\newline

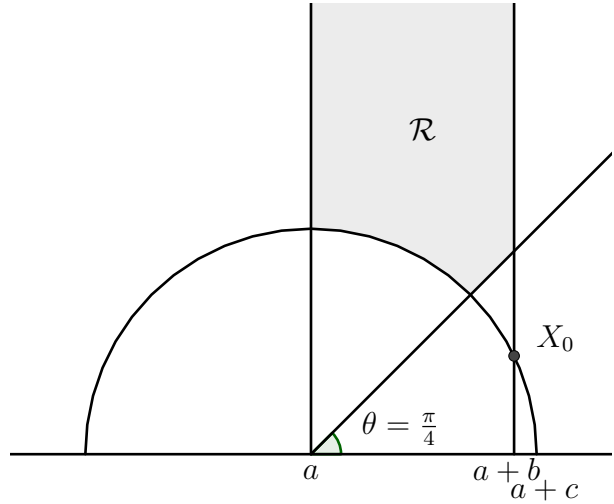
\begin{figure}[h]
\center
\definecolor{qqwuqq}{rgb}{0,0.39215686274509803,0}
\definecolor{uuuuuu}{rgb}{0.26666666666666666,0.26666666666666666,0.26666666666666666}
\begin{tikzpicture}[line cap=round,line join=round,>=triangle 45,x=1cm,y=1cm]
\clip(-4,-0.7) rectangle (4,6);
\fill [color = uuuuuu, opacity=0.1] (0,0) rectangle (2.7,6);
\fill [color = white] (0,0) -- (0:3) arc (0:90:3) -- cycle;
\fill [color = white] (0,0) -- (3,3) -- (3,0) -- cycle; 
\draw [shift={(0,0)},line width=1pt,color=qqwuqq,fill=qqwuqq,fill opacity=0.10000000149011612] (0,0) -- (0:0.4) arc (0:45:0.4) -- cycle;
\draw [line width=1pt] (-4,0)-- (4,0);
\draw [shift={(0,0)},line width=1pt]  plot[domain=0:3.141592653589793,variable=\t]({1*3*cos(\t r)+0*3*sin(\t r)},{0*3*cos(\t r)+1*3*sin(\t r)});
\draw [line width=1pt] (2.7,0) -- (2.7,6.357777777777776);
\draw [line width=1pt] (0,0) -- (0,6.357777777777776);
\draw [line width=1pt,domain=0:6.066666666666663] plot(\x,{(-0--1*\x)/1});
\draw (2.8533333333333304,1.864444444444444) node[anchor=north west] {$X_0$};
\draw (1.16,4.61111111111111) node[anchor=north west] {$\mathcal{R}$};
\draw (0.533333333333331,0.7) node[anchor=north west] {$\theta = \frac{\pi}{4}$};
\draw (0,0) node[below] {$a$};
\draw (2.6,0.1) node[below] {$a+b$};
\draw (3.1,-0.2) node[below] {$a+c$};
\begin{scriptsize}
\draw [fill=uuuuuu] (2.7,1.3076696830622017) circle (2pt);
\end{scriptsize}
\end{tikzpicture}
\caption{The domain $\mathcal{R}$.}
\label{fig:domain_R}
\end{figure}

Then, notice that we are in one of the four following cases:
\begin{enumerate}[label=(\roman*)]
\item $a \leq d < a+b < d'$
\item $d < a < a+2b < d'$
\item $d < a < a+b < d' \leq a+2b$
\item $d < a < d' \leq a+b$
\end{enumerate}

\textbf{Case 1.i:} If $a \leq d < a+b < d'$, we deduce from $(H1)$ and Proposition \ref{prop:sinus_comparison} that
\begin{align*}
K(d,d') \sin \theta(X,d,d') & \leq K(d,d') \frac{\sin \theta(X,d,d')}{\sin \theta(X, \infty, a)} \cdot \sin \theta(X, \infty, a) & \\
& \leq  K(d,d') \frac{\sin \theta(X_0,d,d')}{\sin \theta(X_0, \infty, a)} \cdot \sin \theta(X, \infty, a) & \text{by Proposition \ref{prop:sinus_comparison}}\\
& \leq K(d,d') \sin \theta(X_0,d,d') \cdot \frac{\sin \theta(X, \infty, a)}{\sin \theta(X_0, \infty, a)} & \\
& \leq K(a,\infty) \sin \theta(X_0,a,\infty) \cdot \frac{\sin \theta(X, \infty, a)}{\sin \theta(X_0, \infty, a)} & \text{by }(H1)\\
& \leq K(a,\infty) \sin \theta(X,a,\infty). &
\end{align*}

\textbf{Case 1.ii:} If $d < a < a+2b<d'$, then for all $X \notin \mathcal{R}$, we have 
\[ \sin \theta(X,d,d') \leq \sin \theta(X,a,a+2b) \]
so that by \HTD, we have
\begin{align*}
K(d,d') \sin \theta(X,d,d') &\leq K(d,d') \sin \theta(X,a,a+2b) &\\
& \leq K(\infty,a)  \frac{\sin \theta(X_0,\infty,a)}{\sin \theta(X_0,a,a+2b)} \cdot \sin \theta(X,a,a+2b)  &\text{ by \HTD}\\
& \leq K(\infty,a)  \sin \theta(X_0,\infty,a) \cdot \frac{\sin \theta(X,a,a+2b)}{\sin \theta(X_0,a,a+2b)} &\\
& \leq K(\infty,a)  \sin \theta(X,\infty,a) &\text{ by Proposition \ref{prop:sinus_comparison}} \\
\end{align*}

\textbf{Case 1.iii:} If $d < a < a+b < d' \leq a+2b$, let $g_b$ be the reflection along the geodesic $\gamma_{a+b,\infty}$ and $(e,e') = (g_b \cdot d, g_b \cdot d')$. 
The assumptions on $d$ and $d'$ imply that $0 \leq e' < a+b$ and $e \geq a+2b > a+b$, and since $\gamma_{d,d'} \cap \mathcal{D} \neq \varnothing$ we still have $\gamma_{e,e'} \cap \mathcal{D} \neq \varnothing$, and hence the pair $(e',e)$ corresponds to a pair of direction in Case 1.$i$. In particular for any $X \in \mathcal{D}$
\begin{equation}\label{eq:trefle_casiii} 
K(e',e) \sin \theta(X,e',e) \leq K(\infty, a) \sin \theta(X,\infty, a)
\end{equation}

Further,
\begin{itemize}
\item For any $X$ inside the hyperbolic triangle delimited by the geodesics $\gamma_{a+c,a-c}$, $\gamma_{\infty, a+b}$ and $\gamma_{e,e'}$, we have
\[ \sin \theta(X,d,d') \leq \sin \theta(X,e,e').\]
By \eqref{eq:trefle_casiii} and since $K(d,d')=K(e',e)$, we directly deduce \eqref{eq:trefle}.
\item If $X$ lies outside this domain, we can find $X^* \in \gamma_{e,e'} \cap \mathcal{D}$ such that:
\[ \theta(X^*,\infty,a) \leq \theta(X,\infty,a). \]
This is going to be either $X_1^*=\gamma_{e,e'} \cap \gamma_{a+b, \infty}$ or $X_2^* = \gamma_{e,e'} \cap \gamma_{a-c,a+c}$ (see Figure \ref{fig:argument_sinus_cas1iii}).
Using \eqref{eq:trefle_casiii} and the fact that $X^* \in \gamma_{e,e'}$, we deduce:
\begin{align*}
\underbrace{K(d,d')}_{=K(e,e')} \sin \theta(X,d,d') &\leq K(e,e') \underbrace{\sin \theta(X^*,e,e')}_{=1} \\
& \leq K(\infty, a) \sin \theta(X^*, \infty, a) \\
& \leq K(\infty, a) \sin \theta(X, \infty, a)
\end{align*}
as required.
\end{itemize}

\begin{figure}[h]
\center
\definecolor{ududff}{rgb}{0.30196078431372547,0.30196078431372547,1}
\definecolor{uuuuuu}{rgb}{0.26666666666666666,0.26666666666666666,0.26666666666666666}
\begin{tikzpicture}[line cap=round,line join=round,>=triangle 45,x=1cm,y=1cm]
\clip(-5,-1) rectangle (7,6);
\draw [line width=1pt] (-10,0)-- (10,0);
\draw [shift={(0,0)},line width=1pt]  plot[domain=0:3.141592653589793,variable=\t]({1*3*cos(\t r)+0*3*sin(\t r)},{0*3*cos(\t r)+1*3*sin(\t r)});
\draw [line width=1pt] (2.7,0) -- (2.7,6.357777777777776);
\draw [line width=1pt] (0,0) -- (0,6.357777777777776);
\draw (2.853333333333332,1.864444444444444) node[anchor=north west] {$X_0$};
\draw [shift={(-.75,0)},line width=1pt]  plot[domain=0:3.141592653589793,variable=\t]({1*5*cos(\t r)+0*5*sin(\t r)},{0*5*cos(\t r)+1*5*sin(\t r)});
\draw [shift={(6.18,0)},line width=1pt]  plot[domain=0:3.141592653589793,variable=\t]({1*5*cos(\t r)+0*5*sin(\t r)},{0*5*cos(\t r)+1*5*sin(\t r)});
\draw (0.88,4.557777777777777) node[anchor=north west] {$X$};
\draw (2.7866666666666653,3.9) node[anchor=north west] {$X^*_1$};
\draw (1.75,2.55) node[anchor=north east] {$X^*_2$};
\draw (1.25,0) node[below] {$e'$};
\draw (-5.2,0.07777777777777752) node[anchor=north west] {$\longleftarrow d$};
\draw (-3,0.03) node[below] {$a-c$};
\draw (6.5,0.03) node[below] {$e \longrightarrow$};
\draw (4.2,0) node[below] {$d'$};
\draw (-0.2,0.04) node[anchor=north west] {$a$};
\draw (2.2,0.1577777777777775) node[anchor=north west] {$a+b$};
\draw (2.6,-0.21555555555555578) node[anchor=north west] {$a+c$};
\begin{scriptsize}
\draw [fill=uuuuuu] (2.7,1.3076696830622017) circle (2pt);
\draw [fill=ududff] (-10,0) circle (2pt);
\draw [fill=black] (0.8609237502980095,3.906252717909174) circle (2.5pt);
\draw [fill=uuuuuu] (2.7,3.6) circle (2.5pt);
\draw [fill=uuuuuu] (1.8,2.4) circle (2.5pt);
\end{scriptsize}
\end{tikzpicture}
\caption{Illustration of case 1.$iii$. The geodesic $\gamma_{e,e'}$ is obtained from $\gamma_{d,d'}$ by a reflection along $\gamma_{a+b,\infty}$.For any $X \in \mathcal{D}$ outside the hyperbolic triangle delimited by the geodesics $\gamma_{a+c,a-c}$, $\gamma_{\infty, a+b}$ and $\gamma_{e,e'}$, we have $\min(\theta(X^{\star}_1,a, \infty),\theta(X^{\star}_2,a;\infty)) \leq \theta(X,a, \infty)$.}
\label{fig:argument_sinus_cas1iii}
\end{figure}
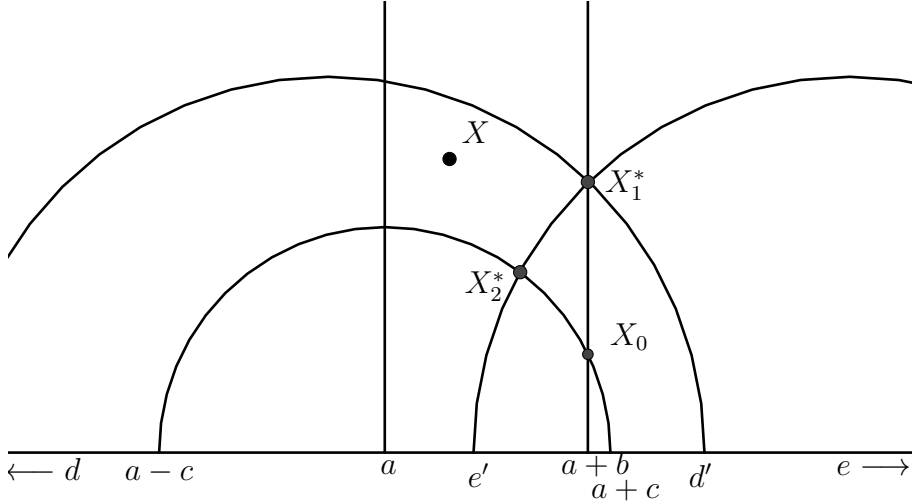

\textbf{Case 1.iv:} Finally, if $d < a < d' \leq a+b$, let $g_c$ be the reflection along the geodesic $\gamma_{a+c,a-c}$ and $(e,e') = (g_c \cdot d, g_c \cdot d')$. 
The assumptions on $d$ and $d'$ imply that $e < a$ and $e' > a+c > a+b$, and since $\gamma_{d,d'} \cap \mathcal{D} \neq \varnothing$ we still have $\gamma_{e,e'} \cap \mathcal{D} \neq \varnothing$, in particular $(e,e')$ corresponds to a pair of direction in Case 1.ii or Case 1.iii, and hence for any $X \in \mathcal{D}$
\begin{equation}\label{eq:trefle_ee'} 
K(e,e') \sin \theta(X,e,e') \leq K(\infty, a) \sin \theta(X,\infty, a)
\end{equation}

Further,
\begin{itemize}
\item For any $X$ inside the hyperbolic triangle delimited by the geodesics $\gamma_{a+c,a-c}$, $\gamma_{\infty, a+b}$ and $\gamma_{e,e'}$, we have
\[ \sin \theta(X,d,d') \leq \sin \theta(X,e,e').\]
By \eqref{eq:trefle_ee'} and since $K(d,d')=K(e,e')$, we directly deduce \eqref{eq:trefle}.
\item If $X$ lies outside this domain, we let $X^* = X_1^*=\gamma_{e,e'} \cap \gamma_{a+b, \infty}$, and we obtain:
\[ \theta(X^*,\infty,a) \leq \theta(X,\infty,a). \]
(see Figure \ref{fig:argument_sinus_cas1iv}).
Note that this case is similar to the previous one, but now it is always $X_1^*$ for which the angle is minimal.
Using \eqref{eq:trefle_ee'} and the fact that $X^* \in \gamma_{e,e'}$, we deduce:
\begin{align*}
\underbrace{K(d,d')}_{=K(e,e')} \sin \theta(X,d,d') &\leq K(e,e') \underbrace{\sin \theta(X^*,e,e')}_{=1} \\
& \leq K(\infty, a) \sin \theta(X^*, \infty, a) \\
& \leq K(\infty, a) \sin \theta(X, \infty, a)
\end{align*}
as required.
\end{itemize}

\begin{figure}[h]
\center
\definecolor{ududff}{rgb}{0.30196078431372547,0.30196078431372547,1}
\definecolor{uuuuuu}{rgb}{0.26666666666666666,0.26666666666666666,0.26666666666666666}
\begin{tikzpicture}[line cap=round,line join=round,>=triangle 45,x=1cm,y=1cm]
\clip(-5,-1) rectangle (7,6);
\draw [line width=1pt] (-10,0)-- (10,0);
\draw [shift={(0,0)},line width=1pt]  plot[domain=0:3.141592653589793,variable=\t]({1*3*cos(\t r)+0*3*sin(\t r)},{0*3*cos(\t r)+1*3*sin(\t r)});
\draw [line width=1pt] (2.7,0) -- (2.7,6.357777777777776);
\draw [line width=1pt] (0,0) -- (0,6.357777777777776);
\draw [line width=0.5pt, dash pattern=on 3pt off 3pt] (0,0)--(2.7,3.4467375879228173);
\draw (2.853333333333332,1.864444444444444) node[anchor=north west] {$X_0$};
\draw [shift={(-4.25,0)},line width=1pt]  plot[domain=0:3.141592653589793,variable=\t]({1*5.75*cos(\t r)+0*5.75*sin(\t r)},{0*5.75*cos(\t r)+1*5.75*sin(\t r)});
\draw [shift={(2.55,0)},line width=1pt]  plot[domain=0:3.141592653589793,variable=\t]({1*3.45*cos(\t r)+0*3.45*sin(\t r)},{0*3.45*cos(\t r)+1*3.45*sin(\t r)});
\draw (0.88,4.557777777777777) node[anchor=north west] {$X$};
\draw (2.7866666666666653,4.104444444444443) node[anchor=north west] {$X^*_1$};
\draw (0.2,2.7311111111111104) node[anchor=north west] {$X^*_2$};
\draw (-1.15,0) node[anchor=north west] {$e$};
\draw (-5.2,0.07777777777777752) node[anchor=north west] {$\longleftarrow d$};
\draw (-3,0.03) node[below] {$a-c$};
\draw (5.65,0.03) node[anchor=north west] {$e'$};
\draw (1.27,0.05111111111111085) node[anchor=north west] {$d'$};
\draw (-0.2,0.04) node[anchor=north west] {$a$};
\draw (2.2,0.1577777777777775) node[anchor=north west] {$a+b$};
\draw (2.6,-0.21555555555555578) node[anchor=north west] {$a+c$};
\begin{scriptsize}
\draw [fill=uuuuuu] (2.7,1.3076696830622017) circle (2pt);
\draw [color=black] (1.5,0)-- ++(-2.5pt,-2.5pt) -- ++(5pt,5pt) ++(-5pt,0) -- ++(5pt,-5pt);
\draw [fill=ududff] (-10,0) circle (2pt);
\draw [color=black] (-0.9,0)-- ++(-2.5pt,-2.5pt) -- ++(5pt,5pt) ++(-5pt,0) -- ++(5pt,-5pt);
\draw [fill=uuuuuu] (6,0) circle (2pt);
\draw [fill=black] (0.8609237502980095,3.906252717909174) circle (2.5pt);
\draw [fill=uuuuuu] (2.7,3.4467375879228173) circle (2.5pt);
\draw [fill=uuuuuu] (0.7058823529411765,2.9157726426808774) circle (2.5pt);
\end{scriptsize}
\end{tikzpicture}
\caption{Illustration of case 1.iv. The geodesic $\gamma_{e,e'}$ is obtained from $\gamma_{d,d'}$ by a reflection along $\gamma_{a-c,a+c}$. For any $X \in \mathcal{D}$ outside the hyperbolic triangle delimited by the geodesics $\gamma_{a+c,a-c}$, $\gamma_{\infty, a+b}$ and $\gamma_{e,e'}$, we have $\theta(X^{\star}_1,a, \infty) \leq \theta(X,a, \infty)$.}
\label{fig:argument_sinus_cas1iv}
\end{figure}
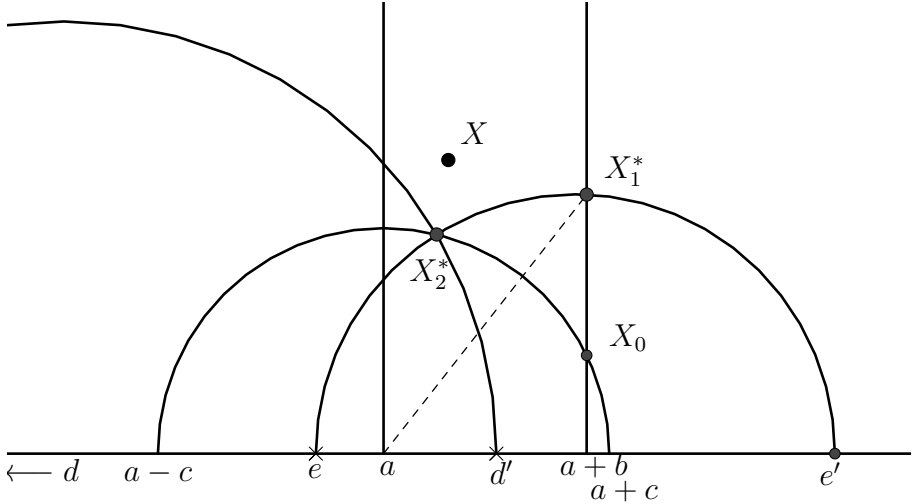

\paragraph{\underline{Case 2: $\gamma_{d,d'} \cap \mathcal{D} = \varnothing$.}} Let us distinguish two cases. \newline

\textbf{Case 2.1:} If $\gamma_{d,d'}$ does not intersect one of images of the domain $\mathcal{D}$ adjacent to $X_0$, then for all $X \in \mathcal{D}$,
\[ \sin \theta(X,d,d') \leq \sin \theta(X,a, \infty) \]
so that \eqref{eq:trefle} holds by $(H2)$.\newline

\textbf{Case 2.2:} Else, for every $X \in \mathcal{D}$, we can always find an element $g$ of the dihedral group generated by the reflections along $\gamma_{\infty,a+b}$ and $\gamma_{a+c,a-c}$ \footnote{which may depend on $X$.} such that:
\begin{itemize}
    \item $\gamma_{g\cdot d, g \cdot d'} \cap \mathcal{D} \neq \varnothing$
    \item $\sin \theta(X,d,d') \leq \sin \theta(X,g\cdot d, g \cdot d').$
\end{itemize}
Hence, we deduce from Case 1 that
\begin{align*}
K(d,d') \sin \theta(X,d,d') & \leq K(g \cdot d,g \cdot d') \sin \theta(X,g \cdot d,g \cdot d')\\
& \leq K(\infty, a) \sin \theta(X,\infty, a)
\end{align*}
as required.

%% file: KVolBM_bibli.bib
@article {CKMcras,
    AUTHOR = {Cheboui, Smail and Kessi, Arezki and Massart, Daniel},
     TITLE = {Algebraic intersection for translation surfaces in the stratum
              {$H(2)$}},
   JOURNAL = {C. R. Math. Acad. Sci. Paris},
  FJOURNAL = {Comptes Rendus Math\'{e}matique. Acad\'{e}mie des Sciences. Paris},
    VOLUME = {359},
      YEAR = {2021},
     PAGES = {65--70},
      ISSN = {1631-073X},
   MRCLASS = {30F60 (53A05)},
  MRNUMBER = {4229038},
MRREVIEWER = {Subhojoy Gupta},
       DOI = {10.5802/crmath.153},
       URL = {https://doi.org/10.5802/crmath.153},
}

@article {CKM,
    AUTHOR = {Cheboui, Smail and Kessi, Arezki and Massart, Daniel},
     TITLE = {Algebraic intersection for translation surfaces in a family of
              {T}eichm\"{u}ller disks},
   JOURNAL = {Bull. Soc. Math. France},
  FJOURNAL = {Bulletin de la Soci\'{e}t\'{e} Math\'{e}matique de France},
    VOLUME = {149},
      YEAR = {2021},
    NUMBER = {4},
     PAGES = {613--640},
      ISSN = {0037-9484},
   MRCLASS = {53C22 (30F60 32G15 37D40)},
  MRNUMBER = {4383613},
       DOI = {10.24033/bsmf.283},
       URL = {https://doi.org/10.24033/bsmf.283},
}

@Article{MM,
 Author = {Daniel {Massart} and Bjoern {Muetzel}},
 Title = {{On the intersection form of surfaces}},
 FJournal = {{Manuscripta Mathematica}},
 Journal = {{Manuscr. Math.}},
 ISSN = {0025-2611; 1432-1785/e},
 Volume = {143},
 Number = {1-2},
 Pages = {19--49},
 Year = {2014},
 Publisher = {Springer, Berlin/Heidelberg},
 MSC2010 = {53C20 30F30 32G15 53C22},
 Zbl = {1296.53071}
}

@article {survey_Wright,
    AUTHOR = {Wright, Alex},
     TITLE = {From rational billiards to dynamics on moduli spaces},
   JOURNAL = {Bull. Amer. Math. Soc. (N.S.)},
  FJOURNAL = {American Mathematical Society. Bulletin. New Series},
    VOLUME = {53},
      YEAR = {2016},
    NUMBER = {1},
     PAGES = {41--56},
      ISSN = {0273-0979},
   MRCLASS = {37D50 (14D20 22E40 32G15 57M50)},
  MRNUMBER = {3403080},
MRREVIEWER = {Jon Chaika},
       DOI = {10.1090/bull/1513},
       URL = {https://doi.org/10.1090/bull/1513},
}

@unpublished{BLM22,
AUTHOR = {Boulanger, Julien AND Lanneau, Erwan AND Massart, Daniel},
TITLE = {Algebraic intersection for a family of {V}eech surfaces},
YEAR = {2022},
URL = {https://hal.archives-ouvertes.fr/hal-03909001},
}

@article{Hooper,
  title={Grid graphs and lattice surfaces},
  author={W. Patrick Hooper},
  journal={International mathematics research notices},
  year={2013},
  volume={12},
  pages={2657-2698}
}

@article{Bou23,
Author = {Boulanger, Julien},
Title = {Algebraic intersection, lengths, and {V}eech groups},
journal = {arXiv:2309.17165},
Year = {2023}}

@incollection{Masur,
    AUTHOR = {Masur, Howard},
     TITLE = {Ergodic theory of translation surfaces},
 BOOKTITLE = {Handbook of dynamical systems. {V}ol. 1{B}},
     PAGES = {527--547},
 PUBLISHER = {Elsevier B. V., Amsterdam},
      YEAR = {2006},
 }

@article{KZ03,
author = {Kontsevich, Maxim and Zorich, Anton},
year = {2002},
month = {01},
pages = {},
title = {Connected components of the moduli space of {A}belian differentials with prescribed singularities},
volume = {153},
journal = {Inventiones Mathematicae},
doi = {10.1007/s00222-003-0303-x}
}

@inproceedings{GraphsOS,
  title={Graphs on {S}urfaces - {D}ualities, {P}olynomials, and {K}nots},
  author={Joanna A. Ellis-Monaghan and Iain Moffatt},
  booktitle={Springer Briefs in Mathematics},
  year={2013}
}

@Article{Veech,
 Author = {W. A. {Veech}},
 Title = {{Teichm\"{u}ller curves in moduli space, {E}isenstein series and an application to triangular billiards}},
 FJournal = {{Inventiones Mathematicae}},
 Journal = {{Invent. Math.}},
 ISSN = {0020-9910; 1432-1297/e},
 Volume = {97},
 Number = {3},
 Pages = {553--583},
 Year = {1989},
 Publisher = {Springer, Berlin/Heidelberg},
 MSC2010 = {32G15 30F30},
 Zbl = {0676.32006}
}

@incollection{HS_survey,
title = {An Introduction to {V}eech Surfaces},
series = {Handbook of Dynamical Systems},
volume = {1},
pages = {501-526},
year = {2006},
author = {Pascal Hubert and Thomas Schmidt}}

@article{BM10,
author = {Bouw, Irene AND Möller, Martin},
title = {Teichmuller curves, triangle groups, and {L}yapunov exponents},
journal = {Annals of Mathematics},
volume = {172},
number = {1},
pages = {85–139},
year = {2010}}

@article{DPU,
author = {Davis, Diana AND Pasquinelli, Irene AND Ulcigrai, Corinna},
title = {Cutting sequences on {B}ouw-{M}\"oller surfaces :
an S-adic characterization},
journal = {Annales scientifiques de l'ENS},
year = {2019}}

@incollection {survey_Massart,
    AUTHOR = {Massart, Daniel},
     TITLE = {A short introduction to translation surfaces, {V}eech
              surfaces, and {T}eichm\"{u}ller dynamics},
 BOOKTITLE = {Surveys in geometry {I}},
     PAGES = {343--388},
 PUBLISHER = {Springer, Cham},
      YEAR = {[2022] \copyright 2022},
   MRCLASS = {37D40 (30F60 32G15)},
  MRNUMBER = {4404400},
       DOI = {10.1007/978-3-030-86695-2\_9},
       URL = {https://doi.org/10.1007/978-3-030-86695-2_9},
}

@article{Hubert_Lanneau_Parabolic,
    title = {Veech groups without parabolic elements},
    author = {Hubert, Pascal AND Lanneau, Erwan},
    journal = {Duke Mathematical journal},
    year = {2006},
volume = {133}
}

@phdthesis{these_massart,
  author  = {Massart, Daniel},
  title   = {Normes stables des surfaces},
  school  = {Ecole Normale Supérieure de Lyon},
  year    = {1996}
}

@article{Bou23b,
author = {Boulanger, Julien},
title ={Lower bound for {KV}ol on the minimal stratum of translation surfaces},
journal = {arXiv:2310.00130},
year = {2023}}

@unpublished{these_boulanger,
author = {Julien Boulanger},
title = {Quelques problèmes géométriques autour des surfaces de translation},
note = {Thèse de Doctorat},
year = {2023}}

@article{Lelievre_Weiss_convex,
author = {Lelièvre, Samuel AND Weiss, Barak},
title = {Translation surfaces with no convex presentation},
journal = {Geom. Funct. Anal.},
Volume = {25},
pages = {1902–1936},
year = {2015},
doi = {https://doi.org/10.1007/s00039-015-0349-0}
}

@article{Jiang_Pan,
author = {Jiang, Manman AND Pan, Huiping},
title = {Algebraic intersections for hyperbolic surfaces},
journal = {arxiv:2404.15921},
year = {2024}
}

@article{Torkaman,
author = {Torkaman, Tina},
title = {Intersection number, length, and systole on compact hyperbolic surfaces},
journal = {arXiv:2306.09249},
year = {2023}
}

@article{BKP21,
  title={The minimal length product over homology bases of manifolds},
  author={Florent Balacheff and Steve Karam and Hugo Parlier},
  journal={Mathematische Annalen},
  year={2021},
  volume={380(1-2)},
  pages={825-854},
}
